\documentclass[a4paper, 11pt, dvipsnames, svgnames]{article}

\usepackage[utf8]{inputenc}
\usepackage[top=1.0in, bottom=1.0in, left=1.1in, right=1.0in]{geometry}
\usepackage[all]{xy}
\usepackage{amsfonts}
\usepackage{amsmath}
\usepackage{amssymb}
\usepackage{amsthm}
\usepackage{calrsfs}
\usepackage{enumerate}
\usepackage{enumitem}
\usepackage{xparse, etoolbox}
\usepackage{ifthen}
\usepackage{mathtools, nccmath, textcomp}
\usepackage[new]{old-arrows}
\usepackage{relsize}
\usepackage{rsfso}
\usepackage{stmaryrd}
\usepackage{sectsty}
\usepackage{setspace}
\usepackage[nottoc]{tocbibind}
\usepackage{tensor}
\usepackage{tikz}
\usepackage{tikz-cd}
\usepackage[sf, bf]{titlesec}	
\usepackage[titles]{tocloft}
\usepackage{xcolor}
\usepackage[backend=biber, style=alphabetic, sorting=nyt, doi=false, url=false, maxbibnames=99, maxalphanames=99]{biblatex}
\usepackage[bbgreekl]{mathbbol}

\usepackage[T1]{fontenc}
\usepackage{lmodern}
\AtBeginBibliography{\small}

\definecolor{bondiblue}{rgb}{0.1, 0.58, 0.71}
\usepackage{hyperref}
\hypersetup{
	colorlinks=true,
	linkcolor=WildStrawberry,
	citecolor=bondiblue,
	urlcolor=black,
}

\usepackage{imakeidx}
\makeindex

\usepackage{fancyhdr}

\newcommand{\addperiod}[1]{#1.}
\titleformat{\section}[block]{\scshape\Large\filcenter}{\thesection.}{1em}{}
\titleformat{\subsection}[runin]{\normalfont\large\bfseries}{\thesubsection.}{1em}{\addperiod}
\titleformat{\subsubsection}[runin]{\normalfont\bfseries}{\thesubsubsection.}{1em}{\addperiod}

\numberwithin{equation}{section}

\DeclareSymbolFontAlphabet{\mathbbl}{bbold}
\newcommand{\Prism}{{\mathlarger{\mathbbl{\Delta}}}}

\makeatletter
\def\keywords{\xdef\@thefnmark{}\@footnotetext}
\makeatother

\theoremstyle{plain}
\newtheorem{thm}{Theorem}[section]
\newtheorem{cor}[thm]{Corollary}
\newtheorem{lem}[thm]{Lemma}
\newtheorem{prop}[thm]{Proposition}
\newtheorem{claim}[thm]{Claim}

\theoremstyle{definition}
\newtheorem{defi}[thm]{Definition}

\newtheorem{rem}[thm]{Remark}

\AtEndEnvironment{const}{\qed}

\theoremstyle{remark}

\newtheorem*{nota}{Notation}

\newenvironment{enumarabicup}
{\begin{enumerate}[font=\upshape, labelindent=\parindent, label=(\arabic*)]}
{\end{enumerate}}

\DeclareMathOperator*{\colim}{\text{\scalebox{0.95}{$\textup{colim}$}}}

\DeclareMathAlphabet{\pazocal}{OMS}{zplm}{m}{n}
\DeclareMathAlphabet{\dutchcal}{U}{dutchcal}{m}{n}

\newcommand{\isomorphic}{\xrightarrow{\hspace{0.5mm} \sim \hspace{0.5mm}}}
\newcommand{\lisomorphic}{\xleftarrow{\hspace{0.5mm} \sim \hspace{0.5mm}}}

\newcommand{\calc}{\mathcal{C}}

\newcommand{\pazn}{\pazocal{N}}
\newcommand{\pazo}{\pazocal{O}}

\newcommand{\pazs}{\pazocal{S}}

\def\CC{\mathbb{C}}

\def\NN{\mathbb{N}}
\def\QQ{\mathbb{Q}}
\def\RR{\mathbb{R}}
\def\ZZ{\mathbb{Z}}

\newcommand{\mbfd}{\mathbf{D}}

\newcommand{\mbfn}{\mathbf{N}}

\newcommand{\mbft}{\mathbf{T}}
\newcommand{\mbfv}{\mathbf{V}}

\newcommand{\smbfe}{\mathbf{e}}
\newcommand{\smbff}{\mathbf{f}}

\newcommand{\smbfi}{\mathbf{i}}
\newcommand{\smbfj}{\mathbf{j}}
\newcommand{\smbfk}{\mathbf{k}}

\newcommand{\frakm}{\mathfrak{m}}
\newcommand{\frakp}{\mathfrak{p}}
\newcommand{\frakq}{\mathfrak{q}}

\newcommand{\ad}{^{\textup{ad}}}
\newcommand{\algebra}{\textrm{-algebra}}

\newcommand{\Ainf}{A_{\inf}}

\newcommand{\Acrys}{A_{\textup{cris}}}
\newcommand{\OAcrys}{\pazo A_{\textup{cris}}}
\newcommand{\Atilde}{\tilde{A}}
\newcommand{\Btilde}{\tilde{B}}
\newcommand{\Dtilde}{\tilde{D}}
\newcommand{\AFpi}{A_{F,\varpi}}

\newcommand{\AF}{A_F}

\newcommand{\Aframe}{A_{\square}}
\newcommand{\AR}{A_R}
\newcommand{\BR}{B_R}
\newcommand{\DR}{\mbfd_R}
\newcommand{\DRn}{D_{R,n}}
\newcommand{\ER}{E_R}
\newcommand{\NR}{\mbfn_R}
\newcommand{\NRn}{N_{R,n}}
\newcommand{\DRbar}{\overline{D}_R}
\newcommand{\DLbar}{\overline{D}_L}
\newcommand{\ALpi}{A_{L,\varpi}}
\newcommand{\OALpi}{\pazo A_{L,\varpi}}
\newcommand{\ARpi}{A_{R,\varpi}}
\newcommand{\OARpi}{\pazo A_{R,\varpi}}

\newcommand{\BdR}{B_{\textup{dR}}}
\newcommand{\OBdR}{\pazo B_{\textup{dR}}}
\newcommand{\Bcrys}{B_{\textup{cris}}}
\newcommand{\OBcrys}{\pazo B_{\textup{cris}}}
\newcommand{\action}{\textrm{-action}}

\newcommand{\Cplusp}{\CC^+(\frakp)}

\newcommand{\CRbar}{\CC(\overline{R})}
\newcommand{\crys}{\textup{cris}}

\newcommand{\Dcrys}{\mbfd_{\textup{cris}}}
\newcommand{\ODcrys}{\pazo\mbfd_{\textup{cris}}}
\newcommand{\ODcrysL}{\pazo\mbfd_{\textup{cris}, L}}
\newcommand{\ODL}{\pazo D_L}

\newcommand{\ODcrysR}{\pazo\mbfd_{\textup{cris}, R}}
\newcommand{\ODR}{\pazo D_R}
\newcommand{\OVcrysR}{\pazo\mbfv_{\textup{cris}, R}}
\newcommand{\DcrysLbreve}{\mbfd_{\textup{cris}, \breve{L}}}
\newcommand{\dlog}{\hspace{0.2mm}d\textup{\hspace{0.3mm}log\hspace{0.3mm}}}

\newcommand{\equivariant}{\textrm{-equivariant}}
\newcommand{\etale}{\textup{\'et}}

\newcommand{\Fil}{\textup{Fil}}

\newcommand{\Fr}{\textup{Frac}}
\newcommand{\GL}{\textup{GL}}
\newcommand{\Gal}{\textup{Gal}}
\newcommand{\GammaLbreve}{\Gamma_{\Lbreve}}
\newcommand{\gr}{\textup{gr}}
\newcommand{\GRp}{G_R(\frakp)}
\newcommand{\GRhatp}{\widehat{G}_R(\frakp)}
\newcommand{\Hom}{\textup{Hom}}
\newcommand{\Id}{\textup{Id}}

\newcommand{\Fbar}{\overline{F}}
\newcommand{\Abar}{\overline{A}}
\newcommand{\Jbar}{\overline{J}}
\newcommand{\Kbar}{\overline{K}}

\newcommand{\kert}{\textup{Ker }}

\newcommand{\Lbreve}{\breve{L}}
\newcommand{\Lbreveinfty}{\breve{L}_{\infty}}
\newcommand{\Lbrevebar}{\overline{\breve{L}}}

\newcommand{\linear}{\textrm{-linear}}

\newcommand{\lattice}{\textrm{-lattice}}
\newcommand{\module}{\textrm{-module}}
\newcommand{\modules}{\textrm{-modules}}
\newcommand{\mubar}{\overline{\mu}}

\newcommand{\MF}{\textup{MF}}

\newcommand{\MLbreve}{M_{\Lbreve}}

\newcommand{\Mcirc}{M^{\circ}}

\newcommand{\Mod}{\textup{-Mod}}

\newcommand{\OFinfty}{O_{F_{\infty}}}
\newcommand{\Linfty}{L_{\infty}}

\newcommand{\Lbar}{\overline{L}}
\newcommand{\Lbarp}{\overline{L}(\frakp)}
\newcommand{\Cp}{\CC_{\frakp}}
\newcommand{\OLbar}{O_{\overline{L}}}
\newcommand{\OLbarp}{O_{\overline{L}(\frakp)}}
\newcommand{\Cpplus}{\CC^+_{\frakp}}
\newcommand{\Cpplusflat}{\CC^{+, \flat}_{\frakp}}
\newcommand{\OLbreve}{O_{\breve{L}}}
\newcommand{\AL}{A_L}
\newcommand{\ALplusp}{A_L^+(\frakp)}
\newcommand{\BL}{B_L}
\newcommand{\DL}{\mbfd_L}
\newcommand{\DLn}{D_{L,n}}
\newcommand{\EL}{E_L}
\newcommand{\NL}{\mbfn_L}
\newcommand{\NLp}{N_L(\frakp)}
\newcommand{\NLbar}{\overline{N}_L}
\newcommand{\NLn}{N_{L,n}}

\newcommand{\ALbreve}{A_{\breve{L}}}

\newcommand{\NLbreve}{N_{\breve{L}}}

\newcommand{\OSmPD}{\pazo S_m^{\PD}}
\newcommand{\SnPD}{S_n^{\PD}}
\newcommand{\OSnPD}{\pazo S_n^{\PD}}

\newcommand{\SnhatPD}{\widehat{S}_n^{\PD}}

\newcommand{\adic}{\textrm{-adic}}
\newcommand{\padic}{p\textrm{-adic}}
\newcommand{\pqheight}{[p]_q\textrm{-height}}

\newcommand{\PD}{\textup{PD}}

\newcommand{\pins}{\frakp \in \pazs}

\newcommand{\qconnection}{q\textrm{-connection}}
\newcommand{\rank}{\textup{rk}}
\newcommand{\Rbar}{\overline{R}}
\newcommand{\Rbarp}{(\Rbar)_{\frakp}}
\newcommand{\Rep}{\textup{Rep}}

\newcommand{\Rframe}{R^{\square}}
\newcommand{\Rinfty}{R_{\infty}}

\newcommand{\regular}{\textrm{-regular}}
\newcommand{\representation}{\textrm{-representation}}

\newcommand{\Spec}{\textup{Spec}\hspace{0.5mm}}
\newcommand{\Spf}{\textup{Spf }}

\newcommand{\submodule}{\textrm{-submodule}}

\newcommand{\TR}{\mbft_R}

\newcommand{\VR}{\mbfv_R}

\newcommand{\textmod}{\textup{ mod }}

\newcommand{\torsion}{\textrm{-torsion}}

\title{\textsc{Crystalline representations and Wach modules in the relative case II}}
\author{\textsc{Abhinandan}}

\newcommand{\Addresses}{{
  \footnotesize

  \rule{2cm}{0.4pt}\vspace{2mm}

  \textsc{Abhinandan}\par\nopagebreak
  \textsc{IMJ-PRG, Sorbonne Universit\'e, 4 Place Jussieu, Paris, France}\par\nopagebreak\vspace{-0.7mm}
  \textit{E-mail}: \footnotesize{\href{abhinandan@imj-prg.fr}{abhinandan@imj-prg.fr}}, \textit{Web}: \footnotesize{\href{https://abhinandan.perso.math.cnrs.fr/}{https://abhinandan.perso.math.cnrs.fr/}}
}}

\date{}

\begin{document}

\fontdimen2\font=0.3em
\pagenumbering{arabic}


\maketitle
{
	\textsc{Abstract.} We study relative Wach modules, generalising our previous works on this subject.
	Our main result shows a categorical equivalence between relative Wach modules and lattices inside relative crystalline representations.
	Using this result, we deduce a purity statement for relative crystalline representations, and provide a criteria for checking the crystallinity of relative $\padic$ representations.
	Furthermore, we interpret relative Wach modules as modules with $q\textrm{-connections}$, and show that for a crystalline representation, its associated Wach module together with the Nygaard filtration is the canonical $q\textrm{-deformation}$ (after inverting $p$) of the filtered $(\varphi,\partial)\module$ associated to the representation.
}


\section{Introduction}

The study of arithmetic Wach modules and their relationship to crystalline representations is of classical nature, having been taken up in the works of Fontaine \cite{fontaine-phigamma}, Wach \cite{wach-free, wach-torsion}, Colmez \cite{colmez-hauteur} and Berger \cite{berger-limites}.
More precisely, in op.\ cit.\ the authors studied the situation of an absolutely unramified extension of $\QQ_p$ with perfect residue field.
In \cite{abhinandan-relative-wach-i} we defined a similar concept in the relative case, i.e.\ for certain étale algebras over a formal torus (see Section \ref{subsec:setup_nota} for precise setup) and showed that such objects give rise to crystalline representations of the fundamental group of the generic fibre.
On the other hand, in \cite{abhinandan-imperfect-wach}, we generalised the theory of Wach modules and their relationship to crystalline representations, to the imperfect residue field case.
In this article, we combine these two generalisations of the classical theory, to discuss the equivalence between Wach modules and crystalline representations in its most natural generality.
In addition, we provide some applications of the preceding result and also show that Wach modules are $q\textrm{-deformations}$ of lattices inside the filtered $(\varphi, \partial)\module$ attached to crystalline representations.

Before providing further motivations for our results, let us remark that recent developments in the theory of prismatic $F\textrm{-crystals}$ \cite{bhatt-scholze-crystals, du-liu-moon-shimizu, guo-reinecke} provide a new approach to the classification of lattices inside crystalline representations.
These exciting new developments have motivated us in seeking the results of the current paper.
However, instead of using the tools from the prismatic theory, we employ techniques from the classical theory of $(\varphi, \Gamma)\modules$ to obtain our results due to the very nature of the objects studied in this article, i.e.\ relative Wach modules.
Additionally, our proof enables us to provide interesting applications as well, for example, using \cite[Theorem 1.5]{abhinandan-imperfect-wach} and Theorem \ref{intro_thm:crystalline_wach_equivalence}, we provide a new criteria for checking the crystallinity of a $\padic$ representation in the relative case (see Theorem \ref{intro_thm:purity_crystalline} and Corollary \ref{intro_cor:purity_crystalline}).
We refer the reader to Section \ref{subsubsec:main_results} for precise statements of these results, to Section \ref{subsubsec:proof_strategy} for a sketch of our proof strategy and to Section \ref{subsec:relate_other_works} for more details on relation of our results to the prismatic theory.

Our motivation for studying relative Wach modules is twofold, largely stemming from geometry.
In \cite{abhinandan-syntomic}, for smooth ($\padic$ formal) schemes, we defined the notion of crystalline syntomic complex with coefficients in global relative Fontaine-Laffaille modules.
Moreover, \cite[Theorem 1.11]{abhinandan-syntomic} showed that such a complex is naturally comparable to the complex of $\padic$ nearby cycles of the associated crystalline $\ZZ_p\textrm{-local system}$ on the (rigid analytic) generic fibre of the (formal) scheme.
The work in loc.\ cit.\ was motivated by the results of \cite{fontaine-messing}, \cite{tsuji-syntomic}, \cite{tsuji-cst} and \cite{colmez-niziol}, and the proof of \cite[Theorem 1.11]{abhinandan-syntomic} follows via careful computations in the local setting in which relative Wach modules play a pivotal role (see \cite[Corollary 1.6]{abhinandan-syntomic}).
To generalise these results beyond the Fontaine-Laffaille case, it is therefore necessary to understand the relationship between crystalline representations of the fundamental group and general relative Wach modules (see Theorem \ref{intro_thm:crystalline_wach_equivalence}).

On the other hand, in \cite{bhatt-morrow-scholze-2}, for smooth $\padic$ formal schemes, the authors defined a prismatic syntomic complex and compared it to the complex of $\padic$ nearby cycles integrally.
In the same vein, comparison results beyond the smooth case, have also been obtained in \cite{antieau-mathew-morrow-nikolaus} and \cite{bhatt-mathew}, where the latter uses the theory of prismatic cohomology from \cite{bhatt-scholze-prisms}.
The aforementioned results were obtained in the case of constant coefficients and it is natural to ask if \cite[Theorem 10.1]{bhatt-morrow-scholze-2} could be generalised to non-constant coefficients, i.e.\ prismatic $F\textrm{-crystals}$.
In our approach to resolving this question, results pertaining to Wach modules from the current paper will play a critical role.

Another motivation for considering Wach modules is to construct a deformation of crystalline cohomology, i.e.\ the functor $\Dcrys$ from classical $\padic$ Hodge theory, to better capture mixed characteristic information.
In \cite[Section B.2.3]{fontaine-phigamma} Fontaine expressed similar expectations which were verified by Berger in \cite[Th\'eor\`eme III.4.4]{berger-limites} and generalised to finer integral conjectures in \cite[Section 6]{scholze-q-deformations}.
Some conjectures of \cite{scholze-q-deformations} were resolved by the introduction of prismatic cohomology \cite{bhatt-scholze-prisms}.
Furthermore, it is also worth mentioning that the proof of the result comparing prismatic syntomic complex to $\padic$ nearby cycles, i.e.\ \cite[Theorem 10.1]{bhatt-morrow-scholze-2}, relies on a local computation of prismatic cohomology using the $q\textrm{-de Rham}$ complex, i.e.\ a $q\textrm{-deformation}$ of the usual de Rham complex.
Additionally, the importance of $q\textrm{-de Rham}$ cohomology in the computation of prismatic cohomology has also been emphasised in \cite[Section 3]{bhatt-lurie}.

In this paper, we interpret Wach modules as $q\textrm{-de Rham}$ complexes (see Theorem \ref{intro_thm:qdeformation_dcrys}).
Moreover, we show that such an object is the $q\textrm{-deformation}$ of a lattice inside the filtered $(\varphi, \partial)\module$ attached to a crystalline representation.
In a subsequent work \cite{abhinandan-prismatic-wach}, we show that in our setting, a relative Wach module may be regarded as the evaluation of a prismatic $F\textrm{-crystal}$ over a covering (by a suitable $q\textrm{-de Rham}$ prism) of the final object of a certain prismatic topos.
Hence, from these apparent tight connections between Wach modules and prismatic $F\textrm{-crystals}$ and $\padic$ crystalline representations, we expect these objects to play a pivotal role in the study of $\padic$ nearby cycles of crystalline $\ZZ_p\textrm{-local systems}$ (for smooth formal schemes) and its comparison to prismatic syntomic complex with coefficients.

In summary, within the overarching program sketched above, this paper realises two of our goals (see Theorem \ref{intro_thm:crystalline_wach_equivalence} and Theorem \ref{intro_thm:qdeformation_dcrys}).
Additionally, we provide interesting applications of our results to purity statements in $\padic$ Hodge theory (see Theorem \ref{intro_thm:purity_crystalline} and Corollary \ref{intro_cor:purity_crystalline}).

\subsection{Crystalline representations and Wach modules}\label{subsec:arithmetic_case}

Let $p$ be a fixed prime number and $\kappa$ a perfect field of characteristic $p$; set $O_F := W(\kappa)$ to be the ring of $p\textrm{-typical}$ Witt vectors with coefficients in $\kappa$ and $F := O_F[1/p]$.
Let $d \in \NN$ and take $X_1, X_2, \ldots, X_d$ to be some indeterminates.
We set $O_F\langle X_1^{\pm 1}, \ldots, X_d^{\pm 1}\rangle$ to be the $\padic$ completion of Laurent polynomial ring $O_F[X_1^{\pm 1}, \ldots, X_d^{\pm 1}]$.
Let $R$ denote the $\padic$ completion of an \'etale algebra over $O_F\langle X_1^{\pm 1}, \ldots, X_d^{\pm 1}\rangle$ with non-empty and geometrically integral special fibre.
Denote by $G_R$ the \'etale fundamental group of $R[1/p]$ and by $\Gamma_R$ the Galois group of $\Rinfty[1/p]$ over $R[1/p]$, where $\Rinfty$ is obtained from $R$ by adjoining to it all $p\textrm{-power}$ roots of unity and all $p\textrm{-power}$ roots of $X_i$, for each $1 \leqslant i \leqslant d$.
Then, we have $\Gamma_R \isomorphic \ZZ_p(1)^d \rtimes \ZZ_p^{\times}$ (see Section \ref{sec:period_rings_padic_reps} for precise definitions).
Set $O_L := (R_{(p)})^{\wedge}$ as a complete discrete valuation ring with uniformiser $p$, residue field a finite \'etale extension of $\kappa(X_1, \ldots, X_d)$ and set $L := O_L[1/p]$.
Let $G_L$ denote the absolute Galois group of $L$ such that we have a continuous homomorphism $G_L \rightarrow G_R$; let $\Gamma_L$ denote the Galois group of $\Linfty$ over $L$, where $\Linfty$ is obtained from $L$ by adjoining to it all $p\textrm{-power}$ roots of unity and all $p\textrm{-power}$ roots of $X_i$, for each $1 \leqslant i \leqslant d$.
The continuous homomorphism $G_L \rightarrow G_R$ induces a continuous isomorphism $\Gamma_L \isomorphic \Gamma_R$.
In this setting, we have the theory of crystalline representations of $G_R$ from \cite{brinon-relatif} and the theory of \'etale $(\varphi, \Gamma)\textrm{-modules}$ from \cite{andreatta-phigamma, andreatta-brinon}.

\subsubsection{Relative Wach modules}

Set $\varepsilon := (1, \zeta_p, \zeta_{p^2}, \ldots)$ in $\Rinfty^{\flat}$ (the tilt of $R_{\infty}$) and its Teichm\"uller lift $[\varepsilon]$ in $\Ainf(\Rinfty) := W(\Rinfty^{\flat})$, the ring of $p\textrm{-typical}$ Witt vectors with coefficients in $\Rinfty^{\flat}$.
Additionally, set $\mu := [\varepsilon]-1$ and $[p]_q := \varphi(\mu)/\mu$, as elements of $\Ainf(\Rinfty)$.
Moreover, for $1 \leqslant i \leqslant d$, fix $X_i^{\flat} := (X_i, X_i^{1/p}, \ldots)$ in $R_{\infty}^{\flat}$ and their Teichm\"uller lifts $[X_i^{\flat}]$ in $\Ainf(\Rinfty)$.
Let $\AR^+$ denote the $(p, \mu)\textrm{-adic}$ completion of the unique extension of the $(p, \mu)\textrm{-adic}$ completion of $O_F\llbracket \mu \rrbracket [X_1^{\flat}]^{\pm 1}, \ldots, [X_d^{\flat}]^{\pm 1}]$ along the $p\textrm{-adically}$ completed \'etale map $O_F\langle X_1^{\pm 1}, \ldots, X_d^{\pm 1}\rangle \rightarrow R$ (see Section \ref{subsec:setup_nota} and Section \ref{subsec:ainf_relative}).
The ring $\AR^+$ is equipped with a Frobenius endomorphism $\varphi$ and a continuous action of $\Gamma_R$; set $\AL^+$ to be the $(p, \mu)\textrm{-adic}$ completion of the localisation $(\AR^+)_{(p, \mu)}$ equipped with an induced Frobenius endomorphism $\varphi$ and a continuous action of $\Gamma_L \isomorphic \Gamma_R$.
With this setup, we define the following:

\begin{defi}\label{intro_defi:wach_mods_relative}
	A \textit{Wach module} over $\AR^+$ with weights in the interval $[a, b]$, for some $a, b \in \ZZ$ with $b \geqslant a$, is a finitely generated $\AR^+\module$ $N$ satisfying the following assumptions:
	\begin{enumarabicup}
		\item The sequences $\{p, \mu\}$ and $\{\mu, p\}$ are regular on $N$.

		\item $N$ is equipped with a semilinear action of $\Gamma_R$ such that the induced action of $\Gamma_R$ on $N/\mu N$ is trivial.

		\item $N$ admits a Frobenius-semilinear operator $\varphi: N[1/\mu] \rightarrow N[1/\varphi(\mu)]$, compatible with the respective actions of $\Gamma_R$, and such that $\varphi(\mu^b N) \subset \mu^b N$ and the cokernel of the $\AR^+\linear$ map $(1 \otimes \varphi) : \varphi^{\ast}(\mu^b N) \rightarrow \mu^b N$ is killed by $[p]_q^{b-a}$.
	\end{enumarabicup}
	Denote by $(\varphi, \Gamma_R)\Mod_{\AR^+}^{[p]_q}$ the category of Wach modules over $\AR^+$, with morphisms between objects being $\AR^+\linear$, $\varphi\equivariant$ (after inverting $\mu$) and $\Gamma_R\equivariant$ morphisms.
\end{defi}

\begin{rem}
	The condition (1) in Definition \ref{intro_defi:wach_mods_relative} is new and relaxes finite projectivity assumption of relative Wach modules in \cite[Definition 4.8]{abhinandan-relative-wach-i}.
	Moreover, condition (1) above is equivalent to the vanishing of local cohomology of $N$ with respect to the ideal $(p, \mu) \subset \AR^+$ in degree 1 (see Lemma \ref{lem:strictreg_koszulcomp} and Remark \ref{rem:strictreg_localcoh}), in particular, it is equivalent to having $\{p, [p]_q\}$ and $\{[p]_q, p\}$ as regular sequences on $N$ (see Lemma \ref{lem:p_pq_regsec}).
	Furthermore, one can also show that $N[1/p]$ is a finite projective $\AR^+[1/p]\module$ (see Proposition \ref{prop:finiteproj_torus}, where we use some ideas from \cite{bhatt-morrow-scholze-1, du-liu-moon-shimizu}) and $N[1/\mu]$ is a finite projective $\AR^+[1/\mu]\module$ (see Proposition \ref{prop:wachmod_proj_pmu}) satisfying $N = N[1/p] \cap N[1/\mu] \subset N[1/p, 1/\mu]$ (see Lemma \ref{lem:wach_saturated}).
\end{rem}

\begin{rem}
	In Definition \ref{intro_defi:wach_mods_relative}, note that in contrast to the definition of Wach modules in the arithmetic case (see \cite[Definition III.4.1]{berger-limites}), we have dropped the assumption on the continuity of the action of $\Gamma_R$ on $N$.
	However, in Lemma \ref{lem:automatic_continuity} we show that the condition (2) in Definition \ref{intro_defi:wach_mods_relative}, i.e.\ triviality of the action of $\Gamma_R$ on $N/\mu N$, automatically implies that the action of $\Gamma_R$ on $N$ is continuous.
\end{rem}

\begin{rem}\label{inro_rem:wach_mods_field}
	Definition \ref{intro_defi:wach_mods_relative} may be adapted to the case when $R[1/p]$ is a field, i.e.\ over the ring $\AF^+ := O_F\llbracket \mu \rrbracket$ (resp.\ $\AL^+$).
	In these cases, from the assumptions of Definition \ref{intro_defi:wach_mods_relative} it follows that a Wach module over $\AF^+$ (resp.\ $\AL^+$) is necessarily finite free.
	Indeed, if $N$ is a Wach module over $\AF^+$ (resp.\ $\AL^+$), in the sense of Definition \ref{intro_defi:wach_mods_relative}, then one first observes that $N$ is torsion-free since $N \subset N[1/p]$ and the latter is finite free over $\AF^+[1/p]$ (resp.\ $\AL^+[1/p]$) by \cite[Lemma 2.14]{abhinandan-imperfect-wach}.
	Then, using \cite[Proposition B.1.2.4]{fontaine-phigamma} (resp.\ Lemma \ref{lem:wach_saturated} and \cite[Remark 2.15]{abhinandan-imperfect-wach}) it follows that $N$ is finite free.
	In particular, Definition \ref{intro_defi:wach_mods_relative} is equivalent to \cite[Definition III.4.1]{berger-limites} over $\AF^+$ (resp.\ \cite[Definition 1.3]{abhinandan-imperfect-wach} over $\AL^+$).
\end{rem}

Set $\AR := \AR^+[1/\mu]^{\wedge}$ as the $\padic$ completion, and note that Frobenius endomorphism $\varphi$ and the continuous action of $\Gamma_R$ on $\AR^+$ extend to $\AR$, and similarly, set $\AL := \AL^+[1/\mu]^{\wedge}$ and note that the Frobenius endomorphism $\varphi$ and the continuous action of $\Gamma_L$ on $\AL^+$ extend to $\AL$.
Let $T$ be a finite free $\ZZ_p\textrm{-representation}$ of $G_R$ and recall that one can functorially attach to $T$ a finite projective \'etale $(\varphi, \Gamma_R)\module$ $\DR(T)$ over $\AR$ of rank $=\rank_{\ZZ_p} T$, equipped with a semilinear and continuous action of $\Gamma_R$ and a Frobenius-semilinear operator $\varphi$ commuting with the action of $\Gamma_R$.
In fact, the preceding functor induces a categorical equivalence between the category of finite free $\ZZ_p\textrm{-representations}$ of $G_R$ and the category of finite projective \'etale $(\varphi, \Gamma_R)\modules$ over $\AR$ (see \cite[Theorem 7.11]{andreatta-phigamma}).
Additionally, the category of Wach modules over $\AR^+$ fully faithfully embeds into the latter category, i.e.\ the category of \'etale $(\varphi, \Gamma_R)\modules$ over $\AR$ (see Proposition \ref{prop:wach_etale_ff_relative}).

\subsubsection{Main results}\label{subsubsec:main_results}

Let $\Rep_{\ZZ_p}^{\crys}(G_R)$ denote the category of $\ZZ_p\textrm{-lattices}$ inside $\padic$ crystalline representations of $G_R$.
For a $\ZZ_p\textrm{-representation}$ $T$ of $G_R$ such that $T[1/p]$ is crystalline, we construct a Wach module $\NR(T)$ over $\AR^+$, functorial in $T$, and contained in $\DR(T)$ (see Theorem \ref{thm:crys_fh_relative}).
Our first main result is as follows:
\begin{thm}[Corollary \ref{cor:crystalline_wach_equivalence_relative}]\label{intro_thm:crystalline_wach_equivalence}
	The Wach module functor induces an equivalence of categories
	\begin{align*}
		\Rep_{\ZZ_p}^{\crys}(G_R) &\isomorphic (\varphi, \Gamma_R)\Mod_{\AR^+}^{[p]_q}\\
		T &\longmapsto \NR(T),
	\end{align*}
	with a quasi-inverse given as $N \mapsto \TR(N) := \big(W\big(\Rbar^{\flat}[1/p^{\flat}]\big) \otimes_{\AR^+} N\big)^{\varphi=1}$.
\end{thm}

\begin{rem}\label{intro_rem:crystalline_wach_equivalence_rational}
	In Theorem \ref{intro_thm:crystalline_wach_equivalence}, we do not expect the functor $\NR$ to be an exact equivalence.
	However, note that after inverting $p$, the Wach module functor induces an exact equivalence of $\otimes\textrm{-categories}$: $\Rep_{\QQ_p}^{\crys}(G_R) \isomorphic (\varphi, \Gamma_R)\Mod_{\BR^+}^{[p]_q}$, via $V \mapsto \NR(V)$, where $\BR^+ := \AR^+[1/p]$, with an exact $\otimes\textrm{-compatible}$ quasi-inverse functor given as $M \mapsto \VR(M) := \big(W(\Rbar^{\flat}[1/p^{\flat}]) \otimes_{\AR^+} M\big)^{\varphi=1}$ (see Corollary \ref{cor:crystalline_wach_rat_equivalence_relative}).
\end{rem}

As an application of Theorem \ref{intro_thm:crystalline_wach_equivalence}, we obtain the following purity statement:
\begin{thm}[Theorem \ref{thm:purity_crystalline}]\label{intro_thm:purity_crystalline}
	Let $V$ be a $\padic$ representation of $G_R$.
	Then, $V$ is crystalline as a representation of $G_R$ if and only if it is crystalline as a representation of $G_L$.
\end{thm}

For a $\padic$ representation $V$ of $G_R$, let $\ODcrysR(V)$ denote the associated filtered $(\varphi, \partial)\module$ over $R[1/p]$ (see \cite[Section 8.2]{brinon-relatif}).
We show the following criterion for checking the crystallinity of $V$:
\begin{cor}[Theorem \ref{thm:purity_crystalline} \& Corollary \ref{cor:odcrys_functoriality}]\label{intro_cor:purity_crystalline}
	Let $V$ be a $\padic$ representation of $G_R$.
	Then, $V$ is crystalline if and only if $\rank_{R[1/p]} \ODcrysR(V) = \dim_{\QQ_p} V$.
	Moreover, under these equivalent conditions, we have a natural isomorphism $L \otimes_{R[1/p]} \ODcrysR(V) \isomorphic \ODcrysL(V)$ of filtered $(\varphi, \partial)\modules$ over $L$.
\end{cor}
Important inputs for the proof of Corollary \ref{intro_cor:purity_crystalline} are Theorem \ref{intro_thm:purity_crystalline} and a careful study of the period rings for the localisation of $\Rbar$ at its minimal primes above $(p) \subset R$ (see Section \ref{subsec:localisation}).

\subsubsection{Strategy for the proof of Theorem \ref{intro_thm:crystalline_wach_equivalence}}\label{subsubsec:proof_strategy}

The motivation for seeking Theorem \ref{intro_thm:crystalline_wach_equivalence} comes from \cite{tsuji-hodge-tate-purity, liu-zhu} and some unpublished works of Tsuji, and it is based on the idea that the crystallinity of a $\padic$ representation of $G_R$ may be deduced from its crystallinity as a $\padic$ representation of $G_L$ (a posteriori, Theorem \ref{intro_thm:purity_crystalline} verifies these expectations).
In a similar vein, \cite{du-liu-moon-shimizu} showed that lattices inside crystalline representations of $G_R$ are equivalent to Breuil--Kisin modules equipped with the prismatic descent data, where the key observation is that for a crystalline $\ZZ_p\textrm{-representation}$ of $G_R$ one may construct its Breuil--Kisin module with descent data from the associated \'etale $\varphi\module$, the Breuil--Kisin module for $L$ and their compatible descent data.
Subsequently, putting these ideas together and a search for analogous results in the realm of $(\varphi, \Gamma)\modules$ led us to Theorem \ref{intro_thm:crystalline_wach_equivalence} (see Section \ref{subsec:relate_other_works} for more on relation to previous works).
Let us make this idea a bit more precise.

The proof of Theorem \ref{intro_thm:crystalline_wach_equivalence} crucially uses the analogous results in the imperfect residue field case (see \cite[Theorem 1.5]{abhinandan-imperfect-wach}) and the results of \cite{andreatta-phigamma, abhinandan-relative-wach-i}.
More concretely, starting with a Wach module $N$ over $\AR^+$, we use ideas from \cite[Theorem 4.25 \& Proposition 4.28]{abhinandan-relative-wach-i}, the observation that $\AL^+ \otimes_{\AR^+} N$ is a Wach module over $\AL^+$ and \cite[Lemma 3.6 \& Theorem 3.12]{abhinandan-imperfect-wach} to establish that $\TR(N)[1/p]$ is a $\padic$ crystalline representation of $G_R$ (see Theorem \ref{thm:fh_crys_relative}).
Conversely, starting with a $\ZZ_p\lattice$ $T$ inside a $\padic$ crystalline representation of $G_R$, we observe that $T[1/p]$ is a $\padic$ crystalline representation of $G_L$, and we use \cite[Theorem 4.1]{abhinandan-imperfect-wach} to obtain a unique Wach module $\NL(T)$ over $\AL^+$.
Moreover, from the theory of $(\varphi, \Gamma)\modules$ we have an \'etale $(\varphi, \Gamma_R)\module$ $\DR(T)$ over $\AR$ (see \cite{andreatta-phigamma}).
So, we set $\NR(T) := \NL(T) \cap \DR(T) \subset \DL(T)$ as an $\AR^+\module$, where $\DL(T)$ is the $(\varphi, \Gamma_L)\module$ over $\AL$, associated to $T$.
Then, using the compatible Frobenius-semilinear endomorphism $\varphi$ and the continuous action of $\Gamma_L \isomorphic \Gamma_R$ on $\NL(T)$ and $\DR(T)$, we equip the $\AR^+\module$ $\NR(T)$ with a natural $(\varphi, \Gamma_R)\action$.

As mentioned earlier, the construction of $\NR(T)$ is parallel to the Breuil--Kisin setting studied in \cite{du-liu-moon-shimizu}, and we employ some (modified) ideas from op.\ cit.\ to show that $\NR(T)$ has ``good'' properties as a $\varphi\module$ over $\AR^+$.
However, there are two key differences: first, op.\ cit.\ uses \cite{brinon-trihan} as an important ingredient but our constructions crucially use \cite{abhinandan-imperfect-wach} instead; second, note that in op.\ cit.\ Breuil--Kisin modules are equipped with prismatic descent data whereas Wach modules admit an action of $\Gamma_R$.
Equipping $\NR(T)$ with a natural action of $\Gamma_R$ is a non-trivial problem and we resolve it by using the theory of Wach modules in the imperfect residue field case from \cite{abhinandan-imperfect-wach} and the theory of \'etale $(\varphi, \Gamma)\modules$ from \cite{andreatta-phigamma} as important inputs.
Finally, we utilise the properties of $\NL(T)$ and $\DR(T)$ to show that $\NR(T)$ is the unique Wach module associated to $T$.

\subsection{Wach modules as \texorpdfstring{$q\textrm{-deformations}$}{-}}\label{subsec:q_deformation}

In Section \ref{sec:wachmod_qconnection} we recall the definition of a $\qconnection$ axiomatically, following \cite{morrow-tsuji}.
Moreover, we show that a Wach module $N$ over $\AR^+$ may also be seen as a $\varphi\module$ equipped with a $\qconnection$.
More precisely, let $D := O_F\llbracket \mu \rrbracket$, and let $\{\gamma_1, \ldots, \gamma_d\}$ be topological generators of the geometric part of $\Gamma_R$, i.e.\ $\Gamma_R'$ (see Section \ref{sec:period_rings_padic_reps}).
Then in Proposition \ref{defi:wachmod_qconnection} we show that the $\qconnection$ defined as
\begin{equation*}
	\nabla_q : N \longrightarrow N \otimes_{\AR^+} \Omega^1_{\AR^+/D}, \hspace{5mm} x \longmapsto \textstyle\sum_{i=1}^d \tfrac{\gamma_i(x)-x}{\mu} \dlog([X_i^{\flat}]),
\end{equation*}
describes $(N, \nabla_q)$ as a $\varphi\module$ with $(p, [p]_q)\textrm{-adically}$ quasi-nilpotent $D\textrm{-linear}$ flat $\qconnection$ over $\AR^+$.
We equip $N$ with the Nygaard filtration $\{\Fil^k N\}_{n \in \ZZ}$ as in Definition \ref{defi:nygaard_fil}.
Then, it follows that $N/\mu N$ is a $\varphi\module$ over $R$ equipped with a $p\textrm{-adically}$ quasi-nilpotent flat connection and we further equip it with a filtration $\Fil^k(N/\mu N)$ given as the image of $\Fil^k N$ under the surjection $N \twoheadrightarrow N/\mu N$.
We equip $(N/\mu N)[1/p]$ with induced structures, in particular, it is a filtered $(\varphi, \partial)\module$ over $R[1/p]$.
\begin{thm}[Theorem \ref{thm:qdeformation_dcrys}]\label{intro_thm:qdeformation_dcrys}
	Let $N$ be a Wach module over $\AR^+$ and $V := \TR(N)[1/p]$, the associated crystalline representation from Theorem \ref{intro_thm:crystalline_wach_equivalence}.
	Then, we have a natural isomorphism $(N/\mu N)[1/p] \isomorphic \ODcrysR(V)$ of filtered $(\varphi, \partial)\modules$ over $R[1/p]$.
\end{thm}
Note that $\ODcrysR(V)$ denotes the filtered $(\varphi, \partial)\module$ over $R[1/p]$ associated to $V$ (see \cite[Section 8.2]{brinon-relatif}).
Our proof of the theorem follows from computations done for the proof of Theorem \ref{thm:fh_crys_relative} (building upon ideas developed in \cite[Theorem 4.25 \& Proposition 4.28]{abhinandan-relative-wach-i} and \cite[Theorem 1.7]{abhinandan-imperfect-wach}).

Finally, let us summarise the relationship between various categories considered in Theorem \ref{intro_thm:crystalline_wach_equivalence} and Theorem \ref{intro_thm:qdeformation_dcrys}.
Recall that $\Rep_{\QQ_p}^{\crys}(G_R)$ is the category of $\padic$ crystalline representations of $G_R$, and let $\MF_R(\varphi, \partial)$ denote the category of filtered $(\varphi, \partial)\modules$ over $R[1/p]$.
From \cite[Section 8.2]{brinon-relatif} we have a natural $\otimes\textrm{-compatible}$ functor $\ODcrysR : \Rep_{\QQ_p}^{\crys}(G_R) \rightarrow \MF_R(\varphi, \partial)$, and let $\MF_R\ad(\varphi, \partial)$ denote its essential image.
Then, from \cite[Th\'eor\`eme 8.5.1]{brinon-relatif}, we have an exact equivalence of $\otimes\textrm{-categories}$ $\ODcrysR : \Rep_{\QQ_p}^{\crys}(G_R) \isomorphic \MF_R\ad(\varphi, \partial)$, with an exact $\otimes\textrm{-compatible}$ quasi-inverse $\OVcrysR$ (see Section \ref{subsec:relative_padicreps}).
So, Remark \ref{intro_rem:crystalline_wach_equivalence_rational} and Theorem \ref{intro_thm:qdeformation_dcrys} may be summarised as follows:
\begin{cor}[Corollary \ref{cor:cat_equiv_diagram}]
	Functors in the following diagram induce exact equivalence of $\otimes\textrm{-categories}$:
	\begin{center}
		\begin{tikzcd}[row sep=36pt]
			\Rep_{\QQ_p}^{\crys}(G_R) \arrow[rr, "\NR", shift left=1.1mm] \arrow[rd, "\ODcrysR", shift left=1.1mm] & & (\varphi, \Gamma_R)\Mod_{\BR^+}^{[p]_q} \arrow[ll, "\VR", shift left=1.1mm] \arrow[ld, "q \mapsto 1"]\\
			& \MF_R\ad(\varphi, \partial) \arrow[ul, "\OVcrysR", shift left=1.1mm].
		\end{tikzcd}
	\end{center}
\end{cor}

\subsection{Relation to previous works}\label{subsec:relate_other_works}

Our first main result, Theorem \ref{intro_thm:crystalline_wach_equivalence}, is a generalisation of arithmetic Wach modules from \cite{wach-free, colmez-hauteur, berger-limites} and \cite[Theorem 1.5]{abhinandan-imperfect-wach}.
That said, the methods of op.\ cit.\ do not directly apply to our current situtation.
In fact, the proof of Theorem \ref{intro_thm:crystalline_wach_equivalence} uses crucial inputs of results and ideas from \cite{abhinandan-relative-wach-i} and \cite{abhinandan-imperfect-wach}.

Recent developments in the theory of prismatic $F\textrm{-crystals}$ in \cite{bhatt-scholze-crystals, du-liu-moon-shimizu, guo-reinecke} would suggest that there is a categorical equivalence between the category of Wach modules over $\AR^+$ and (completed/analytic) prismatic $F\textrm{-crystals}$ on the absolute prismatic site $(\Spf R)_{\Prism}$.
From that perspective, Theorem \ref{intro_thm:crystalline_wach_equivalence} could be seen as an analogue of \cite[Theorem 1.2 \& Proposition 1.4]{du-liu-moon-shimizu}.
In our constructions, for a lattice $T$ inside a crystalline representation of $G_R$, the definition of $\NR(T)$ is parallel to the Breuil--Kisin case studied in op.\ cit.
However, there are two key differences: first, op.\ cit.\ uses \cite{brinon-trihan} as an important ingredient but our constructions use \cite{abhinandan-imperfect-wach} instead; second, note that Wach modules admit a natural action of $\Gamma_R$ whereas relative Breuil--Kisin modules are equipped with the prismatic descent data.
Equipping $\NR(T)$ with a natural action of $\Gamma_R$ is a non-trivial problem and we resolve it by using the theory of Wach modules in the imperfect residue field case from \cite{abhinandan-imperfect-wach} and the theory of \'etale $(\varphi, \Gamma)\modules$ from \cite{andreatta-phigamma} as important inputs.
Furthermore, as our base ring $R$ is absolutely unramified (at $p$), the action of $\Gamma_R$ is rich enough to establish the categorical equivalence claimed in Theorem \ref{intro_thm:crystalline_wach_equivalence}.

In the current paper, we provide two applications of Theorem \ref{intro_thm:crystalline_wach_equivalence}.
The first application, i.e.\ Theorem \ref{intro_thm:purity_crystalline} establishes a certain purity statement for crystalline representations.
Our result is similar to the purity statement for Hodge-Tate representations in \cite[Theorem 9.1]{tsuji-hodge-tate-purity} and rigidity of de Rham local systems in \cite[Theorem 1.3]{liu-zhu}.
It should be noted that the purity result in Theorem \ref{intro_thm:purity_crystalline} may also be obtained by combining \cite[Theorem 1.3]{liu-zhu} and some unpublished works of Tsuji.
Moreover, the result of loc.\ cit.\  works for general ramified (at $p$) small base.
A similar statement has been obtained in \cite[Theorem 1.4]{moon} using the results of \cite{du-liu-moon-shimizu}.

The second application of Theorem \ref{intro_thm:crystalline_wach_equivalence} is given in Corollary \ref{intro_cor:purity_crystalline}.
Our result provides a new criterion for checking the crystallinity of a $\padic$ representation of $G_R$.
Note that the analogous statement for de Rham representations is true from the results of \cite{liu-zhu}.
However, our result in the crystalline case is entirely new and uses Theorem \ref{intro_thm:purity_crystalline} as an important input.
At this point, it is worth mentioning that for general ramified (at $p$) small base, a statement analogous to Corollary \ref{intro_cor:purity_crystalline} appears to be true.
In particular, we expect that one may deduce the statement using \cite[Theorem 1.3]{liu-zhu}, the unpublished results of Tsuji mentioned above and employing arguments similar to our proof of Theorem \ref{thm:purity_crystalline}.

For our second main result, Theorem \ref{intro_thm:qdeformation_dcrys}, the motivation for interpreting a Wach module as a $q\textrm{-de Rham}$ complex and as the $q\textrm{-deformation}$ of crystalline cohomology, i.e.\ $\ODcrys$, comes from \cite[Section B.2.3]{fontaine-phigamma}, \cite[Th\'eor\`eme III.4.4]{berger-limites} and \cite[Section 6]{scholze-q-deformations}.
In particular, we provide a direct generalisation of \cite[Th\'eor\`eme III.4.4]{berger-limites}, as well as verify expectations put forth in \cite[Remark 4.48]{abhinandan-relative-wach-i} and \cite[Remark 1.8]{abhinandan-imperfect-wach} (see Remark \ref{rem:qdeformation_dcrys_imperfect} for the latter).

\subsection{Setup and notations}\label{subsec:setup_nota}

In this section, we will describe our setup and fix some notations, which are essentially the same as in \cite[Section 1.4]{abhinandan-relative-wach-i}.
We will work under the convention that $0 \in \NN$, the set of natural numbers.

Let $p$ be a fixed prime number, $\kappa$ a perfect field of characteristic $p$, $O_F := W(\kappa)$ the ring of $p\textrm{-typical}$ Witt vectors with coefficients in $\kappa$.
Then, $O_F$ is a complete discrete valuation ring with uniformiser $p$ and set $F := O_F[1/p]$ to be the fraction field of $O_F$.
Let $\overline{F}$ denote a fixed algebraic closure of $F$ so that its residue field, denoted as $\overline{\kappa}$, is an algebraic closure of $\kappa$.
Furthermore, denote the absolute Galois group of $F$ to be $G_F := \Gal(\overline{F}/F)$.

\begin{nota}
	Let $\Lambda$ be an $I\textrm{-adically}$ complete algebra for a finitely generated ideal $I \subset \Lambda$.
	Let $Z := (Z_1, \ldots, Z_s)$ denote a set of indeterminates and $\smbfk := (k_1, \ldots, k_s) \in \NN^s$ be a multi-index, and we write $Z^{\smbfk} := Z_1^{k_1} \cdots Z_s^{k_s}$.
	For $\smbfk \rightarrow +\infty$ we will mean that $\sum k_i \rightarrow +\infty$ and we define 
	\begin{equation*}
		\Lambda\langle Z\rangle := \big\{\textstyle\sum_{\smbfk \in \NN^s} a_{\smbfk} Z^{\smbfk}, \hspace{1mm} \textrm{where} \hspace{1mm} a_{\smbfk} \in \Lambda \hspace{1mm} \textrm{and} \hspace{1mm} a_{\smbfk} \rightarrow 0 \hspace{1mm} I\textrm{-adically} \hspace{1mm} \textrm{as} \hspace{1mm} \smbfk \rightarrow +\infty\big\}.
	\end{equation*}
\end{nota}
We fix $d \in \NN$ and let $X := (X_1, X_2, \ldots, X_d)$ be some indeterminates.
Let $R$ be the $\padic$ completion of an \'etale algebra over $\Rframe := O_F\langle X, X^{-1}\rangle$, with non-empty geometrically integral special fibre. 
We fix an algebraic closure $\overline{\Fr(R)}$ of $\Fr(R)$ containing $\overline{F}$.
Let $\overline{R}$ denote the union of finite $R\textrm{-subalgebras}$ $S \subset \overline{\Fr(R)}$, such that $S[1/p]$ is \'etale over $R[1/p]$.
Let $\overline{\eta}$ denote the fixed geometric point of the generic fibre $\Spec R[1/p]$ (defined by $\overline{\Fr(R)}$), and let $G_R := \pi_1^{\etale}\big(\Spec R[1/p], \overline{\eta}\big) = \Gal(\overline{R}[1/p] / R[1/p])$ denote the \'etale fundamental group.
For $k \in \NN$, let $\Omega^k_R$ denote the $\padic$ completion of module of $k\textrm{-differentials}$ of $R$ relative to $\ZZ$.
Then, we have $\Omega^1_R = \oplus_{i=1}^d R \dlog X_i$, and $\Omega^k_R = \wedge_R^k \Omega^1_R$.

Let $\varphi$ denote an endomorphism of $\Rframe$ which extends the natural Frobenius on $O_F$ by setting $\varphi(X_i) = X_i^p$, for all $1 \leqslant i \leqslant d$.
The morphism $\varphi : \Rframe \rightarrow \Rframe$ is flat by \cite[Lemma 7.1.5]{brinon-relatif}, and it is faithfully flat since $\varphi(\frakm) \subset \frakm$ for any maximal ideal $\frakm \subset \Rframe$.
Moreover, using Nakayama Lemma and the fact that the absolute Frobenius on $\Rframe/p$ is evidently of degree $p^d$, it easily follows that $\varphi$ on $\Rframe$ is finite of degree $p^d$.
Recall that the $O_F\algebra$ $R$ is given as the $\padic$ completion of an \'etale algebra $\Rframe$, therefore, the Frobenius endomorphism $\varphi$ on $\Rframe$ admits a unique extension $\varphi : R \rightarrow R$ such that the induced map $\varphi : R/p \rightarrow R/p$ is the absolute Frobenius $x \mapsto x^p$ (see \cite[Proposition 2.1]{colmez-niziol}).
Similar to above, again note that the endomorphism $\varphi : R \rightarrow R$ is faithfully flat and finite of degree $p^d$.

Let $O_L := (R_{(p)})^{\wedge}$, where ${}^{\wedge}$ denotes the $\padic$ completion.
Let $\Lbar$ denote a fixed algebraic closure of $L$ with ring of integers $\OLbar$ such that we have an injective homomorphism $\Rbar \rightarrow \OLbar$.
Then, we get a continuous homomorphism $G_L := \Gal(\Lbar/L) \rightarrow G_R$, inducing an isomorphism $\Gamma_L \isomorphic \Gamma_R$.
The Frobenius on $R$ extends to a unique Frobenius endomorphism $\varphi : O_L \rightarrow O_L$, lifting the absolute Frobenius on $O_L/pO_L$ (see \cite[Proposition 2.1]{colmez-niziol}).
Similar to above, $\varphi$ on $O_L$ is faithfully flat and finite of degree $p^d$.

Let $S$ be a commutative ring with an element $\pi \in S$ such that $\pi^p$ divides $p$ and it admits a compatible system of $p\textrm{-power}$ roots, and the $p\textrm{-power}$ map $S/\pi \rightarrow S/\pi^p$ is bijective.
Then, the tilt of $S$ is defined as $S^{\flat} := \lim_{\varphi} S/p$ and the tilt of $S[1/p]$ is defined as $S[1/p]^{\flat} := S^{\flat}[1/\pi^{\flat}]$, where $\pi^{\flat} := (\pi, \pi^{1/p}, \ldots) \in S^{\flat}$ (see \cite[Chapitre V, Section 1.4]{fontaine-pdivisibles} and \cite[Section 3]{bhatt-morrow-scholze-1}).
Finally, consider a $\ZZ_p\algebra$ $A$ equipped with a lift of the absolute Frobenius on $A/p$, i.e.\ an endomorphism $\varphi : A \rightarrow A$ such that $\varphi$ modulo $p$ is the absolute Frobenius.
Then, for any $A\module$ $M$ we write $\varphi^*(M) := A \otimes_{\varphi, A} M$.

\vspace{2mm}
\noindent \textbf{Outline of the paper.}
This article consists of four main sections.
In Section \ref{sec:period_rings_padic_reps}, we collect relevant results from relative $\padic$ Hodge theory.
In Section \ref{subsec:localisation}, we consider localisations of $\Rbar$ at minimal primes above $(p) \subset R$ and study their properties.
Then, in Section \ref{subsec:ainf_relative}, Section \ref{subsec:derham_relative} \& Section \ref{subsec:crystalline_relative}, we define relative period rings and study their localisations at primes of $\Rbar$ above $(p) \subset R$.
In Section \ref{subsec:phigamma_mod_rings}, we quickly recall important rings from the theory of relative $(\varphi, \Gamma)\modules$, and in Section \ref{subsec:relative_padicreps}, we recall the relation between $(\varphi, \Gamma)\module$ theory and $\padic$ representations, as well as, definition and properties of crystalline representations.
The aim of Section \ref{sec:wach_modules} is to define and study properties of a Wach module in the relative case and the associated representation of $G_R$.
In Section \ref{subsec:technical_results}, we first note some technical lemmas, and then in Section \ref{subsec:wach_mod_props}, we define relative Wach modules, study its properties and relate these objects to \'etale $(\varphi, \Gamma)\modules$ (see Proposition \ref{prop:wach_etale_ff_relative}).
Furthermore, in Section \ref{subsec:wach_mod_rep}, we functorially attach a $\ZZ_p\representation$ of $G_R$ to a relative Wach module, and in Section \ref{subsec:finite_pqheight_reps}, we show that such representations are closely related to finite $\pqheight$ representations studied in \cite{abhinandan-relative-wach-i}.
In Section \ref{subsec:nygaard_fil_wach_mod}, we study the Nygaard filtration on relative Wach modules.
Finally, in Section \ref{subsec:wachmod_crystalline}, we show that the $\ZZ_p\representation$ of $G_R$ associated to a relative Wach module, as in Section \ref{subsec:wach_mod_rep}, is a lattice inside a $\padic$ crystalline representation of $G_R$ (see Theorem \ref{thm:fh_crys_relative}).
In Section \ref{sec:crystalline_finite_height}, we prove our first main result, i.e.\ Theorem \ref{intro_thm:crystalline_wach_equivalence} which is entriely contained in Section \ref{subsec:crys_wach_equiv_proof}.
In Section \ref{subsec:crys_wach_equiv_consequence}, we draw some important conclusions of Theorem \ref{intro_thm:crystalline_wach_equivalence}, in particular, we prove Theorem \ref{intro_thm:purity_crystalline} and Corollary \ref{intro_cor:purity_crystalline}.
In Section \ref{sec:wachmod_qconnection}, we state and prove our second main result, i.e.\ Theorem \ref{intro_thm:qdeformation_dcrys}.
In Section \ref{subsec:qconnection_formalism}, we recall the formalism on $q\textrm{-connections}$.
Then, in Section \ref{subsec:wachmod_qdeformation}, we show that a Wach module may be interpreted as a $\varphi\module$ equipped with a $\qconnection$ (see Proposition \ref{defi:wachmod_qconnection}).
Finally, using the computations done in the proof of Theorem \ref{thm:fh_crys_relative}, we prove Theorem \ref{intro_thm:qdeformation_dcrys}.

\vspace{2mm}

\noindent \textbf{Acknowledgements:} 
I would like to sincerely thank Takeshi Tsuji for discussing many ideas during the course of this project, reading an earlier version of the article and suggesting improvements.
Additionally, I would like to thank Denis Benois for pointing out some inaccuracies in a previous version of the article, and Yong Suk Moon, Koji Shimizu and Alex Youcis for their helpful remarks.
Finally, I would like to thank the referee for carefully reading the article and pointing out issues in the previous version, which greatly helped in rectifying the errors and improving the writing.
This research was supported by JSPS KAKENHI grant numbers 22F22711 and 22KF0094.

\section{Period rings and \texorpdfstring{$\padic$}{--} representations}\label{sec:period_rings_padic_reps}

We shall use the setup and notations from Section \ref{subsec:setup_nota}.
Recall that $R$ is the $\padic$ completion of an \'etale algebra over $O_F\langle X_1^{\pm 1}, \ldots, X_d^{\pm 1} \rangle$ and $O_L \coloneq (R_{(p)})^{\wedge}$.
Set $\Rinfty \coloneq \cup_{i=1}^d R[\mu_{p^{\infty}}, X_i^{1/p^{\infty}}]$, i.e.\ $\Rinfty$ is obtained by adjoining to $R$ all $p\textrm{-power}$ roots of unity and all $p\textrm{-power}$ roots of $X_i$, for $1 \leqslant i \leqslant d$.
Recall that $\Rbar$ is the union of finite $R\textrm{-subalgebras}$ $S$ in a fixed algebraic closure $\overline{\Fr(R)} \supset \Fbar$, such that $S[1/p]$ is \'etale over $R[1/p]$.
Then, we set (see \cite[Section 2 \& Section 3]{abhinandan-relative-wach-i}),
\begin{align*}
	G_R &\coloneq \Gal(\Rbar[1/p]/R[1/p]), \hspace{1mm} H_R \coloneq \Gal(\Rbar[1/p]/\Rinfty[1/p]),\\
	\Gamma_R &\coloneq G_R/H_R = \Gal(\Rinfty[1/p]/R[1/p]) \isomorphic \ZZ_p(1)^d \rtimes \ZZ_p^{\times}, \\
	\Gamma'_R &\coloneq \Gal(R_{\infty}[1/p]/R(\mu_{p^{\infty}})[1/p]) \isomorphic \ZZ_p(1)^d, \hspace{1mm} \Gal\big(R(\mu_{p^{\infty}})[1/p]/R[1/p]) = \Gamma_R/\Gamma'_R \isomorphic \ZZ_p^{\times}.
\end{align*}
We fixed $\Lbar$ as an algebraic closure of $L \coloneq O_L[1/p]$ with the ring of integers $\OLbar$ such that we have an injective homomorphism $\Rbar \rightarrow \OLbar$.
So, we have a continuous homomorphism of groups $G_L \coloneq \Gal(\Lbar/L) \rightarrow G_R$, which induces an isomorphism of Galois groups $\Gamma_L \isomorphic \Gamma_R$.
For $1 \leqslant i \leqslant d$, we fix $X_i^{\flat} \coloneq (X_i, X_i^{1/p}, X_i^{1/p^2}, \ldots)$ in $\Rinfty^{\flat}$ and take $\{\gamma_0, \gamma_1, \ldots, \gamma_d\}$ in $\Gamma_R$ such that $\{\gamma_1, \ldots, \gamma_d\}$ are topological generators of $\Gamma'_R$ satisfying $\gamma_j(X_i^{\flat}) = \varepsilon X_i^{\flat}$, if $i=j$, and $X_i^{\flat}$ otherwise, and $\gamma_0$ is a lift of a topological generator of $\Gamma_R/\Gamma'_R$.

\subsection{Localisation}\label{subsec:localisation}

Let $\pazs$ denote the set of minimal primes of $\Rbar$ above $pR \subset R$.
The set $\pazs$ is equipped with a transitive action of $G_R$ (see \cite[Theorem 9.3]{matsumura}).
For each prime $\pins$, set $\GRp \coloneq \{g \in G_R \textrm{ such that } g(\frakp) = \frakp\}$, i.e.\ the decomposition group of $G$ at $\frakp$.
Recall that $O_L = (R_{(p)})^{\wedge}$ and $L = O_L[1/p]$.
For each $\pins$, let $\Lbarp$ denote an algebraic closure of $L$ with ring of integers $\OLbarp$ containing $\Rbarp$.
Set $\GRhatp \coloneq \Gal(\Lbarp/L)$, so that we have a natural homomorphism $\GRhatp \rightarrow G_R$ which factors as $\GRhatp \twoheadrightarrow \GRp \subset G_R$ (see \cite[Lemme 3.3.1]{brinon-relatif}).
Note that for each $\pins$, we have a natural inclusion $\Rbar \subset \OLbarp$, and hence we have a (non-canonical) isomorphism of Galois groups $\GRhatp \isomorphic G_L$.

Now, for each $\pins$, let $\Cpplus$ denote the $\padic$ completion of $\OLbarp$ and let $\Cp \coloneq \Fr(\Cpplus)$.
Then, $\Cp$ is a complete algebraically closed valuation field equipped with a continuous action of $\GRhatp$ and $(\Cpplus)^{\GRhatp} = O_L$ (see \cite[Theorem 1]{hyodo}).
Furthermore, let $\Cplusp$ denote the $\padic$ completion of $\Rbarp$ and let $\CC(\frakp) \coloneq \Cplusp[1/p]$ equipped with a continuous action of $\GRp$.

\begin{lem}\label{lem:prbarp}
	For each $\pins$, the ring $\CC^+(\frakp)$ is $p\textrm{-torsion}$ free.
	Moreover, we have that $\Rbarp \hookrightarrow \Cplusp$ and the following two equalities hold, for each $n \geqslant 1$:
	\begin{equation*}
		\Rbarp \cap p^n\Cplusp = p^n\Rbarp \subset \Cplusp, \qquad \Rbarp \cap p^n\OLbarp = p^n\Rbarp \subset \OLbarp.
	\end{equation*}
	In particular, we have a natural injective homomorphism of rings $\Cplusp \rightarrow \Cpplus$, and inside $\Cp$, we have that $\Cplusp = \CC(\frakp) \cap \Cpplus$.
\end{lem}
\begin{proof}
	The first claim follows because $\Rbarp$ is $p\textrm{-torsion}$ free.
	To show the other claims, we first make an observation.
	Note that $R_{(p)}$ is a one-dimensional normal noetherian domain, therefore, it is a discrete valuation ring (see \cite[Theorem 11.2]{matsumura}) with field of fractions $\Fr(R)$.
	Let $\upsilon$ denote the normalised valuation on $R_{(p)}$ such that $\upsilon(p) = 1$.
	As $O_L$ is the $\padic$ completion of $R_{(p)}$, the valuation $\upsilon$ extends uniquely to $O_L$ and its fraction field is $L$.
	Moreover, note that $\Lbarp$ is an algebraic closure of $L$, so the valuation $\upsilon$ further uniquely extends to $\Lbarp$ and its ring of integers is denoted as $\OLbarp$.

	Now, to show that $\Rbarp \hookrightarrow \Cplusp$, it is enough to show that $\Rbarp$ is $p\textrm{-adically}$ separated, i.e.\ $\cap_{n \in \NN} p^n \Rbarp = 0$.
	Note that we have a natural injective homomorphism $\Rbarp \hookrightarrow \OLbarp$ and the latter is $p\textrm{-adically}$ separated because it is the ring of integers of the valued field $\Lbarp$.
	Therefore, we get that $\cap_{n \in \NN} p^n \Rbarp \subset \cap_{n \in \NN} p^n \OLbarp = 0$, hence, $\Rbarp \hookrightarrow \Cplusp$.
	Moreover, using the preceding observations, the fact that $\Rbarp/p^n \isomorphic \Cplusp/p^n$, for each $n \geqslant 1$ (see \cite[\href{https://stacks.math.columbia.edu/tag/05GG}{Tag 05GG}]{stacks-project}), and Lemma \ref{lem:injectivity_modulo}, it follows that $\Rbarp \cap p^n\Cplusp = p^n\Rbarp$.

	For the next claim, observe that we have $p^n\Rbarp \subset \Rbarp \cap p^n\OLbarp$.
	To show the converse, let $x$ be an element of $\Rbarp \cap p^n\OLbarp$, i.e.\ $x = p^ny$ for some $y$ in $\OLbarp$, so we have that $\upsilon(y) \geqslant 0$.
	Moreover, note that there exists a finite normal $R\textrm{-subalgebra}$ $S \subset \Rbar$ such that $S[1/p]$ is \'etale over $R[1/p]$, the ideal $\frakq \coloneq \frakp \cap S$ is a height 1 prime ideal of $S$ with $p \in \frakq$ (since $\Rbar$ is integral over $S$) and $x$ is in $S_{\frakq} \subset \Rbarp$.
	So, we see that $y$ is in $S_{\frakq}[1/p]$ and it suffices to show that $y$ is in $S_{\frakq}$.
	To show this, observe that the extension $\Fr(R) \rightarrow \Fr(S)$ is finite and algebraic and the inclusion $S_{\frakq} \subset \Rbarp \subset \OLbarp$ induces an embedding of fields $\Fr(S) \hookrightarrow \OLbarp$.
	Then, by the theory of extension of valuations, we see that this embedding induces a valuation on $\Fr(S)$ extending the valuation $\upsilon$ on $\Fr(R)$; we again denote this valuation on $\Fr(S)$ by $\upsilon$, and note that $S_{\frakq}$ is the ring of integers of $\Fr(S)$.
	As we have that $\upsilon(y) \geqslant 0$, therefore, we conclude that $y$ must be in $S_{\frakq}$, i.e.\ $x$ is in $p^n S_{\frakq} \subset p^n \Rbarp$.
	Hence, it follows that $\Rbarp \cap p^n\OLbarp = p^n\Rbarp \subset \OLbarp$.

	Finally, using Lemma \ref{lem:injectivity_modulo} and the observations from above, we conclude that the natural homomorphism of rings $\Rbarp/p^n \rightarrow \OLbarp/p^n$ is injective, for each $n \geqslant 1$.
	Taking the limit over $n$ and noting that $\lim_n$ is left exact, we get that the natural homomorphism of rings $\Cplusp \rightarrow \Cpplus$ is injective.
	As $\CC(\frakp) = \Cplusp[1/p]$ and $\Cp = \Cpplus[1/p]$, the last claim again follows from the injectivity of $\Rbarp/p^n \rightarrow \OLbarp/p^n$ and Lemma \ref{lem:injectivity_modulo}.
	This completes our proof.
\end{proof}

All of the rings appearing in Lemma \ref{lem:prbarp} are $p\textrm{-torsion free}$.
Moreover, we have a natural injective homomorphism of rings $\Cplusp \rightarrow \Cpplus$, and note that it is further compatible with the respective actions of $\GRhatp$, where the action of $\GRhatp$ on the left-hand term factors through $\GRhatp \twoheadrightarrow \GRp$.
In particular, we get that $\Cplusp^{\GRp} = O_L$ (see \cite[p.\ 24]{brinon-relatif}).

Now, note that we have natural injective homomorphisms of rings $\Rbar \rightarrow \Rbarp \rightarrow O_{\Lbar(\frakp)}$.
Upon passing to the $\padic$ completions and setting $\CC^+(\Rbar) \coloneq \widehat{\Rbar}$, we obtain natural homomorphisms of rings $\CC^+(\Rbar) \rightarrow \Cplusp \rightarrow \Cpplus$, where the first map need not be injective.
However, taking the direct summation over all $\pins$ gives us the following:
\begin{lem}\label{lem:cplus_gequiv}
	For each $n \geqslant 1$, the natural $R\linear$ homomorphism $\Rbar/p^n \rightarrow \oplus_{\pins} \Rbarp/p^n$ is injective.
	Taking the limit over $n$, yields the following natural $R\linear$ injective homomorphisms:
	\begin{equation}\label{eq:cplus_gequiv}
		\CC^+(\Rbar) \longrightarrow \textstyle\prod_{\pins} \Cplusp \longrightarrow \textstyle\prod_{\pins} \Cpplus.
	\end{equation}
	In particular, we have that $\CC^+(\Rbar) = \CC(\Rbar) \cap \textstyle\prod_{\pins} \Cplusp \subset \big(\textstyle\prod_{\pins} \Cplusp\big)[1/p]$.
\end{lem}
\begin{proof}
	To prove the first claim, let $(x_{\frakp})_{\pins}$ be an element of $\oplus_{\pins} \Rbarp$ such that $(p^nx_{\frakp})_{\pins}$ is in the image of the injective homomorphism $\Rbar \rightarrow \oplus_{\pins} \Rbarp$.
	In particular, there exists some $y$ in $\Rbar$ such that $p^nx_{\frakp} = y$, for each $\pins$, and it suffices to show that $(x_{\frakp})_{\pins}$ is in the image of the map $\Rbar \rightarrow \oplus_{\pins} \Rbarp$.
	Viewing the equality $p^nx_{\frakp} = y$ inside $\Fr(\Rbar)$, we conclude that $x_{\frakp} = x_{\frakp'} = y/p^n$, for all $\frakp, \frakp' \in \pazs$.
	Next, recall that $\Rbar$ is a direct limit of finite and normal $R\textrm{-algebras}$, therefore, there must exist some $R\textrm{-subalgebra}$ $S \subset \Rbar$ which is finite and normal over $R$ with $S[1/p]$ \'etale over $R[1/p]$, and such that $y$ is in $S$ and $x_{\frakp}$ is in $S_{\frakq}$, for each $\pins$ and $\frakq \coloneq \frakp \cap S$ a height 1 prime ideal $S$ containing $p$.
	As $x_{\frakp} = x_{\frakp'}$, for all $\frakp, \frakp' \in \pazs$, therefore, we also get that $x_{\frakp}$ is in $S_{\frakq'}$ for each height 1 prime ideal $\frakq' \coloneq \frakp' \cap S$ of $S$ containing $p$.
	Moreover, we have that $x_{\frakp} = y/p^n$, i.e.\ $x_{\frakp}$ is in $S[1/p]$.
	Now, let $I$ denote the set of height 1 prime ideals of $S$ containing $p$ and let $J$ denote the set of height 1 prime ideals of $S$ not containing $p$.
	Then, we see that inside $\Fr(S)$, we have,
	\begin{equation*}
		S = \cap_{\frakq \in I \cup J} S_{\frakq} = \big(\cap_{\frakq \in I} S_{\frakq}\Big) \cap S[1/p],
	\end{equation*}
	where the first equality follows from \cite[Theorem 11.5]{matsumura} and the second equality follows because $S[1/p] \subset S_{\frakq}$, for each $\frakq$ in $J$.
	Hence, from the preceding discussion, we conclude that $x_{\frakp}$ is an element of $S \subset \Rbar$ and $x_{\frakp} = x_{\frakp'}$, for all $\frakp, \frakp' \in \pazs$, in particular, $(x_{\frakp})_{\pins}$ is the image of $y$ under the map $\Rbar \rightarrow \oplus_{\pins} \Rbarp$.
	This proves the first claim.

	For the second claim, we use the conclusion of the first claim and take the limit over $n$ to obtain the first homomorphism in \eqref{eq:cplus_gequiv}, which is injective because $\lim_n$ is left exact.
	For the second map in \eqref{eq:cplus_gequiv}, we take the product of the injective homomorphisms $\Cplusp \rightarrow \Cpplus$, for each $\pins$ (see Lemma \ref{lem:prbarp}).
	Finally, the last claim follows from the injectivity of $\Rbar/p^n \rightarrow \oplus_{\pins} \Rbarp/p^n$ and Lemma \ref{lem:injectivity_modulo}.
	This finishes our proof.
\end{proof}

Note that in \eqref{eq:cplus_gequiv} the leftmost term admits a natural action of $G_R$, the middle term admits a natural action of $\prod_{\pins}\GRp$ and the rightmost term admits a natural action of $\prod_{\pins} \GRhatp$.
The two homomorphisms in \eqref{eq:cplus_gequiv} are compatible with these respective actions.
Moreover, from \cite[Remarque 3.3.2]{brinon-relatif} the middle term of \eqref{eq:cplus_gequiv} may be equipped with an action of $G_R$ and the left homomorphism in \eqref{eq:cplus_gequiv} is equivariant with respect to this action of $G_R$.

\begin{rem}\label{rem:ollin_embed}
	Note that $\Cplusp$ is an $O_L\algebra$ for each $\pins$, so the maps in \eqref{eq:cplus_gequiv} extend to natural injective homomorphisms $O_L \otimes_R \CC^+(\Rbar) \rightarrow \prod_{\pins} \Cplusp \rightarrow \prod_{\pins} \Cpplus$ (see \cite[Proposition 3.3.3]{brinon-relatif}).
\end{rem}

\begin{lem}\label{lem:rbarphat_perfectoid}
	The $R\textrm{-algebras}$ $\widehat{R}_{\infty}$ and $\CC^+(\Rbar)$ are perfectoid in the sense of \cite[Definition 3.5]{bhatt-morrow-scholze-1}.
	Moreover, the $O_L\textrm{-algebras}$ $\Cplusp$ and $\Cpplus$ are also perfectoid.
\end{lem}
\begin{proof}
	The fact that $\CC^+(\Rbar)$ is perfectoid follows from \cite[Propositions 2.0.1 \& 5.1.2]{brinon-relatif}, and that $\widehat{R}_{\infty}$ is perfectoid follows from loc.\ cit.\ and \cite[Corollary 3.7]{andreatta-phigamma}.
	Moreover, from the discussion in \cite[Section 1]{brinon-imparfait}, \cite[Section 3]{scholze-perfectoid} and \cite[Lemma 3.20]{bhatt-morrow-scholze-1}, we see that $\Cpplus$ is also a perfectoid algebra.
	Next, for $\Cplusp$, note that we have $\pi \coloneq p^{1/p}$ in $\Rbar \subset \Rbarp \subset \Cplusp$ and $\pi^p = p$ divides $p$.
	Moreover, it is clear that $\Cplusp$ is $\pi\textrm{-adically}$ complete.
	Now, consider the following commutative diagram:
	\begin{center}
		\begin{tikzcd}[row sep=16pt]
			\Cplusp/\pi^p \arrow[r, twoheadrightarrow] \arrow[d] & \Cplusp/\pi \arrow[r, "\varphi"'] \arrow[d] & \Cplusp/\pi^p \arrow[d]\\
			\Cpplus/\pi^p \arrow[r, twoheadrightarrow] & \Cpplus/\pi \arrow[r, "\sim", "\varphi"'] & \Cpplus/\pi^p,
		\end{tikzcd}
	\end{center}
	where the left and the right vertical arrows are injective by Lemma \ref{lem:prbarp}, the middle vertical arrow is injective by an argument similar to the proof of Lemma \ref{lem:prbarp}, and the bottom right horizontal arrow is bijective because $\Cpplus$ is perfectoid.
	So, it follows that the top right horizontal arrow is injective as well.
	Then, using \cite[Lemma 3.9 and Lemma 3.10]{bhatt-morrow-scholze-1}, we are left to show that $\varphi \colon \Cplusp/p = \Rbarp/p \rightarrow \Rbarp/p = \Cplusp/p$ is surjective.
	Let $x$ be an element of $\Rbarp/p$ and take a lift $y$ in $\Rbarp$ of $x$.
	Then, there exists an $a$ in $\Rbar \setminus \frakp$ such that $ay$ is in $\Rbar$.
	Now, from \cite[Proposition 2.0.1]{brinon-relatif}, there exist some $z$ and $w$ in $\Rbar$ such that $ay = z^p + pw$.
	Moreover, there exist some $b$ in $\Rbar \setminus \frakp$ and $c$ in $\Rbar$ such that $a = b^p + pc$.
	So, we may write $b^py + pcy = z^p + pw$, or equivalently, $y = (z/b)^p + p(cy+w)/b^p$, with $(z/b)^p$ in $\Rbarp$ and $p(cy+w)/b^p$ in $p\Rbarp$.
	Hence, $x = (z/b)^p \mod p\Rbarp$, proving that $\varphi \colon \Rbarp/p \rightarrow \Rbarp/p$ is surjective.
\end{proof}

\subsection{The period ring \texorpdfstring{$\Ainf$}{-}}\label{subsec:ainf_relative}

In this section, we will study the relative version of Fontaine's infinitesimal period ring $\Ainf$ which will be useful in later sections (see \cite[Section 5.1]{brinon-relatif} and \cite[Sections 2 \& 3]{abhinandan-relative-wach-i} for more details on these rings).
Let $\Ainf(\Rinfty) \coloneq W(\Rinfty^{\flat})$ and $\Ainf(\Rbar) \coloneq W(\Rbar^{\flat})$ admitting the Frobenius on Witt vectors and continuous $G_R\textrm{-action}$ (for the weak topology).
Moreover, we have $\Ainf(\Rinfty) = \Ainf(\Rbar)^{H_R}$ (see \cite[Proposition 7.2]{andreatta-phigamma}).
Let $\varepsilon \coloneq (1, \zeta_p, \zeta_{p^2}, \ldots)$ and $\mubar \coloneq \varepsilon-1$ in $\OFinfty^{\flat}$, and set $\mu \coloneq [\varepsilon] - 1$ and $\xi \coloneq \mu/\varphi^{-1}(\mu)$ in $\Ainf(\OFinfty)$.
Let $\chi$ denote the $\padic$ cyclotomic character, and note that for any $g$ in $G_R$, we have that $g(1+\mu) = (1+\mu)^{\chi(g)}$.
Additionally, we have a $G_R\equivariant$ surjection $\theta \colon \Ainf(\Rbar) \rightarrow \CC^+(\Rbar)$ and $\kert \theta = \xi \Ainf(\Rbar)$.
The map $\theta$ further induces a $\Gamma_R\equivariant$ surjection $\theta \colon \Ainf(\Rinfty) \rightarrow \widehat{R}_{\infty}$.

Let $\pazs$ denote the set of minimal primes of $\Rbar$ above $pR \subset R$, and for each prime $\pins$, let $\Cp$ denote the complete valuation field described in Section \ref{subsec:localisation} with its ring of integers being the perfectoid algebra $\Cpplus$.
Moreover, from Lemma \ref{lem:rbarphat_perfectoid}, recall that we also have the perfectoid algebra $\Cplusp$.
So, we set $\Ainf(\Cpplus) \coloneq W(\Cpplusflat)$ (resp.\ $\Ainf(\Cplusp) \coloneq W(\Cplusp^{\flat})$) admitting the Frobenius on Witt vectors and continuous $\GRhatp\action$ (resp.\ $\GRp\action$).
Similar to above, we have a $\GRhatp\equivariant$ surjection $\theta \colon \Ainf(\Cpplus) \rightarrow \Cpplus$ with $\kert \theta = \xi \Ainf(\Cpplus)$ (resp.\ a $\GRp\equivariant$ surjection $\theta \colon \Ainf(\Cplusp) \rightarrow \Cplusp$ with $\kert \theta = \xi \Ainf(\Cplusp)$).

\begin{lem}\label{lem:cplusp_in_cpplus_ainf}
	For each $\pins$ we have natural $(\varphi, \GRhatp)\equivariant$ injective ring homomorphisms $\Ainf(\Cplusp) \rightarrow \Ainf(\Cpplus)$ and $W(\CC(\frakp)^{\flat}) \rightarrow W(\Cp^{\flat})$, where the action of $\GRhatp$ on left-hand terms factor through $\GRhatp \twoheadrightarrow \GRp$.
	Moreover, we have a natural $(\varphi, \GRhatp)\equivariant$ identification $\Ainf(\Cplusp) = \Ainf(\Cpplus) \cap W(\CC(\frakp)^{\flat})$ as subrings of $W(\Cp^{\flat})$.
\end{lem}
\begin{proof}
	From the discussion before \eqref{eq:cplus_gequiv}, we have a $\GRhatp\equivariant$ injective ring homorphism $\Cplusp \rightarrow \Cpplus$.
	By applying the tilting functor, we further obtain a natural  $(\varphi, \GRhatp)\equivariant$ commutative diagram of rings
	\begin{equation}\label{eq:cpplus_cplusp_flat}
		\begin{tikzcd}[row sep=15pt]
			\Cplusp^{\flat} \arrow[r] \arrow[d] & \Cpplusflat \arrow[d]\\
			\CC(\frakp)^{\flat} \arrow[r] & \Cp^{\flat},
		\end{tikzcd}
	\end{equation}
	where the vertical arrows are injective.
	Note that the natural homomorphism of rings $\Cplusp/p = \Rbarp/p \rightarrow \OLbarp/p = \Cpplus/p$ is injective by Lemma \ref{lem:prbarp}, so by the left exactness of $\lim_{\varphi}$, we obtain that in \eqref{eq:cpplus_cplusp_flat} the top horizontal arrow is injective.
	Now, let $p^{\flat} \coloneq (p, p^{1/p}, \cdots)$ in $C^{\flat} \coloneq \overline{F}^{\flat}$ be a pseudo-uniformiser, and note that $\CC(\frakp)^{\flat}$ and $\Cp^{\flat}$ are perfectoid algebras over $C^{\flat}$.
	So, it follows that $\Cplusp^{\flat}$ and $\Cplusp^{\flat}$ are $p^{\flat}\textrm{-torsion}$ free and we have that $\CC(\frakp)^{\flat} = \Cplusp^{\flat}[1/p^{\flat}]$ and $\Cp^{\flat} = \Cpplusflat[1/p^{\flat}]$.
	In particular, from the injectivity of the top horizontal arrow in \eqref{eq:cpplus_cplusp_flat} we obtain that the bottom horizontal arrow is also injective.

	Next, we will show that the following equality holds:
	\begin{equation}\label{eq:cplusp_intersect}
		\Cplusp^{\flat} = \CC(\frakp)^{\flat} \cap \Cpplusflat \subset \Cp^{\flat}.
	\end{equation}
	To obtain the equality in \eqref{eq:cplusp_intersect}, we will use valuations.
	So, we start by observing that $\Cp$ is a complete valuation field, and let $\upsilon_{\frakp}$ denote the normalised valuation on it such that $\upsilon_{\frakp}(p) = 1$ (one may take this to be the natural extension of the valuation on $\Lbarp$ described in the proof of Lemma \ref{lem:cplus_gequiv}).
	For the valuation $\upsilon_{\frakp}$ on $\Cp$, note that if $x$ is an element of $\Cp$ such that $\upsilon_{\frakp}(x) \geqslant 0$, then $x$ is in $\Cpplus$.
	As we have that $\Cplusp = \CC(\frakp) \cap \Cpplus \subset \Cp$ (see Lemma \ref{lem:prbarp}), therefore, we see that if $x$ is an element of $\CC(\frakp)$ such that $\upsilon_{\frakp}(x) \geqslant 0$, then $x$ must be in $\Cplusp$.
	Now, note that we have isomorphisms of multiplicative monoids $\Cp^{\flat} \isomorphic \lim_{x \mapsto x^p} \Cp$ (resp.\ $\Cpplusflat \isomorphic \lim_{x \mapsto x^p} \Cpplus$) and $\CC(\frakp)^{\flat} \isomorphic \lim_{x \mapsto x^p} \CC(\frakp)$ (resp.\ $\Cplusp^{\flat} \isomorphic \lim_{x \mapsto x^p} \Cplusp$).
	Using the preceding identification, let $\sharp \colon \Cp^{\flat} \rightarrow \Cp$ denote the projection to the first component, denoted as $y \mapsto y^{\sharp}$, and note that we also have compatible projection maps for $\Cpplusflat$, $\CC(\frakp)^{\flat}$ and $\Cplusp^{\flat}$.
	The map $\sharp$ is a multiplicative map of monoids and we define a valuation on $\Cp^{\flat}$ given as $\upsilon_{\frakp}^{\flat}(y) = \upsilon_{\frakp}(y^{\sharp})$, for any $y$ in $\Cp^{\flat}$ (see \cite[Lemma 3.4]{scholze-perfectoid}).
	For this valuation, if $\upsilon_{\frakp}(y) \geqslant 0$, then $y$ is in $\Cpplusflat$.
	As $\CC(\frakp)^{\flat} \hookrightarrow \Cp^{\flat}$, so we see that if $y$ is an element of $\CC(\frakp)^{\flat}$ such that $\upsilon_{\frakp}(y) \geqslant 0$, then from the discussion above we get that $y^{\sharp}$ is in $\Cplusp$.
	Write $y = (y_n)_{n \in \NN}$ in $\lim_{x \mapsto x^p} \CC(\frakp)$, with $y_0 = y^{\sharp}$ in $\Cplusp$ and $y_{n+1}^p = y_n$, for $n \geqslant 1$.
	Then, using the identification $\Cplusp = \CC(\frakp) \cap \Cpplus \subset \Cp$ (see Lemma \ref{lem:prbarp}), and induction on $n \geqslant 0$, it follows that $y_n$ must be in $\Cplusp$, for all $n \in \NN$.
	Thus, $y$ is an element of $\Cplusp^{\flat}$.
	Conversely, it is clear that for any $y$ in $\Cplusp^{\flat}$, we have that $\upsilon_{\frakp}^{\flat}(y) \geqslant 0$.
	Hence, we conclude that the equality in \eqref{eq:cplusp_intersect} holds.

	Now, recall that the $p\textrm{-typical}$ Witt vector functor is left exact since it is the right adjoint to the forgetful functor from the category of $\delta\textrm{-rings}$ to the category of rings (see \cite{joyal}).
	Therefore, all the arrows in the following natural $(\varphi, \GRhatp)\equivariant$ commutative diagram are injective:
	\begin{center}
		\begin{tikzcd}[row sep=15pt]
			\Ainf(\Cplusp) \arrow[r] \arrow[d] & \Ainf(\Cpplus) \arrow[d]\\
			W(\CC(\frakp)^{\flat}) \arrow[r] & W(\Cp^{\flat}).
		\end{tikzcd}
	\end{center}
	So, from \eqref{eq:cplusp_intersect} it follows that $\Ainf(\Cplusp) = W(\CC(\frakp)^{\flat}) \cap \Ainf(\Cpplus) \subset W(\Cp^{\flat})$.
\end{proof}

\begin{rem}\label{rem:xiainf_intersect}
	Using the injective ring homomorphism $\Ainf(\Cpplus) \rightarrow \Ainf(\Cpplus)$ from Lemma \ref{lem:cplusp_in_cpplus_ainf}, the fact that $\Ainf(\Cpplus)$ and $\Ainf(\Cpplus)$ are $\xi\textrm{-torsion}$ free (see \cite[Lemma 3.10]{bhatt-morrow-scholze-1}), the injective homomorphism $\Ainf(\Cplusp)/\xi \isomorphic \Cplusp \hookrightarrow \Cpplus \lisomorphic \Ainf(\Cpplus)/\xi$, and Lemma \ref{lem:injectivity_modulo}, it follows that we have $\xi \Ainf(\Cplusp) = \Ainf(\Cplusp) \cap \xi \Ainf(\Cpplus) \subset \Ainf(\Cpplus)$.
\end{rem}

\begin{rem}\label{rem:gr_act_prodainf}
	By the functoriality of the tilting construction and the Witt vector construction, we note that the action of $G_R$ on $\prod_{\pins} \Cplusp$ described after \eqref{eq:cplus_gequiv} (see \cite[Remarque 3.3.2]{brinon-relatif}), extends to respective natural actions of $G_R$ on $\prod_{\pins} \Ainf(\Cplusp)$ and $\prod_{\pins} W(\CC(\frakp)^{\flat})$.
\end{rem}

\begin{lem}\label{lem:ainf_intersection}
	In the notations described above, we have natural $(\varphi, G_R)\equivariant$ injective homomorphisms $\Ainf(\Rbar) \rightarrow \prod_{\pins} \Ainf(\Cplusp)$ and $W(\CC(\Rbar)^{\flat}) \rightarrow \prod_{\pins} W(\CC(\frakp)^{\flat})$, where the right-hand terms are equipped with a $G_R\action$ as described in Remark \ref{rem:gr_act_prodainf}.
	Moreover, we have a natural $(\varphi, G_R)\equivariant$ identification $\Ainf(\Rbar) = W(\CC(\Rbar)^{\flat}) \cap \prod_{\pins} \Ainf(\Cplusp)$ as subrings of $\prod_{\pins} W(\CC(\frakp)^{\flat})$.
\end{lem}
\begin{proof}
	From \eqref{eq:cplus_gequiv} recall that we have an injective homomorphism $\CC^+(\Rbar) \rightarrow \prod_{\pins} \Cplusp$.
	By applying the tilting functor, we further obtain a natural $(\varphi, G_R)\equivariant$ commutative diagram:
	\begin{equation}\label{eq:cc_in_prod_cp}
		\begin{tikzcd}[row sep=15pt]
			\CC^+(\Rbar)^{\flat} \arrow[r] \arrow[d] & \prod_{\pins} \Cplusp^{\flat} \arrow[d]\\
			\CC(\Rbar)^{\flat} \arrow[r] & \big(\prod_{\pins} \Cplusp^{\flat}\big)\big[\tfrac{1}{p^{\flat}}\big] \arrow[r] & \prod_{\pins} \CC(\frakp)^{\flat},
		\end{tikzcd}
	\end{equation}
	where the bottom right horizontal arrow and the vertical arrows are injective.
	From the injectivity of $\Rbar/p \rightarrow \prod_{\pins} \Cplusp/p$ (see Lemma \ref{lem:cplus_gequiv}) and the left exactness of $\lim_{\varphi}$, we obtain that in \eqref{eq:cc_in_prod_cp} the top horizontal arrow is injective and since we have $\CC(\Rbar)^{\flat} = \CC^+(\Rbar)^{\flat}[1/p^{\flat}]$, it also follows that the bottom left horizontal arrow is injective.
	Now, note that we have isomorphisms of multiplicative monoids $\CC(\Rbar)^{\flat} \isomorphic \lim_{x \mapsto x^p} \CC(\Rbar)$ (resp.\ $\CC^+(\Rbar)^{\flat} \isomorphic \lim_{x \mapsto x^p} \CC^+(\Rbar)$).
	Let $\sharp \colon \CC(\Rbar)^{\flat} \rightarrow \CC(\Rbar)$ denote the projection to the first component, denoted as $y \mapsto y^{\sharp}$, and note that we also have a compatible projection map for $\CC^+(\Rbar)^{\flat}$.
	This map is compatible with the $\sharp\textrm{-map}$ defined over $\Cp^{\flat}$ in the proof of Lemma \ref{lem:cplusp_in_cpplus_ainf}.
	Moreover, let $\upsilon_{\frakp}^{\flat}$ denote the valuation on $\Cp^{\flat}$ introduced in the proof of Lemma \ref{lem:cplusp_in_cpplus_ainf}.
	Then, under the composition of the left vertical and the bottom horizontal arrows of \eqref{eq:cc_in_prod_cp}, we see that for any $y$ in $\CC(\Rbar)^{\flat}$ we have that if $\upsilon_{\frakp}^{\flat}(y) \geqslant 0$ for each $\pins$, then $y^{\sharp}$ is an element of $\CC^+(\frakp)$ for each $\pins$.
	Using the following identification from Lemma \ref{lem:cplus_gequiv},
	\begin{equation*}
		\CC^+(\Rbar) = \CC(\Rbar) \cap \textstyle\prod_{\pins} \Cplusp \subset \big(\textstyle\prod_{\pins} \Cplusp\big)[1/p],
	\end{equation*}
	it follows that $y^{\sharp}$ is in $\CC^+(\Rbar)$.
	Write $y = (y_n)_{n \in \NN}$ in $\lim_{x \mapsto x^p} \CC(\Rbar)$, with $y_0 = y^{\sharp}$ in $\CC^+(\Rbar)$ and $y_{n+1}^p = y_n$, for $n \geqslant 1$.
	Then, using the identification $\CC^+(\Rbar) = \CC(\Rbar) \cap \prod_{\pins} \Cplusp$ and induction on $n \geqslant 0$, it follows that $y_n$ must be in $\CC^+(\Rbar)$, for all $n \in \NN$.
	Thus, $y$ is an element of $\CC^+(\Rbar)^{\flat}$.
	Conversely, it is clear that for any $y$ in $\CC^+(\Rbar)^{\flat}$, we have that $\upsilon_{\frakp}^{\flat}(y) \geqslant 0$.
	In particular, we conclude that the following equality holds:
	\begin{equation}\label{eq:tilt_intersect_injective}
		\CC^+(\Rbar)^{\flat} = \CC(\Rbar)^{\flat} \cap \textstyle\prod_{\pins} \Cplusp^{\flat} \subset \prod_{\pins} \CC(\frakp)^{\flat}.
	\end{equation}
	Furthermore, recall that the $p\textrm{-typical}$ Witt vector functor is left exact since it is right adjoint to the forgetful functor from the category of $\delta\textrm{-rings}$ to the category of rings (see \cite{joyal}).
	Therefore, all the arrows in the following natural $(\varphi, G_R)\equivariant$ commutative diagram are injective:
	\begin{center}
		\begin{tikzcd}[row sep=16pt]
			\Ainf(\Rbar) \arrow[r] \arrow[d] & \prod_{\pins} \Ainf(\Cplusp) \arrow[d]\\
			W(\CC(\Rbar)^{\flat}) \arrow[r] & \prod_{\pins} W(\CC(\frakp)^{\flat}).
		\end{tikzcd}
	\end{center}
	So, from \eqref{eq:tilt_intersect_injective} we obtain that $\Ainf(\Rbar) = W(\CC(\Rbar)^{\flat}) \cap \prod_{\pins} \Ainf(\Cplusp)$ as subrings inside $\prod_{\pins} W(\CC(\frakp)^{\flat})$.
\end{proof}

\subsection{de Rham period rings}\label{subsec:derham_relative}

In this section, we will recall the de Rham period rings from \cite[Chapitre 5]{brinon-relatif}, \cite[Section 2.1]{abhinandan-relative-wach-i} and \cite[Section 2]{brinon-imparfait}.
We start by noting that the $\Gamma_R\equivariant$ surjective homomorphism $\theta \colon \Ainf(\Rinfty) \twoheadrightarrow \widehat{R}_{\infty}$ described in Section \ref{subsec:ainf_relative} extends to a $\Gamma_R\equivariant$ surjective homomorphism $\theta \colon \Ainf(\Rinfty)[1/p] \twoheadrightarrow \widehat{R}_{\infty}[1/p]$, and its kernel is principal and generated by $\xi$.
We set
\begin{equation*}
	\BdR^+(\Rinfty) \coloneq \lim_n (\Ainf(\Rinfty)[1/p])/(\kert \theta)^n.
\end{equation*}
Note that $t \coloneq \log(1+\mu)$ converges in $\BdR^+(\Rinfty)$ and the latter is $t\textrm{-torsion free}$; we set $\BdR(\Rinfty) \coloneq \BdR^+(\Rinfty)[1/t]$.
Furthermore, we have big period rings $\OBdR^+(\Rinfty)$ and $\OBdR(\Rinfty)$.
These rings are equipped with a $\Gamma_R\action$, a natural extension of the map $\theta$ (before inverting $t$) and a $\Gamma_R\textrm{-stable}$ decreasing, separated and exhaustive filtration (see \cite[p.\ 52--54]{brinon-relatif}).
Moreover, the ring $\OBdR(\Rinfty)$ is further equipped with a $\BdR(\Rinfty)\linear$ and $\Gamma_R\equivariant$ integrable connection $\partial$ satisfying Griffiths transversality with respect to the filtration; we equip $\OBdR^+(\Rinfty)$ with the induced connection (see \cite[Section 5.3]{brinon-relatif}).
From loc.\ cit., we have that $\OBdR^+(\Rinfty)^{\partial=0} = \OBdR(\Rinfty)^{\partial=0} = R[1/p]$.
Note that we have natural variations of these constructions over $\Rbar$ as well.

Next, let $\pazs$ denote the set of minimal primes of $\Rbar$ above $pR \subset R$ as in Section \ref{subsec:localisation}.
Similar to above, for each $\frakp \in \pazs$, we set 
\begin{equation*}
	\begin{aligned}
		\BdR^+(\Cpplus) &\coloneq \lim_n (\Ainf(\Cpplus)[1/p])/(\kert \theta)^n\\
		(\textrm{resp. } \BdR^+(\Cplusp) &\coloneq \lim_n (\Ainf(\Cplusp)[1/p])/(\kert \theta)^n),
	\end{aligned}
\end{equation*}
and set $\BdR(\Cpplus) \coloneq \BdR^+(\Cpplus)[1/t]$ (resp.\ $\BdR(\Cplusp) \coloneq \BdR^+(\Cplusp)[1/t]$) equipped with a $\GRhatp\action$ (resp.\ $\GRp\action$), a natural extension of the map $\theta$ (before inverting $t$) and a $\GRhatp\textrm{-stable}$ (resp.\ $\GRp\textrm{-stable}$) decreasing, exhaustive and separated filtration given as $\Fil^k \BdR(\Cpplus) \coloneq t^k \BdR(\Cpplus)$ (resp.\ $\Fil^k \BdR(\Cplusp) \coloneq t^k \BdR(\Cplusp)$), for each $k \in \ZZ$.

Moreover, for each $\frakp \in \pazs$, note that we have big de Rham period rings $\OBdR^+(\Cpplus)$ and $\OBdR(\Cpplus)$ equipped with an $L\linear$ $\GRhatp\action$, a natural extension of the map $\theta$ (before inverting $t$), a $\GRhatp\textrm{-stable}$ decreasing, separated and exhaustive filtration and a $\GRhatp\equivariant$ integrable connection satisfying Griffiths transversality with respect to the filtration (see \cite[Section 2]{brinon-imparfait}).
Additionally, from \cite[Proposition 2.9]{brinon-imparfait}, we see that the natural $\GRhatp\equivariant$ injective homomorphism of rings $\BdR^+(\Cpplus) \rightarrow \OBdR^+(\Cpplus)$ extends to a $\BdR^+(\Cpplus)\linear$ isomorphism
\begin{equation}\label{eq:bdrcp+_powerseries}
	\BdR^+(\Cpplus)\llbracket T_1, \ldots, T_d \rrbracket \isomorphic \OBdR^+(\Cpplus), \qquad T_i \longmapsto X_i - [X_i^{\flat}].
\end{equation}

Analogously, following the construction of big de Rham period rings in \cite[Chapitre 5]{brinon-relatif}, for each $\pins$, we can construct period rings for the perfectoid $L\algebra$ $\CC(\frakp)$.
In particular, we have big de Rham period rings $\OBdR^+(\Cplusp)$ and $\OBdR(\Cplusp)$ equipped with an $L\linear$ $\GRp\action$, a natural extension of the map $\theta$ (before inverting $t$ and denoted $\theta_L$), a $\GRp\textrm{-stable}$ decreasing, separated and exhaustive filtration and a $\GRp\equivariant$ integrable connection satisfying Griffiths transversality with respect to the filtration.
Additionally, note that we have the natural $\GRp\equivariant$ injective homomorphism of rings $\BdR^+(\Cplusp) \rightarrow \OBdR^+(\Cplusp)$ and by employing arguments similar to \cite[Proposition 2.9]{brinon-imparfait} and \cite[Proposition 5.2.2]{brinon-relatif} (also see \cite[Proposition 6.10]{scholze-rigid} and \cite[p.\ 2]{scholze-rigid-erratum}), we see that the preceding homomorphism extends to a $\BdR^+(\Cplusp)\linear$ isomorphism
\begin{equation}\label{eq:bdrc+p_powerseries}
	\BdR^+(\Cplusp)\llbracket T_1, \ldots, T_d \rrbracket \isomorphic \OBdR^+(\Cplusp), \qquad T_i \longmapsto X_i - [X_i^{\flat}].
\end{equation}

Using the isomorphism \eqref{eq:bdrc+p_powerseries}, we may describe the supplementary structures on the ring $\OBdR^+(\Cplusp)$ from the previous paragraph, more explicitly.
Indeed, let us first note that from the isomorphism \eqref{eq:bdrc+p_powerseries} (see \cite[Proposition 2.9]{brinon-imparfait} and \cite[Proposition 5.2.2]{brinon-relatif}), it follows that $\kert \theta_L = (t, X_1 - [X_1^{\flat}], \ldots, X_d - [X_d^{\flat}]) \subset \OBdR(\Cplusp)$.
Then, the filtration on $\OBdR^+(\Cplusp)$ is given by powers of $\kert \theta_L$, i.e.\ for each $k \in \NN$, we have
\begin{equation*}
	\Fil^k \OBdR^+(\Cplusp) = (\kert \theta_L)^k = \big(t, X_1 - [X_1^{\flat}], \ldots, X_d - [X_d^{\flat}]\big)^k \subset \OBdR^+(\Cplusp),
\end{equation*}
which is clearly decreasing, separated and exhaustive, and by employing arguments similar to \cite[Proposition 5.2.5]{brinon-relatif}, we see that $\gr^{\bullet}\OBdR^+(\Cplusp) \isomorphic \Cplusp[u_1, \ldots, u_d, t]$, where the grading is given by the degree of $t$ and $u_i$ denotes the image of $X_i-[X_i^{\flat}]$ in $\gr^1 \OBdR^+(\Cplusp)$.
Next, note that we have $\OBdR(\Cplusp) \coloneq \OBdR^+(\Cplusp)[1/t]$, and the filtration on it is given by setting
\begin{equation*}
	\Fil^0 \OBdR(\Cplusp) \coloneq \textstyle\sum_{n \in \NN} t^{-n} \Fil^n \OBdR^+(\Cplusp),
\end{equation*}
and $\Fil^k \OBdR(\Cplusp) \coloneq t^k \Fil^0 \OBdR(\Cplusp)$, for each $k \in \ZZ$ (see \cite[p.\ 52]{brinon-relatif}).
Then, by employing arguments similar to \cite[Proposition 5.2.6]{brinon-relatif}, we see that $\gr^{\bullet}\OBdR(\Cplusp) \isomorphic \Cplusp[z_1, \ldots, z_d, t^{\pm 1}]$, where the grading is given by the degree of $t$, and $z_i$ denotes the image of $(X_i-[X_i^{\flat}])/t$ in $\gr^0 \OBdR(\Cplusp)$.
Additionally, using arguments similar to \cite[Proposition 5.2.8 \& Corollaire 5.2.9]{brinon-relatif}, we see that the filtration on $\OBdR(\Cplusp)$ is separated and exhaustive.

Furthermore, from \cite[Section 5.3]{brinon-relatif}, note that using the isomorphism \eqref{eq:bdrc+p_powerseries}, there exists a unique and $(\kert \theta_L)\adic$ continuous $\BdR^+(\Cplusp)\linear$ derivation $N_i \colon \OBdR^+(\Cplusp) \rightarrow \OBdR^+(\Cplusp)$, for each $1 \leqslant i \leqslant d$, and such that $N_i(X_j - [X_j^{\flat}]) = \delta_{i,j}X_j$, where $\delta_{i,j}$ is the Kronecker delta symbol.
As $N_i(t) = 0$, for each $1 \leqslant i \leqslant d$, the derivation $N_i$ extends to $\OBdR(\Cplusp)$.
Using this, $\OBdR(\Cplusp)$ is equipped with a $\BdR(\Cplusp)\linear$ connection as follows (see \cite[p.\ 56]{brinon-relatif}):
\begin{equation*}
	\partial \colon \OBdR(\Cplusp) \longrightarrow \OBdR(\Cplusp) \otimes_{O_L} \Omega^1_{O_L/O_F}, \qquad x \longmapsto \textstyle\sum_{i=1}^d N_i(x) \otimes \tfrac{dX_i}{X_i}.
\end{equation*}
The connection $\partial$ is integrable, $\GRp\equivariant$ (see \cite[Proposition 5.3.1]{brinon-relatif}) and satisfies Griffiths transversality with respect to the filtration (see \cite[Proposition 5.3.9]{brinon-relatif}).

\begin{lem}\label{lem:cplusp_in_cpplus_obdr}
	The natural $\GRhatp\equivariant$ injective homomorphism of rings $\Ainf(\Cplusp) \rightarrow \Ainf(\Cpplus)$ from Lemma \ref{lem:cplusp_in_cpplus_ainf} extends to a natural $\GRhatp\equivariant$ injective homomorphism of rings $\BdR(\Cplusp) \rightarrow \BdR(\Cpplus)$ compatible with the respective filtrations.
	Moreover, the preceding natural $\GRhatp\equivariant$ homomorphism further extends to a natural $L\linear$ and $\GRhatp\equivariant$ injective homomorphism of rings $\OBdR(\Cplusp) \rightarrow \OBdR(\Cpplus)$ compatible with the respective filtrations and connections.
\end{lem}
\begin{proof}
	To prove the claim, let us consider the following diagram:
	\begin{equation}\label{eq:cplusp_in_cpplus_obdr}
		\begin{tikzcd}[row sep=15pt]
			\Ainf(\Cplusp) & \BdR(\Cplusp) & \BdR^+(\Cplusp)\llbracket T_1, \ldots, T_d \rrbracket[1/t] & \OBdR(\Cplusp) \\
			\Ainf(\Cpplus) & \BdR(\Cpplus) & \BdR^+(\Cpplus)\llbracket T_1, \ldots, T_d \rrbracket[1/t] & \OBdR(\Cpplus).
			\arrow[from=1-1, to=1-2]
			\arrow[from=1-1, to=2-1]
			\arrow[from=1-2, to=1-3]
			\arrow[from=1-2, to=2-2]
			\arrow["\eqref{eq:bdrc+p_powerseries}", "\sim"', from=1-3, to=1-4]
			\arrow[from=1-3, to=2-3]
			\arrow[from=1-4, to=2-4]
			\arrow[from=2-1, to=2-2]
			\arrow[from=2-2, to=2-3]
			\arrow["\eqref{eq:bdrcp+_powerseries}", "\sim"', from=2-3, to=2-4]
		\end{tikzcd}
	\end{equation}
	where the unlabelled horizontal arrows are the natural homomorphisms.
	By definition, we see that the middle horizontal arrow in the top (resp.\ bottom) row is injective.
	Also, the top (resp.\ bottom) left horizontal arrow is injective because we have
	\begin{equation*}
		\begin{aligned}
			\Ainf(\Cplusp) &\longhookrightarrow \Ainf(\Cplusp)[1/p] \longhookrightarrow \BdR^+(\Cplusp)\\
			\textrm{ (resp.\ } \Ainf(\Cpplus) &\longhookrightarrow \Ainf(\Cpplus)[1/p] \longhookrightarrow \BdR^+(\Cpplus)),
		\end{aligned}
	\end{equation*}
	where the last inclusion follows because if we take $x$ to be an element of $\cap_{n \in \NN} \xi^n\Ainf(\Cplusp)[1/p]$ (resp.\ $\cap_{n \in \NN} \xi^n\Ainf(\Cpplus)[1/p]$) and $i \geqslant 0$ such that $p^ix$ is in $\Ainf(\Cplusp)$ (resp.\ $\Ainf(\Cpplus)$), then using that the sequence $\{\xi, p\}$ is regular on $\Ainf(\Cplusp)$ (resp.\ $\Ainf(\Cpplus)$), we get that $p^ix$ must be in $\cap_{n \in \NN} \xi^n\Ainf(\Cplusp) = 0$ (resp.\ $\cap_{n \in \NN} \xi^n\Ainf(\Cpplus) = 0$), i.e.\ $x = 0$.

	Next, the vertical arrows in the first, second and fourth columns of \eqref{eq:cplusp_in_cpplus_obdr} are obtained by using the functoriality of the construction of $\Ainf$, $\BdR$ and $\OBdR$ (for the homomorphism $\Cplusp \rightarrow \Cpplus$), in particular, they are $\GRhatp\equivariant$.
	Note that by definition, the $\GRhatp\equivariant$ vertical arrow in the first column induces a $\GRhatp\equivariant$ homomorphism $\BdR^+(\Cplusp) \rightarrow \BdR^+(\Cpplus)$.
	Then, from Remark \ref{rem:xiainf_intersect} and the fact that $\lim_n$ is left exact, we get that the map $\BdR^+(\Cplusp) \rightarrow \BdR^+(\Cpplus)$ is injective.
	As $\BdR^+(\Cpplus)$ and $\BdR^+(\Cplusp)$ are $t\textrm{-torsion}$ free (see \cite[Proposition 5.1.4]{brinon-relatif}), it follows that the vertical arrows in the second and the third column above are obtained from the homomorphism $\BdR^+(\Cplusp) \rightarrow \BdR^+(\Cpplus)$ in an obvious manner, and thus both the arrows are injective.
	By the description of the filtration on the objects of the second column, it is also clear that the vertical arrow in the second column is compatible with filtrations.
	This shows the first claim.
	As the diagram in \eqref{eq:cplusp_in_cpplus_obdr} commutes by definition, therefore, from the preceding discussion we conclude that the right vertical arrow is also injective.
	Furthermore, from the explicit description of the filtration and the connection on the top right object and the explicit description of the filtration and the connection on the bottom right object (see \cite[Section 2.2]{brinon-imparfait}), and an easy diagram chase shows that the right vertical arrow is compatible with the respective filtrations and connections.
	This proves the second claim.
\end{proof}

\begin{rem}\label{rem:gr_act_prodbdr}
	From diagram \eqref{eq:cplusp_in_cpplus_obdr} in Lemma \ref{lem:cplusp_in_cpplus_obdr}, note that we have natural $\GRhatp\equivariant$ injective homomorphisms of rings $\Ainf(\Cplusp) \rightarrow \BdR(\Cplusp) \rightarrow \OBdR(\Cplusp)$, for each $\pins$.
	As product is an exact functor on the category of abelian groups, therefore, the preceding natural homomorphisms, extend to natural injective homomorphisms
	\begin{equation}\label{eq:gr_act_prodbdr}
		\textstyle\prod_{\pins} \Ainf(\Cplusp) \longrightarrow \textstyle\prod_{\pins} \BdR(\Cplusp) \longrightarrow \textstyle\prod_{\pins} \OBdR(\Cplusp).
	\end{equation}
	Moreover, by using an argument similar to \cite[Remarque 3.3.2]{brinon-relatif} we see that the products $\prod_{\pins} \BdR(\Cplusp)$ and $\prod_{\pins} \OBdR(\Cplusp)$ may, respectively, be equipped with an action of $G_R$, extending the $G_R\action$ on $\prod_{\pins} \Ainf(\Cplusp)$ (see Remark \ref{rem:gr_act_prodainf}).
	In particular, we conclude that the natural injective homomorphisms in \eqref{rem:gr_act_prodbdr} are $G_R\equivariant$.
\end{rem}

\begin{lem}\label{lem:bdr_embedding}
	In the notations described above, we have a natural $R[1/p]\linear$ $G_R\equivariant$ injective homomorphism $\OBdR(\Rbar) \rightarrow \prod_{\pins} \OBdR(\Cplusp)$, where the right-hand term is equipped with a $G_R\action$ as described in Remark \ref{rem:gr_act_prodbdr}.
	Moreover, for each $\pins$, the induced natural $R[1/p]\linear$ $G_R\equivariant$ homomorphism $\OBdR(\Rbar) \rightarrow \OBdR(\Cplusp)$ is compatible with the respective filtrations and connections.
\end{lem}
\begin{proof}
	Note that from Lemma \ref{lem:ainf_intersection} and \eqref{eq:gr_act_prodbdr} in Remark \ref{rem:gr_act_prodbdr}, we have $G_R\equivariant$ injective homomorphisms:
	\begin{equation*}
		\Ainf(\Rbar) \longrightarrow \textstyle\prod_{\pins} \Ainf(\Cplusp) \longrightarrow \textstyle\prod_{\pins} \BdR(\Cplusp) \longrightarrow \textstyle\prod_{\pins} \OBdR(\Cplusp).
	\end{equation*}
	Then, from the definition of $\BdR(\Rbar)$ and $\OBdR(\Rbar)$, the preceding maps naturally induce a $G_R\equivariant$ commutative diagram:
	\begin{equation}\label{eq:bdr_embedding}
		\begin{tikzcd}[row sep=15pt]
			\BdR(\Rbar) \arrow[r] \arrow[d] & \prod_{\pins} \BdR(\Cplusp) \arrow[d] \arrow[r] & \BdR(\Cplusp) \arrow[d]\\
			\OBdR(\Rbar) \arrow[r] & \prod_{\pins} \OBdR(\Cplusp) \arrow[r] & \OBdR(\Cplusp),
		\end{tikzcd}
	\end{equation}
	where the vertical maps are injective, with the left and the right vertical arrows being compatible with the respective filtrations and connections.
	Moreover, let us note that from the explicit description of the filtration on $\BdR$ and $\OBdR$ (see \cite[Section 5.2]{brinon-relatif} and the discussion after \eqref{eq:bdrc+p_powerseries}), it easily follows that the respective compositions of the horizontal arrows in \eqref{eq:bdr_embedding} are compatible with filtrations and connections, i.e.\ for each $k \in \ZZ$, the respective images of $\Fil^k \BdR(\Rbar)$ and $\Fil^k \OBdR(\Rbar)$ are contained in $\Fil^k \BdR(\Cplusp)$ and $\Fil^k \OBdR(\Cplusp)$, under the composition of the horizontal arrows.
	Similarly, from the explicit description of the connection on $\OBdR$ (see \cite[Section 5.3]{brinon-relatif} and the discussion before Lemma \eqref{lem:cplusp_in_cpplus_obdr}), it easily follows that the composition of the bottom horizontal arrows in \eqref{eq:bdr_embedding} is further compatible with the respective connections, for each $\pins$.

	Now, for the diagram in \eqref{eq:bdr_embedding}, we will show that the bottom left horizontal arrow is injective, which will also imply the injectivity of the top left horizontal arrow (our argument is similar to \cite[Proposition 6.2.6]{brinon-relatif}).
	Note that the filtration on $\OBdR(\Rbar)$ is separated (see \cite[Corollaire 5.2.9]{brinon-relatif}), and similarly the filtration on $\OBdR(\Cplusp)$ is separated for each $\pins$ (see the discussion after \eqref{eq:bdrc+p_powerseries}).
	Therefore, to obtain the claim it is enough to show that the induced map on the grading of the filtration is injective.
	From \cite[Proposition 5.2.6]{brinon-relatif} recall that $\gr^{\bullet} \OBdR(\Rbar) \isomorphic \CC^+(\Rbar)[z_1, \ldots, z_d, t^{\pm 1}]$, where $z_i$ denotes the image of $(X_i-[X_i]^{\flat})/t$ in $\gr^0 \OBdR(\Rbar) \isomorphic \CC^+(\Rbar)[z_1, \ldots, z_d]$.
	Similarly, we have that $\gr^{\bullet} \OBdR(\Cplusp) \isomorphic \Cplusp[z_1, \ldots, z_d, t^{\pm 1}]$, for each $\pins$.
	The claim now follows from injectivity of the natural map $\CC^+(\Rbar) \rightarrow \prod_{\pins} \Cplusp$ (see \eqref{eq:cplus_gequiv}).
	This concludes our proof.
\end{proof}

\begin{rem}\label{rem:bdr_embedding_l}
	Note that the natural $R[1/p]\linear$ and $G_R\equivariant$ injective homomorphism $\OBdR(\Rbar) \rightarrow \prod_{\pins} \OBdR(\Cplusp)$ of Lemma \ref{lem:bdr_embedding} extends to a natural $L\linear$ and $G_R\equivariant$ injective homomorphism $L \otimes_{R[1/p]} \OBdR(\Rbar) \rightarrow \prod_{\pins} \OBdR(\Cplusp)$ (see \cite[Proposition 6.2.6]{brinon-relatif}).
	Indeed, to obtain the injectivity of the preceding natural map, note that similar to the proof of Lemma \ref{lem:bdr_embedding}, it is enough to show that the induced map on the grading of the $L\linear$ extension of the filtration on the left and the natural filtration on the right is injective, i.e.\ we need to show that the following composition (in the notation of Lemma \ref{lem:bdr_embedding}) is injective:
	\begin{align*}
		L \otimes_{R[1/p]} \gr^{\bullet} \OBdR(\Rbar) &\isomorphic L \otimes_{R[1/p]} \CC^+(\Rbar)[z_1, \ldots, z_d, t^{\pm 1}]\\
		&\qquad \longrightarrow \textstyle\prod_{\pins} \Cplusp[z_1, \ldots, z_d, t^{\pm 1}] \isomorphic \prod_{\pins} \gr^{\bullet} \OBdR(\Cplusp).
	\end{align*}
	The only non-trivial part is the injectivity of the second homomorphism which follows from Remark \ref{rem:ollin_embed}, thus proving the claim.
	Furthermore, from Lemma \ref{lem:bdr_embedding}, it also follows that for each $\pins$, the induced natural map $L \otimes_{R[1/p]} \OBdR(\Rbar) \rightarrow \OBdR(\Cplusp)$ is compatible with the respective filtrations and connections, where the left-hand term is equipped with an integrable tensor product connection ($\partial \otimes 1 + 1 \otimes \partial$) and a filtration given as the $L\linear$ extension of the filtration on $\OBdR(\Rbar)$.
\end{rem}

\subsection{Crystalline period rings}\label{subsec:crystalline_relative}

In this section, we will recall crystalline period rings (see \cite[Chapitre 6]{brinon-relatif}, \cite[Section 2.2]{abhinandan-relative-wach-i} and \cite[Section 2]{brinon-imparfait}).
We set $\Acrys(\Rinfty) \coloneq \Ainf(\Rinfty)\langle \{\xi^k/k!\}_{k \in \NN} \rangle$, and we have that $t = \log(1+\mu)$ converges in $\Acrys(\OFinfty)$ and $\Acrys(\Rinfty)$ is $p\textrm{-torsion}$ free and $t\textrm{-torsion}$ free.
So, we set $\Bcrys^+(\Rinfty) \coloneq \Acrys(\Rinfty)[1/p]$ and $\Bcrys(\Rinfty) \coloneq \Bcrys^+(\Rinfty)[1/t]$.
Furthermore, one may define big crystalline period rings $\OAcrys(\Rinfty)$, $\OBcrys^+(\Rinfty)$ and $\OBcrys(\Rinfty)$.
All the rings defined above are equipped with a continuous action of $\Gamma_R$, a Frobenius endomorphism $\varphi$ and a natural extension of the map $\theta$ (before inverting $t$) (see \cite[Section 6.1]{brinon-relatif}).
Moreover, we have $\Gamma_R\equivariant$ natural injective homomorphisms (see \cite[Section 6.2]{brinon-relatif}),
\begin{align*}
	\Acrys(\Rinfty) &\hookrightarrow \Bcrys^+(\Rinfty) \hookrightarrow \Bcrys(\Rinfty) \hookrightarrow \BdR(\Rinfty),\\
	\OAcrys(\Rinfty) &\hookrightarrow \OBcrys^+(\Rinfty) \hookrightarrow \OBcrys(\Rinfty) \hookrightarrow \OBdR(\Rinfty),
\end{align*}
and we equip all the crystalline period rings above with a $\Gamma_R\textrm{-stable}$ decreasing, separated and exhaustive filtration induced from the filtration on the de Rham period rings (see \cite[p.\ 71]{brinon-relatif}).
Additionally, using loc.\ cit., we equip $\OBcrys(\Rinfty)$ with a $\Bcrys(\Rinfty)\linear$ and $\Gamma_R\equivariant$ integrable connection induced from the connection $\partial$ on $\OBdR(\Rinfty)$, which satisfies Griffiths transversality with respect to the filtration because the same is true for the connection on $\OBdR(\Rinfty)$; we also equip $\OBcrys^+(\Rinfty)$ and $\OAcrys(\Rinfty)$ with induced connections.
Then, we have that $\OBcrys^+(\Rinfty)^{\partial=0} = \OBcrys(\Rinfty)^{\partial=0} = R[1/p]$ and $\OAcrys(\Rinfty)^{\partial=0} = R$.
Note that we have natural variations of these constructions over $\Rbar$ as well, and from \cite[Corollary 4.34]{morrow-tsuji}, we have a natural $(\varphi, \Gamma_R)\equivariant$ isomorphism $\OAcrys(\Rinfty) \isomorphic \OAcrys(\Rbar)^{H_R}$.

As in Section \ref{subsec:localisation}, let $\pazs$ denote the set of minimal primes of $\Rbar$ above $pR \subset R$.
Similar to above, from \cite[Section 2.3 \& 2.4]{brinon-imparfait}, for each $\frakp \in \pazs$ we have rings $\Acrys(\Cpplus)$, $\Bcrys(\Cpplus)$, $\OAcrys(\Cpplus)$ and $\OBcrys(\Cpplus)$ equipped with a $\GRhatp\action$, a natural extension of the map $\theta$ (before inverting $t$) and a Frobenius endomorphism $\varphi$.
Additionally, from \cite[Proposition 2.39]{brinon-imparfait}, the natural injective homomorphism of rings $\Acrys(\Cpplus) \rightarrow \OAcrys(\Cpplus)$ extends to a $\Acrys(\Cpplus)\linear$ isomorphism
\begin{equation}\label{eq:acryscp+_powerseries}
	\Acrys(\Cpplus)[T_1, \ldots, T_d]_{\textrm{PD}}^{\wedge} \isomorphic \OAcrys(\Cpplus), \qquad T_i \mapsto X_i - [X_i^{\flat}],
\end{equation}
where the source denotes the $\padic$ completion of a PD-polynomial algebra over $\Acrys(\Cpplus)$ in variables $\{T_1, \ldots, T_d\}$.
Using the isomorphisms \eqref{eq:bdrcp+_powerseries} and \eqref{eq:acryscp+_powerseries}, in \cite[Proposition 2.42]{brinon-imparfait}, it has been shown that we have a natural $L\linear$ and $\GRhatp\equivariant$ injective homomorphism of rings $\OAcrys(\Cpplus) \hookrightarrow \OBdR^+(\Cpplus)$.
Using the preceding inclusion, we equip all crystalline period rings for $\Cpplus$ described above with an induced $\GRhatp\textrm{-stable}$ decreasing, separated and exhaustive filtration and an induced $\GRhatp\equivariant$ integrable connection (on rings with a prefix ``$\pazo$'') satisfying Griffiths transversality with respect to the filtration.

Analogously, following the construction of crystalline period rings in \cite[Chapitre 6]{brinon-relatif} for each $\pins$, we can  construct period rings for the perfectoid $L\algebra$ $\Cplusp$ (similar to the de Rham period rings for $\Cplusp$ described in Section \ref{subsec:derham_relative} using \cite[Chapitre 5]{brinon-relatif}).
Consequently, we have period rings $\Acrys(\Cplusp)$, $\Bcrys(\Cplusp)$, $\OAcrys(\Cplusp)$ and $\OBcrys(\Cplusp)$ equipped with a $\GRp\action$, a natural extension of the map $\theta$ (before inverting $t$) and a Frobenius endomorphism $\varphi$.
Moreover, by adapting the proof of \cite[Proposition 6.1.5]{brinon-relatif} to our setting, we see that similar to \eqref{eq:acryscp+_powerseries}, the natural injective homomorphism of rings $\Acrys(\Cplusp) \rightarrow \OAcrys(\Cplusp)$ extends to an $\Acrys(\Cplusp)\linear$ isomorphism
\begin{equation}\label{eq:acrysc+p_powerseries}
	\Acrys(\Cplusp)[T_1, \ldots, T_d]_{\textrm{PD}}^{\wedge} \isomorphic \OAcrys(\Cplusp), \qquad T_i \mapsto X_i - [X_i^{\flat}],
\end{equation}
where the source denotes the $\padic$ completion of a PD-polynomial algebra over $\Acrys(\Cplusp)$ in variables $\{T_1, \ldots, T_d\}$.
Now, using the isomorphisms \eqref{eq:bdrc+p_powerseries} and \eqref{eq:acrysc+p_powerseries}, we can adapt the proof of \cite[Proposition 6.2.1]{brinon-relatif} to obtain a natural $L\linear$ and $\GRp\equivariant$ injective homomorphism of rings $\OAcrys(\Cplusp) \hookrightarrow \OBdR^+(\Cplusp)$.
Using the preceding inclusion, we equip all the crystalline period rings for $\Cplusp$ described above with an induced $\GRp\textrm{-stable}$ decreasing, separated and exhaustive filtration and an induced $\GRp\equivariant$ integrable connection (on rings with a prefix ``$\pazo$'') satisfying Griffiths transversality with respect to the filtration.

\begin{lem}\label{lem:cplusp_in_cpplus_bcrys}
	For each $\pins$, the natural $(\varphi, \GRhatp)\equivariant$ injective homomorphism $\Ainf(\Cplusp) \rightarrow \Ainf(\Cpplus)$ from Lemma \ref{lem:cplusp_in_cpplus_ainf} extends to natural $(\varphi, \GRhatp)\equivariant$ and filtration compatible injective homomorphisms $\Bcrys(\Cplusp) \rightarrow \Bcrys(\Cpplus)$ and $\OBcrys(\Cplusp) \rightarrow \OBcrys(\Cpplus)$, where the latter is $L\linear$ and also compatible with the respective connections.
\end{lem}
\begin{proof}
	By definition, the $(\varphi, \GRhatp)\equivariant$ injective homomorphism $\Ainf(\Cplusp) \rightarrow \Ainf(\Cpplus)$ naturally extends to $(\varphi, \GRhatp)\equivariant$ homomorphisms $\Acrys(\Cplusp) \rightarrow \Acrys(\Cpplus)$ and $\OAcrys(\Cplusp) \rightarrow \OAcrys(\Cpplus)$, where the latter is $O_L\textrm{-linear}$ and compatible with the respective connections.
	Now, consider the following $\GRhatp\equivariant$ commutative diagram:
	\begin{center}
		\begin{tikzcd}[row sep=15pt]
			\Acrys(\Cplusp) \arrow[r] \arrow[d] & \OAcrys(\Cplusp) \arrow[r] \arrow[d] & \OBdR(\Cplusp) \arrow[d]\\
			\Acrys(\Cpplus) \arrow[r] & \OAcrys(\Cpplus) \arrow[r] & \OBdR(\Cpplus),
		\end{tikzcd}
	\end{center}
	where all the horizontal arrows are injective and compatible with respective filtrations (see the respective discussions after \eqref{eq:acrysc+p_powerseries} for the top right arrow and \eqref{eq:acryscp+_powerseries} for the bottom right arrow).
	Moreover, by Lemma \ref{lem:cplusp_in_cpplus_obdr}, the right vertical arrow is injective and compatible with the respective filtrations and connections.
	Therefore, it follows that the left and the middle vertical arrows are also injective and compatible with the respective filtrations and connections.
	Finally, the claims for $\Bcrys$ and $\OBcrys$ follow by inverting $t$ in the left and the middle columns of the diagram.
\end{proof}

\begin{rem}\label{rem:gr_act_prodbcrys}
	From Remark \ref{rem:gr_act_prodbdr} and Lemma \ref{lem:cplusp_in_cpplus_bcrys}, it is easy to see that the diagram of injective homomorphisms in \eqref{eq:gr_act_prodbdr} may be upgraded to the following diagram of injective homomorphisms:
	\begin{equation}\label{eq:gr_act_prodbcrys}
		\textstyle\prod_{\pins} \Ainf(\Cplusp) \longrightarrow \textstyle\prod_{\pins} \Bcrys(\Cplusp) \longrightarrow \textstyle\prod_{\pins} \OBcrys(\Cplusp) \longrightarrow \textstyle\prod_{\pins} \OBdR(\Cplusp),
	\end{equation}
	where the first two homomorphisms are compatible with the respective Frobenii.
	Moreover, by an argument similar to \cite[Remarque 3.3.2]{brinon-relatif}, we see that the products $\prod_{\pins} \Bcrys(\Cplusp)$ and $\prod_{\pins} \OBcrys(\Cplusp)$ are stable under the $G_R\action$ on $\prod_{\pins} \OBdR(\Cplusp)$ (see Remark \ref{rem:gr_act_prodbdr}) and we equip them with the induced action.
	Then, it follows that the injective homomorphims in \eqref{eq:gr_act_prodbcrys} are $G_R\equivariant$ as well.
\end{rem}

\begin{lem}\label{lem:bcrys_embedding}
	Using the notations described above, we have a natural $R[1/p]\linear$ and $(\varphi, G_R)\equivariant$ injective homomorphism $\OBcrys(\Rbar) \rightarrow \prod_{\pins} \OBcrys(\Cplusp)$, where the right-hand term is equipped with a $G_R\action$ as described in Remark \ref{rem:gr_act_prodbcrys}.
	Moreover, for each $\pins$, the induced natural $R[1/p]\linear$ $(\varphi, G_R)\equivariant$ homomorphism $\OBcrys(\Rbar) \rightarrow \OBcrys(\Cplusp)$ is compatible with the respective filtrations and connections.
\end{lem}
\begin{proof}
	From Lemma \ref{lem:ainf_intersection} and \eqref{eq:gr_act_prodbcrys} in Remark \ref{rem:gr_act_prodbcrys}, note that we have $(\varphi, G_R)\equivariant$ injective homomorphisms:
	\begin{equation*}
		\Ainf(\Rbar) \longrightarrow \textstyle\prod_{\pins} \Ainf(\Cplusp) \longrightarrow \textstyle\prod_{\pins} \OBcrys(\Cplusp).
	\end{equation*}
	Then, from the definition of $\OBcrys$ we see that the preceding homomorphisms naturally induce an $R[1/p]\linear$ and $(\varphi, G_R)\equivariant$ homomorphism $\OBcrys(\Rbar) \rightarrow \prod_{\pins} \OBcrys(\Cplusp)$.
	We will show that the latter homomorphism is injective (our argument is similar to \cite[Proposition 6.2.6]{brinon-relatif}).
	Let us consider the following natural $R[1/p]\linear$ $G_R\equivariant$ diagram:
	\begin{center}
		\begin{tikzcd}[row sep=15pt]
			\OBcrys(\Rbar) \arrow[r] \arrow[d] & \prod_{\pins} \OBcrys(\Cplusp) \arrow[d] \\
			\OBdR(\Rbar) \arrow[r] & \prod_{\pins} \OBdR(\Cplusp),
		\end{tikzcd}
	\end{center}
	where the left and the right vertical arrows are naturally injective (by the discussion at the beginning of this section), and the bottom arrow is injective by Lemma \ref{lem:bdr_embedding}.
	The diagram commutes since the top and the bottom horizontal arrows are naturally defined using the injective homomorphism $\Ainf(\Rbar) \rightarrow \prod_{\pins} \Ainf(\Cplusp)$ of Lemma \ref{lem:ainf_intersection}.
	So, it follows that the top horizontal arrow is injective, thus proving the first claim.
	Next, for each $\pins$, the induced natural $R[1/p]\linear$ homomorphism $\OBcrys(\Rbar) \rightarrow \OBcrys(\Cplusp)$ is tautologically compatible with the respective Frobenii and the claims on filtrations and connections follow from the corresponding claims on $\OBdR$ in Lemma \ref{lem:bdr_embedding}.
	Hence, the lemma is proved.
\end{proof}

\begin{rem}\label{lem:bcrys_embedding_l}
	Note that the natural $R[1/p]\linear$ and $(\varphi, G_R)\equivariant$ injective homomorphism $\OBcrys(\Rbar) \rightarrow \prod_{\pins} \OBcrys(\Cplusp)$ from Lemma \ref{lem:bcrys_embedding} extends to a natural $L\linear$ and $(\varphi, G_R)\equivariant$ injective homomorphism $L \otimes_{R[1/p]} \OBcrys(\Rbar) \rightarrow \prod_{\pins} \OBcrys(\Cplusp)$ (see from \cite[Proposition 6.2.6]{brinon-relatif}).
	Indeed, this follows from an argument similar to Lemma \ref{lem:bcrys_embedding} using Remark \ref{rem:bdr_embedding_l}.
	Furthermore, from Lemma \ref{lem:bcrys_embedding}, it also follows that the induced natural $L\linear$ and $(\varphi, G_R)\equivariant$ homomorphism $L \otimes_{R[1/p]} \OBcrys(\Rbar) \rightarrow \OBcrys(\Cplusp)$ is compatible with the respective filtrations and connections, where the left-hand term is equipped with the tensor product Frobenius, the $L\linear$ extension of the filtration on $\OBcrys(\Rbar)$ and an integrable tensor product connection ($\partial \otimes 1 + 1 \otimes \partial$).
\end{rem}

\subsection{Rings of \texorpdfstring{$(\varphi, \Gamma)\modules$}{-}}\label{subsec:phigamma_mod_rings}

Let us fix Teichm\"uller lifts $[X_i^{\flat}]$ in $\Ainf(\Rinfty)$, for $1 \leqslant i \leqslant d$, and let $\Aframe^+$ denote the $(p, \mu)\textrm{-adic}$ completion of $O_F[\mu, [X_1^{\flat}]^{\pm 1}, \ldots, [X_d^{\flat}]^{\pm 1}]$.
By defininition, there exists a natural injective homomorphism of $O_F\textrm{-algebras}$ $\Aframe^+ \hookrightarrow \Ainf(\Rinfty)$ and its image in the target is stable under the Witt vector Frobenius endomorphism $\varphi$ and the $\Gamma_R\action$ on $\Ainf(\Rinfty)$ (see \cite[Section 3]{abhinandan-relative-wach-i}); we equip $\Aframe^+$ with the induced $(\varphi, \Gamma_R)\action$.
Furthermore, note that we have an embedding $\iota \colon \Rframe \rightarrow \Aframe^+$ defined by the $O_F\linear$ homomorphism sending $X_i \mapsto [X_i^{\flat}]$, and it is easy to see that $\iota$ extends to an isomorphism of rings $\Rframe\llbracket \mu \rrbracket \isomorphic \Aframe^+$ (enough to check modulo $\mu$ since both the source and the target are $(p, \mu)\textrm{-adically}$ complete and $\mu\textrm{-torsion}$ free).
We extend the Frobenius endomorphism on $\Rframe$ to a Frobenius endomorphism $\varphi$ on $\Rframe\llbracket \mu \rrbracket$ by setting $\varphi(\mu) = (1+\mu)^p-1$.
Then, the Frobenius on $\Rframe\llbracket \mu \rrbracket$ is faithfully flat and finite of degree $p^{d+1}$ .
Moreover, by the preceding discussion, it also follows that the embedding $\iota$ and the isomorphism $\Rframe\llbracket \mu \rrbracket \isomorphic \Aframe^+$ are Frobenius-equivariant.

Let $\AR^+$ denote the $(p, \mu)\textrm{-adic}$ completion of the unique extension of the injective homomorphism $\Aframe^+ \hookrightarrow \Ainf(\Rinfty)$ along the $p\textrm{-adically}$ completed \'etale map $\Rframe \rightarrow R$ (see \cite[Section 3.3.2]{abhinandan-relative-wach-i} and \cite[Proposition 2.1]{colmez-niziol}).
Then, there exists a natural injective homomorphism of $\Aframe^+\textrm{-algebras}$ $\AR^+ \hookrightarrow \Ainf(\Rinfty)$.
By the $\padic$ \'etaleness of the map $\Rframe \rightarrow R$, the Frobenius $\varphi$ and the $\Gamma_R\action$ on $\Aframe^+$, respectively and uniquely, extend to an endomorphism $\varphi \colon \AR^+ \rightarrow \AR^+$, which is a lift of the absolute Frobenius on $\AR^+/p$, and a continuous action of $\Gamma_R$.
As the $(\varphi, \Gamma_R)\action$ on $\Aframe^+$ is induced from the $(\varphi, \Gamma_R)\action$ on $\Ainf(\Rinfty)$, the uniqueness of the extended $(\varphi, \Gamma_R)\action$ on $\AR^+$ implies that the homomorphism $\AR^+ \hookrightarrow \Ainf(\Rinfty)$ is $(\varphi, \Gamma_R)\equivariant$.
Moreover, the $O_F\linear$ injective homomorphism $\iota \colon \Rframe \rightarrow \Aframe^+ \hookrightarrow \AR^+$ and the isomorphism $\Rframe\llbracket \mu \rrbracket \isomorphic \Aframe^+ \hookrightarrow \AR^+$ naturally extend to a unique $O_F\linear$ injective homomorphism $\iota \colon R \rightarrow \AR^+$ and an isomorphism of rings $R\llbracket \mu \rrbracket \isomorphic \AR^+$.
We extend the Frobenius endomorphism on $R$ to a Frobenius endomorphism $\varphi$ on $R\llbracket \mu \rrbracket$ by setting $\varphi(\mu) = (1+\mu)^p-1$.
Then, the Frobenius on $R\llbracket \mu \rrbracket$ is faithfully flat and finite of degree $p^{d+1}$ .
Furthermore, by the preceding disucssion, it follows that the embedding $\iota$ and the isomorphism $R\llbracket \mu \rrbracket \isomorphic \AR^+$ are Frobenius-equivariant.
In particular, the induced Frobenius endomorphism $\varphi$ on $\AR^+$ is faithfully flat and finite of degree $p^{d+1}$, and we have that $\varphi^*(\AR^+) \coloneq \AR^+ \otimes_{\varphi, \AR^+} \AR^+ \isomorphic \oplus_{\alpha} \varphi(\AR^+) u_{\alpha}$, where $u_{\alpha} \coloneq (1+\mu)^{\alpha_0} [X_1^{\flat}]^{\alpha_1} \cdots [X_d^{\flat}]^{\alpha_d}$, for $\alpha = (\alpha_0, \alpha_1, \ldots, \alpha_d)$ with $\alpha_i \in \{0, 1, \ldots, p-1\}$.

Set $\AR \coloneq \AR^+[1/\mu]^{\wedge}$ as the $\padic$ completion.
Note that we have $\varphi(\mu) = (1+\mu)^p-1 = [p]_q\mu$, and for any $g \in \Gamma_R$, an easy computation shows that $g(\mu) = (1+\mu)^{\chi(g)}-1 = \chi(g) \mu v$, where $\chi$ is the $\padic$ cyclotomic character, so $\chi(g) \in \ZZ_p^{\times}$, and $v$ a unit in $\AR^+$.
Consequently, we see that the Frobenius endomorphism $\varphi$ and the continuous action of $\Gamma_R$ on $\AR^+$ naturally extend to $\AR$.
Similar to above, the induced Frobenius endomorphism $\varphi$ on $\AR$ is faithfully flat and finite of degree $p^{d+1}$ and $\varphi^*(\AR) \coloneq \AR \otimes_{\varphi, \AR} \AR \isomorphic \oplus_{\alpha} \varphi(\AR) u_{\alpha} = (\oplus_{\alpha} \varphi(\AR^+) u_{\alpha}) \otimes_{\varphi(\AR^+)} \varphi(\AR) \lisomorphic \AR^+ \otimes_{\varphi, \AR^+} \AR$.

Recall that $\CRbar = \CC^+(\Rbar)[1/p]$ and we set $\Atilde \coloneq W\big(\CRbar^{\flat}\big)$ and $\Btilde \coloneq \Atilde[1/p]$, equipped with the Frobenius on Witt vectors and a continuous (for the weak topology) action of $G_R$.
Moreover, the natural $(\varphi,\Gamma_R)\equivariant$ injective homomorphism $\AR^+ \hookrightarrow \Ainf(\Rinfty)$ extends to a natural $(\varphi, \Gamma_R)\equivariant$ injective homomorphism $\AR \hookrightarrow \Atilde^{H_R}$ and we set $\BR \coloneq \AR[1/p]$ equipped with the induced $(\varphi, \Gamma_R)\action$.
Take $A$ to be the $\padic$ completion of the maximal unramified extension of $\AR$ inside $\Atilde$ and set $B \coloneq A[1/p] \subset \Btilde$.
The rings $A$ and $B$ are stable under the action of $G_R$ and the Frobenius endomorphism $\varphi$ on $\Btilde$, and we equip $A$ and $B$ with the induced $(\varphi, G_R)\action$.
Moreover, from \cite[Proposition 7.8]{andreatta-phigamma}, we have that $\AR = A^{H_R}$ and $\BR = B^{H_R}$.
Next, let us set $A^+ \coloneq \Ainf(\Rbar) \cap A \subset \Atilde$ and $B^+ \coloneq A^+[1/p] \subset B$, and note that these rings are stable under the $(\varphi, G_R)\action$ on $B$.
Then, from the discussion above, it follows that we have $\AR^+ = (A^+)^{H_R}$ and $\BR^+ = (B^+)^{H_R}$.

Note that we have analogous rings for $O_L$ (see \cite[Section 2.1.2]{abhinandan-imperfect-wach}) and by identifying the groups $\Gamma_L \isomorphic \Gamma_R$, we have a $(\varphi, \Gamma_L)\equivariant$ isomorphism $\AL^+ \isomorphic ((\AR^+)_{(p, \mu)})^{\wedge}$, where ${}^{\wedge}$ denotes the $(p, \mu)\textrm{-adic}$ completion.
The preceding isomorphism extends to a $(\varphi, \Gamma_L)\equivariant$ isomorphism $\AL = \AL^+[1/\mu]^{\wedge} \isomorphic ((\AR)_{(p)})^{\wedge}$, where ${}^{\wedge}$ denotes the $\padic$ completion.
\begin{lem}\label{lem:ar_intersect_al+}
	As subrings of $\AL$, we have that $\AR^+ = \AR \cap \AL^+$, and by inverting $p$, we get that $\BR^+ \coloneq \AR^+[1/p] = \BL^+ \cap \BR$ as subrings of $\BL$.
\end{lem}
\begin{proof}
	From the preceding discussion, we have that $\AR^+ \isomorphic R\llbracket \mu \rrbracket$ as $O_F\textrm{-algebras}$, and therefore, $\AR \isomorphic R\llbracket \mu \rrbracket[1/\mu]^{\wedge}$, where ${}^{\wedge}$ denotes the $\padic$ completion.
	Moreover, we have that $\AL^+ \isomorphic O_L\llbracket \mu \rrbracket$ as $O_F\textrm{-algebras}$ (see \cite[Section 2.1.2]{abhinandan-imperfect-wach}).
	So, to get the claim, it is enough to show that 
	\begin{equation*}
		R\llbracket \mu \rrbracket = R\llbracket \mu \rrbracket[1/\mu]^{\wedge} \cap O_L\llbracket \mu \rrbracket \subset O_L\llbracket \mu \rrbracket[1/\mu]^{\wedge}.
	\end{equation*}
	Note that any element $x$ in $R\llbracket \mu \rrbracket[1/\mu]^{\wedge}$ has a unique presentation as $x = \sum_{k \in \ZZ} a_k \mu^k$, with $a_k$ in $R$ and $p\textrm{-adically}$ $a_k \rightarrow 0$ as $k \rightarrow -\infty$.
	Now, as $R \hookrightarrow O_L$, therefore, this presentation is also unique in $O_L\llbracket \mu \rrbracket[1/\mu]^{\wedge}$ and we see that $x$ is in $O_L\llbracket \mu \rrbracket$, if and only if, $a_k = 0$ for all $k < 0$, and $a_k$ is in $O_L$ for all $k \geqslant 0$.
	So, it follows that $x$ is in $R\llbracket \mu \rrbracket[1/\mu]^{\wedge} \cap O_L\llbracket \mu \rrbracket$, if and only if, $x = \sum_{k \in \ZZ} a_k \mu^k$, with $a_k = 0$ for $k < 0$, and $a_k$ is in $R$ for $k \geqslant 0$, i.e.\ $x$ is in $R\llbracket \mu \rrbracket$.
	This proves the lemma.
\end{proof}

To summarise, we have the following $(\varphi, G_R)\equivariant$ commutative diagram with injective arrows:
\begin{center}
	\begin{tikzcd}[row sep=small]
		\AL^+ & \AR^+ && A^+ \\
		&& \Ainf(\Rinfty) && \Ainf(\Rbar) \\
		\AL & \AR && A \\
		&& \Atilde^{H_R} && \Atilde.
		\arrow[from=1-1, to=3-1]
		\arrow[from=1-2, to=1-1]
		\arrow[from=1-2, to=1-4]
		\arrow[from=1-2, to=2-3]
		\arrow[from=1-2, to=3-2]
		\arrow[from=1-4, to=2-5]
		\arrow[from=1-4, to=3-4]
		\arrow[crossing over, from=2-3, to=2-5]
		\arrow[from=2-3, to=4-3]
		\arrow[from=2-5, to=4-5]
		\arrow[from=3-2, to=3-1]
		\arrow[crossing over, from=3-2, to=3-4]
		\arrow[from=3-2, to=4-3]
		\arrow[from=3-4, to=4-5]
		\arrow[from=4-3, to=4-5]
	\end{tikzcd}
\end{center}
Note that the second (resp.\ third) column coincides the $H_R\textrm{-invariants}$ of the fourth (resp.\ fifth) column.
Moreover, after inverting $p$, we have an anologous commutative diagram with injective arrows.

\subsection{\texorpdfstring{$\padic$}{-} representations}\label{subsec:relative_padicreps}

Let $T$ be a finite free $\ZZ_p\representation$ of $G_R$.
By the theory of \'etale $(\varphi, \Gamma)\modules$ (see \cite{fontaine-phigamma} and \cite{andreatta-phigamma}), one can functorially associate to $T$ a finite projective étale $(\varphi, \Gamma_R)\module$ $\DR(T) \coloneq (A \otimes_{\ZZ_p} T)^{H_R}$ over $\AR$ of rank $= \rank_{\ZZ_p} T$.
Moreover, $\Dtilde_R(T) \coloneq (\Atilde \otimes_{\ZZ_p} T)^{H_R} \isomorphic \Atilde^{H_R} \otimes_{\AR} \DR(T)$ and we have a natural $(\varphi, G_R)\equivariant$ isomorphism
\begin{equation}\label{eq:phigamma_comp_iso_relative}
	A \otimes_{\AR} \DR(T) \isomorphic A \otimes_{\ZZ_p} T.
\end{equation}
These constructions are functorial in $\ZZ_p\textrm{-representations}$ and induce an exact equivalence of $\otimes\textrm{-categories}$ (see \cite[Theorem 7.11]{andreatta-phigamma}),
\begin{equation}\label{eq:rep_phigamma_relative}
	\Rep_{\ZZ_p}(G_R) \isomorphic (\varphi, \Gamma_R)\Mod_{\AR}^{\etale},
\end{equation}
with an exact $\otimes\textrm{-compatible}$ quasi-inverse given as $\TR(D) \coloneq (A \otimes_{\AR} D)^{\varphi=1}$, and similar to \eqref{eq:phigamma_comp_iso_relative} we have a natural $(\varphi, G_R)\equivariant$ isomorphism $A \otimes_{\ZZ_p} \TR(D) \isomorphic A \otimes_{\AR} D$.
By extending scalars of the preceding isomorphism along the $(\varphi, G_R)\equivariant$ homomorphism $A \rightarrow \Atilde$ and taking ``$\varphi=1$'' of the resulting isomorphism yields a $G_R\equivariant$ isomorphism $\TR(D) \isomorphic (\Atilde \otimes_{\AR} D)^{\varphi=1}$, since $\tilde{A}^{\varphi=1} = \mathbb{Z}_p$ by the Artin--Schreier theory (see \cite[Proposition B.1]{andreatta-iovita-phigamma}).
Similar statements are also true for $\padic$ representations of $G_R$.
Furthermore, we set $\DR^+(T) \coloneq (A^+ \otimes_{\ZZ_p} T)^{H_R}$ as the $(\varphi, \Gamma_R)\module$ over $\AR^+$ associated to $T$ and for $V \coloneq T[1/p]$ set $\DR^+(V) \coloneq \DR^+(T)[1/p]$ as the $(\varphi, \Gamma_R)\module$ over $\BR^+$ associated to $V$.

Let $V$ be a $\padic$ representation of $G_R$.
From the $\padic$ Hodge theory of $G_R$ (see \cite[Chapitre 8]{brinon-relatif}), we can attach to $V$ a filtered $(\varphi, \partial)\module$ over $R[1/p]$, finite projective of rank $ \leqslant \dim_{\QQ_p} V$, and described via the functor
\begin{align*}
	\ODcrysR \colon \Rep_{\QQ_p}(G_R) &\longrightarrow \MF_R(\varphi, \partial)\\
		V &\longmapsto (\OBcrys(\Rbar) \otimes_{\QQ_p} V)^{G_R}.
\end{align*}
The representation $V$ is said to be crystalline if the following natural homomorphism:
\begin{equation}\label{eq:crys_comp}
	\OBcrys(\Rbar) \otimes_{R[1/p]} \ODcrysR(V) \longrightarrow \OBcrys(\Rbar) \otimes_{\QQ_p} V,
\end{equation}
is an isomorphism, in particular, if $V$ is crystalline then $\rank_{R[1/p]} \ODcrysR(V) = \dim_{\QQ_p} V$.
Restricting the functor $\ODcrysR$ to the category of crystalline representations of $G_R$ and writing $\MF_R\ad(\varphi, \partial)$ for the essential image of the restricted functor, we obtain an exact equivalence of $\otimes\textrm{-categories}$ (see \cite[Th\'eor\`eme 8.5.1]{brinon-relatif}):
\begin{equation}\label{eq:odcrysr_func}
	\ODcrysR \colon \Rep_{\QQ_p}^{\crys}(G_R) \isomorphic \MF_R\ad(\varphi, \partial),
\end{equation}
with an exact $\otimes\textrm{-compatible}$ quasi-inverse functor given as $\OVcrysR(D) \coloneq (\Fil^0(\OBcrys(\Rbar) \otimes_{R[1/p]} D))^{\partial=0, \varphi=1}$.

Finally, note that we have a continuous homomorphism $G_L \rightarrow G_R$, i.e.\ $V$ is also a $\padic$ representation of $G_L$.
So, by base changing the $(\varphi, G_R)\equivariant$ isomorphism in \eqref{eq:crys_comp} along the $(\varphi, G_R)\equivariant$ homomorphism $\OBcrys(\Rbar) \rightarrow \OBcrys(\OLbar)$, we obtain a $(\varphi, G_L)\equivariant$ isomorphism 
\begin{equation*}
	\OBcrys(\OLbar) \otimes_{R[1/p]} \ODcrysR(V) \isomorphic \OBcrys(\OLbar) \otimes_{\QQ_p} V.
\end{equation*}
By taking $G_L\textrm{-invariants}$ in the preceding isomorphism, we further obtain a natural isomorphism of $L\textrm{-vector}$ spaces $L \otimes_{R[1/p]} \ODcrysR(V) \isomorphic \ODcrysL(V)$ compatible with the respective Frobenii and connections.
Consequently, we also get that $V$ is a crystalline representation of $G_L$.

\section{Relative Wach modules}\label{sec:wach_modules}

In this section, we will describe relative Wach modules and finite $\pqheight$ representations of $G_R$, and relate them to crystalline representations.
We start by noting some technical lemmas.

\subsection{Some technical results}\label{subsec:technical_results}

In $\Ainf(\OFinfty)$, let us fix $q \coloneq [\varepsilon]$, $\mu \coloneq [\varepsilon]-1 = q-1$ and $[p]_q \coloneq \varphi(\mu)/\mu$.
\begin{defi}
	Let $N$ be a finitely generated $\AR^+\module$.
	The sequence $\{p, \mu\}$ in $\AR^+$ is said to be \textit{strictly $N\regular$} if both $\{p, \mu\}$ and $\{\mu, p\}$ are $N\regular$ sequences.
\end{defi}

For the relation between the definition above and the local cohomology of $N$ with respect to the ideal $(p, \mu)$, see Lemma \ref{lem:strictreg_koszulcomp} and Remark \ref{rem:strictreg_localcoh}.

\begin{lem}\label{lem:p_pq_regsec}
	Let $N$ be a finitely generated $\AR^+\module$.
	Then, the sequence $\{p, \mu\}$ is strictly $N\regular$ if and only if the sequence $\{p, [p]_q\}$ is strictly $N\regular$.
\end{lem}
\begin{proof}
	Note that we have $Z \coloneq V(p, \mu) = V(p, [p]_q) \subset \Spec(\AR^+)$.
	Therefore, from Lemma \ref{lem:strictreg_koszulcomp} and Remark \ref{rem:strictreg_localcoh}, we conclude that the sequence $\{p, \mu\}$ is strictly $N\regular$ if and only if $H^1_Z(N) = 0$ which is equivalent to the sequence $\{p, [p]_q\}$ being strictly $N\regular$.
\end{proof}

\begin{lem}\label{lem:wach_saturated}
	Let $N$ be a finitely generated $\AR^+\module$ such that $\{p, \mu\}$ is strictly $N\regular$.
	Then, we have $N = N[1/p] \cap N[1/\mu] \subset N[1/p, 1/\mu]$ as $\AR^+\modules$.
	Moreover, $N = N[1/p] \cap N[1/\mu]^{\wedge} \subset N[1/\mu]^{\wedge}[1/p]$, where ${}^{\wedge}$ denotes the $\padic$ completion.
\end{lem}
\begin{proof}
	From the definitions we have equalities $(N/p)[\mu] = (N/\mu)[p] = 0$ and $(N[1/\mu])/p = (N/p)[1/\mu]$.
	So, it follows that $N/p^nN \hookrightarrow (N/p^n)[1/\mu]$, for all $n \in \NN$, and therefore, $N[1/p] \cap N[1/\mu] = N$ by Lemma \ref{lem:injectivity_modulo}.
	Furthermore, since $(N[1/\mu]^{\wedge})/p^n = (N[1/\mu])/p^n = (N/p^n)[1/\mu]$ and $N$ is $p\textrm{-adically}$ complete, therefore, taking the $\lim_n$ of the inclusions $N/p^nN \hookrightarrow (N/p^n)[1/\mu]$ gives that $N \hookrightarrow N[1/\mu]^{\wedge}$.
	So, similar to above, by using Lemma \ref{lem:injectivity_modulo}, we get that $N = N[1/p] \cap N[1/\mu]^{\wedge}$ as submodules of $N[1/\mu]^{\wedge}[1/p]$.
\end{proof}

Finally, let us note an important observation for the action of $\Gamma_R$ on $\AR^+\modules$.
Note that the action of $\Gamma_R$ is continuous on $\AR^+$ for the $(p, \mu)\adic$ topology and the induced action of $\Gamma_R$ on $\AR^+/\mu \isomorphic R$ is trivial.
More generally, we claim the following:
\begin{lem}\label{lem:automatic_continuity}
	Let $N$ be a finitely generated $\mu\textrm{-torsion}$ free $\AR^+\module$ equipped with a semilinear action of $\Gamma_R$ such that the induced action of $\Gamma_R$ on $N/\mu N$ is trivial.
	Then, the action of $\Gamma_R$ on $N$ is continuous for the $(p, \mu)\adic$ topology.
\end{lem}
\begin{proof}
	Recall that from Section \ref{sec:period_rings_padic_reps} we have $\Gamma_R \isomorphic \Gamma_R' \rtimes \Gamma_F \isomorphic \ZZ_p(1)^d \rtimes \ZZ_p^{\times}$.
	Moreover, we fixed $\{\gamma_1, \ldots, \gamma_d\}$ to be topological generators of $\Gamma_R'$ and $\gamma_0$ in $\Gamma_R$ to be a lift of a topological generator of $\Gamma_R/\Gamma_R'$.
	Additionally, we may assume that $\chi(\gamma_0) = 1 + pa$, for $p \geqslant 3$, and $\chi(\gamma_0) = 1 + 4a$, for $p = 2$, where $\chi$ is the $\padic$ cyclotomic character and $a$ is a unit in $\ZZ_p$.
	To show that the action of $\Gamma_R$ is continuous on $N$, for the $(p, \mu)\adic$ topology, we need to show that for any $x$ in $N$, any $n \geqslant 1$ and for each $\gamma_i$, there exists an $m \in \NN$ such that $\gamma_i^{p^m}(x) = x \textrm{ mod } (p, \mu)^n$.
	As the action of $\Gamma_R$ is trivial on $N/\mu$, let us note that for each $0 \leqslant i \leqslant d$, the following operators are well-defined (see Section \ref{subsec:wachmod_qdeformation} for more on such operators):
	\begin{equation*}
		\nabla_{q, i} \coloneq \tfrac{\gamma_i-1}{\mu} \colon N \longrightarrow N.
	\end{equation*}
	Moreover, note that for any $a$ in $\AR^+$, $x$ in $N$ and $0 \leqslant i \leqslant d$, we have that $(\gamma_i-1)(ax) = (\gamma_i-1)a \cdot x + \gamma_i(a) \cdot (\gamma_i-1)(x)$, and therefore, $\nabla_{q,i}(ax) = \nabla_{q,i}(a) \cdot x + \gamma_i(a) \cdot \nabla_{q,i}(x)$.
	Now, for $1 \leqslant i \leqslant d$, note that $\nabla_{q, i}(\mu) = \mu$, so by setting $m = n$, we get that
	\begin{equation*}
		\gamma_i^{p^n}(x) = (1 + \mu \nabla_{q, i})^{p^n}(x) = x + \textstyle \sum_{k=1}^{p^n} \binom{p^n}{k} \mu^k \nabla_{q, i}^k(x),
	\end{equation*}
	where the summation in the third term is easily seen to be an element of $(p, \mu)^n N$.
	Next, let $i = 0$ and using the action of $\gamma_0$ on $\mu$ we have that
	\begin{align*}
		(\gamma_0-1)\mu = (1+\mu)^{\chi(\gamma_0)}-(1+\mu) &= \big(1 + \chi(\gamma_0)\mu + \tfrac{\chi(\gamma_0)(\chi(\gamma_0)-1)}{2}\mu^2 + \cdots\big) - (1+\mu)\\
				&= (\chi(\gamma_0)-1)\mu + \tfrac{\chi(\gamma_0)(\chi(\gamma_0)-1)}{2}\mu^2 + \cdots,
	\end{align*}
	where the second equality follows from the binomial expansion formula, and using that the power series expansion above is unique in $F\llbracket \mu \rrbracket$ and the fact that $(\gamma_0-1)\mu$ belongs to $O_F\llbracket \mu \rrbracket$ it follows that the coefficient of $\mu^k$ is in $O_F$, for each $k \geqslant 0$.
	As $\chi(\gamma_0)-1 = pa$, for $p \geqslant 3$ and $\chi(\gamma_0)-1 = 4a$, for $p=2$, therefore, we obtain that $\nabla_{q, 0}(\mu) = \tfrac{(\gamma_0-1)\mu}{\mu}$ is an element of $(p, \mu)\AF^+$.
	Now, an easy induction on $k \geqslant 1$ shows that for any $x$ in $N$, we must have that $(\mu \nabla_{q, 0})^k(x)$ is an element of $(p, \mu)^k N$.
	In particular, by setting $m = n$, we get that
	\begin{equation*}
		\gamma_0^{p^n}(x) = (1 + \mu \nabla_{q, 0})^{p^n}(x) = x + \textstyle \sum_{k=1}^{p^n} \binom{p^n}{k} (\mu \nabla_{q, 0})^k(x),
	\end{equation*}
	where the summation in the third term is again an element of $(p, \mu)^n N$.
	Hence, we conclude that the action of $\Gamma_R$ is continuous on $N$.
\end{proof}

\subsection{Wach modules over \texorpdfstring{$\AR^+$}{-}}\label{subsec:wach_mod_props}

We start with the definition of Wach modules.
\begin{defi}\label{defi:wach_mods_relative}
	A \textit{Wach module} over $\AR^+$ with weights in the interval $[a, b]$, for some $a, b \in \ZZ$ with $b \geqslant a$, is a finitely generated $\AR^+\module$ $N$ satisfying the following assumptions:
	\begin{enumarabicup}
		\item The sequences $\{p, \mu\}$ and $\{\mu, p\}$ are regular on $N$.

		\item $N$ is equipped with a semilinear action of $\Gamma_R$ such that the induced action of $\Gamma_R$ on $N/\mu$ is trivial.

		\item There is a $\Gamma_R\equivariant$ Frobenius-semilinear operator $\varphi \colon N[1/\mu] \rightarrow N[1/\varphi(\mu)]$ such that $\varphi(\mu^b N) \subset \mu^b N$ and the map $(1 \otimes \varphi) \colon \varphi^{\ast}(\mu^b N) \rightarrow \mu^b N$ is injective with its cokernel killed by $[p]_q^{b-a}$.
	\end{enumarabicup}
	We define the $[p]_q\textit{-height}$ of $N$ to be the largest value of $-a$ for $a \in \ZZ$ as above.
	The module $N$ is said to be \textit{effective} if we can take $b = 0$ and $a \leqslant 0$.
	Denote by $(\varphi, \Gamma_R)\Mod_{\AR^+}^{[p]_q}$, the category of Wach modules over $\AR^+$ with morphisms between objects being $\AR^+\linear$ $\varphi\equivariant$ (after inverting $\mu$) and $\Gamma_R\equivariant$.

	A Wach module over $\BR^+$ is a finitely generated module $M$ equipped with a semilinear action of $\Gamma_R$ and a $\Gamma_R\equivariant$ Frobenius-semilinear operator $\varphi \colon M[1/\mu] \rightarrow M[1/\varphi(\mu)]$, and such that there exists a $\Gamma_R\textrm{-stable}$ and $\varphi\textrm{-stable}$ (after inverting $\mu$) $\AR^+\submodule$ $N \subset M$ with $N[1/p] = M$ and equipped with the induced $(\varphi, \Gamma_R)\action$ $N$ is a Wach module over $\AR^+$.
\end{defi}

\begin{rem}
	In Definition \ref{defi:wach_mods_relative}, note that from the triviality of the action of $\Gamma_R$ on $N/\mu$ and Lemma \ref{lem:automatic_continuity}, it follows that the action of $\Gamma_R$ on $N$ is continuous.
\end{rem}

Next, we note some structural properties of Wach modules.

\begin{lem}\label{lem:finite_pqheight_equiv}
	Let $N$ be a finitely generated $\AR^+\module$ satisfying conditions (1) and (2) of Definition \ref{defi:wach_mods_relative}.
	Then, (3) of Definition \ref{defi:wach_mods_relative} is equivalent to giving an $\AR^+\linear$ and $\Gamma_R\equivariant$ isomorphism $\varphi_N \colon (\varphi^*N)[1/[p]_q] = (\AR^+ \otimes_{\varphi, \AR^+} N)[1/[p]_q] \isomorphic N[1/[p]_q]$.
\end{lem}
\begin{proof}
	Suppose that $N$ satisfes condition (3) of Definition \ref{defi:wach_mods_relative}.
	Then, the map $1 \otimes \varphi \colon \varphi^*(\mu^b N) \rightarrow \mu^b N$ induces an isomorphism $1 \otimes \varphi \colon (\mu^b \varphi^*N)[1/[p]_q] \isomorphic (\mu^b N)[1/[p]_q]$.
	Therefore, we obtain an isomorphism
	\begin{equation*}
		\varphi_N \colon (\varphi^*N)[1/[p]_q] \xrightarrow[\hspace{1mm} \mu^b \hspace{1mm}]{\sim} (\mu^b \varphi^*N)[1/[p]_q] \xrightarrow[\hspace{1mm} 1 \otimes \varphi \hspace{1mm}]{\sim} (\mu^b N)[1/[p]_q] \xleftarrow[\hspace{1mm} \mu^b \hspace{1mm}]{\sim} N[1/[p]_q].
	\end{equation*}
	Since, $1 \otimes \varphi$ commutes with the action of $\Gamma_R$, we deduce that $\varphi_N$ is $\Gamma_R\equivariant$.

	Conversely, let us suppose that we are given an $\AR^+\linear$ $\Gamma_R\equivariant$ isomorphism $\varphi_N \colon (\varphi^*N)[1/[p]_q] \isomorphic N[1/[p]_q]$.
	Then, as $N$ is $[p]_q\textrm{-torsion}$ free by Lemma \ref{lem:p_pq_regsec}, note that for some $a, b \in \ZZ$ with $b \geqslant a$, we may write $[p]_q^b \varphi_N(\varphi^* N) \subset N \subset [p]_q^a \varphi_N(\varphi^*N)$.
	Therefore, we get an $\AR^+\textrm{-semilinear}$ and $\Gamma_R\equivariant$ morphism defined via the composition $\varphi \colon \mu^b N \xrightarrow{\hspace{1mm} \textrm{can} \hspace{1mm}} \varphi^*(\mu^b N) \xrightarrow{\hspace{1mm} \varphi_N \hspace{1mm}} \mu^b N$.
	This extends to an $\AR^+\textrm{-semilinear}$ $\Gamma_R\equivariant$ map $\varphi \colon N[1/\mu] \rightarrow N[1/\varphi(\mu)]$, and we have 
	\begin{equation*}
		\varphi_N(\varphi^*(\mu^b N)) = \mu^b [p]_q^b \varphi_N(\varphi^*N) \subset \mu^b N \subset [p]_q^{a-b} \varphi_N(\varphi^*(\mu^b N)).
	\end{equation*}
	Then, it easily follows that $1 \otimes \varphi = \varphi_N \colon \varphi^*(\mu^b N) \rightarrow \mu^b N$ is injective, its cokernel is killed by $[p]_q^{b-a}$ and it commutes with the action of $\Gamma_R$.
	Hence, $N$ satisfies condition (3) of Definition \ref{defi:wach_mods_relative}.
\end{proof}

\begin{prop}\label{prop:wachmod_proj_pmu}
	Let $N$ be a Wach module over $\AR^+$.
	Then, $N[1/p]$ is finite projective over $\AR^+[1/p]$ and $N[1/\mu]$ is finite projective over $\AR^+[1/\mu]$.
\end{prop}
\begin{proof}
	For $r \in \NN$ large enough, note that the Wach module $\mu^r N(-r)$ is always effective.
	So without loss of generality, we may assume that $N$ is effective.
	Then, the first claim follows from Lemma \ref{lem:finite_pqheight_equiv} and Proposition \ref{prop:finiteproj_torus}.
	For the second claim, note that $N$ is $p\textrm{-torsion}$ free, so $\AR \otimes_{\AR^+} N$ is a $p\textrm{-torsion}$ free \'etale $(\varphi, \Gamma_R)\module$ over $\AR$, and therefore, finite projective by \cite[Lemma 7.10]{andreatta-phigamma}.
	Next, as $\AR^+[1/\mu]$ is noetherian, therefore, we have that $N[1/\mu]^{\wedge} \isomorphic \AR \otimes_{\AR^+[1/\mu]} N[1/\mu] = \AR \otimes_{\AR^+} N$, where ${}^{\wedge}$ denotes the $\padic$ completion.
	Moreover, the following natural map is a faithfully flat cover:
	\begin{equation*}
		\Spec(\AR^+[1/\mu]^{\wedge}) \cup \Spec(\AR^+[1/\mu, 1/p]) \longrightarrow \Spec(\AR^+[1/\mu]).
	\end{equation*}
	Therefore, by faithfully flat descent it follows that $N[1/\mu]$ is finite projective over $\AR^+[1/\mu]$.
\end{proof}

\begin{rem}\label{rem:wachmod_pq_finproj}
	Note that the map $\Spec(\AR^+[1/[p]_q]^{\wedge}) \cup \Spec(\AR^+[1/[p]_q, 1/p]) \rightarrow \Spec(\AR^+[1/[p]_q])$ is a faithfully flat cover and $\AR^+[1/\mu]^{\wedge} = \AR^+[1/[p]_q]^{\wedge}$.
	Now, for a Wach module $N$ over $\AR^+$, we have that the $\AR^+[1/p]\module$ $N[1/p]$ is finite projective and the $\AR^+[1/\mu]\module$ $N[1/\mu]$ is finite projective (see Proposition \ref{prop:wachmod_proj_pmu}).
	Therefore, by faithfully flat descent, we get that the $\AR^+[1/[p]_q]\module$ $N[1/[p]_q]$ is also finite projective.
	Moreover, from Lemma \ref{lem:p_pq_regsec} we have that the sequence $\{p, [p]_q\}$ is strictly $N\regular$ and equivalent to condition (1) in Definition \ref{defi:wach_mods_relative}.
\end{rem}

\begin{rem}\label{rem:wachmod_torsionfree}
	Note that for a Wach module $N$ over $\AR^+$, we have that $N$ is $p\textrm{-torsion}$ free, in particular, $N$ is contained in $N[1/p]$.
	As $N[1/p]$ is finite projective over $\AR^+[1/p]$ by Proposition \ref{prop:wachmod_proj_pmu}, therefore, we obtain that $N$ is a torsion free $\AR^+\module$.
\end{rem}

\begin{lem}\label{lem:wach_intersection_lemma}
	Let $N$ be a Wach module over $\AR^+$, then we have $N = (\AL^+ \otimes_{\AR^+} N) \cap (\AR \otimes_{\AR^+} N) \subset \AL \otimes_{\AR^+} N$ as $\AR^+\modules$.
\end{lem}
\begin{proof}
	Let $N_R \coloneq N$, $N_L \coloneq \AL^+ \otimes_{\AR^+} N$ and $D_R \coloneq \AR \otimes_{\AR^+} N$.
	Note that we have $N_L[1/p] = \BL^+ \otimes_{\BR^+} N_R[1/p]$ and $D_R[1/p] = \BR \otimes_{\BR^+} N_R[1/p]$, therefore, it follows that
	\begin{equation*}
		N_L[1/p] \cap D_R[1/p] = (\BL^+ \cap \BR) \otimes_{\BR^+} N_R[1/p] = N_R[1/p],
	\end{equation*}
	where the first equality holds because $N_R[1/p]$ is finite projective over $\BR^+$ and the second equality holds because we have $\BR^+ = \BL^+ \cap \BR \subset \BL$ from Lemma \ref{lem:ar_intersect_al+}.
	Moreover, we have that $N_L \cap D_R \subset N_L[1/p] \cap D_R[1/p] = N_R[1/p]$, and using Lemma \ref{lem:wach_saturated} we see that $N_L \cap D_R = N_L \cap D_R \cap N_R[1/p] = N_R$.
\end{proof}

From the proof of Proposition \ref{prop:wachmod_proj_pmu}, it is clear that extending scalars along $\AR^+ \rightarrow \AR$ induces a functor $(\varphi, \Gamma_R)\Mod_{\AR^+}^{[p]_q} \rightarrow (\varphi, \Gamma_R)\Mod_{\AR}^{\etale}$, and we make the following claim:
\begin{prop}\label{prop:wach_etale_ff_relative}
	The following natural functor is fully faithful:
	\begin{align*}
		(\varphi, \Gamma_R)\Mod_{\AR^+}^{[p]_q} &\longrightarrow (\varphi, \Gamma_R)\Mod_{\AR}^{\etale}\\
		N &\longmapsto \AR \otimes_{\AR^+} N.
	\end{align*}
\end{prop}
\begin{proof}
	Let $N, N'$ be two Wach modules over $\AR^+$.
	Write $N_R \coloneq N$, $N_L \coloneq \AL^+ \otimes_{\AR^+} N$, $D_L \coloneq \AR \otimes_{\AR^+} N$ and similarly for $N'$.
	We need to show that for the Wach modules $N_R$ and $N_R'$, we have a natural bijection
	\begin{equation}\label{eqref:homset_bijection_relative}
		\Hom_{(\varphi, \Gamma_R)\Mod_{\AR^+}^{[p]_q}}(N_R, N_R') \isomorphic \Hom_{(\varphi, \Gamma_R)\Mod_{\AR}^{\etale}}(D_R, D_R')
	\end{equation}
	As the homomorphism $\AR^+ \rightarrow \AR = \AR^+[1/\mu]^{\wedge}$ is injective, therefore, we see that the map in \eqref{eqref:homset_bijection_relative} is injective.
	To check that \eqref{eqref:homset_bijection_relative} is surjective, take an $\AR\linear$ and $(\varphi, \Gamma_R)\equivariant$ map $f \colon D_R \rightarrow D_R'$.
	We need to show that $f(N_R) \subset N_R'$.
	Base changing $f$ along $\AR \rightarrow \AL$ and using the isomorphism $\Gamma_L \isomorphic \Gamma_R$, induces an $\AL\linear$ and $(\varphi, \Gamma_L)\equivariant$ map $f \colon D_L \rightarrow D_L'$.
	Then, from \cite[Proposition 3.3]{abhinandan-imperfect-wach} we have that $f(N_L) \subset N_L'$.
	So, using Lemma \ref{lem:wach_intersection_lemma}, we get that inside $D_L'$ we have $f(N_R) = f(N_L \cap D_R) \subset f(N_L) \cap f(D_R) \subset N_L' \cap D_R' = N_R'$, thus concluding the proof.
\end{proof}

Analogous to above, one can define categories $(\varphi, \Gamma_R)\Mod_{\BR^+}^{[p]_q}$ and $(\varphi, \Gamma_R)\Mod_{\BR}^{\etale}$ and a functor from the former to the latter by extending scalars along $\BR^+ \rightarrow \BR$.
Then, passing to the associated isogeny categories in Proposition \ref{prop:wach_etale_ff_relative}, we get the following:
\begin{cor}\label{cor:wach_etale_ff_relative}
	The natural functor $(\varphi, \Gamma_R)\Mod_{\BR^+}^{[p]_q} \rightarrow (\varphi, \Gamma_R)\Mod_{\BR}^{\etale}$ is fully faithful.
\end{cor}

\subsection{\texorpdfstring{$G_R\textrm{-representations}$}{-} attached to Wach modules}\label{subsec:wach_mod_rep}

Composition of the functor in Proposition \ref{prop:wach_etale_ff_relative} with the equivalence in \eqref{eq:rep_phigamma_relative}, yield the following fully faithful functor:
\begin{equation}\label{eq:wach_reps_relative}
	\begin{aligned}
		\TR \colon (\varphi, \Gamma_R)\Mod_{\AR^+}^{[p]_q} &\longrightarrow \Rep_{\ZZ_p}(G_R)\\
		N &\longmapsto \big(A \otimes_{\AR^+} N\big)^{\varphi = 1} \isomorphic \big(W(\CC(\Rbar)^{\flat}) \otimes_{\AR^+} N\big)^{\varphi = 1}.
	\end{aligned}
\end{equation}

\begin{prop}\label{prop:wachmod_comp_relative}
	Let $N$ be a Wach module over $\AR^+$ and $T \coloneq \TR(N)$, the associated finite free $\ZZ_p\representation$ of $G_R$.
	Then, we have a natural $G_R\equivariant$ comparison isomorphism
	\begin{equation}\label{eq:wachmod_comp_relative_ainf}
		\Ainf(\Rbar)[1/\mu] \otimes_{\AR^+} N \isomorphic \Ainf(\Rbar)[1/\mu] \otimes_{\ZZ_p} T.
	\end{equation}
	Additionally, \eqref{eq:wachmod_comp_relative_ainf} is compatible with the respective Frobenii after base change along the natural map $\Ainf(\Rbar)[1/\mu] \rightarrow W(\CC(\Rbar)^{\flat})$.
\end{prop}
\begin{proof}
	Note that for $T = \TR(N)$, from the equivalence in \eqref{eq:rep_phigamma_relative}, we have that $\DR(T) \isomorphic \AR \otimes_{\AR^+} N$ as \'etale $(\varphi, \Gamma_R)\modules$ over $\AR$.
	Then, extending scalars of the isomorphism in \eqref{eq:phigamma_comp_iso_relative} along $A \rightarrow W(\CC(\Rbar)^{\flat})$ gives a $(\varphi, G_R)\equivariant$ isomorphism
	\begin{equation}\label{eq:phigamma_comp_tilde_relative}
		W(\CC(\Rbar)^{\flat}) \otimes_{\AR^+} N \isomorphic W(\CC(\Rbar)^{\flat}) \otimes_{\ZZ_p} T.
	\end{equation}
	Now, for $r \in \NN$ large enough, the Wach module $\mu^r N (-r)$ is always effective and we have $\TR(\mu^rN(-r)) = T(-r)$ (the twist $(-r)$ denotes the Tate twist on which $\Gamma_R$ acts via the cyclotomic character).
	Therefore, we see that it is enough to show the claim for effective Wach modules (see Definition \ref{defi:wach_mods_relative}), in particular, in the rest of the proof we will assume that $N$ is effective.

	Let $\pazs$ denote the set of minimal primes of $\Rbar$ above $pR \subset R$.
	From Section \ref{subsec:localisation}, recall that for each $\pins$, we have $\Lbar(\frakp) \subset \Cp$, an algebraic closure of $L$ containing $\Rbarp$, and we have $\GRhatp = \Gal(\Lbar(\frakp)/L)$.
	Moreover, we have an isomorphism of groups $\Gamma_L \isomorphic \Gamma_R$ and for each prime $\pins$, let $\ALplusp$ denote the base ring for Wach modules in the imperfect residue field case (see \cite[Section 2.1.2]{abhinandan-imperfect-wach}).
	To avoid confusion, let us write $N_R \coloneq N$ and $\NLp \coloneq \ALplusp \otimes_{\AR^+} N$, in particular, $\NLp$ is a Wach module over $\ALplusp$ finite free of rank $= \rank_{\ZZ_p} T$.
	From \cite[Lemma 3.6]{abhinandan-imperfect-wach} note that we have $\GRhatp\equivariant$ inclusions for each $\pins$,
	\begin{equation}\label{eq:wachmod_almost_comp_imperfect}
		\mu^s \Ainf(\Cpplus) \otimes_{\ZZ_p} T \subset \Ainf(\Cpplus) \otimes_{\ALplusp} \NLp \subset \Ainf(\Cpplus) \otimes_{\ZZ_p} T.
	\end{equation}
	Now, observe that the $(\varphi, \GRp)\equivariant$ composition $\AR^+ \rightarrow W(\CC(\Rbar)^{\flat}) \rightarrow W(\CC(\frakp)^{\flat})$ naturally factors as the $(\varphi, \GRp)\equivariant$ homomorphisms $\AR^+ \rightarrow \ALplusp \rightarrow W(\CC(\frakp)^{\flat})$.
	So, by base changing the $(\varphi, G_R)\equivariant$ isomorphism in \eqref{eq:phigamma_comp_tilde_relative} along the $(\varphi, \GRp)\equivariant$ homomorphism $W(\CC(\Rbar)^{\flat}) \rightarrow W(\CC(\frakp)^{\flat})$, we obtain a natural $(\varphi, \GRp)\equivariant$ isomorphism
	\begin{equation}\label{eq:phigamma_comp_tildep_relative}
		W(\CC(\frakp)^{\flat}) \otimes_{\ALplusp} \NLp \isomorphic W(\CC(\frakp)^{\flat}) \otimes_{\ZZ_p} T.
	\end{equation}
	All terms in \eqref{eq:wachmod_almost_comp_imperfect} and \eqref{eq:phigamma_comp_tildep_relative} admit $(\varphi, \GRhatp)\equivariant$ embedding into $W(\Cp^{\flat}) \otimes_{\ALplusp} \NLp \isomorphic W(\Cp^{\flat}) \otimes_{\ZZ_p} T$, where the action of $\GRhatp$ on \eqref{eq:phigamma_comp_tildep_relative} factors through $\GRhatp \twoheadrightarrow \GRp$.
	Therefore, taking the intersection of \eqref{eq:wachmod_almost_comp_imperfect} with \eqref{eq:phigamma_comp_tildep_relative} inside $W(\Cp^{\flat}) \otimes_{\ALplusp} \NLp \isomorphic W(\Cp^{\flat}) \otimes_{\ZZ_p} T$, and using Lemma \ref{lem:cplusp_in_cpplus_ainf}, for each $\pins$, we obtain the following $(\varphi, \GRp)\equivariant$ inclusions:
	\begin{equation}\label{eq:wachmod_almost_comp_p}
		\mu^s \Ainf(\Cplusp) \otimes_{\ZZ_p} T \subset \Ainf(\Cplusp) \otimes_{\ALplusp} \NLp \subset \Ainf(\Cplusp) \otimes_{\ZZ_p} T,
	\end{equation}
	where the middle term may be written as $\Ainf(\Cplusp) \otimes_{\ALplusp} \NLp = \Ainf(\Cplusp) \otimes_{\AR^+} N_R$.

	Next, from Remark \ref{rem:gr_act_prodainf}, recall that $\prod_{\pins} \Ainf(\Cplusp)$ is equipped with a $G_R\action$ and from Lemma \ref{lem:ainf_intersection} we have a $(\varphi, G_R)\equivariant$ injective homomorphism $\Ainf(\Rbar) \rightarrow \prod_{\pins} \Ainf(\Cplusp)$.
	Then, we may equip 
	\begin{equation*}
		\textstyle\prod_{\pins} (\Ainf(\Cplusp) \otimes_{\ZZ_p} T) = (\textstyle\prod_{\pins} \Ainf(\Cplusp)) \otimes_{\ZZ_p} T,
	\end{equation*}
	with the diagonal action of $(\varphi, G_R)$, and similarly for 
	\begin{equation*}
		\textstyle\prod_{\pins} \big(\Ainf(\Cplusp) \otimes_{\ALplusp} \NLp\big) = \textstyle\prod_{\pins} \big(\Ainf(\Cplusp) \otimes_{\AR^+} N_R\big) = \big(\textstyle\prod_{\pins} \Ainf(\Cplusp)\big) \otimes_{\AR^+} N_R,
	\end{equation*}
	where the second equality follows from the fact that product is an exact functor on the category of $\AR^+\modules$ and $N_R$ is finitely presented over the noetherian ring $\AR^+$ (see \cite[\href{https://stacks.math.columbia.edu/tag/059K}{Tag 059K}]{stacks-project}).
	So, taking the product of \eqref{eq:wachmod_almost_comp_p} over all $\pins$ and using the discussion above, we obtain $(\varphi, G_R)\equivariant$ inclusions:
	\begin{equation}\label{eq:wachmod_prod_almost_comp}
		\mu^s \textstyle\prod_{\pins} \big(\Ainf(\Cplusp) \otimes_{\ZZ_p} T\big) \subset \big(\textstyle\prod_{\pins} \Ainf(\Cplusp)\big) \otimes_{\AR^+} N_R \subset \textstyle\prod_{\pins} \big(\Ainf(\Cplusp) \otimes_{\ZZ_p} T\big).
	\end{equation}
	Inverting $\mu$ in \eqref{eq:wachmod_prod_almost_comp} yields the top horizontal isomorphism in the following $(\varphi, G_R)\equivariant$ commutative diagram:
	\begin{equation}\label{eq:wachmod_prod_muinverse_comp}
		\begin{tikzcd}[row sep=15pt]
			\big(\textstyle\prod_{\pins} \Ainf(\Cplusp)\big)[1/\mu] \otimes_{\AR^+[1/\mu]} N_R[1/\mu] & \big(\textstyle\prod_{\pins} \Ainf(\Cplusp)\big)[1/\mu] \otimes_{\ZZ_p} T\\
			\big(\textstyle\prod_{\pins} W(\CC(\frakp)^{\flat})\big) \otimes_{\AR^+} N_R & \big(\textstyle\prod_{\pins} W(\CC(\frakp)^{\flat})\big) \otimes_{\ZZ_p} T\\
			W(\CC(\Rbar)^{\flat}) \otimes_{\AR^+[1/\mu]} N_R[1/\mu] & W(\CC(\Rbar)^{\flat}) \otimes_{\ZZ_p} T,
			\arrow["\sim", from=1-1, to=1-2]
			\arrow[hook, from=1-1, to=2-1]
			\arrow[hook, from=1-2, to=2-2]
			\arrow["\sim", from=2-1, to=2-2]
			\arrow[hook', from=3-1, to=2-1]
			\arrow["\sim", "\eqref{eq:phigamma_comp_tilde_relative}"', from=3-1, to=3-2]
			\arrow[hook', from=3-2, to=2-2]
		\end{tikzcd}
	\end{equation}
	where the vertical arrows are natural inclusions by Lemma \ref{lem:ainf_intersection}, and the horizontal isomorphism in the middle row is obtained by taking the product of \eqref{eq:phigamma_comp_tildep_relative} over all $\pins$ and using that (see \cite[\href{https://stacks.math.columbia.edu/tag/059K}{Tag 059K}]{stacks-project}):
	\begin{equation*}
		\textstyle\prod_{\pins} \big(W(\CC(\frakp)^{\flat}) \otimes_{\ALplusp} \NLp\big) = \textstyle\prod_{\pins} \big(W(\CC(\frakp)^{\flat}) \otimes_{\AR^+} N_R\big) = \big(\textstyle\prod_{\pins} W(\CC(\frakp)^{\flat})\big) \otimes_{\AR^+} N_R.
	\end{equation*}
	Note that $N_R[1/\mu]$ is finite projective over $\AR^+[1/\mu]$ (see Proposition \ref{prop:wachmod_proj_pmu}), so in diagram \eqref{eq:wachmod_prod_muinverse_comp}, taking the intersection of the top left term and the bottom left term, inside the middle left term, gives
	\begin{align*}
		\big(W(\CC(\Rbar)^{\flat}) \otimes_{\AR^+[1/\mu]} N_R[1/\mu]\big) &\cap \big(\big(\textstyle\prod_{\pins} \Ainf(\Cplusp)\big)[1/\mu] \otimes_{\AR^+[1/\mu]} N_R[1/\mu]\big)\\
		&= \Ainf(\Rbar)[1/\mu] \otimes_{\AR^+[1/\mu]} N_R[1/\mu] = \Ainf(\Rbar)[1/\mu] \otimes_{\AR^+} N_R,
	\end{align*}
	where the first equality follows from Lemma \ref{lem:ainf_intersection}.
	Similarly, taking the intersection of the top right term and bottom right term, inside the middle right term, gives
	\begin{align*}
		\big(W(\CC(\Rbar)^{\flat}) \otimes_{\ZZ_p} T\big) \cap \big(\big(\textstyle\prod_{\pins} \Ainf(\Cplusp)\big)[1/\mu] \otimes_{\ZZ_p} T\big) = \Ainf(\Rbar)[1/\mu] \otimes_{\ZZ_p} T,
	\end{align*}
	where the equality again follows from Lemma \ref{lem:ainf_intersection}.
	As the horizontal arrows in \eqref{eq:wachmod_prod_muinverse_comp} are bijective, therefore, we obtain that the intersection of the top and the bottom rows, inside the middle row, is naturally a $G_R\equivariant$ isomorphism, thus yielding the claimed isomorphism in \eqref{eq:wachmod_comp_relative_ainf}.
	From the preceding discussion, it also follows that the isomorphism in \eqref{eq:wachmod_comp_relative_ainf} is compatible with the respective Frobenii after base change along $\Ainf(\Rbar) \rightarrow W(\CC(\Rbar)^{\flat})$.
\end{proof}

\begin{cor}\label{cor:wachmod_comp_relative}
	Let $N$ be a Wach module over $\AR^+$ and let $T \coloneq \TR(N)$ denote the associated finite free $\ZZ_p\representation$ of $G_R$.
	Then, we have a natural $G_R\equivariant$ comparison isomorphism
	\begin{equation}\label{eq:wachmod_comp_relative_a+}
		A^+[1/\mu] \otimes_{\AR^+} N \isomorphic A^+[1/\mu] \otimes_{\ZZ_p} T.
	\end{equation}
	Additionally, the isomorphism above is compatible with the respective Frobenii after base change along $A^+[1/\mu] \rightarrow A$.
\end{cor}
\begin{proof}
	From Section \ref{subsec:phigamma_mod_rings}, recall that $A^+ = \Ainf(\Rbar) \cap A \subset W(\CC(\Rbar)^{\flat})$.
	So, taking the intersection of the isomorphism \eqref{eq:wachmod_comp_relative_ainf} in Proposition \ref{prop:wachmod_comp_relative} with the isomorphism in \eqref{eq:phigamma_comp_iso_relative}, inside $W(\CC(\Rbar)^{\flat}) \otimes_{\ZZ_p} T$, and noting that $N[1/\mu]$ is a finite projective module over $\AR^+[1/\mu]$, yields the $G_R\equivariant$ isomorphism in \eqref{eq:wachmod_comp_relative_a+}.
	Moreover, from Proposition \ref{prop:wachmod_comp_relative}, it also follows that the isomorphism in \eqref{eq:wachmod_comp_relative_a+} is further compatible with the respective Frobenii after base change along $A^+[1/\mu] \rightarrow A$.
\end{proof}

\begin{prop}\label{prop:wachmod_almost_comp}
	Let $N$ be an effective Wach module over $\AR^+$ and $T \coloneq \TR(N)$ the associated finite free $\ZZ_p\representation$ of $G_R$.
	Then, we have $(\varphi, \Gamma_R)\equivariant$ inclusions $\mu^s \DR^+(T) \subset N \subset \DR^+(T)$ (see Section \ref{subsec:relative_padicreps} for notations).
\end{prop}
\begin{proof}
	The proof follows in a manner similar to the proof of Proposition \ref{prop:wachmod_comp_relative}, so we will freely use the notations therein.
	By inverting $p$ in \eqref{eq:wachmod_prod_almost_comp} we have $(\varphi, G_R)\equivariant$ inclusions
	\begin{equation}\label{eq:wachmod_prod_pinverse_almost_comp}
		\begin{aligned}
			\mu^s \big(\textstyle\prod_{\pins} \big(\Ainf(\Cplusp) \otimes_{\ZZ_p} T\big)\big)[1/p] &\subset \big(\textstyle\prod_{\pins} \Ainf(\Cplusp)\big)[1/p] \otimes_{\BR^+} N_R[1/p]\\
			&\subset \big(\textstyle\prod_{\pins} \big(\Ainf(\Cplusp) \otimes_{\ZZ_p} T\big)\big)[1/p],
		\end{aligned}
	\end{equation}
	where the last term may also be written as $\big(\prod_{\pins} \Ainf(\Cplusp)\big)[1/p] \otimes_{\QQ_p} V$ and similarly for the first term.
	Moreover, by inverting $p$ in \eqref{eq:phigamma_comp_tilde_relative}, we have the following $(\varphi, G_R)\equivariant$ comparison isomorphism:
	\begin{equation}\label{eq:phigamma_comp_pinverse}
		W(\CC(\Rbar)^{\flat})[1/p] \otimes_{\BR^+} N_R[1/p] \isomorphic W(\CC(\Rbar)^{\flat})[1/p] \otimes_{\QQ_p} V.
	\end{equation}
	Using Lemma \ref{lem:ainf_intersection} and diagram \eqref{eq:wachmod_prod_muinverse_comp} (after inverting $p$), we embed all terms of \eqref{eq:wachmod_prod_pinverse_almost_comp} and \eqref{eq:phigamma_comp_pinverse}, in a $(\varphi, G_R)\equivariant$ manner, inside
	\begin{equation}\label{eq:phigamma_comp_prod_pinverse}
		\big(\textstyle\prod_{\pins} W(\CC(\frakp)^{\flat})\big)[1/p] \otimes_{\BR^+} N_R[1/p] \isomorphic \big(\textstyle\prod_{\pins} W(\CC(\frakp)^{\flat})\big)[1/p] \otimes_{\QQ_p} V,
	\end{equation}
	where the isomorphism above is obtained by inverting $p$ in the middle row of \eqref{eq:wachmod_prod_muinverse_comp}.
	Since $N_R[1/p]$ is a finite projective module over $\BR^+$, therefore, the intersection of the middle term of \eqref{eq:wachmod_prod_pinverse_almost_comp} and the left-hand term of \eqref{eq:phigamma_comp_pinverse}, inside the left-hand term of \eqref{eq:phigamma_comp_prod_pinverse}, gives
	\begin{equation*}
		\begin{aligned}
			\big(W(\CC(\Rbar)^{\flat})[1/p] \otimes_{\BR^+} N_R[1/p]\big) \cap \big(\big(\textstyle\prod_{\pins} \Ainf(\Cplusp)\big)[1/p] &\otimes_{\BR^+} N_R[1/p]\big)\\
			&= \Ainf(\Rbar)[1/p] \otimes_{\BR^+} N_R[1/p],
		\end{aligned}
	\end{equation*}
	where the equality follows from Lemma \ref{lem:ainf_intersection}.
	Similarly, taking the intersection of the right-most terms of \eqref{eq:wachmod_prod_pinverse_almost_comp} and \eqref{eq:phigamma_comp_pinverse}, inside the right-hand term of \eqref{eq:phigamma_comp_prod_pinverse}, gives
	\begin{align*}
		\big(W(\CC(\Rbar)^{\flat})[1/p] \otimes_{\QQ_p} V\big) \cap \big(\big(\textstyle\prod_{\pins} \Ainf(\Cplusp)\big)[1/p] \otimes_{\QQ_p} V\big) = \Ainf(\Rbar)[1/p] \otimes_{\QQ_p} V,
	\end{align*}
	where the equality again follows from Lemma \ref{lem:ainf_intersection}.
	Therefore, from \eqref{eq:wachmod_prod_pinverse_almost_comp} and using the $(\varphi, G_R)\textrm{-equivariance}$ of \eqref{eq:phigamma_comp_pinverse}, we obtain the following $(\varphi, G_R)\equivariant$ inclusions:
	\begin{equation}\label{eq:wachmodrat_almost_comp}
		\mu^s \big(\Ainf(\Rbar)[1/p] \otimes_{\QQ_p} V\big) \subset \Ainf(\Rbar)[1/p] \otimes_{\BR^+} N_R[1/p] \subset \Ainf(\Rbar)[1/p] \otimes_{\QQ_p} V.
	\end{equation}
	Inverting $p$ in the isomorphism \eqref{eq:wachmod_comp_relative_a+} of Corollary \ref{cor:wachmod_comp_relative} and taking its intersection with \eqref{eq:wachmodrat_almost_comp}, inside \eqref{eq:phigamma_comp_pinverse}, we obtain the following $(\varphi, G_R)\equivariant$ inclusions:
	\begin{equation*}
		\mu^s \big(B^+ \otimes_{\QQ_p} V\big) \subset B^+ \otimes_{\BR^+} N_R[1/p] \subset B^+ \otimes_{\QQ_p} V.
	\end{equation*}
	In the preceding equation, first we take the $H_R\textrm{-invariants}$ and then its intersection with $\DR(T) = N_R[1/\mu]^{\wedge}$, inside $\DR(V)$, to obtain $(\varphi, \Gamma_R)\equivariant$ inclusions $\mu^s \DR^+(T) \subset N_R \subset \DR^+(T)$, because we have $N_R = N_R[1/p] \cap N_R[1/\mu]^{\wedge}$ from Lemma \ref{lem:wach_saturated} and $\DR^+(T) = \DR(T) \cap \DR^+(V) \subset \DR(V)$, by definition.
	Hence, the proposition is proved.
\end{proof}

\subsection{Finite \texorpdfstring{$\pqheight$}{-} representations}\label{subsec:finite_pqheight_reps}

In this section, we will generalise the definition of finite $\pqheight$ representations from \cite[Definition 4.9]{abhinandan-relative-wach-i} in the relative case.

\begin{defi}\label{defi:finite_pqheight}
	A finite $\pqheight$ $\ZZ_p\representation$ of $G_R$ is a finite free $\ZZ_p\module$ $T$ admitting a linear and continuous action of $G_R$ such that there exists a finitely generated $\AR^+\submodule$ $\NR(T) \subset \DR(T)$, stable under the action of $\Gamma_R$ on $\DR(T)$, and such that $\NR(T)$, equipped with the induced actions of $\varphi$ and $\Gamma_R$, satisfies the following:
	\begin{enumarabicup}
	\item The module $\NR(T)$ is a Wach module in the sense of Definition \ref{defi:wach_mods_relative}.
		
	\item The $\AR\linear$ extension of the inclusion $\NR(T) \subset \DR(T)$ induces a $(\varphi, \Gamma_R)\equivariant$ isomorphism $\AR \otimes_{\AR^+} \NR(T) \isomorphic \DR(T)$.
	\end{enumarabicup}
	The height of $T$ is defined to be the height of $\NR(T)$.
	Say that $T$ is \textit{positive} if $\NR(T)$ is effective.

	A finite $\pqheight$ $\padic$ representation of $G_R$ is a finite dimensional $\QQ_p\textrm{-vector space}$ admitting a linear and continuous action of $G_R$ such that there exists a $G_R\textrm{-stable}$ $\ZZ_p\textrm{-lattice}$ $T \subset V$, with $T$ of finite $\pqheight$.
	We set $\NR(V) \coloneq \NR(T)[1/p]$ satisfying properties analogous to (1) and (2) above.
	The height of $V$ is defined to be the height of $T$.
	Say that $V$ is positive if $\NR(V)$ is effective.
\end{defi}

\begin{lem}\label{lem:wach_unique_relative}
	Let $T$ be a finite $\pqheight$ $\ZZ_p\representation$ of $G_R$.
	Then, the $\AR^+\module$ $\NR(T)$, associated to $T$ in Definition \ref{defi:finite_pqheight}, is unique, i.e.\ if there exists an $\AR^+\textrm{-submodule}$ $N \subset \DR(T)$ satisfying conditions (1) and (2) of Definition \ref{defi:finite_pqheight}, then we must have that $N = \NR(T) \subset \DR(T)$.
\end{lem}
\begin{proof}
	By definition, $\AR \otimes_{\AR^+} \NR(T) \isomorphic \DR(T)$ and this scalar extension induces a fully faithful functor in Proposition \ref{prop:wach_etale_ff_relative}.
	So from \eqref{eq:rep_phigamma_relative} we obtain the uniqueness of $\NR(T)$.
\end{proof}

\begin{rem}\label{rem:finiteheight_props_relative}
	Let $V$ be a finite $\pqheight$ $\padic$ representation of $G_R$ and $T \subset V$ a finite $\pqheight$ $G_R\textrm{-stable}$ $\ZZ_p\textrm{-lattice}$.
	Then, we have $\NR(V) = \NR(T)[1/p]$ and from Corollary \ref{prop:wach_etale_ff_relative} it follows that $\NR(V)$ is unique, i.e.\ if there exists a $\BR^+\textrm{-submodule}$ $M \subset \DR(V)$ satisfying the conditions of Definition \ref{defi:finite_pqheight}, then we must have that $M = \NR(V) \subset \DR(V)$.
	Consequently, we get that $\NR(V)$ is independent of the choice of a $G_R\textrm{-stable}$ $\ZZ_p\textrm{-lattice}$ $T \subset V$.
\end{rem}

\begin{rem}\label{rem:finitepqheight_wach_relative}
	By the definition of finite $\pqheight$ representations, Lemma \ref{lem:wach_unique_relative} and the fully faithful functor in \eqref{eq:wach_reps_relative} it follows that the data of a finite $\pqheight$ representation of $G_R$ is equivalent to the data of a Wach module.
	Subsequently, in Theorem \ref{thm:crys_fh_relative}, we will show that the data of an integral Wach module is equivalent to the data of a $G_R\textrm{-stable}$ $\ZZ_p\textrm{-lattice}$ inside a crystalline representation of $G_R$.
	Thus, our notion of finite $\pqheight$ is equivalent to the notion of crystallinity for a $\ZZ_p\textrm{-representation}$ (resp.\ $\padic$ representation) of $G_R$.
\end{rem}

\begin{rem}
	Let us note that the notion of finite height representations and several closely related variations were first studied by Fontaine in \cite[Section B]{fontaine-phigamma}.
	Fontaine's definition is general, and in the unramified arithmetic case (i.e.\ $R = O_F$), there are further refinements of these notions where one imposes various conditions on the \'etale $(\varphi, \Gamma)\module$ associated to a $\ZZ_p\textrm{-representation}$ (resp.\ $\padic$ representaion) of $G_F$, as in \cite[Section B.2]{fontaine-phigamma} and in the works of Wach \cite{wach-free, wach-torsion}, Colmez \cite{colmez-hauteur} and Berger \cite{berger-limites}.
	More precisely, in the case $R = O_F$, our notion of finite $\pqheight$ representations in Definition \ref{defi:finite_pqheight} coincides with Fontaine's notion of finite crystalline height representations (\textit{cr-hauteur finie}) in \cite[Paragraph B.2.1.6]{fontaine-phigamma}.
	Also note that in \cite{wach-free, wach-torsion, colmez-hauteur, berger-limites}, several of these notions were shown to be equivalent.
	This is the reason behind naming the objects in Definition \ref{defi:finite_pqheight} as such.
	Finally, let us also remark that somewhat different notion of finite height representations with less restrictive conditions have been studied in the works of Kisin \cite{kisin-deformation-rings}, Liu \cite{liu-semistable-lattices} and others.
	We hope that this terminology is clear from the different contexts and it is not too confusing for the reader.
\end{rem}

\subsection{Nygaard filtration on Wach modules}\label{subsec:nygaard_fil_wach_mod}

In this section, we shall consider the Nygaard filtration on Wach modules.
Let $N$ be a Wach module over $\AR^+$ (resp.\ $\AL^+$).
\begin{defi}\label{defi:nygaard_fil}
	Define a decreasing filtration by $\AR^+\textrm{-submodules}$ (resp.\ $\AL^+\textrm{-submodules}$) on $N$, called the \textit{Nygaard filtration}, by setting
	\begin{equation*}
		\Fil^k N \coloneq \{ x \in N \textrm{ such that } \varphi(x) \in [p]_q^k N\}, \qquad k \in \ZZ.
	\end{equation*}
	From the definition it is clear that $N$ is effective if and only if $\Fil^0 N = N$.
	Similarly, we define the Nygaard filtration on $M \coloneq N[1/p]$.
\end{defi}

\begin{lem}\label{lem:nygaard_fil_twist}
	Let $N$ be a Wach module over $\AR^+$ (resp.\ $\AL^+$) and set $M \coloneq N[1/p]$.
	\begin{enumarabicup}
		\item For any $k, r \in \ZZ$, and the Wach module $\mu^{-r}N(r)$ over $\AR^+$ (resp.\ $\AL^+$), we have that $\Fil^k(\mu^{-r}N(r)) = \mu^{-r} (\Fil^{r+k} N) (r)$.
		\item For all $k \in \ZZ$, we have that $\Fil^k N \cap \mu N = \mu\Fil^{k-1} N \subset N$.

		\item For all $k \in \ZZ$, we have that $\Fil^k N = N \cap \Fil^k M \subset M$, in particular, $\Fil^k M = (\Fil^k N)[1/p]$.
			Moreover, statements analogous to (1) and (2) are also true for the $\BR^+\module$ (resp.\ $\BL^+\module$) $M$.
	\end{enumarabicup}
\end{lem}
\begin{proof}
	The proof follows from arguments similar to \cite[Lemma 3.3 \& Lemma 3.4]{abhinandan-syntomic}.
	In (1), the inclusion $\mu^{-r} (\Fil^{r+k} N) (r) \subset \Fil^k(\mu^{-r}N(r))$ is obvious.
	For the converse, let $\mu^{-r} x \otimes \epsilon^{\otimes r}$ be an element of $\Fil^k (\mu^{-r} N(r))$, where $x$ is an element of $N$ and $\epsilon^{\otimes r}$ is a $\ZZ_p\textrm{-basis}$ of $\ZZ_p(r)$.
	By assumption, note that 
	\begin{equation*}
		\varphi(\mu^{-r} x \otimes \epsilon^{\otimes r}) = ([p]_q \mu)^{-r} \varphi(x) \otimes \epsilon^{\otimes r},
	\end{equation*}
	is in $[p]_q^k \mu^{-r} N(r)$.
	Therefore, $\varphi(x)$ belongs to $[p]_q^{r+k} N$, i.e.\ $x$ is in $\Fil^{r+k} N$.

	For (2), note that using (1) we may assume that $N$ is effective.
	The claim is obvious if $\Fil^{k-1} N = N$, so we assume that $k \geqslant 2$.
	Let $x$ be an element of $\Fil^k N \cap \mu N$ and write $x = \mu y$, for some $y$ in $N$.
	Now, as $\varphi(x)$ is in $[p]_q^k N$, therefore, $\mu \varphi(y)$ is in $[p]_q^{k-1} N$, i.e.\ $\mu \varphi(y) = [p]_q^{k-1} z$, for some $z$ in $N$.
	In particular, we have that $[p]_q^{k-1}z = p^{k-1}z = 0 \textrm{ mod } \mu N$.
	But $N/\mu$ is $p\textrm{-torsion}$ free, so it follows that $z = 0 \textrm{ mod } \mu N$, i.e.\ $y$ is in $\Fil^{k-1} N$.
	The other inclusion is obvious because $\mu \Fil^{k-1} N \subset \Fil^k N$.

	For (3), note that using (1) we may assume that $N$ is effective and $k \geqslant 0$.
	By definitions, it is clear that $\Fil^k N \subset N \cap \Fil^k M$.
	Conversely, let $x$ be an element of $N \cap \Fil^k M$.
	Then, we see that $\varphi(x)$ is an element of $[p]_q^k M \cap N \subset M$.
	As $\{[p]_q, p\}$ is a regular sequence on $N$, we have that $N/[p]_q^k \hookrightarrow (N/[p]_q^k)[1/p]$, and by Lemma \ref{lem:injectivity_modulo} it follows that $[p]_q^k N = [p]_q^k M \cap N$, i.e.\ $x$ is an element of $\Fil^k N$.
	From the preceding conclusion, it easily follows that $\Fil^k M  = (\Fil^k N)[1/p]$.
	Then, by inverting $p$ in (1) and (2), we get the analogous claims for $M$. 
	This allows us to conclude.
\end{proof}

\begin{rem}\label{rem:gamma_minus1_image}
	The Nygaard filtration from Definition \ref{defi:nygaard_fil}, on a Wach module $N$ over $\AR^+$ (resp.\ $\AL^+$), is stable under the action of $\Gamma_R$ (resp.\ $\Gamma_L$).
	Therefore, using Lemma \ref{lem:nygaard_fil_twist} (2) we see that for any $g$ in $\Gamma_R$ (resp.\ $\Gamma_L$) and $k \in \ZZ$, we have that $(g-1)\Fil^k N \subset (\Fil^k N) \cap \mu N = \mu \Fil^{k-1} N \subset N$.
\end{rem}

Note that the base ring $\AR^+$ (resp.\ $\AL^+$) is trivially a Wach module and we have a Nygaard filtration on it as in Definition \ref{defi:nygaard_fil}.
It is clear that $\Fil^k \AR^+ = \AR^+$ (resp.\ $\Fil^k \AL^+ = \AL^+$), for $k \leqslant 0$.
In positive degrees, we claim the following:
\begin{lem}\label{lem:nygaard_muadic_base}
	For $k \geqslant 0$ the Nygaard filtration on $\AR^+$ (resp.\ $\AL^+$) coincides with the $\mu\adic$ filtration, i.e.\ $\Fil^k \AR^+ = \mu^k \AR^+$ (resp.\ $\Fil^k \AL^+ = \mu^k \AL^+$).
	Analogous statement is also true for the ring $\BR^+$ (resp.\ $\BL^+$).
\end{lem}
\begin{proof}
	It is enough to show the claims for $k \geqslant 1$.
	So, we begin by proving the first claim for $\AR^+$ and $k=1$.
	Note that we have $\mu \AR^+ \subset \Fil^1 \AR^+$ because $\varphi(\mu) = [p]_q\mu$.
	For the converse, let us recall that we have a Frobenius equivariant $O_F\llbracket \mu \rrbracket\linear$ isomorphism $R\llbracket \mu \rrbracket \isomorphic \AR^+$ from the discussion in Section \ref{subsec:phigamma_mod_rings}, and we set $\Fil^k R\llbracket \mu \rrbracket$ to be the inverse image of $\Fil^k \AR^+$ under the preceding isomorphism.
	Now, let $a$ be an element of $\Fil^1 R\llbracket \mu \rrbracket$, so we have that $\varphi(a) = [p]_q b$ for some $b$ in $R\llbracket \mu \rrbracket$.
	Observe that have a unique presentation $a = \sum_{i \geqslant 0} a_i \mu^i$ in $R\llbracket \mu \rrbracket$, with $a_i$ in $R$.
	Using this, we compute that $\varphi(a) = \sum_i\varphi(a_i)\varphi(\mu^i) = [p]_qy$, in particular, we see that $\varphi(a_0)$ is in $[p]_q\AR^+$.
	As the composition
	\begin{equation*}
		R \longhookrightarrow R\llbracket \mu \rrbracket \longrightarrow R\llbracket \mu \rrbracket/[p]_q \isomorphic R[\zeta_p]
	\end{equation*}
	is the natural inclusion, we conclude that $\varphi(a_0) = 0$, hence, $a_0 = 0$ because $\varphi \colon R \rightarrow R$ is injective.
	In particular, it follows that $a = \sum_{i \geqslant 1} a_i \mu^i$ is in $\mu R\llbracket \mu \rrbracket$, and thus $\Fil^1 \AR^+ = \mu \AR^+$.
	To show the claim for any $k \geqslant 1$, we shall proceed by induction.
	Indeed, the case $k = 1$ shown above, and assuming that $\Fil^k \AR^+ = \mu^k \AR^+$, from Lemma \ref{lem:nygaard_fil_twist} we get that
	\begin{equation*}
		\Fil^{k+1} \AR^+ = \Fil^{k+1} \AR^+ \cap \Fil^1 \AR^+ = \Fil^{k+1} \AR^+ \cap \mu \AR^+ = \mu \Fil^k \AR^+ = \mu^{k+1} \AR^+.
	\end{equation*}
	This proves the claim for $\AR^+$ and a similar reasoning for $\AL^+ \isomorphic O_L\llbracket \mu \rrbracket$ shows that $\Fil^k \AL^+ = \mu^k \AL^+$, for $k \geqslant 0$.
	Finally, we have that $\Fil^k \BR^+ = (\Fil^k \AR^+)[1/p] = \mu^k \BR^+$ and $\Fil^k \BL^+ = (\Fil^k \AL^+)[1/p] = \mu^k \BL^+$.
\end{proof}

\subsubsection{Nygaard filtration on scalar extension of Wach modules}

Let $N_R$ be a Wach module over $\AR^+$.
Then, by Remark \ref{inro_rem:wach_mods_field} we know that $N_L \coloneq \AL^+ \otimes_{\AR^+} N_R$ is a Wach module over $\AL^+$ equipped with the natural action of $\varphi$ and $\Gamma_L \isomorphic \Gamma_R$ (in the sense of \cite[Definition 3.1]{abhinandan-imperfect-wach}), and we claim the following:
\begin{prop}\label{prop:nygaard_fil_nr_nl}
	For each $k \in \ZZ$, we have a natural $\AR^+\linear$ isomorphism 
	\begin{equation*}
		\Fil^k N_R \isomorphic N_R \cap \Fil^k N_L \subset N_L,
	\end{equation*}
	in particular, $\gr^k N_R \hookrightarrow \gr^k N_L$.
	Moreover, we have a natural $\AL^+\linear$ isomorphism
	\begin{equation}\label{eq:nygaard_fil_nr_nl}
		\AL^+ \otimes_{\AR^+} \Fil^k N_R \isomorphic \Fil^k N_L,
	\end{equation}
	in particular, a natural $O_L\linear$ isomorphism $O_L \otimes_R \gr^k N_R \isomorphic \gr^k N_L$.
	Analogous statements are also true after inverting $p$.
\end{prop}
\begin{proof}
	Using Lemma \ref{lem:nygaard_fil_twist} (1), note that it is enough to prove the claim for effective Wach modules, in particular, we will assume that $\Fil^0 N_R = N_R$.
	Note that the case $k=0$ is trivial.
	We shall first prove the claim rationally and use it to deduce the integral claim.
	Set $M_R \coloneq N_R[1/p]$ and $M_L \coloneq N_L[1/p]$, equipped with induced structures.
	Now, consider the following commutative diagram with exact rows:
	\begin{equation}\label{eq:gr_mr_ml_1}
		\begin{tikzcd}[row sep=15pt]
			0 \arrow[r] & \Fil^{k+1} M_R \arrow[r] \arrow[d] & \Fil^k M_R \arrow[r] \arrow[d] & \gr^k M_R \arrow[r] \arrow[d] & 0\\
			0 \arrow[r] & \Fil^{k+1} M_L \arrow[r] & \Fil^k M_L \arrow[r] & \gr^k M_L \arrow[r] & 0,
		\end{tikzcd}
	\end{equation}
	where the left and the middle vertical arrows are natural inclusions.
	For the induced right vertical arrow note that we have $\Fil^k M_R \cap \Fil^{k+1} M_L = \Fil^{k+1} M_R \subset M_L$, which easily follows from the defintions and the following equalities:
	\begin{equation*}
		[p]_q^k M_R \cap [p]_q^{k+1} M_L = (\BR^+ \cap [p]_q \BL^+) \otimes_{\BR^+} [p]_q^k M_R = [p]_q^{k+1} M_R \subset M_L,
	\end{equation*}
	where the first equality follows because $\BL^+ \otimes_{\BR^+} M_R \isomorphic M_L$ and $M_R$ is finite projective over $\BR^+$, and the second equality follows by Lemma \ref{lem:injectivity_modulo} using that $\AR^+/[p]_q \isomorphic R[\zeta_p] \hookrightarrow O_L[\zeta_p] \lisomorphic \AL^+/[p]_q$.
	In particular, we get that the right vertical arrow of \eqref{eq:gr_mr_ml_1} is injective, and 
	\begin{equation*}
		\Fil^k M_R = M_R \cap \Fil^k M_L \subset M_L \hspace{2mm} \textrm{ and } \hspace{2mm} \Fil^k N_R = N_R \cap \Fil^k N_L \subset N_L,
	\end{equation*}
	where the equality on the right follows because $\Fil^k N_R = N_R \cap \Fil^k M_R$ and $\Fil^k N_L = N_L \cap \Fil^k M_L$ (see Lemma \ref{lem:nygaard_fil_twist} (3)).
	The preceding equalities imply that $\gr^k N_R \hookrightarrow \gr^k N_L$ and $\gr^k M_R \hookrightarrow \gr^k M_L$.
	This proves the first claim.

	Next, from Lemma \ref{lem:nygaard_fil_twist}, note that we have $\mu\Fil^k M_R \subset \Fil^{k+1} M_R$, for any $k \geqslant 0$.
	Therefore, we get an induced natural surjective homomorphism $(\Fil^k M_R)/\mu \twoheadrightarrow \gr^k M_R$ of modules over $\gr^0\BR^+ \isomorphic \BR^+/\mu \isomorphic R[1/p]$ (see Lemma \ref{lem:nygaard_muadic_base}).
	In particular, we see that $\gr^k M_R$ is a finitely generated $R[1/p]\module$.
	A similar reasoning for the $\BL^+\module$ $M_L$ shows that $\gr^k M_L$ is a finite dimensional vector space over $\gr^0 \BL^+ \isomorphic \BL^+/\mu \isomorphic L$ (see Lemma \ref{lem:nygaard_muadic_base}).
	Now, consider the following $\BL^+\linear$ diagram with exact rows:
	\begin{equation}\label{eq:gr_mr_ml_2}
		\begin{tikzcd}[row sep=15pt]
			0 \arrow[r] & \BL^+ \otimes_{\BR^+} \Fil^{k+1} M_R \arrow[r] \arrow[d] & \BL^+ \otimes_{\BR^+} \Fil^k M_R \arrow[r] \arrow[d] & \BL^+ \otimes_{\BR^+} \gr^k M_R \arrow[r] \arrow[d] & 0\\
			0 \arrow[r] & \Fil^{k+1} M_L \arrow[r] & \Fil^k M_L \arrow[r] & \gr^k M_L \arrow[r] & 0,
		\end{tikzcd}
	\end{equation}
	where the top row is the scalar extension of the top exact row in \eqref{eq:gr_mr_ml_1} along the natural flat map $\BR^+ \rightarrow \BL^+$ and the vertical maps are the $\BL^+\textrm{-linear}$ extension of the corresponding maps in \eqref{eq:gr_mr_ml_1} (the left and middle vertical arrows are naturally compatible with the isomorphism $\BL^+ \otimes_{\BR^+} M_R \isomorphic M_L$).
	From the preceding discussion, note that the right vertical arrow of \eqref{eq:gr_mr_ml_2} coincides with the extension of the injective right vertical arrow of \eqref{eq:gr_mr_ml_1} along the flat homomorphism $R[1/p] \rightarrow L$, i.e.\ the following homomorphism:
	\begin{equation*}
		\BL^+ \otimes_{\BR^+} \gr^k M_R \isomorphic L \otimes_{R[1/p]} \gr^k M_R \longrightarrow \gr^k M_L.
	\end{equation*}
	Using the Frobenius on $M_R$, we will show that the second homomorphism above is injective.
	Indeed, let us first note that by definition the semilinear homomorphism $\varphi \colon \Fil^k M_R \hookrightarrow [p]_q^k M_R$ induces a semilinear injective homomorphism $\overline{\varphi} \colon \gr^k M_R \hookrightarrow [p]_q^k M_R/[p]_q^{k+1} M_R$.
	In other words, $\overline{\varphi}$ is linear over the source of the map $\varphi \colon R[1/p] \rightarrow R[1/p]$, and so we see that extension of scalars of $\overline{\varphi}$ along the flat homomorphism $R[1/p] \rightarrow L$, yields an injective map $\varphi \otimes \overline{\varphi} \colon L \otimes_{R[1/p]} \gr^k M_R \hookrightarrow \varphi(L) \otimes_{\varphi(R[1/p])} [p]_q^k M_R/[p]_q^{k+1} M_R$.
	Note that the target of $\varphi \otimes \overline{\varphi}$ identifies with $L \otimes_{R[1/p]} [p]_q^k M_R/[p]_q^{k+1} M_R$.
	To see this, recall that $\varphi \colon R \rightarrow R$ is finite free of degree $p^d$ with a set of basis elements given as $\{\prod_{k=1}^d X_k^{j_k}\}_{0 \leqslant j_k \leqslant p-1}$, and similarly for $\varphi \colon O_L \rightarrow O_L$.
	In particular, we may write $R = \oplus_{\smbfj \in J} \varphi(R) (\prod_{k=1}^d X_k^{j_k})$, where $J = \{(j_1, \ldots, j_d), \textrm{ with } 0 \leqslant j_k \leqslant p-1\}$, and similarly $O_L = \oplus_{\smbfj \in J} \varphi(O_L) (\prod_{k=1}^d X_k^{j_k}) = (\oplus_{\smbfj \in J} \varphi(R) (\prod_{k=1}^d X_k^{j_k})) \otimes_{\varphi(R)} \varphi(O_L)$, thus yielding the equality in the following diagram:
	\begin{equation*}
		\varphi \otimes \overline{\varphi} \colon L \otimes_{R[1/p]} \gr^k M_R \hookrightarrow \varphi(L) \otimes_{\varphi(R[1/p])} [p]_q^k M_R/[p]_q^{k+1} M_R = L \otimes_{R[1/p]} [p]_q^k M_R/[p]_q^{k+1} M_R.
	\end{equation*}
	Similar to above, note that by definition we also have a semilinear injective homomorphism $\overline{\varphi} \colon \gr^k M_L \hookrightarrow [p]_q^k M_L/[p]_q^{k+1} M_L$, and we consider the following diagram:
	\begin{center}
		\begin{tikzcd}
			L \otimes_{R[1/p]} \gr^k M_R \arrow[r, hookrightarrow, "\varphi \otimes \overline{\varphi}"] \arrow[d] & L \otimes_{R[1/p]} [p]_q^k M_R/[p]_q^{k+1} M_R \arrow[d, "\wr"]\\
			\gr^k M_L \arrow[r, hookrightarrow, "\overline{\varphi}"] & {[p]_q^k M_L/[p]_q^{k+1} M_L},
		\end{tikzcd}
	\end{center}
	where the right vertical arrow is bijective because $M_L = \BL^+ \otimes_{\BR^+} M_R$, and the left vertical arrow is the right vertical arrow of \eqref{eq:gr_mr_ml_2}.
	As the Frobenius structure on $M_R$ and $M_L$ are compatible, it follows that the diagram commutes and its left vertical arrow, and hence the right vertical arrow of \eqref{eq:gr_mr_ml_2}, is injective.

	Next, to show the bijectivity of \eqref{eq:nygaard_fil_nr_nl} (after inverting $p$, i.e.\ for $M_R$), we will proceed by induction on $k$, i.e.\ assume that the middle vertical arrow of \eqref{eq:gr_mr_ml_2} is an isomorphism for some $k \geqslant 0$.
	Then, it follows that the right vertical arrow of \eqref{eq:gr_mr_ml_2} is surjective and it is injective by the discussion above, hence bijective.
	So, we conclude that the left vertical arrow of \eqref{eq:gr_mr_ml_2} must be bijective as well, i.e.\ $\BL^+ \otimes_{\BR^+} \Fil^{k+1} M_R \isomorphic \Fil^{k+1} M_L$.
	To get the bijectivity of \eqref{eq:nygaard_fil_nr_nl}, note that the natural map $\AR^+ \rightarrow \AL^+$ is flat, therefore, for any $k \in \ZZ$, we have that
	\begin{align*}
		\AL^+ \otimes_{\AR^+} \Fil^k N_R &= \AL^+ \otimes_{\AR^+} (\Fil^k M_R \cap N_R)\\
			&= (\AL^+ \otimes_{\AR^+} \Fil^k M_R) \cap (\AL^+ \otimes_{\AR^+}N_R) \isomorphic \Fil^k M_L \cap N_L = \Fil^k N_L.
	\end{align*}
	Taking the associated graded pieces and using the preceding isomorphism, it follows that $O_L \otimes_R \gr^k N_R \isomorphic \gr^k N_L$, thus concluding our proof.
\end{proof}

\begin{rem}\label{rem:nygaard_fil_nl_nlbreve}
	The ideas employed in the proof of Proposition \ref{prop:nygaard_fil_nr_nl} enables us to relate the Nygaard filtration on $N_L$ to the Nygaard filtration on classical Wach modules.
	Indeed, set $\OLbreve \coloneq (\cup_{i=1}^d O_L[X_i^{1/p^{\infty}}])^{\wedge}$, where ${}^{\wedge}$ denotes the $\padic$ completion, and note that the $O_L\algebra$ $\OLbreve$ is a complete discrete valuation ring with perfect residue field, uniformiser $p$ and fraction field $\Lbreve \coloneq \OLbreve[1/p]$.
	The Witt vector Frobenius on $\OLbreve$ is given by the Frobenius on $O_L$ described in Section \ref{subsec:setup_nota} and setting $\varphi(X_i^{1/p^n}) = X_i^{1/p^{n-1}}$, for all $1 \leqslant i \leqslant d$ and $n \in \NN$.
	Let $\Lbreveinfty \coloneq \Lbreve(\mu_{p^{\infty}})$ and let $\Lbrevebar \supset \Lbar$ denote a fixed algebraic closure of $\Lbreve$.
	We have the Galois groups $G_{\Lbreve} \coloneq \Gal(\Lbrevebar/\Lbreve) \isomorphic \Gal(\Lbar/\cup_{i=1}^d L(X_i^{1/p^{\infty}}))$ and $\GammaLbreve \coloneq \Gal(\Lbreveinfty/\Lbreve) \isomorphic \Gal(L(\mu_{p^{\infty}})/L) \isomorphic \ZZ_p^{\times}$.
	Note that $G_{\Lbreve}$ may be identified with a subgroup of $G_L$ and $\GammaLbreve$ may be identified with a quotient of $\Gamma_L$.
	Next, recall that from \cite{berger-limites}, we have the theory of Wach modules over $\ALbreve^+ = \OLbreve\llbracket \mu \rrbracket$ (see \cite[Section 4.1]{abhinandan-imperfect-wach} for a quick recollection).
	Now, if $N_L$ is a Wach module over $\AL^+$, then $\NLbreve \coloneq \ALbreve^+ \otimes_{\AL^+} N_L$ is naturally a Wach module over $\ALbreve^+$ (see \cite[Corollary 4.27]{abhinandan-imperfect-wach}).
	We equip $\NLbreve$ with the Nygaard filtration similar to Definition \ref{defi:nygaard_fil}, and observe that statements analogous to Lemma \ref{lem:nygaard_fil_twist} and Lemma \ref{lem:nygaard_muadic_base} hold for the filtration on $\NLbreve$ (follows by arguments similar to the proofs of these statements).
	Then, employing an argument similar to the proof of Proposition \ref{prop:nygaard_fil_nr_nl} shows that, for each $k \in \ZZ$, we have that $\Fil^k N_L = N_L \cap \Fil^k \NLbreve \subset \NLbreve$ and $\ALbreve^+ \otimes_{\AL^+} \Fil^k N_L \isomorphic \Fil^k \NLbreve$.
	In particular, we see that $\gr^k N_L \hookrightarrow \gr^k \NLbreve$ and $\OLbreve \otimes_{O_L} \gr^k N_L \isomorphic \gr^k \NLbreve$.
	Analogous statements are also true after inverting $p$.
\end{rem}

\subsubsection{Reduction modulo \texorpdfstring{$\mu$}{-} of the Nygaard filtration}\label{subsubsec:nygaardfil_modmu}

Let $N_R$ be a Wach module over $\AR^+$ and set $M_R \coloneq N_R[1/p]$.
Note that $N_R/\mu$ is equipped with a decreasing $R\linear$ filtration $\Fil^k(N_R/\mu)$ given as the image of $\Fil^k N_R$ under the surjection $N_R \twoheadrightarrow N_R/\mu$.
We equip $M_R/\mu$ with the induced $R[1/p]\linear$ filtration $\Fil^k(M_R/\mu) \coloneq (\Fil^k(N_R/\mu))[1/p]$, and note that $M_R/\mu$ is a $\varphi\module$ over $R[1/p]$ since $[p]_q = p \textmod \mu \AR^+$, in particular, it is a filtered $\varphi\module$ over $R[1/p]$.
Similarly, we equip the reduction modulo $\mu$ of a Wach module over $\AL^+$ (resp.\ $\BL^+$) with a decreasing $O_L\linear$ (resp.\ $L\linear$) filtration and note that, after inverting $p$, it is a filtered $\varphi\module$ over $L$.

\begin{lem}\label{lem:fil_gr_modmu}
	For each $k \in \ZZ$, the following natural $\AR^+\linear$ sequence is exact:
	\begin{equation}\label{eq:filnmodmu}
		0 \longrightarrow \Fil^{k-1} N_R \xrightarrow{\hspace{1mm} \mu \hspace{1mm}} \Fil^k N_R \longrightarrow \Fil^k (N_R/\mu) \longrightarrow 0.
	\end{equation}
	Moreover, by taking the associated graded pieces, we obtain the following natural $R\linear$ exact sequence:
	\begin{equation}\label{eq:grnmodmu}
		0 \longrightarrow \gr^{k-1} N_R \xrightarrow{\hspace{1mm} \mu \hspace{1mm}} \gr^k N_R \longrightarrow \gr^k (N_R/\mu) \longrightarrow 0.
	\end{equation}
	Additionally, \eqref{eq:filnmodmu} induces the following natural $R\linear$ exact sequence:
	\begin{equation}\label{eq:filn_modmu_filnmodmu}
		0 \longrightarrow \gr^{k-1} N_R \xrightarrow{\hspace{1mm} \mu \hspace{1mm}} (\Fil^k N_R)/\mu \longrightarrow \Fil^k (N_R/\mu) \longrightarrow 0.
	\end{equation}
	Furthermore, the natural inclusion $\mu\Fil^k N_R \subset \Fil^{k+1} N_R$ induces the following $R\linear$ exact sequence:
	\begin{equation}\label{eq:filn_modmu_grn}
		0 \longrightarrow \Fil^{k+1} (N_R/\mu) \longrightarrow (\Fil^k N_R)/\mu \longrightarrow \gr^k N_R \longrightarrow 0.
	\end{equation}
	Analogous statements are also true after inverting $p$, i.e.\ over $\BR^+$, and we have that $\Fil^k(M_R/\mu)$ coincides with the image of $\Fil^k M_R$ under the surjective map $M_R \twoheadrightarrow M_R/\mu$.
	Moreover, analogous claims also hold for Wach modules over $\AL^+$ and $\BL^+$.
\end{lem}
\begin{proof}
	We will only show the claim for Wach modules over $\AR^+$ and $\BR^+$, and note that a similar reasoning shows the claim for Wach modules over $\AL^+$ and $\BL^+$.
	For the first claim, observe that the kernel of the surjective map $\Fil^k N_R \twoheadrightarrow \Fil^k(N_R/\mu)$ is given as $\Fil^k N_R \cap \mu N_R = \mu\Fil^{k-1} N_R$ (see Lemma \ref{lem:nygaard_fil_twist} (2)).
	This shows the exactness of \eqref{eq:filnmodmu}.
	For the second claim, consider the following commutative diagram:
	\begin{center}
		\begin{tikzcd}[row sep=15pt]
			0 \arrow[r] & \Fil^k N_R \arrow[r, "\mu"] \arrow[d, hookrightarrow] & \Fil^{k+1} N_R \arrow[r] \arrow[d, hookrightarrow] & \Fil^{k+1}(N_R/\mu) \arrow[r] \arrow[d, hookrightarrow] & 0\\
			0 \arrow[r] & \Fil^{k-1} N_R \arrow[r, "\mu"] \arrow[d, twoheadrightarrow] & \Fil^k N_R \arrow[r] \arrow[d, twoheadrightarrow] & \Fil^k(N_R/\mu) \arrow[r] \arrow[d, twoheadrightarrow] & 0\\
			0 \arrow[r] & \gr^{k-1} N_R \arrow[r, "\mu"] & \gr^k N_R \arrow[r] & \gr^k(N_R/\mu) \arrow[r] & 0,
		\end{tikzcd}
	\end{center}
	where the top and the middle rows are exact by \eqref{eq:filnmodmu} and the columns are exact by definition.
	Therefore, it follows that the bottom row is exact as well, thus yielding the exact sequence in \eqref{eq:grnmodmu}.

	For the third claim, consider the following commutative diagram with exact rows:
	\begin{center}
		\begin{tikzcd}[row sep=15pt]
			0 \arrow[r] & \Fil^k N_R \arrow[r, "\mu"] \arrow[d] & \Fil^k N_R \arrow[r] \arrow[d, equal] & (\Fil^k N_R)/\mu \arrow[r] \arrow[d] & 0\\
			0 \arrow[r] & \Fil^{k-1} N_R \arrow[r, "\mu"] & \Fil^k N_R \arrow[r] & \Fil^k(N_R/\mu) \arrow[r] & 0,
		\end{tikzcd}
	\end{center}
	where the left vertical arrow is the natural inclusion and the right vertical arrow is surjective.
	Then, an application of the snake lemma, gives us the exact sequence in \eqref{eq:filn_modmu_filnmodmu}.
	For the fourth claim, note that the natural inclusion $\mu\Fil^k N_R \subset \Fil^{k+1} N_R$ induces a natural surjective $R\linear$ homomorphism $(\Fil^k N_R)/\mu \twoheadrightarrow (\Fil^k N_R)/(\Fil^{k+1} N_R) = \gr^k N_R$, whose kernel is given as $(\Fil^{k+1} N_R)/(\mu\Fil^k N_R) \isomorphic \Fil^{k+1}(N_R/\mu)$ by \eqref{eq:filnmodmu}.
	Hence, the sequence \eqref{eq:filn_modmu_grn} is exact.

	Next, as we have that $\Fil^k M_R = (\Fil^k N_R)[1/p]$ and $\Fil^k(M_R/\mu) = \Fil^k(N_R/\mu)[1/p]$, the analogous claims for $M_R$ follow by inverting $p$ in the preceding argument.
	Finally, let $\textup{F}^k(M_R/\mu)$ denote the image of $\Fil^k M_R$ under the surjective map $M_R \twoheadrightarrow M_R/\mu$.
	Then, the kernel of the surjection $\Fil^k M_R \twoheadrightarrow \textup{F}^k(M_R/\mu)$ is given as $\Fil^k M_R \cap \mu M_R = \mu\Fil^{k-1} M_R$ (see Lemma \ref{lem:nygaard_fil_twist} (2)).
	So, from the exactness of \eqref{eq:filnmodmu} (after inverting $p$), it follows that $\textup{F}^k(M_R/\mu) = \Fil^k(M_R/\mu)$.
	This concludes our proof.
\end{proof}

\begin{rem}\label{rem:fil_gr_nlbreve_modmu}
	In the notations of Remark \ref{rem:nygaard_fil_nl_nlbreve}, let $\NLbreve$ denote a Wach module over $\ALbreve^+$.
	Similar to above, we equip $\NLbreve/\mu$ with a decreasing $O_{\Lbreve}\linear$ filtration given as the image of $\Fil^k \NLbreve$ under the surjective map $\NLbreve \twoheadrightarrow \NLbreve/\mu$.
	Then, claims analogous to the ones in Lemma \ref{lem:fil_gr_modmu} are also true for $\NLbreve$, by using similar arguments as in the proof of the lemma.
	Moreover, after equipping $\MLbreve \coloneq (\NLbreve/\mu)[1/p]$ with the filtration $\Fil^k(\MLbreve/\mu) \coloneq (\Fil^k(\NLbreve/\mu))[1/p]$, we note that it is a filtered $\varphi\module$ over $\Lbreve$ and analogous claims from Lemma \ref{lem:fil_gr_modmu} also hold true for $\MLbreve$.
\end{rem}

Let $N_R$ be a Wach module over $\AR^+$ as above and $M_R = N_R[1/p]$.
Set $N_L \coloneq \AL^+ \otimes_{\AR^+} N_R$ and $M_L \coloneq N_L[1/p]$ equipped with the natural actions of $\varphi$ and $\Gamma_L \isomorphic \Gamma_R$.
Our next goal is to compare $\Fil^k(N_R/\mu)$ and $\Fil^k(N_L/\mu)$.
To that end, let us first note that we have the following commutative diagram:
\begin{equation}\label{eq:nr_nl_modmu_inject}
	\begin{tikzcd}[row sep=15pt]
		\Fil^k(N_R/\mu) \arrow[r] \arrow[d] & N_R/\mu \arrow[r] \arrow[d] & M_R/\mu \arrow[d] \\
		\Fil^k(N_L/\mu) \arrow[r] & N_L/\mu \arrow[r] & M_L/\mu,
	\end{tikzcd}
\end{equation}
where all arrows are injective.
Indeed, the top and the bottom horizontal arrows are natural inclusions (because $N_R/\mu$ is $p\textrm{-torsion}$ free and $R \rightarrow O_L$ is flat).
Moreover, the right vertical arrow in \eqref{eq:nr_nl_modmu_inject} is injective because it identifies with the tensor product of the injective map $R[1/p] \hookrightarrow L$ with the projective $R[1/p]\module$ $M_R/\mu$.
So, it follows that the left and the middle vertical arrows are injective as well.
Moreover, we have the following:
\begin{lem}\label{lem:filnrnl_modmu_injective}
	The natural isomorphism in \eqref{eq:nygaard_fil_nr_nl} of Proposition \ref{prop:nygaard_fil_nr_nl}, induces a natural injective homomorphism
	\begin{equation*}
		(\Fil^k N_R)/\mu \longrightarrow O_L \otimes_R (\Fil^k N_R)/\mu \isomorphic (\Fil^k N_L)/\mu.
	\end{equation*}
	Analogous statements are also true after inverting $p$.
\end{lem}
\begin{proof}
	To show that the induced homomorphism $(\Fil^k N_R)/\mu \rightarrow (\Fil^k N_L)/\mu$ is injective, it is enough to show that the following equality holds:
	\begin{equation*}
		\mu \Fil^k N_R = \Fil^k N_R \cap \mu \Fil^k N_L \subset N_L.
	\end{equation*}
	It is clear that $\mu \Fil^k N_R \subset \Fil^k N_R \cap \mu \Fil^k N_L$.
	For the converse, note that inside $N_L$, the following hold:
	\begin{align*}
		\Fil^k N_R \cap \mu \Fil^k N_L \subset \Fil^k N_R \cap \Fil^{k+1} N_L &= \Fil^{k+1} N_R,\\
		\Fil^k N_R \cap \mu \Fil^k N_L \subset N_R \cap \mu N_L &= \mu N_R,
	\end{align*}
	where the equality in the first line follows from Proposition \ref{prop:nygaard_fil_nr_nl} and the equality in the second line follows from the injectivity of the middle vertical arrow in \eqref{eq:nr_nl_modmu_inject} and Lemma \ref{lem:injectivity_modulo}.
	Combining the two equations above and using Lemma \ref{lem:nygaard_fil_twist} (2), it follows that 
	\begin{equation*}
		\Fil^k N_R \cap \mu \Fil^k N_L \subset \Fil^{k+1} N_R \cap \mu N_R = \mu\Fil^k N_R.
	\end{equation*}
	The analogous claims involving $M_R$ and $M_L$ follow by inverting $p$ in the argument above.
\end{proof}

\begin{lem}\label{lem:nygaard_fil_nr_nl_modmu}
	The natural isomorphism of $\varphi\modules$ $O_L \otimes_R (N_R/\mu) \isomorphic N_L/\mu$ is compatible with filtrations, i.e. for each $k \in \ZZ$, we have that
	\begin{equation}\label{eq:nygaard_fil_nr_nl_modmu}
		O_L \otimes_R \Fil^k (N_R/\mu) \isomorphic \Fil^k (N_L/\mu),
	\end{equation}
	in particular, we get that $O_L \otimes_R \gr^k(N_R/\mu) \isomorphic \gr^k(N_L/\mu)$.
	Moreover, if $N_R$ is an effective Wach module, then we have that
	\begin{equation}\label{eq:fil1_nrnlmodmu_induced}
		\Fil^1 (N_R/\mu) \isomorphic (N_R/\mu) \cap \Fil^1 (N_L/\mu) \subset N_L/\mu.
	\end{equation}
	Analogous statements are also true after inverting $p$.
\end{lem}
\begin{proof}
	Let us consider the following $\AL^+\linear$ commutative diagram with exact rows:
	\begin{equation}\label{lem:nygaard_fil_modmu}
		\begin{tikzcd}[column sep=18pt, row sep=16pt]
			0 \arrow[r] & \AL^+ \otimes_{\AR^+} \Fil^{k-1} N_R \arrow[r, "\mu"] \arrow[d, "\wr"', "\eqref{eq:nygaard_fil_nr_nl}"] & \AL^+ \otimes_{\AR^+} \Fil^k N_R \arrow[r] \arrow[d, "\wr"', "\eqref{eq:nygaard_fil_nr_nl}"] & O_L \otimes_R \Fil^k (N_R/\mu N_R) \arrow[r] \arrow[d, "\eqref{eq:nygaard_fil_nr_nl_modmu}"] & 0\\
			0 \arrow[r] & \Fil^{k-1} N_L \arrow[r, "\mu"] & \Fil^k N_L \arrow[r] & \Fil^k (N_L/\mu N_L) \arrow[r] & 0,
		\end{tikzcd}
	\end{equation}
	where the top row is the extension of the exact sequence \eqref{eq:filnmodmu} along the flat map $\AR^+ \rightarrow \AL^+$ and the bottom row is the exact sequence (analogous to \eqref{eq:filnmodmu}) for Wach modules over $\AL^+$.
	Now, note that the map in \eqref{eq:nygaard_fil_nr_nl_modmu} (also see the right vertical arrow in \eqref{lem:nygaard_fil_modmu}) is compatible with the natural map in \eqref{eq:nygaard_fil_nr_nl}, so the bottom right square in \eqref{lem:nygaard_fil_modmu} commutes, and therefore, it follows that \eqref{eq:nygaard_fil_nr_nl_modmu} is bijective, thus proving the first claim.
	For the second claim, from the diagram \eqref{eq:nr_nl_modmu_inject}, it is clear that \eqref{eq:fil1_nrnlmodmu_induced} is injective and it remains to show that it is also surjective.
	So, let $x$ be an element of $(N_R/\mu) \cap \Fil^1 (N_L/\mu)$ and let $y$ in $N_R$ (resp.\ $z$ in $\Fil^1 N_L$) denote a lift of $x$.
	As we have that $N_R \subset N_L$, therefore, we see that $y-z$ is in $\mu N_L \subset \Fil^1 N_L$.
	But, then we get that $y$ is in $N_R \cap \Fil^1 N_L = \Fil^1 N_R$ (see Proposition \ref{prop:nygaard_fil_nr_nl}), in particular, the image of $y$ under the homomorphism $N_R \twoheadrightarrow N_R/\mu$, i.e.\ $x$, is in $\Fil^1(N_R/\mu)$.
	Finally, after inverting $p$ in the arguments above, the analogous claims involving $M_R/\mu$ and $M_L/\mu$ follow.
	This allows us to conclude.
\end{proof}

\begin{rem}\label{rem:nygaard_fil_nl_nlbreve_modmu}
	In the notation of Remark \ref{rem:nygaard_fil_nl_nlbreve} and Remark \ref{rem:fil_gr_nlbreve_modmu}, let $N_L$ denote a Wach module over $\AL^+$ and set $\NLbreve \coloneq \ALbreve^+ \otimes_{\AL^+} N_L$.
	Then, arguments similar to the proof of Lemma \ref{lem:nygaard_fil_nr_nl_modmu} show that the natural isomorphism of $\varphi\modules$ $\OLbreve \otimes_{O_L} (N_L/\mu N_L) \isomorphic \NLbreve/\mu \NLbreve$ is compatible with filtrations, i.e.\ for each $k \in \ZZ$, we have that
	\begin{equation*}
		\OLbreve \otimes_{O_L} \Fil^k (N_L/\mu N_L) \isomorphic \Fil^k (\NLbreve/\mu \NLbreve),
	\end{equation*}
	in particular, $O_{\Lbreve} \otimes_{O_L} \gr^k(N_L/\mu) \isomorphic \gr^k(N_{\Lbreve}/\mu)$.
	Analogous statements are also true after inverting $p$.
\end{rem}

\subsection{Wach modules are crystalline}\label{subsec:wachmod_crystalline}

The goal of this section is to prove Theorem \ref{thm:fh_crys_relative}.
To state and prove our result, we will need some auxiliary period rings $\ARpi^{\PD}$ and $\OARpi^{\PD}$ from \cite[Section 4.3.1]{abhinandan-relative-wach-i}.
We briefly recall their definitions.
Let $\varpi \coloneq \zeta_{p^m}-1$, where $m = 1$, for $p \geqslant 3$, and $m = 2$, for $p = 2$, and set 
\begin{equation*}
	\ARpi^+ \coloneq \AR^+[Y]/(Y^{p^m}-q) \isomorphic \AR^+[q^{1/p^m}] = \AR^+[\varphi^{-m}(q)] \subset \Ainf(\Rinfty),
\end{equation*}
where $q = 1+\mu = [\varepsilon]$.
By the definition, it is clear that $\ARpi^+$ is finite free as a module over $\AR^+$, with a basis given as $\{1, q^{1/p^m}, \ldots, q^{(p^m-1)/p^m}\}$.
Moreover, it is also clear that $\ARpi^+$ is stable under the Frobenius on $\Ainf(\Rinfty)$, and as the action of $\Gamma_R$ on $\Ainf(\Rinfty)$ is Frobenius-equivariant, therefore, it follows that $\ARpi^+$ is stable under the action of $\Gamma_R$; we equip it with the induced structures.

Next, we note that the restriction to $\ARpi^+$ of the map $\theta$ on $\Ainf(\Rinfty)$ (see Section \ref{subsec:ainf_relative}), gives a surjective ring homomorphism $\theta: \ARpi^+ \twoheadrightarrow R[\varpi]$.
We define $\ARpi^{\PD}$ to be the $\padic$ completion of the divided power envelope of $\ARpi^+$, with respect to $\kert \theta$.
Furthermore, the map $\theta$ extends $R\textrm{-linearly}$ to a surjective ring homomorphism $\theta_R \colon R \otimes_{\ZZ} \ARpi^+ \twoheadrightarrow R[\varpi]$, given as $x \otimes y \mapsto x\theta(y)$.
Similar to above, we define $\OARpi^{\PD}$ to be the $\padic$ completion of the divided power envelope of $(R \otimes_{\ZZ} \ARpi^+)$, with respect to $\kert \theta_R$.
The morphisms $\theta$ and $\theta_R$ naturally extend to respective surjections $\theta \colon \ARpi^{\PD} \twoheadrightarrow R[\varpi]$ and $\theta_R \colon \OARpi^{\PD} \twoheadrightarrow R[\varpi]$.

\begin{lem}\label{lem:oarpipd_oacrys}
	The natural $(\varphi, \Gamma_R)\equivariant$ injective homomorphism $\ARpi^+ \rightarrow \Ainf(\Rinfty)$ extends to natural injective homomorphisms $\ARpi^{\PD} \rightarrow \Acrys(\Rinfty)$ and $\OARpi^{\PD} \rightarrow \OAcrys(\Rinfty)$.
	Moreover, the source rings are stable under the $(\varphi, \Gamma_R)\action$ on the target rings.
\end{lem}
\begin{proof}
	By the universal property of PD-envelopes, the $(\varphi, \Gamma_R)\equivariant$ injective homomorphism $\ARpi^+ \rightarrow \Ainf(\Rinfty)$ extends to natural homomorphisms $\ARpi^{\PD} \rightarrow \Acrys(\Rinfty)$ and $\OARpi^{\PD} \rightarrow \OAcrys(\Rinfty)$, and we need to show that these maps are injective.
	Let us first note that $\ARpi^{\PD}$ and $\Acrys(\Rinfty)$ are $p\textrm{-adically}$ complete and $p\torsion$ free (see \cite[Proposition 6.1.3]{brinon-relatif} and \cite[Lemma 2.43]{bhatt-scholze-prisms}).
	Moreover, as $\lim_n$ is left exact, therefore, it is enough to show that the induced homomorphism $\ARpi^{\PD}/p^n \rightarrow \Acrys(\Rinfty)/p^n$ is injective, where the case $n=1$ would suffice as the injectivity for $n > 1$ follows by an easy induction.
	Now, consider the following commutative diagram:
	{\small
	\begin{equation*}
		\begin{tikzcd}[column sep=12pt, row sep=16pt]
			\ARpi^{\PD}/p \arrow[r, "\sim"] \arrow[d] & (\ARpi^+/p)[Y_0, Y_1, \ldots]/(pY_0-\xi^p, pY_{k+1}-Y_k^p)_{k \geqslant 0} \arrow[r, "\sim"] \arrow[d] & (R[\zeta_{p^m}]/p)[Y_0, Y_1, \ldots]/(Y_k^p)_{k \geqslant 0} \arrow[d]\\
			\Acrys(\Rinfty)/p \arrow[r, "\sim"] & (\Ainf(\Rinfty)/p)[Y_0, Y_1, \ldots]/(pY_0-\xi^p, pY_{k+1}-Y_k^p)_{k \geqslant 0} \arrow[r, "\sim"] & (\Rinfty/p)[Y_0, Y_1, \ldots]/(Y_k^p)_{k \geqslant 0},
		\end{tikzcd}
	\end{equation*}
	}
	where the left vertical arrow is the natural homomorphism described above and the top (resp.\ bottom) left isomorphism is obtained by an argument similar to \cite[Corollaire 6.1.2]{brinon-relatif}.
	For the top right isomorphism, recall that we have $\ARpi^+ = \AR^+[q^{1/p^m}]\isomorphic R\llbracket q-1 \rrbracket[q^{1/p^m}]$ and $[p]_q = \varphi(\xi) = \xi^p \textrm{ mod } p$, therefore, we see that $\ARpi^+/(p, \xi^p) = \ARpi^+/([p]_q, p) \isomorphic R[\zeta_{p^m}]/p$.
	Similarly, for the bottom right isomorphism, observe that $\Ainf(\Rinfty)/(p, \xi^p) = \Ainf(\Rinfty)/([p]_q, p) \isomorphic \Rinfty/p$.
	Moreover, the middle and the right vertical arrows are defined using the left vertical arrow and an easy diagram chase shows that $Y_k$ in the top row goes to $Y_k$ in the bottom row, for each $k \geqslant 0$.
	It is clear that the right vertical arrow is injective, so it follows that the left vertical arrow is also injective.
	Furthermore, using the fact that $\OAcrys(\Rinfty)$ (resp.\ $\OARpi^{\PD}$) is the $p\adic$ completion of a PD-polynomial algebra over $\Acrys(\Rinfty)$ (resp.\ $\ARpi^{\PD}$) in $d$ variables (see \cite[Proposition 6.1.5]{brinon-relatif}, resp.\ \cite[Lemma 4.20]{abhinandan-relative-wach-i}), we get that the natural map $\OARpi^{\PD} \rightarrow \OAcrys(\Rinfty)$ is injective.
	From the preceding descriptions, it is easy to see that $\ARpi^{\PD}$ (resp.\ $\OARpi^{\PD}$) is stable under the $(\varphi, \Gamma_R)\action$ on $\Acrys(\Rinfty)$ (resp.\ $\OAcrys(\Rinfty)$).
\end{proof}

Using the injectivity of the map $\ARpi^{\PD} \hookrightarrow \Acrys(\Rinfty)$ (resp.\ $\OARpi^{\PD} \hookrightarrow \OAcrys(\Rinfty)$) from Lemma \ref{lem:oarpipd_oacrys}, we equip the source rings with structures induced from the target rings, in particular, a Frobenius endomorphism $\varphi$, a continuous action of $\Gamma_R$, and an induced $\Gamma_R\textrm{-stable}$ decreasing, separated and exhaustive filtration.
In particular, the natural homomorphism $\ARpi^{\PD} \rightarrow \OARpi^{\PD}$ is injective and compatible with all the structures.
Moreover, as the filtration on $\Acrys(\Rinfty)$ (resp.\ $\OAcrys(\Rinfty)$) is given by the divided powers of the ideal $\kert \theta \subset \Acrys(\Rinfty)$ (resp.\ $\kert \theta_R \subset \OAcrys(\Rinfty)$), therefore, the induced filtration on $\ARpi^{\PD}$ (resp.\ $\OARpi^{\PD}$) is also given by the divided powers of the ideal $\kert \theta \subset \ARpi^{\PD}$ (resp.\ $\kert \theta_R \subset \OARpi^{\PD}$).
Furthermore, the $\Acrys(\Rinfty)\linear$ integrable connection on $\OAcrys(\Rinfty)$, denoted as $\partial$, induces a $\Gamma_R\equivariant$ $\ARpi^{\PD}\linear$ integrable connection $\partial_A \colon \OARpi^{\PD} \rightarrow \OARpi^{\PD} \otimes \Omega^1_R$.
Note that the connection $\partial_A$ satisfies Griffiths transversality with respect to the induced filtration: indeed, by definitions we have that
\begin{align*}
	\partial_A(\Fil^k \OARpi^{\PD}) &\subset (\OARpi^{\PD} \otimes \Omega^1_R) \cap \partial(\Fil^k \OAcrys(\Rinfty)) \\
		&\subset (\OARpi^{\PD} \cap \Fil^{k-1} \OAcrys(\Rinfty)) \otimes \Omega^1_R = \Fil^{k-1} \OARpi^{\PD} \otimes \Omega^1_R,
\end{align*}
where the second inclusion follows because the connection on $\OAcrys(\Rinfty)$ satisfies Griffiths transversality with respect to the filtration (see Section \ref{subsec:crystalline_relative}).
Additionally, we have that $\OAcrys(\Rinfty)^{\partial = 0} = \Acrys(\Rinfty)$, and so it follows that $(\OARpi^{\PD})^{\partial_A=0} = \ARpi^{\PD}$.

\begin{rem}\label{rem:fil1ar+_intersect}
	Note that for the composition $\AR^+ \hookrightarrow \Ainf(\Rinfty) \xrightarrow{\hspace{1mm}\theta\hspace{1mm}} \widehat{R}_{\infty}$, we have that $\theta(\AR^+) = R$, and therefore, $\kert \theta = \mu \AR^+ \subset \AR^+$.
	So, it follows that $\Fil^1 \OARpi^{\PD} \cap \AR^+ = \Fil^1 \ARpi^{\PD} \cap \AR^+ = \mu \AR^+ \subset \OARpi^{\PD}$.
\end{rem}

Using the period rings described above, we are ready to state the main result of this section:
\begin{thm}\label{thm:fh_crys_relative}
	Let $N$ be a Wach module over $\AR^+$ and let $T \coloneq \TR(N)$ denote the associated finite free $\ZZ_p\representation$ of $G_R$ from \eqref{eq:wach_reps_relative}.
	Then, $V \coloneq T[1/p]$ is a $\padic$ crystalline representation of $G_R$ and we have a natural isomorphism of $R[1/p]\textrm{-modules}$ $(\OARpi^{\PD} \otimes_{\AR^+} N[1/p])^{\Gamma_R} \isomorphic \ODcrysR(V)$ compatible with the respective Frobenii and connections.
\end{thm}

In order to prove Theorem \ref{thm:fh_crys_relative}, we need the following key statement:
\begin{prop}\label{prop:oarpd_comparison}
	Let $N$ be an effective Wach module over $\AR^+$.
	Then, the $R[1/p]\module$ $\ODR \coloneq \big(\OARpi^{\PD} \otimes_{\AR^+} N[1/p]\big)^{\Gamma_R}$ is finite projective of rank $= \rank_{\BR^+} N[1/p]$ and equipped with a Frobenius and an integrable connection.
	Moreover, we have a natural comparison isomorphism
	\begin{equation}\label{eq:oarpd_comparison}
		\begin{aligned}
			f \colon \OARpi^{\PD} \otimes_R \ODR &\isomorphic \OARpi^{\PD} \otimes_{\AR^+} N[1/p]\\
				a \otimes b \otimes x &\longmapsto ab \otimes x,
		\end{aligned}
	\end{equation}
	compatible with the respective Frobenii, connections and actions of $\Gamma_R$.
\end{prop}

\begin{rem}\label{rem:oarpd_comparison_struct}
	In \eqref{eq:oarpd_comparison}, the Frobenius on each term is given as $\varphi \otimes \varphi$; the integrable connection on the right-hand term is given as the natural $\ARpi^{\PD}\textrm{-linear}$ differential operator $\partial_A \otimes 1$ and on the left-hand term, it is given as $\partial_A \otimes 1 + 1 \otimes \partial_D$, where $\partial_D$ is the integrable connection on $\ODR$ (induced from the integrable connection on $\OARpi^{\PD}$); the action of any $g$ in $\Gamma_R$ on the left-hand term is given as $g \otimes 1$ and on the right-hand term, it is given as $g \otimes g$.
\end{rem}

\begin{proof}[Proof of Proposition \ref{prop:oarpd_comparison}]
	We will adapt the proof of \cite[Proposition 4.28]{abhinandan-relative-wach-i}.
	Recall the following rings from \cite[Section 4.4.1]{abhinandan-relative-wach-i}:
	for $n \geqslant 1$, consider the following $p\torsion$ free $\AR^+\algebra$:
	\begin{equation*}
		\AR^+\big[\tfrac{\mu}{p^n}, \tfrac{\mu^2}{2!p^{2n}}, \ldots, \tfrac{\mu^k}{k!p^{kn}}, \ldots\big] \hookrightarrow \AR^+[1/p] = \BR^+,
	\end{equation*}
	and define $\SnPD$ to be the $\padic$ completion of the left-hand term, in particular, $\SnPD$ is a $p\torsion$ free $\AR^+\algebra$.
	Moreover, we have an injective Frobenius homomorphism $\varphi \colon \SnPD \rightarrow S_{n-1}^{\PD}$, and the $n\textrm{-fold}$ composition of $\varphi$ induces an inclusion $\varphi^n(\SnPD) \subset \ARpi^{\PD}$, and the ring $\SnPD$ is further equipped with a continuous (for the $\padic$ topology) action of $\Gamma_R$ which commutes with $\varphi$.
	The injectivity of the homomorphism $\varphi$ may be seen as follows: using the $\varphi\equivariant$ isomorphism $\AR^+ \isomorphic R\llbracket \mu \rrbracket$, we may view $\SnPD$ as the $\padic$ completion of $R\big[\mu, \tfrac{\mu}{p^n}, \tfrac{\mu^2}{2!p^{2n}}, \ldots, \tfrac{\mu^k}{k!p^{kn}}, \ldots\big]$, in particular, as a subring of $R[1/p]\llbracket \mu \rrbracket$, and note that on the latter the endomorphism $\varphi$, defined as the extension of the Frobenius on $R[1/p]$ by setting $\varphi(\mu) = (1+\mu)^p-1$, is injective.
	Furthermore, let us remark that in \cite[Section 4.4.1]{abhinandan-relative-wach-i} we considered a further completion of $\SnPD$, with respect to a certain filtration by PD-ideals, which we denote as $\SnhatPD$ in loc.\ cit.
	However, such a completion is not strictly necessary and all the proofs of loc.\ cit.\ may be carried out without it.
	In particular, note that many good properties of $\SnhatPD$ restrict to good properties on $\SnPD$ as well (for example, the $(\varphi, \Gamma_R)\action$ above).

	Next, let us consider the $O_F\textrm{-linear}$ homomorphism of rings $\iota \colon R \rightarrow \SnPD$, defined by sending $X_j \mapsto [X_j^{\flat}]$, for $1 \leqslant j \leqslant d$.
	Using $\iota$ we define an $O_F\textrm{-linear}$ homomorphism of rings $\jmath \colon R \otimes_{O_F} \SnPD \rightarrow \SnPD$, sending $a \otimes b \mapsto \iota(a) b$.
	Let $\OSnPD$ denote the $\padic$ completion of the divided power envelope of $R \otimes_{O_F} \SnPD$, with respect to $\kert \jmath$.
	On $\OSnPD$, the tensor product Frobenius induces an injective homomorphism $\varphi \colon \OSnPD \rightarrow \pazo S_{n-1}^{\PD}$, so that $\varphi^n(\OSnPD) \subset \OARpi^{\PD}$.
	Moreover, we have a natural Frobenius-equivariant injective homomorphism of rings $\SnPD \hookrightarrow \OSnPD$.

\begin{lem}
	The continuous $\Gamma_R\textrm{-action}$ on $\SnPD$ naturally extends to a continuous $\Gamma_R\action$ on $\OSnPD$.
	Moreover, the $\Gamma_R\textrm{-invariant}$ elements of $\OSnPD$ are given by $R$, i.e. $(\OSnPD)^{\Gamma_R} = R$.
\end{lem}
\begin{proof}
	Let us first note that using the composition $\iota \colon R[1/p] \rightarrow \BR^+ \hookrightarrow \BdR^+(\Rbar)$, any summation $\sum_{k \in \NN} \iota(a_k) \mu^k$, with $a_k$ in $R[1/p]$ for all $k \in \NN$, converges in $\BdR^+(\Rbar)$.
	As any element $x$ in $\SnPD$ has a unique presentation as the summation above (using that $\SnPD \hookrightarrow R[1/p]\llbracket \mu \rrbracket$ from the preceding paragraphs), therefore, we see that the natural $G_R\equivariant$ injective homomorphism $\BR^+ \hookrightarrow \BdR^+(\Rbar)$ extends to a natural $G_R\equivariant$ injective homomorphism $\SnPD \hookrightarrow \BdR^+(\Rbar)$.
	Furthermore, from the definition of $\OSnPD$ and $\OBdR^+(\Rbar)$, we see that the preceding homomorphism induces a natural $R\linear$ homomorphism $\OSnPD \rightarrow \OBdR^+(\Rbar)$.
	Now, note that using an argument similar to \cite[Proposition 6.1.5]{brinon-relatif} and \cite[Lemma 4.20]{abhinandan-relative-wach-i}, we may write $\OSnPD$ as the $\padic$ completion of a PD-polynomial algebra over $\SnPD$ in $d$ variables (see the diagram below), and therefore, using the presentation of $\OBdR^+(\Rbar)$ as a power series algebra (in the same variables) over $\BdR^+(\Rbar)$ (see \cite[Proposition 5.2.2]{brinon-relatif}), we conclude that the map $\OSnPD \rightarrow \OBdR^+(\Rbar)$ is injective.
	In particular, we have the following commutative diagram:
	\begin{equation}\label{eq:OSnPD_in_OBdR+}
		\begin{tikzcd}
			\SnPD[z_1, \ldots, z_d]_{\PD}^{\wedge} \arrow[r, hookrightarrow] \arrow[d, "\wr"] & \BdR^+\llbracket z_1, \ldots, z_d\rrbracket \arrow[d, "\wr"]\\
			\OSnPD \arrow[r, hookrightarrow] & \OBdR^+(\Rbar),
		\end{tikzcd}
	\end{equation}
	where the injective horizontal arrows are as described above, the left vertical arrow sends $z_j^{[k_j]}$ to $(1 \otimes [X_j^{\flat}] - X_j \otimes 1)^{[k_j]}$, and similarly for the right vertical arrow.
	In the following, we shall abuse notations and write $[X_j]^{\flat}$ for $1 \otimes [X_j^{\flat}]$ and $X_j$ for $X_j \otimes 1$.

	We equip the top right corner of \eqref{eq:OSnPD_in_OBdR+} with a $\BdR^+(\Rbar)\textrm{-semilinear}$ action of $G_R$ by transport of structure from the bottom right corner.
	Note that from Section \ref{sec:period_rings_padic_reps} we have that $g([X_j^{\flat}] - X_j) = [X_j^{\flat}]-X_j$ for $g$ in $H_R = \kert(G_R \twoheadrightarrow \Gamma_R)$, and $\gamma_j([X_j^{\flat}]-X_j) = ([X_j^{\flat}]-X_j) + \mu [X_j^{\flat}]$ and $\gamma_i([X_j^{\flat}]-X_j) = [X_j^{\flat}]-X_j$ for $i \neq j$.
	So, using the right vertical arrow, we get that $g(z_j) = z_j$ for $g$ in $H_R = \kert(G_R \twoheadrightarrow \Gamma_R)$, and $\gamma_j(z_j) = z_j + \mu [X_j^{\flat}]$ and $\gamma_i(z_j) = z_j$ for $i \neq j$.
	Using this we compute that $g(z_j^{[k]}) = z_j^{[k]}$ for $g$ in $H_R$, and $\gamma_i(z_j^{[k]}) = z_j^{[k]}$ for $i \neq j$ and
	\begin{equation*}
		\gamma_j(z_j^{[k]}) = \textstyle\sum_{r=0}^{k}\binom{k}{r}\tfrac{z_j^{k-r} (\mu[X_i^{\flat}])^r}{k!} = \textstyle\sum_{r=0}^{k} p^{nr} z_j^{[k-r]} \tfrac{(\mu[X_i^{\flat}])^r}{r!p^{rn}},
	\end{equation*}
	which clearly converges in $\SnPD[z_1, \ldots, z_d]_{\PD}^{\wedge}$ for all $k \geqslant 0$; the key observation here is that $\SnPD$ already admits divided powers of $\tfrac{\mu}{p^n}$.
	In particular, from the preceding discussion it follows that for any element $x = \sum_{\smbfk \in \NN^d} x_{\smbfk} \prod_{j=1}^d z_j^{[k_j]}$ in $\SnPD[z_1, \ldots, z_d]_{\PD}^{\wedge}$, we have that $\gamma_i(x)$ converges in $\SnPD[z_1, \ldots, z_d]_{\PD}^{\wedge}$ and $g(x) = x$ for all $g$ in $H_R$.
	Thus, using the diagram in \eqref{eq:OSnPD_in_OBdR+} and the action of $G_R$ on $\OBdR^+(\Rbar)$, we obtain an induced action of $\Gamma_R$ on $\SnPD[z_1, \ldots, z_d]_{\PD}^{\wedge}$ which is continuous for the $\padic$ topology, Frobenius-equivariant and compatible with the action of $\Gamma_R$ on $\SnPD$.
	We equip the bottom left corner of \eqref{eq:OSnPD_in_OBdR+}, i.e.\ $\OSnPD$, with a continuous, Frobenius-equivariant and $\SnPD\textrm{-semilinear}$ action of $\Gamma_R$ by transport of structure from the top left corner.
	By construction, we see that the bottom horizontal arrow of \eqref{eq:OSnPD_in_OBdR+} is $G_R\textrm{-equivariant}$.

	For the second claim, it is clear that we have $R \subset (\OSnPD)^{\Gamma_R}$.
	Moreover, using \eqref{eq:OSnPD_in_OBdR+} we see that $(\OSnPD)^{\Gamma_R} \hookrightarrow \OBdR(\Rbar)^{G_R} = R[1/p]$ (see \cite[Proposition 5.2.12]{brinon-relatif}).
	Furthermore, using the commutativity of the action of $\varphi$ and $\Gamma_R$ on $\OSnPD$, and the compatibility of $G_R\action$ on $\OARpi^{\PD} \subset \OAcrys(\Rbar) \subset \OBdR^+(\Rbar)$, we observe that
	\begin{equation*}
		\varphi^n\big((\OSnPD)^{\Gamma_R}\big) = (\varphi^n(\OSnPD))^{\Gamma_R} \subset (\OARpi^{\PD})^{\Gamma_R} \longhookrightarrow (\OAcrys(\Rinfty))^{\Gamma_R} = \OAcrys(\Rbar)^{G_R} = R,
	\end{equation*}
	where the last two equalities follow from \cite[Corollary 4.34]{morrow-tsuji} and \cite[Proposition 6.2.9]{brinon-relatif}, respectively.
	As $\varphi \colon R/p \rightarrow R/p$ is injective, therefore, from Lemma \ref{lem:injectivity_modulo} we get that $(\OSnPD)^{\Gamma_R} \subset R$.
	Hence, the claim follows.
\end{proof}

	From the preceding discussion, note that we have a $(\varphi, \Gamma_R)\equivariant$ injective homomorphism $\SnPD \hookrightarrow \OSnPD$.
	Moreover, $\OSnPD$ is equipped with a $\Gamma_R\equivariant$ $\SnPD\linear$ integrable connection given as the universal continuous $\SnPD\linear$ de Rham differential $d \colon \OSnPD \rightarrow \Omega^1_{\OSnPD/\SnPD}$.
	By setting $V_j = \frac{X_j \otimes 1}{1 \otimes [X_j^{\flat}]}$, for $1 \leqslant j \leqslant d$, we have the $p\textrm{-adically}$ complete divided power ideals of $\OSnPD$ as follows:
	\begin{equation*}
		J^{[i]}\OSnPD \coloneq \Big\langle \tfrac{\mu^{[k_0]}}{p^{nk_0}} \textstyle\prod_{j=1}^d (1-V_j)^{[k_j]}, \hspace{1mm} \smbfk = (k_0, k_1, \ldots, k_d) \in \NN^{d+1} \hspace{1mm} \textrm{such that} \hspace{1mm} \sum_{j=0}^d k_j \geqslant i\Big\rangle.
	\end{equation*}

	Now, we equip $\OSnPD \otimes_{\AR^+} N$ with the tensor product Frobenius, and the integrable connection on $\OSnPD$ induces an $\SnPD\linear$ integrable connection on $\OSnPD \otimes_{\AR^+} N$.
	Then, $D_n \coloneq \big(\OSnPD \otimes_{\AR^+} N[1/p]\big)^{\Gamma_R}$ is an $R[1/p]\module$ equipped with a semilinear Frobenius homomorphism $\varphi \colon D_n \rightarrow D_{n-1}$ and an integrable connection.
	In particular, we have that
	\begin{equation*}
		\varphi^n(D_n) \longhookrightarrow \ODR = \big(\OARpi^{\PD} \otimes_{\AR^+} N[1/p]\big)^{\Gamma_R} \longhookrightarrow \big(\OAcrys(\Rbar) \otimes_{\AR^+} N[1/p]\big)^{H_R},
	\end{equation*}
	where we have used that $\OARpi^{\PD} \hookrightarrow \OAcrys(\Rinfty) = \OAcrys(\Rbar)^{H_R}$ (see \cite[Corollary 4.34]{morrow-tsuji} for the equality).
	Let $T \coloneq \TR(N)$ denote the finite free $\ZZ_p\representation$ of $G_R$, associated to $N$ from \eqref{eq:wach_reps_relative}, and set $V \coloneq T[1/p]$.
	Then, we have $R[1/p]\linear$ injective homomorphisms
	\begin{equation}\label{eq:dr_in_dcrys}
		\begin{aligned}
			\ODR \hookrightarrow \big(\OBcrys^+(\Rbar) \otimes_{\BR^+} N[1/p]\big)^{G_R} \subset &\big(\OBcrys(\Rbar) \otimes_{\BR^+} N[1/p]\big)^{G_R}\\
			&\xrightarrow[\eqref{eq:wachmod_comp_relative_ainf}]{\hspace{1mm}\sim\hspace{1mm}} \big(\OBcrys(\Rbar) \otimes_{\QQ_p} V\big)^{G_R} = \ODcrysR(V),
		\end{aligned}
	\end{equation}
	where the isomorphism follows by taking $G_R\textrm{-fixed}$ elements of the extension along $\Ainf(\Rbar)[1/\mu] \rightarrow \OBcrys(\Rbar)$ of the isomorphism \eqref{eq:wachmod_comp_relative_ainf} in Proposition \ref{prop:wachmod_comp_relative}.
	Using the injectivity of \eqref{eq:dr_in_dcrys} and the fact that $\ODcrysR(V)$ is a finitely generated module over the noetherian ring $R[1/p]$ (see Section \ref{subsec:relative_padicreps}), it follows that $\ODR$ is also a finitely generated $R[1/p]\module$.
	Moreover, as explained in Remark \ref{rem:oarpd_comparison_struct}, the $R[1/p]\module$ $\ODR$ is equipped with an integrable connection, so from \cite[Proposition 7.1.2]{brinon-relatif} it follows that $\ODR$ is a finite projective $R[1/p]\module$.
	Now, note that we have an injective homomorphism $\varphi^n \colon D_n \rightarrow \ODR$, and we may view it as a homomorphism of modules over the source of the finite and faithfully flat ring homomorphism $\varphi^n \colon R[1/p] \rightarrow R[1/p]$ (see Section \ref{subsec:setup_nota}).
	So it follows that $\ODR$ is a finite module over the source of $\varphi^n$, and since $R[1/p]$ is noetherian, we get that $D_n$ is also a finite module over the source of $\varphi^n$, i.e.\ $D_n$ is a finitely generated $R[1/p]\module$.
	As $D_n$ is further equipped with an integrable connection, therefore, again from \cite[Proposition 7.1.2]{brinon-relatif} we conclude that $D_n$ is a finite projective $R[1/p]\module$.
	Moreover, note that the following $R[1/p]\linear$ homomorphism is injective:
	\begin{equation}\label{eq:phindn_odr}
		1 \otimes \varphi^n \colon R[1/p] \otimes_{\varphi^n, R[1/p]} D_n \longrightarrow \ODR.
	\end{equation}
	To see this, let $y = \sum_{i=1}^r a_i \otimes y_i$ be an element of $R[1/p] \otimes_{\varphi^n, R[1/p]} D_n$ such that $(1 \otimes \varphi^n)y = \sum_{i=1}^r a_i\varphi^n(y_i) = 0$.
	As $\varphi^n \colon R[1/p] \rightarrow R[1/p]$ is finite free of degree $p^{nd}$ with a basis given by elements $\{\prod_{k=1}^{d} X_k^{j_k}\}_{0 \leqslant j_k \leqslant p^n-1}$, therefore, we may write $a_i = \sum_{\smbfj \in J} \varphi^n(b_{i\smbfj}) (\prod_{k=1}^d X_k^{j_k})$, where $J = \{(j_1, \ldots, j_d) \textrm{ with } 0 \leqslant j_k \leqslant p^n-1\}$ and $b_{i\smbfj} \in R[1/p]$.
	Plugging this into the expression for $y$, we get that $(1 \otimes \varphi^n)y = \sum_{\smbfj \in J} (\sum_{i=1}^r \varphi^n(b_{i\smbfj}y_i))\prod_{k=1}^d X_k^{j_k}$.
	But, note that $\varphi^n(D_n)$ is a $\varphi^n(R[1/p])\module$.
	So, using the injectivity of $\varphi^n$ and the linear independence of the basis elements described above, we observe that for each $\smbfj \in J$, we must have $\sum_{i=1}^n \varphi^n(b_{i\smbfj} y_i) = 0$, or equivalently, $\sum_{i=1}^n b_{i\smbfj}y_i = 0$.
	In particular, we obtain that $y = \sum_{i=1}^n a_i \otimes y_i = \sum_{\smbfj \in J} (\prod_{k=1}^d X_k^{j_k} \otimes \sum_{i=1}^n b_{i\smbfj}y_i) = 0$, implying the injectivity of \eqref{eq:phindn_odr}.

	Next, note that $N[1/p]$ is a finite projective $\BR^+\module$ (see Proposition \ref{prop:wachmod_proj_pmu}), so $\OSnPD \otimes_{\AR^+} N[1/p]$ is a finite projective $\OSnPD[1/p]\module$.
	Moreover, as $\OSnPD$ is $p\textrm{-adically}$ complete and $p\torsion$ free, therefore, from Lemma \ref{lem:pcomplete_finprojrational} it follows that $\OSnPD \otimes_{\AR^+} N$ is $p\textrm{-adically}$ complete.
	Now, for $n \geqslant 1$, similar to the proof of \cite[Lemmas 4.32 \& 4.36]{abhinandan-relative-wach-i}, it is easy to show that $\log \gamma_i \coloneq \sum_{k \in \NN} (-1)^k \frac{(\gamma_i-1)^{k+1}}{k+1}$ converges as a series of operators on $\OSnPD \otimes_{\AR^+} N$, where $\{\gamma_0, \gamma_1, \ldots, \gamma_d\}$ are topological generators of $\Gamma_R$ (see Section \ref{sec:period_rings_padic_reps}).

\begin{lem}\label{lem:osmpd_comp}
	Let $m \geqslant 1$ (let $m \geqslant 2$ if $p=2$), then the following map induces a natural $\Gamma_R\equivariant$ isomorphism:
	\begin{equation}\label{eq:osmpd_dm_n}
		\begin{aligned}
			\OSmPD \otimes_R D_m &\isomorphic \OSmPD \otimes_{\AR^+} N[1/p]\\
			a \otimes b \otimes x &\longmapsto ab \otimes x.
		\end{aligned}
	\end{equation}
\end{lem}
\begin{proof}
	The map in \eqref{eq:osmpd_dm_n} is obviously compatible with the respective actions of $\Gamma_R$, so we need to check that it is bijective.
	Let us first check the injectivity of \eqref{eq:osmpd_dm_n} by considering the following commutative diagram with injective arrows:
	\begin{center}
		\begin{tikzcd}[row sep=15pt]
			R[1/p] \arrow[r, "\varphi^m"] \arrow[d] & R[1/p] \arrow[d] \arrow[rd]\\
			\OSmPD[1/p] \arrow[r, "\varphi^m"] & \OARpi^{\PD}[1/p] \arrow[r] & \OBcrys(\Rbar).
		\end{tikzcd}
	\end{center}
	As $D_m$ is a finite projective $R[1/p]\module$, therefore, the tensor product of the injective homomorphism $\OSmPD \rightarrow \OBcrys(\Rbar)$ with $D_m$, over $R[1/p]$ in the top left corner of the diagram, yields the following natural injective map:
	\begin{equation}\label{eq:osmpd_obcrys_dm}
		\OSmPD \otimes_R D_m = \OSmPD[1/p] \otimes_{R[1/p]} D_m \longrightarrow \OBcrys(\Rbar) \otimes_{\varphi^m, R[1/p]} D_m.
	\end{equation}
	Next, we have $V = T[1/p]$ and we consider the following natural composition:
	\begin{equation}\label{eq:dm_obcrys_dcrys}
		\OBcrys(\Rbar) \otimes_{\varphi^m, R[1/p]} D_m \xrightarrow[1 \otimes \varphi^m]{\hspace{1mm} \eqref{eq:phindn_odr} \hspace{1mm}} \OBcrys(\Rbar) \otimes_{R[1/p]} \ODR \xrightarrow{\hspace{1mm} \eqref{eq:dr_in_dcrys} \hspace{1mm}} \OBcrys(\Rbar) \otimes_{R[1/p]} \ODcrysR(V).
	\end{equation}
	As $R[1/p] \rightarrow \OBcrys(\Rbar)$ is faithfully flat (see \cite[Th\'eor\`eme 6.3.8]{brinon-relatif}), therefore, from the injectivity of \eqref{eq:phindn_odr} and \eqref{eq:dr_in_dcrys}, it follows that \eqref{eq:dm_obcrys_dcrys} is injective.

	Now, set $\pazn_m \coloneq \OSmPD \otimes_{\AR^+} N$, and note that since $N[1/p]$ is a finite projective $\BR^+\module$, therefore, similar to \eqref{eq:osmpd_obcrys_dm}, we see that the following natural map is injective:
	\begin{equation*}
		\pazn_m[1/p] = \OSmPD[1/p] \otimes_{\BR^+} N[1/p] \longrightarrow \OBcrys(\Rbar) \otimes_{\varphi^m, \BR^+} N[1/p].
	\end{equation*}
	Moreover, from the definition of Wach modules and Lemma \ref{lem:finite_pqheight_equiv}, we have a natural isomorphism
	\begin{equation*}
		1 \otimes \varphi \colon \BR^+ \otimes_{\varphi, \BR^+} N[1/p, 1/[p]_q] \isomorphic N[1/p, 1/[p]_q],
	\end{equation*}
	and by iterating the Frobenius twist $m\textrm{-times}$, extending scalars along $\BR^+ \rightarrow \OBcrys(\Rbar)$, and noting that $[p]_q$ is unit in $\OBcrys(\Rbar)$, we obtain the following natural isomorphism
	\begin{equation*}
		1 \otimes \varphi \colon \OBcrys(\Rbar) \otimes_{\varphi^m, \BR^+} N[1/p] \isomorphic \OBcrys(\Rbar) \otimes_{\BR^+} N[1/p].
	\end{equation*}	
	So, from the preceding observations, it follows that the following natural composition is injective:
	\begin{equation}\label{eq:osmpd_obcrys_n}
		\pazn_m[1/p] \longrightarrow \OBcrys(\Rbar) \otimes_{\varphi^m, \BR^+} N[1/p] \xrightarrow[\sim]{\hspace{1mm} 1 \otimes \varphi^m \hspace{1mm}} \OBcrys(\Rbar) \otimes_{\BR^+} N[1/p].
	\end{equation}
	Let us now consider the following diagram with injective horizontal arrows:
	\begin{center}
		\begin{tikzcd}[row sep=16pt]
			\OSmPD \otimes_R D_m \arrow[r, "\eqref{eq:osmpd_obcrys_dm}"] \arrow[d, "\eqref{eq:osmpd_dm_n}"'] & \OBcrys(\Rbar) \otimes_{\varphi^m, R[1/p]} D_m \arrow[r, "\eqref{eq:dm_obcrys_dcrys}"] & \OBcrys(\Rbar) \otimes_{R[1/p]} \ODcrysR(V) \arrow[d]\\
			\pazn_m[1/p] \arrow[r, "\eqref{eq:osmpd_obcrys_n}"] & \OBcrys(\Rbar) \otimes_{\BR^+} N[1/p] \arrow[r, "\eqref{eq:wachmod_comp_relative_ainf}"] & \OBcrys(\Rbar) \otimes_{\QQ_p} V,
		\end{tikzcd}
	\end{center}
	and where the right vertical arrow is the natural injective map (see \cite[Proposition 8.2.6]{brinon-relatif}).
	From the definitions, it easily follows that the diagram commutes, therefore, we conclude that the left vertical arrow, i.e.\ \eqref{eq:osmpd_dm_n} is injective.

	Next, let us check the surjectivity of the map in \eqref{eq:osmpd_dm_n}.
	Note that we have the following well-defined operators on $M$ (see \cite[Section 4.4.2]{abhinandan-relative-wach-i}):
	\begin{equation*}
		\partial_i \coloneq \left\{
			\begin{array}{ll}
				-(\log \gamma_0)/t & \textrm{for} \hspace{1mm} i = 0,\\
				(\log \gamma_i)/(tV_i) & \textrm{for} \hspace{1mm} 1 \leqslant i \leqslant d,
			\end{array}
		\right.
	\end{equation*}
	where $V_i = \frac{X_i \otimes 1}{1 \otimes [X_i^{\flat}]}$, for $1 \leqslant i \leqslant d$.

\begin{claim}\label{claim:connection_osmpdn}
	The following operator defines an $R\linear$ and $p\textrm{-adically}$ quasi-nilpotent integrable connection on $\pazn_m$:
	\begin{align*}
		\partial \colon \pazn_m &\longrightarrow \pazn_m \otimes_{\OSmPD} \Omega^1_{\OSmPD/R}\\
		x &\longmapsto \partial_0(x) dt + \textstyle\sum_{i=1}^d \partial_i(x) d [X_i^{\flat}].
	\end{align*}
\end{claim}
\begin{proof}
	The operator $\partial_i$ is well defined on $\pazn_m$ and to get that $\partial$ defines a connection it remains to check that each $\partial_i$ satisfy a Leibniz rule.
	We will only show the case $i \neq 0$ as the case $i = 0$ follows by a similar argument.
	Our proof will follow an argument appearing in \cite[Proof of Theorem 4.2]{morrow-tsuji}.
	Observe that for any $x$ in $\pazn_m$, we have that
	\begin{equation}\label{eq:log_as_lim}
		\lim_{a \rightarrow +\infty} \tfrac{\gamma_i^{p^a}-1}{p^a} (x) = tV_i\partial_i(x).
	\end{equation}
	Indeed, note that $t^k/k!$ converges $p\textrm{-adically}$ to $0$ in $\OSmPD$ as $k$ goes to $+\infty$, therefore, we may write $\gamma_i^b = \exp(btV_i\partial_i)$, for any $b \in \ZZ$.
	Expanding the preceding exponential, yields the following well-defined operator:
	\begin{equation*}
		\tfrac{\gamma_i^b-1}{b} = tV_i\partial_i + b \textstyle\sum_{k \geqslant 2}b^{k-2} \tfrac{t^k}{k!} (tV_i\partial_i)^k \colon \pazn_m \longrightarrow \pazn_m,
	\end{equation*}
	where $(tV_i\partial_i)^k$ denotes the $k\textrm{-fold}$ composition of $tV_i\partial_i$.
	By setting $b = p^a$ and letting $a \rightarrow +\infty$, we obtain the formula in \eqref{eq:log_as_lim}.
	Now, for any $f$ in $\OSmPD$ and $x$ in $\pazn_m$, note that the following equality holds:
	\begin{equation*}
		(\gamma_i^{p^a} - 1)(fx) = (\gamma_i^{p^a} - 1)(f) \cdot x + \gamma_i^{p^a}(f) (\gamma_i^{p^a} - 1)(x).
	\end{equation*}
	Dividing out the preceding equality by $p^atV_i$, letting $a \rightarrow +\infty$ and using \eqref{eq:log_as_lim}, we obtain that
	\begin{equation*}
		\partial_i(fx) = \partial_i(f)x + f\partial_i(x),
	\end{equation*}
	where the first operator on the right is $\partial_i \coloneq \log(\gamma_i)/(tV_i) \colon \OSmPD \rightarrow \OSmPD$, which is well-defined by \cite[Section 4.4.2]{abhinandan-relative-wach-i} and satisfies the identity in \eqref{eq:log_as_lim} by an argument similar to above.
	In particular, we have shown that the operators $\partial_i$ satisfy a Leibniz rule.
	Moreover, note that $\partial$ is a well-defined differential operator by the discussion in \cite[Section 4.4.2]{abhinandan-relative-wach-i}, and from \cite[Lemma 4.38]{abhinandan-relative-wach-i} we have that the operators $\partial_i$ commute with each other, therefore, the connection $\partial$ is integrable.
	Furthermore, using the finite $\pqheight$ property of $N$, similar to \cite[Lemma 4.39]{abhinandan-relative-wach-i}, it is easy to show that $\partial$ is $p\textrm{-adically}$ quasi-nilpotent, thus establishing the claim.
\end{proof}

	Let us get back to checking the surjectivity of \eqref{eq:osmpd_dm_n}.
	Note that similar to the proof of \cite[Lemma 4.39 \& Lemma 4.41]{abhinandan-relative-wach-i}, we have that for any $x$ in $N$, the following sum converges in $\pazn_m^{\partial=0} = \pazn_m^{\Gamma_R}$:
	\begin{equation}\label{eq:horizontal_elems}
		\tau(x) \coloneq \textstyle\sum_{\smbfk \in \NN^{d+1}} \partial_0^{k_0} \circ \partial_1^{k_1} \circ \cdots \circ \partial_d^{k_d} (x) \tfrac{t^{[k_0]}}{p^{mk_0}} (1-V_1)^{[k_1]} \cdots (1-V_d)^{[k_d]}.
	\end{equation}
	Now, to show the surjectivity of \eqref{eq:osmpd_dm_n}, we will work with a set of generators of $N$.
	So, let $\{y_1, \ldots, y_r\}$ denote a set of generators of $N$ over $\AR^+$, and note that these set of elements also generate $\pazn_m$ over $\OSmPD$.
	Then, using \eqref{eq:horizontal_elems} for $y_i$, we may rewrite the sum in terms of the generators of $\pazn_m$ to obtain that $y_i = \tau(y_i) + \sum_{j=1}^r a_{ij} y_j$, for some $a_{ij}$ in the PD-ideal $J^{[1]}\OSmPD$.
	From these observations, note that for any $x$ in $N$, we have
	\begin{equation}\label{eq:xastauyi}
		\begin{aligned}
			x = \tau(x) + \textstyle\sum_{i=1}^r b_{1i} y_i &= \tau(x) + \textstyle\sum_{i=1}^r b_{1i} (\tau(y_i)+\sum_{j=1}^r a_{ij} y_j)\\
				&= \tau(x) + \textstyle\sum_{i=1}^r b_{1i}\tau(y_i) + \textstyle\sum_{i=1}^r b_{2i} y_i\\
				&= \tau(x) + \textstyle\sum_{i=1}^r b_{1i}\tau(y_i) + \cdots + \textstyle\sum_{i=1}^r b_{ki}\tau(y_i) + \textstyle\sum_{i=1}^r b_{(k+1),i}y_i\\
				&= \tau(x) + \textstyle\sum_{i=1}^r (\sum_{k=1}^{\infty}b_{ki}) \tau(y_i),
		\end{aligned}
	\end{equation}
	for some $b_{ki}$ in $\OSmPD$ obtained as follows: in \eqref{eq:xastauyi}, the second term is obtained by rearranging the terms of \eqref{eq:horizontal_elems}, the third term is obtained by plugging $y_i = \tau(y_i) + \sum_{j=1}^r a_{ij} y_j$ into the second term, the fourth term is a rearrangement of the third term, the fifth term is obtained by repeating the previous two steps $k \geqslant 1$ number of times and the last term is obtained by letting $k \rightarrow +\infty$.
	From the preceding description and an easy induction on $k \geqslant 1$ shows that each $b_{ki}$ is an element of $(J^{[1]}\OSmPD)^k$, for $k \geqslant 1$ and $1 \leqslant i \leqslant d$.
	As we have that $(J^{[1]}\OSmPD)^{p^n(d+1)} \subset p^n J^{[p^n(d+1)]}\OSmPD$, for all $n \in \NN$, therefore, it follows that $p\textrm{-adically}$ $b_{ki}$ goes to $0$ as $k \rightarrow +\infty$.
	Consequently, for $1 \leqslant j \leqslant r$, plugging $x = y_j$ in \eqref{eq:xastauyi}, we see that its final term converges in $\sum_{i=1}^r \OSmPD \cdot \tau(y_i) \subset \OSmPD \otimes_R D_m$.
	Since $\{y_1, \ldots y_r\}$ generate $\pazn_m$ over $\OSmPD$, therefore, from the preceding discussion we obtain that any $z = \sum_{i=1}^r c_i y_i$ in $\pazn_m[1/p]$ may be rewritten as $z = \sum_{i=1}^r b_i \tau(y_i)$, for some $b_i$ in $\OSmPD[1/p]$.
	Hence, it follows that \eqref{eq:osmpd_dm_n} is surjective, thus allowing us to conclude.
\end{proof}

	Next, let us show the comparison isomorphism in \eqref{eq:oarpd_comparison}.
	Recall that $\ODR$ is a finite projective $R[1/p]\module$ equipped with an integrable connection (see the discussion after \eqref{eq:dr_in_dcrys}).
	Moreover, $\ODR$ is equipped with a Frobenius-semilinear operator $\varphi$.
	Now, consider the following diagram:
	\begin{equation}\label{eq:oarpipd_commdiag}
		\begin{tikzcd}[row sep=15pt]
			\OARpi^{\PD} \otimes_{\varphi^m, R} D_m \arrow[r, "1 \otimes \varphi^m"] \arrow[d, "\eqref{eq:osmpd_dm_n}"', "\wr"] & \OARpi^{\PD} \otimes_R \ODR \arrow[d, "\eqref{eq:oarpd_comparison}"]\\
			\OARpi^{\PD} \otimes_{\varphi^m, \AR^+} N[1/p] \arrow[r, "\sim"] & \OARpi^{\PD} \otimes_{\AR^+} N[1/p],
		\end{tikzcd}
	\end{equation}
	where the left vertical arrow is the extension along $\varphi^m \colon \OSmPD \rightarrow \OARpi^{\PD}$ of the isomorphism \eqref{eq:osmpd_dm_n} in Lemma \ref{lem:osmpd_comp} and the bottom horizontal isomorphism follows from an argument similar to \cite[Lemma 4.46]{abhinandan-relative-wach-i}.
	By the description of the arrows, it follows that the diagram is $(\varphi, \Gamma_R)\equivariant$ and commutative.
	Taking $\Gamma_R\textrm{-invariants}$ of the diagram \eqref{eq:oarpipd_commdiag}, we obtain an isomorphism of $R[1/p]\modules$ $1 \otimes \varphi^m \colon R \otimes_{\varphi^m, R} D_m \isomorphic \ODR$.
	In particular, it follows that the top horizontal arrow of \eqref{eq:oarpipd_commdiag} is an isomorphism.
	Hence, we conclude that the right vertical arrow of \eqref{eq:oarpipd_commdiag} is bijective as well, in particular, the comparison in \eqref{eq:oarpd_comparison} is an isomorphism compatible with the respective Frobenii, connections and actions of $\Gamma_R$.
	This finishes our proof of the proposition.
\end{proof}

\begin{proof}[Proof of Theorem \ref{thm:fh_crys_relative}]
	For $r \in \NN$ large enough, note that the Wach module $\mu^r N (-r)$ is always effective and $\TR(\mu^rN(-r)) = \TR(N)(-r)$ (the twist $(-r)$ denotes a Tate twist on which $\Gamma_R$ acts via $\chi^{-r}$, where $\chi$ is the $\padic$ cyclotomic character).
	Therefore, we see that it is enough to show the claim for effective Wach modules.
	So, in the rest of the proof, we will assume that $N$ is effective.
	Let us set $\ODR \coloneq (\OARpi^{\PD} \otimes_{\AR^+} N[1/p])^{\Gamma_R} \hookrightarrow \ODcrysR(V)$, and using Proposition \ref{prop:oarpd_comparison}, we note that $\ODR$ is a finite projective $R[1/p]\module$ of rank $= \rank_{\BR^+} N[1/p]$ equipped with the tensor product Frobenius and an integrable connection induced from the integrable connection on $\OARpi^{\PD}$.
	Now, consider the following diagram:
	\begin{equation}\label{eq:fh_crys_relative_diag}
		\begin{tikzcd}[row sep=15pt]
			\OBcrys(\Rbar) \otimes_{R[1/p]} \ODR \arrow[r, "\eqref{eq:oarpd_comparison}", "\sim"'] \arrow[d, "\eqref{eq:dr_in_dcrys}"'] & \OBcrys(\Rbar) \otimes_{\BR^+} N[1/p] \arrow[d, "\eqref{eq:wachmod_comp_relative_ainf}"', "\wr"]\\
			\OBcrys(\Rbar) \otimes_{R[1/p]} \ODcrysR(V) \arrow[r] & \OBcrys(\Rbar) \otimes_{\QQ_p} V,
		\end{tikzcd}
	\end{equation}
	where the top horizontal arrow is the extension along $\OARpi^{\PD}[1/p] \rightarrow \OBcrys(\Rbar)$ of the isomorphism \eqref{eq:oarpd_comparison} in Proposition \ref{prop:oarpd_comparison}, the right vertical arrow is the extension along $A^+[1/\mu] \rightarrow \OBcrys(\Rbar)$ of the natural isomorphism in Proposition \ref{prop:wachmod_comp_relative} and the bottom horizontal arrow is obtained as the extension along $R[1/p] \rightarrow \OBcrys(\Rbar)$ of the $G_R\textrm{-invariants}$ of the bottom right corner which is naturally injective (see \cite[Proposition 8.2.6]{brinon-relatif}).
	Moreover, in \eqref{eq:fh_crys_relative_diag}, the left vertical arrow is obtained by first taking the $G_R\textrm{-invariants}$ of the composition of the right vertical arrow with the top horizontal arrow (see \eqref{eq:dr_in_dcrys}), and then extending it along the faithfully flat ring homomorphism $R[1/p] \rightarrow \OBcrys(\Rbar)$ (see \cite[Th\`eor\'eme 6.3.8]{brinon-relatif}).
	By the preceding description of the arrows, it follows that the diagram \eqref{eq:fh_crys_relative_diag} is commutative and compatible with the respective actions of $\varphi$ and $G_R$.
	Now, as the top horizontal and the right vertical arrows in \eqref{eq:fh_crys_relative_diag} are bijective, therefore, we conclude that its left vertical arrow and the bottom horizontal arrow are also bijective.
	Hence, we obtain that $V$ is a $\padic$ crystalline representation of $G_R$.
	Finally, recall that $\OBcrys(\Rbar)^{G_R} = R[1/p]$ (see \cite[Proposition 6.2.9]{brinon-relatif}) and $\ODR$ and $\ODcrys(V)$ are fixed under the action of $G_R$.
	So, by taking the $G_R\textrm{-invariants}$ of the left vertical arrow in \eqref{eq:fh_crys_relative_diag}, we obtain a natural $R[1/p]\linear$ isomorphism
	\begin{equation}\label{eq:dr_dcrys_comp}
		\ODR \isomorphic \ODcrysR(V),
	\end{equation}
	compatible with the respective Frobenii and connections.
	This concludes our proof.
\end{proof}

\begin{rem}\label{rem:relate_con_qcon}
	Let us make an observation that will be useful for the proof of Theorem \ref{thm:qdeformation_dcrys}.
	In the basis $\{\dlog(X_1), \ldots, \dlog(X_d)\}$ of $\Omega^1_R$, let $\partial_{A, i}$ denote the $i^{\textrm{th}}$ component of the $\ARpi^{\PD}\linear$ connection on $\OARpi^{\PD}$, for $1 \leqslant i \leqslant d$, and let $\partial_{D, i}$ denote the induced operator on $\ODR$.
	Note that by employing arguments similar to \cite[Lemmas 4.12, 5.17 \& 5.18]{abhinandan-syntomic}, one may show that, for $1 \leqslant i \leqslant d$, the operator $\nabla_i = (\log \gamma_i)/t = \frac{1}{t}\sum_{k \in \NN} (-1)^k \frac{(\gamma_i-1)^{k+1}}{k+1}$ converges as a series of operators on $\OARpi^{\PD} \otimes_{\AR^+} N$, which is $p\textrm{-adically}$ complete by Lemma \ref{lem:pcomplete_finprojrational}.
	Now, using \eqref{eq:horizontal_elems} and the top horizontal arrow in diagram \eqref{eq:oarpipd_commdiag}, we note that for any $x$ in $N[1/p]$, there exists $y$ in $\ODR$ and $z$ in $(\Fil^1\OARpi^{\PD}) \otimes_{\AR^+} N[1/p]$, such that $x = f(y) + z$, where $f$ is the isomorphism in \eqref{eq:oarpd_comparison}.
	Then, an easy computation shows that $\nabla_i(x) - f(\partial_{D, i}(y)) = \nabla_i(z) + \partial_{A, i}(z)$ belongs to $(\Fil^1 \OARpi^{\PD}) \otimes_{\AR^+} N[1/p]$.
\end{rem}

\section{Crystalline implies finite height}\label{sec:crystalline_finite_height}

The goal of this section is to prove the following claim and provide some applications.
\begin{thm}\label{thm:crys_fh_relative}
	Let $T$ be a finite free $\ZZ_p\representation$ of $G_R$ such that $V \coloneq T[1/p]$ is a $\padic$ crystalline representation of $G_R$.
	Then, there exists a unique Wach module $\NR(T)$ over $\AR^+$ attached to $T$.
	In other words, $T$ is of finite $\pqheight$.
\end{thm}

The proof of Theorem \ref{thm:crys_fh_relative} is entirely contained in Section \ref{subsec:crys_wach_equiv_proof}.
Readers interested only in the applications may directly skip to Section \ref{subsec:crys_wach_equiv_consequence}.

\subsection{Proof of Theorem \ref{thm:crys_fh_relative}}\label{subsec:crys_wach_equiv_proof}

To prove Theorem \ref{thm:crys_fh_relative}, we need some preparations.
So, we begin with the following assumptions:
let $T$ be a finite free $\ZZ_p\representation$ of $G_R$ such that $T$ is positive and of finite $\pqheight$ as a $\ZZ_p\representation$ of $G_L$ (see Definition \ref{defi:finite_pqheight} and \cite[Definition 3.7]{abhinandan-imperfect-wach}).
In particular, naturally associated to $T$, we have an \'etale $(\varphi, \Gamma_R)\module$ $\DR(T)$ over $\AR$ and a Wach module $\NL(T)$ over $\AL^+$.
Let us note that $\DR(T)$ is a finite projective module over $\AR$ (see \cite[Lemma 7.10]{andreatta-phigamma}) and $\NL(T)$ is a finite free module over $\AL^+$ (see \cite[Definition 3.1]{abhinandan-imperfect-wach}).
A key ingredient for the proof of Theorem \ref{thm:crys_fh_relative} is the following claim:
\begin{prop}\label{prop:wach_module_constr}
	The $\AR^+\module$ $\NR(T) \coloneq \NL(T) \cap \DR(T) \subset \DL(T)$ satisfies all the axioms of Definition \ref{defi:finite_pqheight}.
	In particular, $T$ is a finite $\pqheight$ representation of $G_R$.
\end{prop}

The proof of Proposition \ref{prop:wach_module_constr} requires some preparations which we carry out next.

\begin{lem}\label{lem:modp_fingen}
	The $\AR^+\module$ $\NR(T) \coloneq \NL(T) \cap \DR(T)$ is finitely generated.
\end{lem}
\begin{proof}
	We first claim that for each $n \in \NN_{\geqslant 1}$, the following natural $\AR^+/p^n\linear$ map is injective:
	\begin{equation*}
		\NR(T)/p^n \longrightarrow (\NL(T)/p^n) \cap (\DR(T)/p^n) \subset \DL(T)/p^n,
	\end{equation*}
	and the $\AR^+/p^n\module$ $(\NL(T)/p^n) \cap (\DR(T)/p^n)$ is finitely generated.
	As $\NR(T)$, $\NL(T)$ and $\DR(T)$ are $p\textrm{-torsion free}$, therefore, it is enough to show the claim for $n=1$ and the claim for any $n > 1$ may be deduced by an easy induction.
	So, we are reduced to showing that the following natural map is injective:
	\begin{equation*}
		\NR(T)/p \longrightarrow (\NL(T)/p) \cap (\DR(T)/p) \subset \DL(T)/p,
	\end{equation*}
	and the module $(\NL(T)/p) \cap (\DR(T)/p)$ is finitely generated over $\ER^+ \coloneq \AR^+/p$.
	We have that $p \NL(T) \cap \NR(T) \subset p \DL(T) \cap \DR(T) = p \DR(T)$, and so we get that
	\begin{equation*}
		p\NL(T) \cap \NR(T) \subset p \NL(T) \cap p \DR(T) = p\NR(T),
	\end{equation*}
	in particular, $\NR(T)/p \hookrightarrow \NL(T)/p$ by Lemma \ref{lem:injectivity_modulo}.
	Similarly, we have that $p \DR(T) \cap \NR(T) \subset p \DL(T) \cap \NL(T) = p \NL(T)$, so we get that 
	\begin{equation*}
		p\DR(T) \cap \NR(T) \subset p \DR(T) \cap p \NL(T) = p\NR(T),
	\end{equation*}
	in particular, $\NR(T)/p \hookrightarrow \DR(T)/p$ by Lemma \ref{lem:injectivity_modulo}.

	Next, we will show that $(\NL(T)/p) \cap (\DR(T)/p)$ is a finitely generated $\ER^+\module$, where $\ER^+ = \AR^+/p$.
	Set $\ER \coloneq \AR/p = \ER^+[1/\mu]$, where we abuse notations by writing $\mu$ for its image modulo $p$, and set $\EL^+ \coloneq \AL^+/p$ and $\EL \coloneq \AL/p = \EL^+[1/\mu]$.
	Also, set $\DRbar \coloneq \DR(T)/p$, $\NLbar \coloneq \NL(T)/p$ and $\DLbar \coloneq \DL(T)/p$.
	Note that $\DRbar$ is a finite projective $\ER\module$, so we may choose an $\ER\module$ $D'$ such that $\DRbar \oplus D' = \ER^{\oplus r}$, for some $r \in \NN$.
	Moreover, we have that $\DLbar \isomorphic \EL \otimes_{\ER} \DRbar$ and by setting $D_L' \coloneq \EL \otimes_{\ER} D'$, we get that $\DLbar \oplus D_L' = \EL^{\oplus r}$.
	As $\EL$ is a discrete valuation field with ring of integers $\EL^+$, therefore, we may choose a lattice of $D_L'$ over $\EL^+$, i.e.\ there exists a free $\EL^+\textrm{-submodule}$ $N_L' \subset D_L'$ such that $N_L'[1/\mu] = D_L'$.
	So, we see that $\NLbar \oplus N_L'$ is a free $\EL^+\module$ such that $\EL \otimes_{\EL^+} (\NLbar \oplus N_L') = \DLbar \oplus D_L' = \EL^{\oplus r}$.
	Now, inside $\EL^{\oplus r}$, consider the following inclusion of $\ER^+\modules$:
	\begin{equation}\label{eq:modp_fingen}
		(\DRbar \cap \NLbar) \oplus (D' \cap N_L') \subset (\DRbar \oplus D') \cap (\NLbar \oplus N_L') = \ER^{\oplus r} \cap (\NLbar \oplus N_L').
	\end{equation}
	Then, from Lemma \ref{lem:modp_freecap_fingen} below, we see that the last term in \eqref{eq:modp_fingen} is a finite $\ER^+\module$.
	Therefore, it follows that $\DRbar \cap \NLbar$ is also a finite $\ER^+\module$.

	To complete our proof, it remains to show that $\NR(T)$ is $p\textrm{-adically}$ complete.
	From the claim above note that, for all $n \in \NN_{\geqslant 1}$, $\NR(T)/p^n$ is a finitely generated $\AR^+/p^n\module$.
	Since $\AR^+$ is noetherian, therefore, for each $n \in \NN$ and $r \in \NN$ as in the previous paragraph, we have a presentation $0 \rightarrow M_n \rightarrow (\AR^+/p^n)^{\oplus r} \rightarrow \NR(T)/p^n \rightarrow 0$, where $M_n$ is a finitely generated $\AR^+/p^n\module$.
	By taking a finite presentation of $M_n$ as an $\AR^+/p^n\module$, it is easy to see that the system $\{M_n\}_{n \in \NN_{\geqslant 1}}$ is Mittag-Leffler.
	In particular, it follows that $\lim_n \NR(T)/p^n$ is a finitely generated module over $\lim_n \AR^+/p^n = \AR^+$.
	Now, from the discussion above, we have the following natural $\AR^+/p^n\linear$ commutative diagram with injective arrows:
	\begin{center}
		\begin{tikzcd}[row sep=15pt]
			\NR(T)/p^n & \NL(T)/p^n \cap \DR(T)/p^n & \DR(T)/p^n\\
			& \NL(T)/p^n & \DL(T)/p^n.
			\arrow[from=1-1, to=1-2]
			\arrow[from=1-2, to=1-3]
			\arrow[from=1-2, to=2-2]
			\arrow[from=1-3, to=2-3]
			\arrow[from=2-2, to=2-3]
		\end{tikzcd}
	\end{center}
	As the collection of the diagram above, for $n \in \NN$, form an inverse system, therefore, by taking the limit over $n$ we obtain the following $\AR^+\linear$ natural commutative diagram with injective arrows:
	\begin{equation}\label{eq:intersection_limit}
		\begin{tikzcd}[row sep=15pt]
			\lim_n \NR(T)/p^n & \lim_n (\NL(T)/p^n \cap \DR(T)/p^n) & \DR(T)\\
			& \NL(T) & \DL(T),
			\arrow[from=1-1, to=1-2]
			\arrow[from=1-2, to=1-3]
			\arrow[from=1-2, to=2-2]
			\arrow[from=1-3, to=2-3]
			\arrow[from=2-2, to=2-3]
		\end{tikzcd}
	\end{equation}
	where we have used that $\DR(T)$, $\NL(T)$ and $\DL(T)$ are $p\textrm{-adically}$ complete.
	Now, consider the composition of the following natural $\AR^+\linear$ maps:
	\begin{equation*}
		\begin{aligned}
			f \colon \NR(T) \longrightarrow \lim_n \NR(T)/p^n &\longrightarrow \lim_n ((\NL(T)/p^n) \cap (\DR(T)/p^n))\\
					&\longrightarrow \NL(T) \cap \DR(T) = \NR(T),
		\end{aligned}
	\end{equation*}
	where the first map is the natural projection map, the second map is the top left injective horizontal arrow in \eqref{eq:intersection_limit} and the third map is induced by the injectivity of the top right horizontal arrow and the left vertical arrow in \eqref{eq:intersection_limit} (note that $\NL(T) \cap \DR(T) \subset \DL(T)$ is well defined), in particular, the third map is injective as well.
	Chasing an element of $x$ in $\NR(T)$ through the composition above, we see that $f(x) = x$.
	Hence, we get that $\NR(T) \isomorphic \lim_n \NR(T)/p^n$, in particular, it is a finitely generated $\AR^+\module$.
\end{proof}

The following observation was used above:
\begin{lem}\label{lem:modp_freecap_fingen}
	Let $D_R$ be a finite free module over $\ER \coloneq \AR/p$ and $N_L$ a finite free module over $\EL^+ \coloneq \AL^+/p$ such that $\EL \otimes_{\ER} D_R \isomorphic \EL \otimes_{\EL^+} N_L$, where $\EL \coloneq \AL/p$.
	Then, the module $D_R \cap N_L \subset \EL \otimes_{\ER} D_R$ is finitely generated over $\ER^+ \coloneq \AR^+/p$.
\end{lem}
\begin{proof}
	Let $r \in \NN$ be the $\ER\textrm{-rank}$ of $D_R$.
	Take $\smbff \coloneq \{f_1, \ldots, f_r\}$ to be a basis of $D_R$ over $\ER$ and take $\smbfe \coloneq \{e_1, \ldots, e_r\}$ to be a basis of $N_L$ over $\EL^+$.
	Then, we have that $\smbfe = A \smbff$, for some $A \coloneq (a_{i j}) \in \GL(r, \EL)$.
	In particular, there exists some integer $k \gg 0$ such that $a_{i j}$ is in $\mu^{-k} \EL^+$, for all $1 \leqslant i, j \leqslant r$; we fix such a $k$.
	Set $M \coloneq \oplus_{i=1}^r \ER^+ f_i$, so that $M[1/\mu] = D_R$.
	Let $x$ be any element of the $\ER^+\module$ $M[1/\mu] \cap N_L = D_R \cap N_L \subset D_L \coloneq \EL \otimes_{\ER} D_R$, and write 
	\begin{equation*}
		x = \textstyle\sum_{i=1}^r c_i e_i = \textstyle\sum_{i=1}^r d_i f_i,
	\end{equation*}
	with $c_i$ in $\EL^+$ and $d_i$ in $\ER$, for all $1 \leqslant i \leqslant r$.
	Using that $\smbfe = A\smbff$, we obtain that $d_i = \sum_{j=1}^r a_{i j} c_j$, for all $1 \leqslant i \leqslant r$.
	As $c_j$ is in $\EL^+$, and $a_{i j}$ is in $\mu^{-k} \EL^+$, therefore, it follows that $d_i$ is in $\mu^{-k} \EL^+$, for all $1 \leqslant i \leqslant r$.
	In particular, we get that $d_i$ is in $\mu^{-k} \EL^+ \cap \ER \subset \EL$.
	Now, consider the following composition:
	\begin{equation*}
		\ER^+/(\mu^k) \isomorphic (R/p)\llbracket \mu \rrbracket/(\mu^k) \hookrightarrow (O_L/p)\llbracket \mu \rrbracket/(\mu^k) \lisomorphic \EL^+/(\mu^k),
	\end{equation*}
	where for the first (resp.\ last) isomorphism we used that $\AR^+ \isomorphic R\llbracket \mu \rrbracket$ (resp.\ $\AL^+ \isomorphic O_L\llbracket \mu \rrbracket$) and the map in the middle is injective because $O_L/p = \Fr(R/p)$.
	Therefore, from Lemma \ref{lem:injectivity_modulo}, we get that $\mu^{-k} \EL^+ \cap \ER = \mu^{-k} \ER^+$.
	So, it follows that $d_i$ is in $\mu^{-k}\ER^+$, and thus $x = \sum_{i=1}^r d_i f_i$ is in $\mu^{-k} M$.
	As $x$ is arbitrary and $k \in \NN$ is independent of $x$, therefore, we conclude that $M[1/\mu] \cap N_L \subset \mu^{-k} M$ as modules over the noetherian ring $\ER^+$.
	Hence, it follows that $M[1/\mu] \cap N_L = D_R \cap N_L$ is finitely generated over $\ER^+$.
\end{proof}

\begin{lem}\label{lem:finite_height}
	The $\AR^+\module$ $\NR(T)$ is of finite $\pqheight$, i.e.\ the cokernel of the injective map $1 \otimes \varphi \colon \varphi^*(\NR(T)) \rightarrow \NR(T)$ is killed by $[p]_q^s$, for some $s \in \NN$.
\end{lem}
\begin{proof}
	Note that $\varphi \colon \AR^+ \rightarrow \AR^+$ is faithfully flat and finite of degree $p^{d+1}$ (see Section \ref{subsec:ainf_relative}).
	Moreover, from Section \ref{subsec:ainf_relative} we have that $\varphi^*(\AR) \isomorphic \AR^+ \otimes_{\varphi, \AR^+} \AR$ and 
	\begin{equation*}
		\varphi^*(\AL^+) \coloneq \AL^+ \otimes_{\varphi, \AL^+} \AL^+ \isomorphic  \oplus_{\alpha} \varphi(\AL^+) u_{\alpha} = (\oplus_{\alpha} \varphi(\AR^+) u_{\alpha}) \otimes_{\varphi(\AR^+)} \varphi(\AL^+) \lisomorphic \AR^+ \otimes_{\varphi, \AR^+} \AL^+.
	\end{equation*}
	Therefore, we also obtain that
	\begin{align*}
		\varphi^*(\NL(T)) \coloneq \AL^+ \otimes_{\varphi, \AL^+} \NL(T) &\isomorphic \AR^+ \otimes_{\varphi, \AR^+} \NL(T),\\
		\varphi^*(\DR(T)) \coloneq \AR \otimes_{\varphi, \AR} \DR(T) &\isomorphic \AR^+ \otimes_{\varphi, \AR^+} \DR(T).
	\end{align*}
	So, it follows that, as $\AR^+\textrm{-submodules}$ of $\varphi^*(\DL(T))$, we have that 
	\begin{align*}
		\varphi^*(\NR(T)) &\coloneq \AR^+ \otimes_{\varphi, \AR^+} \NR(T) = \AR^+ \otimes_{\varphi, \AR^+} (\NL(T) \cap \DR(T))\\
			&\hspace{1mm}= (\AR^+ \otimes_{\varphi, \AR^+} \NL(T)) \cap (\AR^+ \otimes_{\varphi, \AR^+} \DR(T))\\
			&\isomorphic \varphi^*(\NL(T)) \cap \varphi^*(\DR(T)) \subset \varphi^*(\DL(T)).
	\end{align*}
	Now, let $x$ be an element of $\NR(T) = \NL(T) \cap \DR(T)$.
	As the cokernel of the injective map $(1 \otimes \varphi) \colon \varphi^*(\NL(T)) \rightarrow \NL(T)$ is killed by $[p]_q^s$, for some $s \in \NN$, therefore, there exists some $y$ in $\varphi^*(\NL(T))$ such that $(1 \otimes \varphi) y = [p]_q^s x$.
	Moreover, we have a natural isomorphism $(1 \otimes \varphi) \colon \varphi^*(\DR(T)) \isomorphic \DR(T)$, so it follows that there exists some $z$ in $\varphi^*(\DR(T))$ such that $(1 \otimes \varphi) z = [p]_q^s x$.
	As the Frobenius structure on $\NL(T)$ and $\DR(T)$ are compatible (since both have a Frobenius-equivariant embedding into $\DL(T)$), so it follows that $(1 \otimes \varphi)(y-z) = 0$ in $\DL(T)$.
	But, we also have that $(1 \otimes \varphi) \colon \varphi^*(\DL(T)) \isomorphic \DL(T)$, therefore, we conclude that $y = z$ is in $\varphi^*(\NL(T)) \cap \varphi^*(\DR(T)) \isomorphic \varphi^*(\NR(T))$.
	Hence, the cokernel of the injective map $(1 \otimes \varphi) \colon \varphi^*(\NR(T)) \rightarrow \NR(T)$ is killed by $[p]_q^s$.
\end{proof}

Finally, we will show that $\AL^+ \otimes_{\AR^+} \NR(T) \isomorphic \NL(T)$ and $\AR \otimes_{\AR^+} \NR(T) \isomorphic \DR(T)$, using an approach parallel to \cite[Proposition 4.24 \& Lemma 4.25]{du-liu-moon-shimizu}.
For $n \in \NN_{\geqslant 1}$, let $\NRn \coloneq \NR(T)/p^n$, $\DRn \coloneq \DR(T)/p^n$, $\NLn \coloneq \NL(T)/p^n$, $\DLn \coloneq \DL(T)/p^n$ and $M_n \coloneq \NLn \cap \DRn \subset \DLn$.
Then, have the following commutative diagram:
\begin{center}
	\begin{tikzcd}[row sep=15pt]
		M_n \arrow[r, "f_n"] \arrow[d] & M_1 \arrow[d]\\
		\DRn \arrow[r, twoheadrightarrow, "f_n"] & D_{R,1},
	\end{tikzcd}
\end{center}
where the vertical arrows are natural inclusions, the bottom horizontal arrow $f_n$ is the natural projection map and the top arrow is the induced map.
We have a similar diagram with the bottom row replaced by $\NLn \twoheadrightarrow N_{L, 1}$.

\begin{lem}\label{lem:Mn_large}
	The following statements are true:
	\begin{enumarabicup}
		\item $M_n$ is a finitely generated $\AR^+/p^n\module$ and $\NR(T) \isomorphic \lim_n M_n$.
	
		\item $M_n$ is of finite $\pqheight$ $s$, for $s \in \NN$ from Lemma \ref{lem:finite_height}.
	
		\item $M_n[1/\mu] = \AR \otimes_{\AR^+} M_n \isomorphic \DRn$ and $\AL^+ \otimes_{\AR^+} M_n \isomorphic \NLn$.
	\end{enumarabicup}
\end{lem}
\begin{proof}
	The claim in (1) follows from the proof of Lemma \ref{lem:modp_fingen} and the claim in (2) follows by an argument similar to Lemma \ref{lem:finite_height}.
	To obtain the first isomorphism in (3), note that we have natural $\AR\linear$ inclusions $\AR \otimes_{\AR^+} M_n = M_n[1/\mu] \hookrightarrow \DRn \hookrightarrow \DLn$, because $M_n$ is $\mu\textrm{-torsion}$ free.
	Moreover, we have a natural $\AL\linear$ isomorphism $\NLn[1/\mu] \isomorphic \DLn$.
	Then, inside $\DLn$ we have that
	\begin{equation}\label{eq:mn_muinverse}
		M_n[1/\mu] \isomorphic \DRn \cap \NLn[1/\mu] \isomorphic \DRn \cap \DLn = \DRn,
	\end{equation}
	where the first isomorphism follows because we have a natural $\AR\linear$ inclusion $M_n[1/\mu] \hookrightarrow \DRn \cap \NLn[1/\mu]$, and to obtain its surjectivity let $x$ be an element of $\DRn \cap \NLn[1/\mu]$, then there exists some $k \in \NN$ such that $\mu^kx$ is in $\NLn$, i.e.\ $\mu^k x$ is in $\DRn \cap \NLn = M_n$, as claimed.
	This gives us the first claim in (3).

	To obtain the second isomorphism in (3), note that from \eqref{eq:mn_muinverse} and Lemma \ref{lem:injectivity_modulo}, we have a natural $\AR^+/\mu\linear$ injective homomorphism $M_n/\mu \hookrightarrow \NLn/\mu$.
	Moreover, we have that $\AR^+/\mu \isomorphic R$ and $\AL^+/\mu \isomorphic O_L$ (see Section \ref{subsec:phigamma_mod_rings}), and $R/p^n \hookrightarrow (R_{(p)})/p^n \isomorphic O_L/p^n$.
	But, recall that localisation commutes with passing to quotient by ideals (see \cite[Theorem 4.2]{matsumura}), therefore, we have that $(R/p^n)_{(p)} \isomorphic (R_{(p)})/p^n \isomorphic O_L/p^n$.
	As $\NLn/\mu$ is a free $O_L/p^n\textrm{-module}$ and $M_n/\mu \hookrightarrow \NLn/\mu$, so we see that for any $a$ in $(R/p^n) \setminus (p)$, the $R/p^n\module$ $M_n/\mu$ is $a\textrm{-torsion}$ free.
	Therefore, by localising the preceding inclusion we obtain that $(R/p^n)_{(p)} \otimes_{R/p^n} M_n/\mu \hookrightarrow \NLn/\mu$, and hence $(\AL^+ \otimes_{\AR^+} M_n)/\mu \isomorphic O_L/p^n \otimes_{R/p^n} M_n/\mu \hookrightarrow \NLn/\mu$.
	In particular, we see that inside $\DLn$ we have
	\begin{equation*}
		\AL^+ \otimes_{\AR^+} M_n = \NLn \cap \big(\AL^+ \otimes_{\AR^+} M_n[1/\mu]\big) \isomorphic \NLn \cap \big(\AL^+ \otimes_{\AR^+} \DRn\big) \isomorphic \NLn \cap \DLn = \NLn,
	\end{equation*}
	where we used Lemma \ref{lem:injectivity_modulo} for the equality and \eqref{eq:mn_muinverse} for the middle isomorphism.
	This completes our proof.
\end{proof}

Let $\pazs$ denote the set of $\AR^+\textrm{-submodules}$ $M' \subset M_1$ such that $M'$ is stable under the action of $\varphi$, it is of finite $[p]_q\textrm{-height}$ $s \in \NN$ and $M'[1/\mu] = M_1[1/\mu] = D_{R,1} = \DR(T)/p$.
Set $\Mcirc \coloneq \cap_{M' \in \pazs} M' \subset M_1$.

\begin{lem}\label{lem:mcirc_props}
	The $\AR^+\module$ $\Mcirc$ belongs to $\pazs$ and $f_n(M_n)$ is also in $\pazs$, for all $n \in \NN_{\geqslant 1}$.
\end{lem}
\begin{proof}
	The idea of the proof is motivated from \cite[Lemma 4.25]{du-liu-moon-shimizu}.
	Let $M'$ be an element of $\pazs$.
	First, let us show that there exists some $r \in \NN$ such that $\mu^r M_1 \subset M' \subset M_1$.
	Let $M'' \coloneq M_1/M'$ such that $M'' \neq 0$ and let $k = p(p-1)s \in \NN$.
	Also, let $\varphi^*(M'') \coloneq \varphi^*(M_1)/\varphi^*(M')$ and let $1 \otimes \varphi_{M''} \colon \varphi^*(M'') \rightarrow M''$ denote the map induced from $1 \otimes \varphi_M$.
	Since $M_1$ (resp.\ $M'$) is of finite $\pqheight$ $k$ (since $s < k$), we define $\psi_M \colon M_1 \xrightarrow{\mu^k} \mu^k M_1 \rightarrow \varphi^*(M_1)$ (resp.\ $\psi_{M'} \colon M' \xrightarrow{\mu^k} \mu^k M' \rightarrow \varphi^*(M')$) to be the unique $\AR^+/p\textrm{-linear}$ map such that $\psi_M \circ (1 \otimes \varphi_M) = \mu^k \Id_{\varphi^*_M}$ (resp.\ $\psi_{M'} \circ (1 \otimes \varphi_{M'})  = \mu^k \Id_{\varphi^*_{M'}}$).
	Let $\psi_{M''} \colon M'' \rightarrow \varphi^*(M'')$ denote the map induced from $\psi_M$.
	Now, consider the following commutative diagram:
	\begin{center}
		\begin{tikzcd}
			0 \arrow[r] & \varphi^*(M') \arrow[r] \arrow[d, "1 \otimes \varphi_{M'}"] & \varphi^*(M_1) \arrow[r] \arrow[d, "1 \otimes \varphi_{M}"] & \varphi^*(M'') \arrow[r] \arrow[d, "1 \otimes \varphi_{M''}"] & 0\\
			0 \arrow[r] & M' \arrow[r] \arrow[d, "\psi_{M'}"] & M_1 \arrow[r] \arrow[d, "\psi_{M}"] & M'' \arrow[r] \arrow[d, "\psi_{M''}"] & 0\\
			0 \arrow[r] & \varphi^*(M') \arrow[r] & \varphi^*(M_1) \arrow[r] & \varphi^*(M'') \arrow[r] & 0.
		\end{tikzcd}
	\end{center}
	Note that $[p]_q = \mu^{p-1} \mod p$, $\varphi(\mu) = \mu^p \mod p$ and $\varphi([p]_q) = \mu^{p(p-1)} \mod p$.
	Since $M_1[1/\mu] = M'[1/\mu]$, let $i \in \NN_{\geqslant 1}$ such that $\mu^{pi} M'' = 0$ and $\mu^{p(i-1)} M'' \neq 0$.
	Let $x$ be in $M''$ such that $\mu^{p(i-1)} x \neq 0$ and set $y = 1 \otimes x$ to be in $\varphi^*(M'')$.
	Then, $\varphi(\mu^{pi}) y = 1 \otimes \mu^{pi} x = 0$, but 
	\begin{equation*}
		\mu^{p^2(i-1)}y = \varphi(\mu^{p(i-1)}) y = 1 \otimes \mu^{p(i-1)} x \neq 0.
	\end{equation*}
	Let $z = (1 \otimes \varphi_{M''}) y$ be in $M''$, then $\mu^{pi} z = 0$.
	So, we have that
	\begin{equation*}
		0 = \psi_{M''}(\mu^{pi} z) = \mu^{pi} (\psi_{M''} \circ (1 \otimes \varphi_{M''}) y) = \mu^{pi+k} y.
	\end{equation*}
	Therefore, we get that $pi+k = pi+p(p-1)s > p^2 (i-1)$, i.e.\ $i < s + \frac{p}{p-1}$.
	Hence, it follows that $\mu^{p(s+1)} M'' = 0$.
	Since the constant $i$ obtained above is independent of $M'$, we also get that $\mu^{p(s+1)} M_1 \subset \Mcirc \subset M_1$ and $\Mcirc[1/\mu] = M_1[1/\mu]$.

	Next, we will show that $\Mcirc$ is of finite height $s$.
	Let $x$ be in $\Mcirc$, so that $x$ is in $M'$, for each $M'$ in $\pazs$, and there exists some $y$ in $\varphi^*(M') \subset \varphi^*(M_1)$ such that $(1 \otimes \varphi)y = [p]_q^s x$.
	Note that $y$ is unique in $\varphi^*(M_1)$ and since $\varphi: \AR^+ \rightarrow \AR^+$ is flat, we get that $y$ is in $\cap_{M' \in \pazs} (\AR^+ \otimes_{\varphi, \AR^+} M') = \AR^+ \otimes_{\varphi, \AR^+} (\cap_{M' \in \pazs} M') = \varphi^*(\Mcirc)$.
	Therefore, we concldue that $\Mcirc$ is an element of $\pazs$.

	For the second part of the claim, note that $M_n[1/\mu] = \DRn$ and $f_n(\DRn) = \DRn/p = \DR(T)/p$ (see Lemma \ref{lem:Mn_large}).
	So, we get that $f_n\big(M_n[1/\mu]\big) = \DR(T)/p$ and we are left to show that $f_n(M_n)$ is of finite height $s$.
	Note that we have the following commutative diagram with exact rows:
	\begin{center}
		\begin{tikzcd}
			0 \arrow[r] & \varphi^*(\textrm{kernel}) \arrow[r] \arrow[d, "1 \otimes \varphi"] & \varphi^*(M_n) \arrow[r] \arrow[d, "1 \otimes \varphi"] & \varphi^*(f_n(M_n)) \arrow[r] \arrow[d, "1 \otimes \varphi"] & 0\\
			0 \arrow[r] & \textrm{kernel} \arrow[r] & M_n \arrow[r, "f_n"] & f_n(M_n) \arrow[r] & 0.
		\end{tikzcd}
	\end{center}
	The rightmost vertical arrow is injective because $f_n(M_n) \subset \DRn$ and the cokernel of the middle vertical arrow is killed by $[p]_q^s$ (see Lemma \ref{lem:Mn_large}).
	Hence, the cokernel of the rightmost vertical arrow is also killed by $[p]_q^s$.
	This concludes our proof.
\end{proof}

\begin{prop}\label{prop:wach_phigamm_comp}
	The natural inclusion $\NR(T) \subset \DR(T)$ extends to a natural $(\varphi, \Gamma_R)\equivariant$ isomorphism $\AR \otimes_{\AR^+} \NR(T) \isomorphic \DR(T)$.
	Similarly, the natural inclusion $\NR(T) \subset \NL(T)$ extends to a natural $(\varphi, \Gamma_R)\equivariant$ isomorphism $\AL^+ \otimes_{\AR^+} \NR(T) \isomorphic \NL(T)$.
\end{prop}
\begin{proof}
	For the first claim, note that $\DR(T)$ and $\NR(T)$ are $p\torsion$ free and $p\textrm{-adically}$ complete, so it is enough to show the claim modulo $p$.
	From the proof of Lemma \ref{lem:modp_fingen}, recall that we have $\NR(T)/p \subset M_1 = \DR(T)/p \cap \NL(T)/p \subset \DL(T)/p$, and from Lemma \ref{lem:mcirc_props} we have that $\Mcirc \subset \NR(T)/p$.
	Therefore,
	\begin{equation*}
		\DR(T)/p = \Mcirc[1/\mu] \subset \AR/p \otimes_{\AR^+/p} \NR(T)/p \subset M_1[1/\mu] = \DR(T)/p.
	\end{equation*}
	Hence, $(\NR(T)/p)[1/\mu] \isomorphic \DR(T)/p$, thus proving the first claim.

	For the second claim, note that $T$ is a $\ZZ_p\representation$ of $G_L$ such that $V \coloneq T[1/p]$ is crystalline for $G_L$, and from \cite[Theorem 4.1]{abhinandan-imperfect-wach} there exists a unique Wach module $\NL(T)$ over $\AL^+$ associated to $T$.
	Moreover, it is easy to verify that $\AL^+ \otimes_{\AR^+} \NR(T)$ is also a Wach module over $\AL^+$ associated to $T$, using that $\Gamma_L \isomorphic \Gamma_R$.
	Now, from the first claim, we have that $\AL \otimes_{\AR^+} \NR(T) \isomorphic \AL \otimes_{\AR} \DR(T) \isomorphic \DL(T)$ as \'etale $(\varphi, \Gamma_L)\modules$ over $\AL$.
	Hence, by the uniqueness of the Wach module associated to $T$ over $\AL^+$ (see \cite[Lemma 3.9]{abhinandan-imperfect-wach}), it follows that $\AL^+ \otimes_{\AR^+} \NR(T) \isomorphic \NL(T)$ as $(\varphi, \Gamma_L)\modules$ over $\AL^+$.
\end{proof}

\begin{proof}[Proof of Proposition \ref{prop:wach_module_constr}]
	It is immediate that $\NR(T)$ is $p\textrm{-torsion}$ free and $\mu\textrm{-torsion}$ free.
	From Lemma \ref{lem:modp_fingen} and its proof, note that $\NR(T)$ is finitely generated over $\AR^+$ and we have that $\NR(T)/p \subset (\NL(T)/p) \cap (\DR(T)/p) \subset \DL(T)/p$, in particular, $\NR(T)/p$ is $\mu\textrm{-torsion}$ free.
	Then, by using Lemma \ref{lem:ab_torsion}, we also get that $(\NR(T)/\mu)[p] = (\NR(T)/p)[\mu] = 0$, i.e.\ $\NR(T)/\mu$ is $p\torsion$ free.
	Next, from Lemma \ref{lem:finite_height}, we know that $\NR(T)$ is of finite $\pqheight$, i.e.\ the cokernel of the injective map $1 \otimes \varphi \colon \varphi^*(\NR(T)) \rightarrow \NR(T)$ is killed by $[p]_q^s$, where $s$ is the height of $\NL(T)$.
	Furthermore, from Proposition \ref{prop:wach_phigamm_comp} we have that $\AR \otimes_{\AR^+} \NR(T) \isomorphic \DR(T)$.
	Finally, recall that the action of $\Gamma_L$ is trivial on $\NL(T)/\mu \NL(T)$ and $\Gamma_L \isomorphic \Gamma_R$, so for any $g \in \Gamma_R$, we have that $(g-1)\NL(T) \subset \mu \NL(T)$.
	Therefore, we get that $(g-1)\NR(T) \subset (\mu \NL(T)) \cap \DR(T) = \mu \NR(T)$, so it follows that $\Gamma_R$ acts trivially on $\NR(T)/\mu$.
	This concludes our proof.
\end{proof}

We have all the ingredients to prove Theorem \ref{thm:crys_fh_relative}, which we show next.
\begin{proof}[Proof of Theorem \ref{thm:crys_fh_relative}]
	For a $\padic$ representation, the property of being crystalline (resp.\ of being finite $\pqheight$) is invariant under twisting the representation by $\chi^r$, where $\chi$ is the $\padic$ cyclotomic character and $r \in \NN$.
	Therefore, we may assume that $V$ is positive crystalline.
	In this case, note that $V$ is also a positive crystalline representation of $G_L$, and therefore, from \cite[Theorem 4.1]{abhinandan-imperfect-wach} it follows that $T$ is positive and of finite $\pqheight$ as a $\ZZ_p\representation$ of $G_L$ (in the sense of \cite[Definition 3.7]{abhinandan-imperfect-wach}).
	Moreover, associated to $T$, from \cite[Theorem 4.1]{abhinandan-imperfect-wach} we have the Wach module $\NL(T)$ over $\AL^+$ and we set $\NR(T) \coloneq \NL(T) \cap \DR(T) \subset \DL(T)$ as an $\AR^+\module$.
	Then, from Proposition \ref{prop:wach_module_constr}, the module $\NR(T)$ satisfies all the axioms of Definition \ref{defi:wach_mods_relative} and Definition \ref{defi:finite_pqheight}.
	Hence, it follows that $\NR(T)$ is the unique Wach module associated to $T$, or equivalently, $T$ is of finite $\pqheight$ as representation of $G_R$.
	In particular, from Proposition \ref{prop:wach_phigamm_comp}, we have that the natural inclusion $\NR(T) \subset \DR(T)$ (resp.\ $\NR(T) \subset \NL(T)$) extends to a natural $(\varphi, \Gamma_R)\equivariant$ isomorphism $\AR \otimes_{\AR^+} \NR(T) \isomorphic \DR(T)$ (resp.\ $\AL^+ \otimes_{\AR^+} \NR(T) \isomorphic \NL(T)$).
	This allows us to conclude.
\end{proof}

\subsection{Consequences of Theorem \ref{thm:crys_fh_relative}}\label{subsec:crys_wach_equiv_consequence}

In this section, we provide some applications of Theorem \ref{thm:crys_fh_relative}.
Let $\Rep_{\ZZ_p}^{\crys}(G_R)$ denote the category of $\ZZ_p\textrm{-lattices}$ inside $\padic$ crystalline representations of $G_R$.
Then, by combining Theorem \ref{thm:fh_crys_relative} and Theorem \ref{thm:crys_fh_relative}, we obtain the following:
\begin{cor}\label{cor:crystalline_wach_equivalence_relative}
	The Wach module functor induces a natural equivalence of categories
	\begin{align*}
		\Rep_{\ZZ_p}^{\crys}(G_R) &\isomorphic (\varphi, \Gamma_R)\Mod_{\AR^+}^{[p]_q}\\
		T &\longmapsto \NR(T),
	\end{align*}
	with a quasi-inverse functor given as $N \mapsto \TR(N) \coloneq \big(W\big(\Rbar^{\flat}[1/p^{\flat}]\big) \otimes_{\AR^+} N\big)^{\varphi=1}$.
\end{cor}

Passing to the associated isogeny categories, we obtain the following:
\begin{cor}\label{cor:crystalline_wach_rat_equivalence_relative}
	The Wach module functor induces an exact equivalence of $\otimes\textrm{-categories}$ 
	\begin{align*}
		\Rep_{\QQ_p}^{\crys}(G_R) &\isomorphic (\varphi, \Gamma_R)\Mod_{\BR^+}^{[p]_q}\\
		V &\longmapsto \NR(V),
	\end{align*}
	with an exact $\otimes\textrm{-compatible}$ quasi-inverse functor given as $M \mapsto \VR(M) \coloneq \big(W\big(\Rbar^{\flat}[1/p^{\flat}]\big) \otimes_{\AR^+} M\big)^{\varphi=1}$.
\end{cor}
\begin{proof}
	The equivalence of categories follows from Theorem \ref{thm:crys_fh_relative}.
	For the rest of the proof, let us remark that for a $\padic$ crystalline representation $V$ of $G_R$, from Proposition \ref{prop:wach_module_constr}, we have that
	\begin{equation*}
		\NR(V) = \NL(V) \cap \DR(V) \subset \DL(V),
	\end{equation*}
	as finite projective $(\varphi, \Gamma_R)\modules$ over $\BR^+$.
	Moreover, from Theorem \ref{thm:crys_fh_relative} and Proposition \ref{prop:wach_phigamm_comp}, note that $\BL^+ \otimes_{\BR^+} \NR(V) \isomorphic \NL(V)$ and $\BR \otimes_{\BR^+} \NR(V) \isomorphic \DR(V)$ compatible with the respective natural actions of $\varphi$ and $\Gamma_R$.

	Now, let $V_1$ and $V_2$ be two crystalline representations of $G_R$, then $V_1 \otimes_{\QQ_p} V_2$ is again crystalline (see \cite[Th\'eor\`em 8.4.2]{brinon-relatif}).
	So, inside $\DL(V_1 \otimes_{\QQ_p} V_2)$, we have the following:
	\begin{align*}
		\NR(V_1) \otimes_{\BR^+} \NR(V_2) &= \NR(V_1) \otimes_{\BR^+} (\NL(V_2) \cap \DR(V_2))\\
			&= (\NR(V_1) \otimes_{\BR^+} \NL(V_2)) \cap (\NR(V_1) \otimes_{\BR^+} \DR(V_2))\\
			&= (\NL(V_1) \otimes_{\BL^+} \NL(V_2)) \cap (\DR(V_1) \otimes_{\BR} \DR(V_2))\\
			&\isomorphic \NL(V_1 \otimes_{\QQ_p} V_2) \cap \DR(V_1 \otimes_{\QQ_p} V_2) = \NR(V_1 \otimes_{\QQ_p} V_2),
	\end{align*}
	where the first equality follows from the discussion above, the second equality follows because $\NR(V_1)$ is finite projective over $\BR^+$, the third equality again follows from the discussion above and the last isomorphism is $(\varphi, \Gamma_R)\equivariant$ and follows from \cite[Corollary 4.3]{abhinandan-imperfect-wach} and \eqref{eq:rep_phigamma_relative}.
	This shows the compatibility of $\NR$ with tensor products.

	Conversely, let $M_1$ and $M_2$ be two Wach modules over $\BR^+$.
	Then, note that
	\begin{equation}\label{eq:nr_v1otimesv2}
		\NR(\VR(M_1) \otimes_{\QQ_p} \VR(M_2)) \isomorphic \NR(\VR(M_1)) \otimes_{\BR^+} \NR(\VR(M_2)) \isomorphic M_1 \otimes_{\BR^+} M_2,
	\end{equation}
	where the first isomorphism is $(\varphi, \Gamma_R)\equivariant$ and follows from the compatibility of $\NR$ with tensor products shown in the previous paragraph, and the second isomorphism is also $(\varphi, \Gamma_R)\equivariant$ and follows because we have a natural isomorphism of functors $\NR \circ \VR \isomorphic id$.
	Applying the functor $\VR$ to \eqref{eq:nr_v1otimesv2}, noting that we have a natural isomorphism of functors $\VR \circ \NR \isomorphic id$ and using \eqref{eq:rep_phigamma_relative} (after inverting $p$), we obtain a natural $G_R\equivariant$ isomorphism $\VR(M_1) \otimes_{\QQ_p} \VR(M_2) \isomorphic \VR(M_1 \otimes_{\BR^+} M_2)$, thus establishing the compatibility of $\VR$ with tensor products.

	Next, we will show the exactness of the functors $\VR$ and $\NR$.
	The functor $\VR$ is defined to be the composition of the quasi-inverse of the functor in \eqref{eq:rep_phigamma_relative} (after inverting $p$) with the fully faithful functor in Corollary \ref{cor:wach_etale_ff_relative}.
	As the functor in Corollary \ref{cor:wach_etale_ff_relative} is given by extension of scalars along the flat homomorphism $\BR^+ \rightarrow \BR$ and the equivalence in \eqref{eq:rep_phigamma_relative} has an exact quasi-inverse, therefore, it follows that $\VR$ is exact.
	It remains to show that $\NR$ is also exact.
	So, let us consider an exact sequence of $\padic$ crystalline representations of $G_R$ as $0 \rightarrow V_1 \rightarrow V_2 \rightarrow V_3 \rightarrow 0$, and our goal is to show that the sequence
	\begin{equation}\label{eq:nrv_exact}
		0 \longrightarrow \NR(V_1) \longrightarrow \NR(V_2) \longrightarrow \NR(V_3) \longrightarrow 0,
	\end{equation}
	is exact.
	Let $T_2 \subset V_2$ be a $G_R\textrm{-stable}$ $\ZZ_p\lattice$, then $T_1 \coloneq V_1 \cap T_2 \subset V_2$ is a $G_R\textrm{-stable}$ $\ZZ_p\lattice$ inside $V_1$ and set $T_3 \coloneq T_2/T_1 \subset V_3$ as a $G_R\textrm{-stable}$ $\ZZ_p\lattice$.
	By definition, we have Wach modules $\NR(T_1)$, $\NR(T_2)$ and $\NR(T_3)$ and we set $N \coloneq \NR(T_2)/\NR(T_1)$ as a finitely generated $\AR^+\module$ equipped with a Frobenius $\varphi \colon N[1/\mu] \rightarrow N[1/\varphi(\mu)]$ and a continuous action of $\Gamma_R$ induced from the corresponding structures on $\NR(T_2)$.

	We claim that $N[1/p] \isomorphic \NR(V_3)$ as $(\varphi, \Gamma_R)\modules$ over $\BR^+$.
	Indeed, first recall that $\DR$ is an exact functor from the category of $\ZZ_p\textrm{-representations}$ of $G_R$ to the category of \'etale $(\varphi, \Gamma_R)\modules$ over $\AR$ (see Section \ref{subsec:relative_padicreps}).
	So, the following is an exact sequence of \'etale $(\varphi, \Gamma_R)\modules$ over $\AR$:
	\begin{equation*}
		0 \longrightarrow \DR(T_1) \longrightarrow \DR(T_2) \longrightarrow \DR(T_3) \longrightarrow 0.
	\end{equation*}
	As we have that $\AR \otimes_{\AR^+} \NR(T_i) \isomorphic \DR(T_i)$, for $i \in \{1, 2, 3\}$, and $\AR^+ \rightarrow \AR$ is flat, therefore, we get that $\AR \otimes_{\AR^+} N \isomorphic \DR(T_2)/\DR(T_1) \isomorphic \DR(T_3) \lisomorphic \AR \otimes_{\AR^+} \NR(T_3)$ as \'etale $(\varphi, \Gamma_R)\modules$ over $\AR$.
	Moreover, since $\varphi \colon \AR^+ \rightarrow \AR^+$ is flat, therefore, using the finite $\pqheight$ property of $\NR(T_1)$ and $\NR(T_2)$, we get that $N$ is of finite $\pqheight$, i.e.\ $ 1 \otimes \varphi \colon (\varphi^*N)[1/[p]_q] \isomorphic N[1/[p]_q]$.
	In particular, $N[1/p]$ is finite projective over $\BR^+$ by Proposition \ref{prop:finiteproj_torus}.
	Next, for $i \in \{1, 2, 3\}$, considering $V_i$ as a $\padic$ crystalline representation of $G_L$, from \cite[Corollary 4.3]{abhinandan-imperfect-wach}, the following is an exact sequence of Wach modules over $\BL^+ = \AL^+[1/p]$:
	\begin{equation*}
		0 \longrightarrow \NL(V_1) \longrightarrow \NL(V_2) \longrightarrow \NL(V_3) \longrightarrow 0.
	\end{equation*}
	As we have that $\BL^+ \otimes_{\BR^+} \NR(V_i) \isomorphic \NL(V_i)$, for $i \in \{1, 2, 3\}$ (see Theorem \ref{thm:crys_fh_relative} and Proposition \ref{prop:wach_phigamm_comp}), and $\BR^+ \rightarrow \BL^+$ is flat, therefore, we get that $\BL^+ \otimes_{\BR^+} N[1/p] \isomorphic \NL(V_3) \lisomorphic \BL^+ \otimes_{\BR^+} \NR(V_3)$ as $(\varphi, \Gamma_L)\modules$ over $\BL^+$.
	Now, as submodules of $\DL(V_3)$, we have the following isomorphism of $(\varphi, \Gamma_R)\modules$ over $\BR^+$:
	\begin{equation*}
		N[1/p] = (\BL^+ \otimes_{\BR^+} N[1/p]) \cap (\BR \otimes_{\BR^+} N[1/p]) \isomorphic \NL(V_3) \cap \DR(V_3) \lisomorphic \NR(V_3),
	\end{equation*}
	where the first equality follows because $N[1/p]$ is finite projective over $\BR^+$ and $\BR^+ = \BR \cap \BL^+ \subset \BL$ (see Lemma \ref{lem:ar_intersect_al+}).
	Hence, we obtain that \eqref{eq:nrv_exact} is exact, thus concluding our proof.
\end{proof}

We also have the following purity result as another application of Theorem \ref{thm:crys_fh_relative}:
\begin{thm}\label{thm:purity_crystalline}
	Let $V$ be a $\padic$ representation of $G_R$.
	Then, the following are equivalent:
	\begin{enumarabicup}
		\item $V$ is crystalline as a representation of $G_R$;
	
		\item $V$ is crystalline as a representation of $G_L$;
	
		\item $\rank_{R[1/p]} \ODcrysR(V) = \dim_{\QQ_p} V$.
	\end{enumarabicup}
\end{thm}
\begin{proof}
	Let $V$ be an object of $\Rep_{\QQ_p}^{\crys}(G_R)$, then obviously we have that $V$ is in $\Rep_{\QQ_p}^{\crys}(G_L)$.
	Conversely, let $V$ be an object of $\Rep_{\QQ_p}^{\crys}(G_L)$ and choose a $G_R\textrm{-stable}$ $\ZZ_p\textrm{-lattice}$ $T \subset V$ such that $T$ is of finite $\pqheight$ as a representation of $G_L$.
	Then, using Proposition \ref{prop:wach_module_constr}, note that $T$ is of finite $\pqheight$ as a representation of $G_R$.
	Therefore, $V = T[1/p]$ is a crystalline representation of $G_R$ by Theorem \ref{thm:fh_crys_relative}.
	This shows the equivalence of (1) and (2).
	Moreover, if $V$ is an object of $\Rep_{\QQ_p}^{\crys}(G_R)$, then $\rank_{R[1/p]} \ODcrysR(V) = \dim_{\QQ_p} V$ (see Section \ref{subsec:relative_padicreps}), thus proving that (1) implies (3).
	So, it remains to show that (3) implies (2).

	Let $V$ be a $\padic$ representation of $G_R$ such that $\rank_{R[1/p]} \ODcrysR(V) = \dim_{\QQ_p} V$.
	From \cite[Section 3.3]{brinon-imparfait} recall that $V$ is crystalline as a representation of $G_L$ if and only if we have $\dim_L \ODcrysL(V) = \dim_{\QQ_p} V$.
	So we will show that $\dim_L \ODcrysL(V) = \dim_{\QQ_p} V$ by constructing a natural isomorphism of $L\textrm{-vector spaces}$ $L \otimes_{R[1/p]} \ODcrysR(V) \isomorphic \ODcrysL(V)$.
	Since we always have that $\dim_L \ODcrysL(V) \leqslant \dim_{\QQ_p} V$ (see \cite[Proposition 3.22]{brinon-imparfait}), therefore, it is enough to construct a natural $L\textrm{-linear}$ injective map $L \otimes_{R[1/p]} \ODcrysR(V) \rightarrow \ODcrysL(V)$ and the claim would follow by considering $L\textrm{-dimensions}$.

	To construct this injective map, from Remark \ref{lem:bcrys_embedding_l}, note that we have a natural $L\linear$ and $(\varphi, G_R)\equivariant$ injective homomorphism $L \otimes_{R[1/p]} \OBcrys(\Rbar) \rightarrow \prod_{\pins} \OBcrys(\Cplusp)$.
	Tensoring this homomorphism with $V$ (over $\QQ_p$) and considering the diagonal action of $G_R$, we obtain a $(\varphi, G_R)\equivariant$ injective map
	\begin{equation}\label{eq:lbcrysv_embed}
		L \otimes_{R[1/p]} \OBcrys(\Rbar) \otimes_{\QQ_p} V \longrightarrow \big(\textstyle\prod_{\pins} \OBcrys(\Cplusp)\big) \otimes_{\QQ_p} V = \textstyle\prod_{\pins} (\OBcrys(\Cplusp) \otimes_{\QQ_p} V).
	\end{equation}
	The composition in \eqref{eq:lbcrysv_embed} further induces a natural homomorphism
	\begin{equation*}
		L \otimes_{R[1/p]} \OBcrys(\Rbar) \otimes_{\QQ_p} V \longrightarrow \OBcrys(\Cplusp) \otimes_{\QQ_p} V,
	\end{equation*}
	compatible with the respective Frobenii, filtrations and connections (see Remark \ref{lem:bcrys_embedding_l}).
	Now, we take the $G_R\textrm{-invariant}$ part of \eqref{eq:lbcrysv_embed} and note that product commutes with left exact functors, in particular, with taking $G_R\textrm{-invariants}$.
	So, we obtain the following $\varphi\equivariant$ $L\linear$ injective homomorphisms:
	\begin{equation}\label{eq:ldcrysrv_embed}
		\begin{aligned}
			L \otimes_{R[1/p]} \ODcrysR(V) &\longrightarrow \big(\textstyle\prod_{\pins} \OBcrys(\Cplusp) \otimes_{\QQ_p} V\big)^{G_R}\\
			&\xrightarrow{\hspace{1mm} = \hspace{1mm}} \textstyle\prod_{\pins} (\OBcrys(\Cplusp) \otimes_{\QQ_p} V)^{G_R}\\
			&\longrightarrow \textstyle\prod_{\pins} (\OBcrys(\Cplusp) \otimes_{\QQ_p} V)^{\GRp},
		\end{aligned}
	\end{equation}
	where we note that the last arrow is injective because $\GRp \subset G_R$ is a subgroup.
	Moreover, since $G_R$ acts transitively on $\pazs$, it transitively permutes the components of the product $\prod_{\pins} (\OBcrys(\Cplusp) \otimes_{\QQ_p} V)^{\GRp}$.
	So, if $x \neq 0$ is an element of $L \otimes_{R[1/p]} \ODcrysR(V)$, then its image $(x_{\frakp})_{\pins}$ under the composition \eqref{eq:ldcrysrv_embed} satisfies that $x_{\frakp} \neq 0$, for all $\pins$.
	Therefore, for each $\pins$, composing \eqref{eq:ldcrysrv_embed} with the natural $\varphi\equivariant$ $L\linear$ projection
	\begin{equation*}
		\textstyle\prod_{\pins} (\OBcrys(\Cplusp) \otimes_{\QQ_p} V)^{\GRp} \longrightarrow (\OBcrys(\Cplusp) \otimes_{\QQ_p} V)^{\GRp},
	\end{equation*}
	gives a natural $\varphi\equivariant$ $L\linear$ injective homomorphism
	\begin{equation}\label{eq:ldrcysrv_embed_p}
		L \otimes_{R[1/p]} \ODcrysR(V) \longrightarrow (\OBcrys(\Cplusp) \otimes_{\QQ_p} V)^{\GRp},
	\end{equation}
	compatible with the respective Frobenii, filtrations and connections (see above and Remark \ref{lem:bcrys_embedding_l}), where the left-hand term is equipped with the tensor product Frobenius, $L\linear$ extension of the filtration on $\ODcrysR(V)$ and an integrable tensor product connection ($\partial \otimes 1 + 1 \otimes \partial$).

	Next, from Lemma \ref{lem:cplusp_in_cpplus_bcrys}, recall that we have a natural $(\varphi, \GRhatp)\equivariant$ injective homomorphism of $L\textrm{-algebras}$ $\OBcrys(\Cplusp) \rightarrow \OBcrys(\Cpplus)$ compatible with the respective filtrations and connections, and where the $\GRhatp\action$ on the left-hand term factors through $\GRhatp \twoheadrightarrow \GRp$.
	Tensoring the preceding injective homomorphism with $V$ (over $\QQ_p$), equipping each term with the diagonal action of $\GRhatp$ and taking the $\GRhatp\textrm{-invariants}$ yields a natural $L\linear$ injective homomorphism
	\begin{equation*}
		(\OBcrys(\Cplusp) \otimes_{\QQ_p} V)^{\GRp} \longrightarrow \ODcrysL(V),
	\end{equation*}
	compatible with the respective Frobenii, filtrations and connections.
	Composing \eqref{eq:ldrcysrv_embed_p} with the preceding $L\linear$ homomorphism gives a natural $L\linear$ injective homomorphism
	\begin{equation}\label{eq:odcrys_functoriality}
		L \otimes_{R[1/p]} \ODcrysR(V) \longrightarrow \ODcrysL(V),
	\end{equation}
	compatible with the respective Frobenii, filtrations and connections.
	By considering $L\textrm{-dimensions}$, it follows that \eqref{eq:odcrys_functoriality} is bijective (see Corollary \ref{cor:odcrys_functoriality} for a stronger statement).
	Hence, we get that $\dim_L \ODcrysL(V) = \rank_{R[1/p]} \ODcrysR(V) = \dim_{\QQ_p} V$, showing that (3) implies (2).
	This concludes our proof.
\end{proof}

\begin{cor}\label{cor:odcrys_functoriality}
	Let $V$ be a $\padic$ representation of $G_R$.
	Under the equivalent conditions of Theorem \ref{thm:purity_crystalline}, the homomorphism in \eqref{eq:odcrys_functoriality} induces a natural isomorphism of filtered $(\varphi, \partial)\modules$ over $L$.
\end{cor}
\begin{proof}
	Note that $V$ is a crystalline representation of $G_R$ and the homomorphism in \eqref{eq:odcrys_functoriality} is filtered.
	By taking the associated graded, we obtain the following homomorphism of graded $L\textrm{-vector}$ spaces:
	\begin{equation}\label{eq:gr_odcrys_compatibility}
		L \otimes_{R[1/p]} \gr^{\bullet} \ODcrysR(V) \longrightarrow \gr^{\bullet} \ODcrysL(V).
	\end{equation}
	We claim that \eqref{eq:gr_odcrys_compatibility} is bijective.

	As $V$ is a crystalline representation of $G_R$, so it is also a de Rham representation of $G_R$ (in the sense of \cite[Section 8.2]{brinon-relatif}), and in our setting, we have a natural isomorphism of filtered $R[1/p]\modules$ $\ODcrysR(V) \isomorphic \pazo\mbfd_{\textup{dR}, R}(V) \coloneq (\OBdR(\Rbar) \otimes_{\QQ_p} V)^{G_R}$ (see \cite[Proposition 8.2.12]{brinon-relatif}; in fact, the filtration on $\ODcrysR(V)$ is induced from $\pazo\mbfd_{\textup{dR}, R}(V)$).
	Furthermore, $V$ is also a Hodge--Tate representation of $G_R$ (in the sense of \cite[Proposition-Definition 2.1]{hyodo-hodge-tate} and \cite[D\'efinition 4.2.1]{brinon-relatif}), and from \cite[p.\ 131, Proof of Proposition 8.4.3]{brinon-relatif} we have a natural isomorphism of graded $R[1/p]\modules$ 
	\begin{equation}\label{eq:grdr_ht_R}
		\gr^{\bullet} \pazo\mbfd_{\textup{dR}, R}(V) \isomorphic \pazo\mbfd_{\textup{HT}, R}(V) \coloneq (\pazo B_{\textup{HT}}(\Rbar) \otimes_{\QQ_p} V)^{G_R},
	\end{equation}
	where $\pazo B_{\textup{HT}}(\Rbar)$ is the Hodge--Tate period ring of Hyodo (see \cite{hyodo-hodge-tate} and \cite[Section 4.1]{brinon-relatif}), and we have a natural $G_R\equivariant$ isomorphism $\pazo B_{\textup{HT}}(\Rbar) \isomorphic \gr^{\bullet} \OBdR(\Rbar)$ (see \cite[Corollaire 5.2.7]{brinon-relatif}; one may even take the latter as a definition of $\pazo B_{\textup{HT}}(\Rbar)$).

	Similarly, as $V$ is a crystalline representation of $G_L$, therefore, it is also a de Rham representation of $G_L$ (in the sense of \cite[Section 3.3]{brinon-imparfait}) and in our setting, we have a natural isomorphism of filtered $L\textrm{-vector}$ spaces $\ODcrysL(V) \isomorphic \pazo\mbfd_{\textup{dR}, L}(V) \coloneq (\OBdR(\OLbar) \otimes_{\QQ_p} V)^{G_L}$ (see \cite[Proposition 3.30]{brinon-imparfait}; in fact, the filtration on $\ODcrysL(V)$ is induced from $\pazo\mbfd_{\textup{dR}, L}(V)$).
	Furthermore, $V$ is also a Hodge--Tate representation of $G_L$ (in the sense of \cite[Proposition-Definition 2.1]{hyodo-hodge-tate}) and from \cite[p.\ 976, Proof of Proposition 3.35]{brinon-imparfait}, we have a natural isomorphism of graded $L\textrm{-vector}$ spaces
	\begin{equation}\label{eq:grdr_ht_L}
		\gr^{\bullet} \pazo\mbfd_{\textup{dR}, L}(V) \isomorphic \pazo\mbfd_{\textup{HT}, L}(V) \coloneq (\pazo B_{\textup{HT}}(\OLbar) \otimes_{\QQ_p} V)^{G_L},
	\end{equation}
	where $\pazo B_{\textup{HT}}(\OLbar)$ is the Hodge--Tate period ring of Hyodo, and we have a natural $G_L\equivariant$ isomorphism $\pazo B_{\textup{HT}}(\OLbar) \isomorphic \gr^{\bullet} \OBdR(\OLbar)$ (see \cite{hyodo-hodge-tate} and \cite[Proposition 2.22]{brinon-imparfait}; one may even take the latter as a definition of $\pazo B_{\textup{HT}}(\OLbar)$).

	Now, using the observations of the preceding paragraphs, in particular, the inverse of the isomorphism in \eqref{eq:grdr_ht_R}, the isomorphism in \eqref{eq:grdr_ht_L} and the homomorphism in \eqref{eq:gr_odcrys_compatibility}, we obtain the following natural homomorphism of graded $L\textrm{-vector}$ spaces:
	\begin{align*}\label{eq:gr_odcrys_compatibility}
		L \otimes_{R[1/p]} \pazo\mbfd_{\textup{HT}, R}(V) &\isomorphic L \otimes_{R[1/p]} \gr^{\bullet} \pazo\mbfd_{\textup{dR}, R}(V) \isomorphic L \otimes_{R[1/p]} \gr^{\bullet} \ODcrysR(V)\\
			&\qquad \xrightarrow{\hspace{1mm}\eqref{eq:gr_odcrys_compatibility}\hspace{1mm}} \gr^{\bullet} \ODcrysL(V) \isomorphic \gr^{\bullet} \pazo\mbfd_{\textup{dR}, L}(V) \isomorphic \pazo\mbfd_{\textup{HT}, L}(V).
	\end{align*}
	Note that from \cite[p.\ 856, Proof of Theorem 9.1]{tsuji-hodge-tate-purity}, we have that the composition above is bijective (observe that our constructions are compatible with \cite{hyodo-hodge-tate, brinon-imparfait, brinon-relatif} and the constructions of \cite{tsuji-hodge-tate-purity} are also compatible with \cite{hyodo-hodge-tate}, in particular, we may directly use Tsuji's result).
	Hence, it follows that \eqref{eq:gr_odcrys_compatibility} is also bijective.

	As the respective filtrations on $\ODcrysR(V)$ and $\ODcrysL(V)$ are finite, there exists some $r \in \ZZ$ such that $\Fil^k(\ODcrysL(V)) = \ODcrysL(V)$ and $\Fil^k \ODcrysR(V) = \ODcrysR(V)$, for all $k \leqslant r$.
	Then, using that \eqref{eq:odcrys_functoriality} and \eqref{eq:gr_odcrys_compatibility} are bijective, an easy induction on $k \geqslant r$ establishes that
	\begin{equation*}
		L \otimes_{R[1/p]} \Fil^k \ODcrysR(V) \isomorphic \Fil^k \ODcrysL(V),
	\end{equation*}
	i.e.\ the isomorphism in \eqref{eq:odcrys_functoriality} is filtered.
	This allows us to conclude.
\end{proof}

\section{Wach modules and \texorpdfstring{$q\textrm{-connections}$}{-}}\label{sec:wachmod_qconnection}

In this section, we will interpret Wach modules over $\AR^+$ (resp.\ $\BR^+$) as modules with $\qconnection$ and show that Wach modules over $\BR^+$ may be seen as the $q\textrm{-deformation}$ of filtered $(\varphi, \partial)\modules$ over $R[1/p]$, coming from $\padic$ crystalline representations of $G_R$ (see Theorem \ref{thm:qdeformation_dcrys}).
For our definitions, we will follow \cite[Section 2]{morrow-tsuji}, with slight modifications.

\subsection{Formalism on \texorpdfstring{$\qconnection$}{-}}\label{subsec:qconnection_formalism}

Let $D$ be a commutative ring and consider a $D\algebra$ $A$ equipped with $d$ commuting $D\algebra$ automorphisms $\gamma_1 \ldots, \gamma_d$, i.e.\ an action of $\ZZ^d$.
Moreover, fix an element $q$ in $D$ such that $q-1$ is a nonzerodivisor of $D$ and $\gamma_i = 1 \textmod (q-1)A$, for all $1 \leqslant i \leqslant d$.
Assume that we have units $U_1, \ldots, U_d \in A^{\times}$ such that $\gamma_i(U_j) = q U_j$, if $i=j$, or $U_j$, if $i \neq j$.
We fix these choices for the rest of the section.

\begin{defi}[{\cite[Definition 2.1]{morrow-tsuji}}]\label{defi:qdeRham_complex}
	Let $q\Omega^{\bullet}_{A/D} \coloneq \oplus_{k=0}^d q\Omega^k_{A/D}$ be a differential graded $D\algebra$ defined as:
	\begin{itemize}
		\item $q\Omega^0_{A/D} \coloneq A$ and $q\Omega^1_{A/D}$ is a free left $A\module$ on formal basis elements $\dlog(U_i)$.

		\item The right $A\module$ structure on $q\Omega^1_{A/D}$ is twisted by the rule $\dlog(U_i) \cdot f = \gamma_i(f) \dlog(U_i)$.

		\item $\dlog(U_i) \dlog(U_j) = - \dlog(U_j) \dlog(U_i)$, if $i \neq j$, and $0$, if $i=j$.

		\item The following map is an isomorphism of left $A\modules$:
			\begin{align*}
				\oplus_{\smbfi \in I_k} A &\isomorphic q\Omega^k_{A/D}\\
				(f_{\smbfi}) &\longmapsto \textstyle\sum_{\smbfi \in I_k} f_{\smbfi} \dlog(U_{i_1}) \cdots \dlog(U_{i_k}),
			\end{align*}
			where $I_k = \{\smbfi = (i_1, \ldots, i_k) \in \NN^k \textrm{ such that } 1 \leqslant i_1 < \cdots < i_k \leqslant d\}$.

		\item The zeroth differential $d_q \colon A \rightarrow \Omega^1_{A/D}$ is given as $f \mapsto \sum_{i=1}^d \frac{\gamma_i(f)-f}{q-1} \dlog(U_i)$.

		\item The elements $\dlog(U_i)$ in $q\Omega^1_{D/A}$ are cocycles, for all $1 \leqslant i \leqslant d$.
	\end{itemize}
	The data $d_q \colon A \rightarrow q\Omega^1_{A/D}$ forms a differential ring over $D$, i.e.\ $q\Omega^1_{A/D}$ is a $D\textrm{-bimodule}$ and $d_q$ is $D\linear$ satisfying the Leibniz rule $d_q(fg) = d_q(f)g + fd_q(g)$ (see \cite[Section II.1.2.1]{andre}).
\end{defi}

\begin{defi}[{\cite[Definition 2.2]{morrow-tsuji}}]\label{defi:qconnection}
	A module with $\textit{q-connection}$ over $A$ is a right $A\module$ $N$ equipped with a $D\linear$ map $\nabla_q \colon N \rightarrow N \otimes_A q\Omega^1_{A/D}$ satisfying the Leibniz rule $\nabla_q(xf) = \nabla_q(x)f + x \otimes d_q(f)$, for all $f$ in $A$ and $x$ in $N$.
	The $\qconnection$ $\nabla_q$ extends uniquely to a map of graded $D\modules$ $\nabla_q \colon N \otimes_A q\Omega^{\bullet}_{A/D} \rightarrow N \otimes_A q\Omega^{\bullet+1}_{A/D}$ satisfying 
	\begin{equation*}
		\nabla_q((x \otimes \omega) \cdot \omega') = \nabla_q(x \otimes \omega) \cdot \omega' + (-1)^{\textrm{deg } \omega} (x \otimes \omega) \cdot d_q(\omega').
	\end{equation*}
	The $\qconnection$ $\nabla_q$ is said to be \textit{flat} or \textit{integrable} if $\nabla_q \circ \nabla_q = 0$.
\end{defi}

\begin{nota}
	For a module with $\qconnection$ $(N, \nabla_q)$ over $A$ as in Definition \ref{defi:qconnection}, we will denote by $\nabla_{q,i} \colon N \rightarrow N$ the $D\linear$ maps characterised by $\nabla_q(x) = \sum_{i=1}^d \nabla_{q,i}(x) \otimes \dlog(U_i)$, for all $x$ in $N$.
	Also, we will set $\partial_{q,i} \coloneq U_i^{-1}\nabla_{q, i} \colon N \rightarrow N$.
\end{nota}

Next, assume that $D$ is equipped with an endomorphism $\varphi \colon D \rightarrow D$ such that it is a lift of the absolute Frobenius on $D/p$ and $\varphi(q) = q^p$.
Moreover, assume that $A$ is equipped with a compatible (with $\varphi$ on $D$) endomorphism $\varphi \colon A \rightarrow A$ such that it is a lift of the absolute Frobenius on $A/p$, and $\varphi$ commutes with the action of $\gamma_1, \ldots, \gamma_d$ on $A$.
The endomorphism $\varphi$ induces an endomorphism $\varphi_{\Omega}$ on $q\Omega^1_{A/D}$ given as $\varphi_{\Omega}(\sum_{i=1}^d f_i \dlog(U_i)) = [p]_q \sum_{i=1}^d \varphi(f_i) \dlog(U_i)$.
In particular, from \cite[Lemma 2.12]{morrow-tsuji} the following diagram commutes:
\begin{center}
	\begin{tikzcd}[row sep=15pt]
		A \arrow[r, "d_q"] \arrow[d, "\varphi"'] & q\Omega^1_{A/D} \arrow[d, "\varphi_{\Omega}"]\\
		A \arrow[r, "d_q"] & q\Omega^1_{A/D}.
	\end{tikzcd}
\end{center}
It follows that given a $\qconnection$ $(N, \nabla_q)$ we can define the base change via Frobenius, of the $\qconnection$, denoted $\varphi^*\nabla_q$ on $\varphi^*N \coloneq N \otimes_{A, \varphi} A$, as 
\begin{align*}
	\varphi^* \nabla_q \colon \varphi^* N &\longrightarrow N \otimes_{A, \varphi} q\Omega^1_{A/D} = \varphi^* N \otimes q\Omega^1_{A/D}\\
	x \otimes f &\longmapsto (1 \otimes \varphi_{\Omega})(\nabla_q(x)) \cdot f + x \otimes d_q(f).
\end{align*}
A $\varphi\module$ with $\qconnection$ is given by a pair $(N, \nabla_q)$ equipped with an $A\linear$ isomorphism $\varphi_N \colon (\varphi^*N)[1/[p]_q] \isomorphic N[1/[p]_q]$ such that the following diagram commutes:
\begin{equation}\label{eq:qconnection_horizontal_frob}
	\begin{tikzcd}[row sep=15pt]
		(\varphi^*N)[1/[p]_q] \arrow[r, "\varphi^* \nabla_q"] \arrow[d, "\varphi_N"'] & (\varphi^*N)[1/[p]_q] \otimes q\Omega^1_{A/D} \arrow[d, "\varphi_N \otimes 1"]\\
		N[1/[p]_q] \arrow[r, "\nabla_q"] & N[1/[p]_q] \otimes q\Omega^1_{A/D}.
	\end{tikzcd}
\end{equation}

\subsection{Wach modules as \texorpdfstring{$q\textrm{-deformations}$}{-}}\label{subsec:wachmod_qdeformation}

In this section, we take $D \coloneq O_F\llbracket \mu \rrbracket$, $A \coloneq \AR^+$ equipped with the action of $\Gamma_R$ and $\{\gamma_1, \ldots, \gamma_d\}$ as topological generators of $\Gamma_R'$, the geometric part of $\Gamma_R$ (see Section \ref{sec:period_rings_padic_reps}).
Then, by setting $q \coloneq 1+\mu$ and $U_i \coloneq [X_i^{\flat}]$, for $1 \leqslant i \leqslant d$, we have that $\gamma_i = 1 \textmod \mu \AR^+$, for all $1 \leqslant i \leqslant d$.
In particular, $\AR^+$ satisfies the hypotheses of Definition \ref{defi:qdeRham_complex}.
Moreover, the Frobenius endomorphism on $\AR^+$ extends the Frobenius on $D$ given by the natural Frobenius on $O_F$ and $\varphi(\mu) = (1+\mu)^p-1$.
Furthermore, in this case, $q\Omega^1_{\AR^+/D}$ identifies with $\Omega^1_{\AR^+/D}$, which is given as $(p, \mu)\textrm{-adic}$ completion of the module of K\"ahler differentials of $\AR^+$ with respect to $D$.

Note that we have a Frobenius-equivariant isomorphism of rings $\AR^+/\mu \isomorphic R$, so from \cite[Remarks 2.4 \& 2.10]{morrow-tsuji}, reduction modulo $q-1$ of the differential ring $d_q \colon \AR^+ \rightarrow \Omega^1_{\AR^+/D}$ is the usual de Rham differential $d \colon R \rightarrow \Omega^1_R$.
Similarly, the reduction modulo $q-1$ of a module with $\qconnection$ over $\AR^+$ (see Definition \ref{defi:qconnection}) is an $R\module$ with connection.
We say that a $\qconnection$ is $(p, [p]_q)\textit{-adically}$ \textit{quasi-nilpotent} (equivalently, $(p, q-1)\textit{-adically}$ \textit{quasi-nilpotent}) if $\nabla_q \textmod q-1$ is $p\textrm{-adically}$ quasi-nilpotent, i.e.\ for each $1 \leqslant i \leqslant d$, the operator $\partial_{q,i} \textmod q-1$ is $p\textrm{-adically}$ quasi-nilpotent (see after Definition \ref{defi:qconnection} for notations, and \cite[Definition 2.18 \& Remark 2.19]{morrow-tsuji} for similar definitions).

\begin{prop}\label{defi:wachmod_qconnection}
	Let $N$ be a Wach module over $\AR^+$.
	Then, the geometric $\qconnection$
	\begin{align*}
		\nabla_q \colon N &\longrightarrow N \otimes_{\AR^+} \Omega^1_{\AR^+/D}\\
		x &\longmapsto \textstyle\sum_{i=1}^d \tfrac{\gamma_i(x)-x}{\mu} \dlog([X_i^{\flat}]),
	\end{align*}
	describes $(N, \nabla_q)$ as a $\varphi\module$ equipped with a $(p, [p]_q)\textrm{-adically}$ quasi-nilpotent flat $\qconnection$ over $\AR^+$.
\end{prop}
\begin{proof}
	Set $\nabla_{q, i} \coloneq \tfrac{\gamma_i-1}{\mu}$ and $\partial_{q, i} \coloneq [X_i^{\flat}]^{-1} \nabla_{q, i}$.
	Then, from \cite[Lemma 2.3]{morrow-tsuji}, the $\qconnection$ $\nabla_q$ is flat if and only if the operators $\nabla_{q, 1}, \ldots \nabla_{q, d}$ commute pairwise, which is equivalent to the pairwise commutativity of $\gamma_1, \ldots, \gamma_d$.
	As $\Gamma_R'$ is a commutative group, therefore, it follows that $\nabla_q$ is flat.
	Moreover, from Definition \ref{defi:wach_mods_relative} and Lemma \ref{lem:finite_pqheight_equiv}, note that we have $\varphi \otimes 1: (N \otimes_{\AR^+, \varphi} \AR^+)[1/[p]_q] \isomorphic N[1/[p]_q]$.
	So we see that the pair $(N, \nabla_q)$ is a $\varphi\module$ equipped with a flat $\qconnection$ over $\AR^+$.
	Furthermore, since the action of $\varphi$ and $\Gamma_R'$ commute on $N$, therefore, it follows that the corresponding diagram \eqref{eq:qconnection_horizontal_frob} is commutative.

	Next, from the commutativity of the action of $\varphi$ and $\Gamma_R$ and the diagram \eqref{eq:qconnection_horizontal_frob}, note that we have $\nabla_q \circ \varphi = [p]_q \varphi \circ \nabla_q$.
	Furthermore, from the Frobenius finite height condition on $N$, we have that for any $x$ in $N$, there exists $r \in \NN$ large enough, such that $[p]_q^r x$ belongs to $\varphi^*(N)$.
	So, using the relation $\nabla_q \circ \varphi = [p]_q \varphi \circ \nabla_q$ and the fact that $[p]_q = p \textrm{ mod } q-1$, we see that $\partial_{q, i}^k([p]_q^rx) \textrm{ mod } q-1$ converges $p\textrm{-adically}$ to $0$ as $k \rightarrow +\infty$ (see after Definition \ref{defi:qconnection} for notations).
	Hence, it follows that $\partial_{q,i}^k(x) = [p]_q^{-r} \partial_{q,i}^k([p]_q^r x)$ modulo $q-1$ converges $p\textrm{-adically}$ to 0, i.e.\ $\nabla_q$ is $(p, [p]_q)\textrm{-adically}$ quasi-nilpotent.
	This concludes our proof.
\end{proof}

\begin{rem}
	In Proposition \ref{defi:wachmod_qconnection} we call the $\qconnection$ ``geometric'' because in the definition we only use the geometric part of $\Gamma_R$, i.e.\ $\Gamma_R'$.
\end{rem}

\begin{rem}\label{rem:qconnection_arpipd}
	From Section \ref{subsec:wachmod_crystalline} recall that we have the $(\varphi, \Gamma_R)\equivariant$ inclusion $\ARpi^{\PD} \subset \Acrys(\Rinfty)$.
	For $R = O_F$, we denote the aforementioned ring, i.e.\ $\AFpi^{\PD}$ by $D^{\PD}$ and for general $R$, we denote it by $A^{\PD} \coloneq \ARpi^{\PD}$ (we do not use $D$ and $A$ for these rings to avoid conflict with assumptions at the beginning of this section).
	Then, it is easy to see that the hypotheses of Definition \ref{defi:qdeRham_complex} are satisfied for $D^{\PD}$, $A^{\PD}$ with $\Gamma_R\action$ and $U_i \coloneq [X_i^{\flat}]$.
	Now, given a Wach module $N$ over $\AR^+$ note that $N^{\PD} \coloneq A^{\PD} \otimes_{\AR^+} N$ is $p\textrm{-adically}$ complete (see Lemma \ref{lem:pcomplete_finprojrational}), and similar to Propostion \ref{defi:wachmod_qconnection}, one may show that the $\qconnection$
	\begin{equation*}
		\nabla_q \colon N^{\PD} \longrightarrow N^{\PD} \otimes_{A^{\PD}} \Omega^1_{A^{\PD}/D^{\PD}}, \hspace{5mm} x \longmapsto \textstyle\sum_{i=1}^d \tfrac{\gamma_i(x)-x}{\mu} \dlog([X_i^{\flat}]),
	\end{equation*}
	describes $(N^{\PD}, \nabla_q)$ as a $\varphi\module$ equipped with a $p\textrm{-adically}$ quasi-nilpotent flat $\qconnection$ over $A^{\PD}$.
	Set $\nabla_{q, i} \coloneq (\gamma_i-1)/\mu$, for $1 \leqslant i \leqslant d$.
	Then, by employing arguments similar to \cite[Lemmas 4.12, 5.17 \& 5.18]{abhinandan-syntomic} we see that for $1 \leqslant i \leqslant d$, the operator $\nabla_i \coloneq (\log \gamma_i)/t = \frac{1}{t}\sum_{k \in \NN} (-1)^k \frac{(\gamma_i-1)^{k+1}}{k+1}$ converges as a series of operators on $N^{\PD}$.
	So, using the explicit formulas described above, it is easy to see that for any $x$ in $N[1/p]$, we have that $\nabla_{q, i}(x) - \nabla_i(x) = (\frac{\gamma_i-1}{\mu} - \frac{\log\gamma_i}{t})(x)$ belongs to $(\Fil^1 A^{\PD}) \otimes_{\AR^+} N[1/p]$, since $t/\mu$ is a unit in $A^{\PD}$ by \cite[Lemma 3.14]{abhinandan-relative-wach-i}.
\end{rem}

We are now ready to state the main result of this section.
Let $N$ be a Wach module over $\AR^+$ equipped with a $\qconnection$ as in Proposition \ref{defi:wachmod_qconnection} and a Nygaard filtration as in Definition \ref{defi:nygaard_fil}.
Then, from the discussion preceding Proposition \ref{defi:wachmod_qconnection}, we note that $N/\mu$ is a $\varphi\module$ over $R$ equipped with a $p\textrm{-adically}$ quasi-nilpotent flat (integrable) connection and a filtration $\Fil^k(N/\mu)$ given as the image of $\Fil^k N$ under the surjection $N \twoheadrightarrow N/\mu$.
Using Remark \ref{rem:gamma_minus1_image} note that the connection on $N/\mu$ satisfies Griffiths transversality with respect to the filtration, i.e.\ $\nabla(\Fil^k (N/\mu)) \subset \Fil^{k-1}(N/\mu) \otimes \Omega^1_R$.
By inverting $p$, these structures naturally extend to $(N/\mu)[1/p]$ and we equip it with these structures, in particular, we note that $(N/\mu)[1/p]$ is a filtered $(\varphi, \partial)\module$ over $R[1/p]$.

\begin{thm}\label{thm:qdeformation_dcrys}
	Let $N$ be a Wach module over $\AR^+$ and $V \coloneq \TR(N)[1/p]$ the associated crystalline representation from Theorem \ref{thm:fh_crys_relative}.
	Then, we have a natural isomorphism $(N/\mu)[1/p] \isomorphic \ODcrysR(V)$ of filtered $(\varphi, \partial)\modules$ over $R[1/p]$.
\end{thm}
\begin{proof}
	For $r \in \NN$ large enough, note that the Wach module $\mu^r N (-r)$ is always effective and we have that $\TR(\mu^rN(-r)) = \TR(N)(-r)$ (the twist $(-r)$ denotes a Tate twist on which $\Gamma_R$ acts via $\chi^{-r}$, where $\chi$ is the $\padic$ cyclotomic character).
	Therefore, it is enough to show both the claims for effective Wach modules.
	So, let us assume that $N$ is effective and set $M \coloneq N[1/p]$ to which the action of $\Gamma_R$ and the Frobenius-semilinear operator $\varphi$ naturally extend, and we also equip $M$ with the Nygaard filtration as in Definition \ref{defi:nygaard_fil}.
	It follows that the finite projective $R[1/p]\module$ $M/\mu$ is equipped with a Frobenius-semilinear operator $\varphi$, induced from $M$.
	Note that $[p]_q = p \mod \mu \AR^+$, therefore, we have that $1 \otimes \varphi \colon \varphi^*(M/\mu) \isomorphic M/\mu$ .
	Moreover, the filtration $\Fil^k (M/\mu)$ is given as the image of $\Fil^k M$ under the surjective map $M \twoheadrightarrow M/\mu$ (see Lemma \ref{lem:fil_gr_modmu}).
	Furthermore, from Theorem \ref{thm:fh_crys_relative}, we have the $R[1/p]\module$ $\ODR \coloneq (\OARpi^{\PD} \otimes_{\AR^+} M)^{\Gamma_R}$ equipped with a Frobenius-semilinear operator $\varphi$ and a connection, and an $R[1/p]\linear$ isomorphism $\ODR \isomorphic \ODcrysR(V)$ compatible with the respective Frobenii and connections (see \eqref{eq:dr_dcrys_comp} in Theorem \ref{thm:fh_crys_relative}).

	Now, note that we have $(\Fil^1 \OARpi^{\PD} \otimes_{\AR^+} M) \cap M = (\Fil^1 \OARpi^{\PD} \cap \AR^+) \otimes_{\AR^+} M = \mu M$, where the first equality follows because $M$ is finite projective over $\BR^+ = \AR^+[1/p]$, so flat over $\AR^+$, and the second equality follows from Remark \ref{rem:fil1ar+_intersect}.
	So, let us consider the following diagram with exact rows:
	\begin{center}
		\begin{tikzcd}[row sep=16pt]
			0 \arrow[r] & \mu M \arrow[r] \arrow[d] & M \arrow[r] \arrow[d] & M/\mu \arrow[r] \arrow[d] & 0\\
			0 \arrow[r] & (\Fil^1 \OARpi^{\PD}) \otimes_{\AR^+} M \arrow[r] & \OARpi^{\PD} \otimes_{\AR^+} M \arrow[r] & R[\varpi] \otimes_R (M/\mu M) \arrow[r] & 0\\
			0 \arrow[r] & (\Fil^1 \OARpi^{\PD}) \otimes_R \ODR \arrow[r] \arrow[u, "\wr"] & \OARpi^{\PD} \otimes_R \ODR \arrow[r] \arrow[u, "\wr", "\eqref{eq:oarpd_comparison}"'] & R[\varpi] \otimes_R \ODR \arrow[r] \arrow[u, "\wr"] & 0,
		\end{tikzcd}
	\end{center}
	where from the exactness of the second row and the discussion above, it follows that the vertical maps from the first to the second row are natural inclusions.
	Moreover, the middle vertical arrow from the third to the second row is the isomorphism \eqref{eq:oarpd_comparison} in Proposition \ref{prop:oarpd_comparison}, and the left vertical arrow is the tensor product of the $\OARpi^{\PD}\linear$ isomorphism in the middle vertical arrow with the $\OARpi^{\PD}\module$ $\Fil^1 \OARpi^{\PD}$, in particular, the left vertical arrow is bijective as well.
	So, we conclude that the right vertical arrow is also an isomorphism.
	Taking the $\Gal(R[1/p][\varpi]/R[1/p]) = \Gal(F(\zeta_{p^m})/F)\textrm{-invariants}$ (recall that $m = 1$, for $p \geqslant 3$, and $m = 2$, for $p = 2$) of the right vertical arrows gives a natural isomorphism compatible with the respective Frobenii:
	\begin{equation}\label{eq:dr_iso_mmodmu}
		\ODR \isomorphic M/\mu.
	\end{equation}

	We claim that \eqref{eq:dr_iso_mmodmu} is compatible with the respective connections as well.
	Indeed, note that the connection on $M/\mu$ is obtained by first reducing, the $\qconnection$ $\nabla_q$ on $N$, modulo $\mu = q-1$ and then inverting $p$.
	On the other hand, the connection $\partial_D$ on $\ODR = (\OARpi^{\PD} \otimes_{\AR^+} M)^{\Gamma_R}$ is induced from the natural $\ARpi^{\PD}\linear$ connection on $\OARpi^{\PD}$.
	Let $\nabla_{q, i}$ and $\partial_{D, i}$ respectively denote the $i^{\textrm{th}}$ component of the $\qconnection$ on $N$ and the connection on $\ODR$.
	Now, let $x$ be in $M$, and note that from Remark \ref{rem:relate_con_qcon} there exists some $y$ in $\OARpi^{\PD} \otimes_R \ODR$ such that $x = f(y) \textmod (\Fil^1 \OARpi^{\PD}) \otimes_{\AR^+} M$, where $f$ is the isomorphism in \eqref{eq:oarpd_comparison}.
	Then, it follows that to check the compatibility of the isomorphism $\ODR \isomorphic M/\mu$ with connections, it is enough to show that $\nabla_{q, i}(x) - f(\partial_{D, i}(y))$ belongs to $(\Fil^1 \OARpi^{\PD}) \otimes_{\AR^+} M$.
	From Remark \ref{rem:relate_con_qcon} for $\nabla_i = (\log \gamma_i)/t$, we know that $\nabla_i(x) - f(\partial_{D, i}(y))$ is in $(\Fil^1 \OARpi^{\PD}) \otimes_{\AR^+} M$.
	Furthermore, from Remark \ref{rem:qconnection_arpipd} we have that $\nabla_{q, i}(x) - \nabla_i(x)$ is in $(\Fil^1 \OARpi^{\PD}) \otimes_{\AR^+} M$.
	Upon combining the two, we get that $\nabla_{q, i}(x) - f(\partial_{D_i}(y))$ is in $(\Fil^1 \OARpi^{\PD}) \otimes_{\AR^+} M$, i.e.\ the isomorphism \eqref{eq:dr_iso_mmodmu} is compatible with the respective connections.

	Finally, by composing the inverse of \eqref{eq:dr_iso_mmodmu} with \eqref{eq:dr_dcrys_comp} from Theorem \ref{thm:fh_crys_relative}, we obtain isomorphisms
	\begin{equation}\label{eq:qdeformation_dcrys}
		M/\mu \isomorphic \ODR \isomorphic \ODcrysR(V),
	\end{equation}
	compatible with the respective Frobenii and connections.
	By transport of structure, we equip $\ODR$ with a filtration induced from the Hodge filtration on $\ODcrysR(V)$.
	Then, by Proposition \ref{prop:fil_compatibility} we get that the isomorphisms in \eqref{eq:qdeformation_dcrys} are further compatible with the respective filtrations.
	This allows us to conclude.
\end{proof}

The following observation was used above:
\begin{prop}\label{prop:fil_compatibility}
	Let $N$ be a Wach module over $\AR^+$, set $M \coloneq N[1/p]$ and let $V \coloneq \TR(N)[1/p]$ denote the associated crystalline representation of $G_R$.
	Then, the isomorphism $f \colon M/\mu \isomorphic \ODcrysR(V)$ from Theorem \ref{thm:qdeformation_dcrys} is compatible with filtrations, i.e.\ for each $k \in \ZZ$, we have a natural $R[1/p]\linear$ isomorphism
	\begin{equation}\label{eq:fil_compatibility}
		\Fil^k(M/\mu) \isomorphic \Fil^k \ODcrysR(V).
	\end{equation}
\end{prop}
\begin{proof}
	For $r \in \NN$ large enough, note that the Wach module $\mu^r N (-r)$ is always effective and we have that $\TR(\mu^rN(-r)) = \TR(N)(-r)$ (the twist $(-r)$ denotes a Tate twist on which $\Gamma_R$ acts via $\chi^{-r}$, where $\chi$ is the $\padic$ cyclotomic character).
	Therefore, it is enough to show the claim for effective Wach modules.
	Let us set $N_R \coloneq N$, $M_R \coloneq M$, $N_L \coloneq \AL^+ \otimes_{\AR^+} N_R$ (a Wach module over $\AL^+$, see Proposition \ref{prop:wach_phigamm_comp}) and $M_L \coloneq N_L[1/p]$, equipped with the induced actions of $\varphi$ and $\Gamma_L \isomorphic \Gamma_R$.
	Recall that we have the finite projective $R[1/p]\module$ $\ODR \coloneq (\OARpi^{\PD} \otimes_{\AR^+} M_R)^{\Gamma_R}$ and similarly we have the finite dimensional $L\textrm{-vector space}$ $\ODL \coloneq (\OALpi^{\PD} \otimes_{\AL^+} M_L)^{\Gamma_L}$, where the ring $\OALpi^{\PD}$ (depending on $L$) is analogous to $\OARpi^{\PD}$ (see \cite[Section 3.3]{abhinandan-imperfect-wach} for precise definitions), and admits a natural map $\OARpi^{\PD} \rightarrow \OALpi^{\PD}$ compatible with supplementary structures.
	From \cite[Theorem 1.8 \& Corollary 3.16]{abhinandan-imperfect-wach}, recall that we have isomorphisms $M_L/\mu \isomorphic \ODL \isomorphic \ODcrysL(V)$ of $\varphi\modules$ over $L$ (similar to \eqref{eq:qdeformation_dcrys}), and note that the constructions of loc.\ cit.\ are compatible with the constructions of this paper.
	Now, consider the following diagram:
	\begin{equation}\label{eq:fil_compatibility_l}
		\begin{tikzcd}[row sep=15pt]
			L \otimes_{R[1/p]} (M_R/\mu) \arrow[r, "\sim"] \arrow[d, "\wr"] & L \otimes_{R[1/p]} \ODR \arrow[d] \arrow[r, "\sim", "\eqref{eq:dr_dcrys_comp}"'] & L \otimes_{R[1/p]} \ODcrysR(V) \arrow[d, "\wr"', "\eqref{eq:odcrys_functoriality}"]\\
			M_L/\mu \arrow[r, "\sim"] & \ODL \arrow[r, "\sim"] & \ODcrysL(V),
		\end{tikzcd}
	\end{equation}
	where the top row is the scalar extension of \eqref{eq:qdeformation_dcrys} along the flat homomorphism $R[1/p] \rightarrow L$ and the bottom row is as discussed above (see the proof of \cite[Corollary 3.16]{abhinandan-imperfect-wach} for details).
	In \eqref{eq:fil_compatibility_l}, the left and the middle vertical arrows are the natural maps.
	Then, by the discussion above we see that the left square commutes.
	Moreover, as the top right and the bottom right horizontal isomorphisms are induced by natural inclusions (see \eqref{eq:dr_in_dcrys} and \cite[Equation (3.7)]{abhinandan-imperfect-wach}, respectively) and the crystalline period rings over $R$ and $L$ are compatible, therefore, it follows that the right square commutes as well.
	Furthermore, in \eqref{eq:fil_compatibility_l}, the left vertical arrow is a filtered isomorphism by Lemma \ref{lem:nygaard_fil_nr_nl_modmu}, the composition of the bottom arrows is a filtered isomorphism by \cite[Theorem 1.8]{abhinandan-imperfect-wach} (see Remark \ref{rem:fil_compatibility_ml_mlbreve} for another proof) and the right vertical arrow is a filtered isomorphism by Corollary \ref{cor:odcrys_functoriality}.

	Now, we note that the composition of the arrows in the top row of \eqref{eq:fil_compatibility_l} is the extension along $R[1/p] \rightarrow L$ of the isomorphism $f \colon M_R/\mu \isomorphic \ODcrysR(V)$, and from Lemma \ref{lem:fil_compatibility}, it follows that the map $f$ induces the map in \eqref{eq:fil_compatibility} and the induced map is bijective.
	This concludes our proof.
\end{proof}

\begin{lem}\label{lem:fil_compatibility}
	The map $f$ induces the map in \eqref{eq:fil_compatibility} and the induced map is bijective.
\end{lem}
\begin{proof}
	The case $k=0$ follows from the isomorphism $f$, and for $k \geqslant 1$, we shall proceed via induction, by assuming the following:
	\begin{enumerate}
		\item[(1)] The morphism $f$ induces the isomorphism in \eqref{eq:fil_compatibility} for $k-1$.

		\item[(2)] The following isomorphism of $R[1/p]\textrm{-modules}$ holds:
		\begin{equation}\label{eq:rational_filmodmu_induced}
			\Fil^k (M_R/\mu) \isomorphic (M_R/\mu) \cap \Fil^k (M_L/\mu) \subset M_L/\mu.
		\end{equation}
	\end{enumerate}
	For $k = 1$, we already have that $f$ is an isomorphism and from \eqref{eq:fil1_nrnlmodmu_induced} of Lemma \ref{lem:nygaard_fil_nr_nl_modmu} note that \eqref{eq:rational_filmodmu_induced} holds.
	So, we may assume that for some $k$, it is given that (1) holds for $k-1$ and (2) holds for $k$, and we need to show that $f$ induces the isomorphism in \eqref{eq:fil_compatibility} for $k$ and \eqref{eq:rational_filmodmu_induced} holds for $k+1$.

	For the first claim, note that using the isomorphism \eqref{eq:rational_filmodmu_induced}, the filtered isomorphism $M_L/\mu \isomorphic \ODcrysL(V)$ and its compatibility with $f$ (see \eqref{eq:fil_compatibility_l}), it follows that we have the following isomorphisms of $R[1/p]\textrm{-modules}$:
	\begin{equation*}
		\Fil^k (M_R/\mu) \isomorphic (M_R/\mu) \cap \Fil^k(M_L/\mu) \isomorphic \ODcrysR(V) \cap \Fil^k \ODcrysL(V) \lisomorphic \Fil^k \ODcrysR(V),
	\end{equation*}
	where the intersection in the third term is taken inside $\ODcrysL(V)$ via the filtered isomorphism \eqref{eq:odcrys_functoriality} (see Corollary \ref{cor:odcrys_functoriality}), and the last isomorphism follows because we have that $\gr^k \ODcrysR(V) \hookrightarrow \gr^k \ODcrysL(V)$ using the isomorphism \eqref{eq:gr_odcrys_compatibility} and finite projectivity of $\gr^k \ODcrysR(V)$ as an $R[1/p]\module$ (see \cite[Proposition 8.3.2]{brinon-relatif}).
	Therefore, it follows that \eqref{eq:fil_compatibility} is bijective for $k$, proving the first claim.
	Moreover, recall that $\Fil^k \ODcrysR(V)$ is a finite projective $R[1/p]\module$ (see \cite[Proposition 8.3.2]{brinon-relatif}), so from the preceding isomorphism we also get that $\Fil^k (M_R/\mu)$ is a finite projective $R[1/p]\module$.

	Next, for the second claim, using the diagram \eqref{eq:nr_nl_modmu_inject} it is clear that \eqref{eq:rational_filmodmu_induced} is injective and to show that it is surjective, let us consider the following diagram with exact rows:
	\begin{equation}\label{eq:filgr_modmu_diagram}
		\begin{tikzcd}[row sep=8pt, column sep=2pt]
			\Fil^{k+1} (M_R/\mu) && (\Fil^k M_R)/\mu && \gr^k M_R \\
			& \Fil^{k+1} (M_L/\mu) && (\Fil^k M_L)/\mu && \gr^k M_L \\
			\Fil^{k+1} (M_R/\mu) && \Fil^k (M_R/\mu) && \gr^k(M_R/\mu) \\
			& \Fil^{k+1} (M_L/\mu) && \Fil^k (M_L/\mu) && \gr^k(M_L/\mu),
			\arrow[hook, from=1-1, to=1-3]
			\arrow[dotted, hook, from=1-1, to=2-2]
			\arrow[equal, from=1-1, to=3-1]
			\arrow[two heads, from=1-3, to=1-5]
			\arrow[dotted, hook, from=1-3, to=2-4]
			\arrow[two heads, from=1-3, to=3-3]
			\arrow[dotted, hook, from=1-5, to=2-6]
			\arrow[two heads, from=1-5, to=3-5]
			\arrow[two heads, from=2-6, to=4-6]
			\arrow[hook, from=3-1, to=3-3]
			\arrow[dotted, hook, from=3-1, to=4-2]
			\arrow[shift left=2, bend left=40pt, dashed, "s"{pos=0.65}, from=3-3, to=1-3]
			\arrow[two heads, from=3-3, to=3-5]
			\arrow[dotted, hook, from=3-3, to=4-4]
			\arrow[dotted, from=3-5, to=4-6]
			\arrow[hook, from=4-2, to=4-4]
			\arrow[shift left=2, bend left=40pt, dashed, "1 \otimes s"{pos=0.65}, from=4-4, to=2-4]
			\arrow[two heads, from=4-4, to=4-6]
			\arrow[hook, crossing over, from=2-2, to=2-4]
			\arrow[equal, crossing over, from=2-2, to=4-2]
			\arrow[two heads, crossing over, from=2-4, to=2-6]
			\arrow[two heads, crossing over, from=2-4, to=4-4]
		\end{tikzcd}
	\end{equation}
	where the squares involving $M_R$ will be referred to as the bottom layer and it is $R[1/p]\linear$, the squares involving $M_L$ will be referred to as the top layer and it is $L\linear$, and the dotted arrows from the bottom layer to the top layer are $R[1/p]\linear$.
	In the diagram \eqref{eq:filgr_modmu_diagram}, the exactness of the first row of the top (resp.\ bottom) layer follows from \eqref{eq:filn_modmu_grn} of Lemma \ref{lem:fil_gr_modmu}, and the second row of the top (resp.\ bottom) layer is exact by definition.
	Moreover, the surjective map in the second column of the top (resp.\ bottom) layer is the surjective map in \eqref{eq:filn_modmu_filnmodmu} of Lemma \ref{lem:fil_gr_modmu}, and the surjective map in the third column of the top (resp.\ bottom) layer is the surjective map in \eqref{eq:grnmodmu} of Lemma \ref{lem:fil_gr_modmu}.
	Using Proposition \ref{prop:nygaard_fil_nr_nl}, Lemma \ref{lem:filnrnl_modmu_injective} and Lemma \ref{lem:nygaard_fil_nr_nl_modmu}, we see that the top layer is the base change of the bottom layer along the flat homomorphism $R \rightarrow O_L$ and each arrow from the bottom layer to the top layer (except for the lower right corner) is already injective.
	Also, note that by definition each square involving solid and dotted arrows in the diagram \eqref{eq:filgr_modmu_diagram} is commutative.
	Furthermore, the dashed arrow in the bottom layer labelled ``$s$'' is an $R[1/p]\linear$ section to the surjective arrow in the middle column such that the composition is identity on $\Fil^k(M_R/\mu)$, and it exists because $\Fil^k(M_R/\mu)$ is a finite projective $R[1/p]\module$; the dashed arrow in the top layer denotes the extension of scalars of $s$ along the flat homomorphism $R \rightarrow O_L$ and it determines a section to the surjective arrow in the middle column such that the composition is identity on $\Fil^k(M_L/\mu)$.

	Now, we claim that the dotted arrow between the lower right corner of the two layers in \eqref{eq:filgr_modmu_diagram} is also injective which is clearly equivalent to the bijectivity of the map in \eqref{eq:rational_filmodmu_induced} for $k+1$.
	For clarity, using \eqref{eq:filgr_modmu_diagram} let us consider the following diagram with exact rows, $R[1/p]\linear$ bottom layer and $L\linear$ top layer:
	\begin{equation}\label{eq:filgr_modmu_diagram_s}
		\begin{tikzcd}[row sep=8pt, column sep=2pt]
			\Fil^{k+1} (M_R/\mu) && \Fil^k (M_R/\mu) && \gr^k(M_R/\mu) \\
			& \Fil^{k+1} (M_L/\mu) && \Fil^k (M_L/\mu) && \gr^k(M_L/\mu) \\
			\Fil^{k+1} (M_R/\mu) && (\Fil^k M_R)/\mu && Q_R \\
			& \Fil^{k+1} (M_L/\mu) && (\Fil^k M_L)/\mu && Q_L,
			\arrow[hook, from=1-1, to=1-3]
			\arrow[dotted, hook, from=1-1, to=2-2]
			\arrow[equal, from=1-1, to=3-1]
			\arrow[two heads, from=1-3, to=1-5]
			\arrow[dotted, hook, from=1-3, to=2-4]
			\arrow[hook, "s"{pos=0.20}, from=1-3, to=3-3]
			\arrow[dotted, from=1-5, to=2-6]
			\arrow[hook, from=1-5, to=3-5]
			\arrow[hook, from=2-6, to=4-6]
			\arrow[hook, "s"{pos=0.25}, from=3-1, to=3-3]
			\arrow[dotted, hook, from=3-1, to=4-2]
			\arrow[two heads, from=3-3, to=3-5]
			\arrow[dotted, hook, from=3-3, to=4-4]
			\arrow[dotted, hook, from=3-5, to=4-6]
			\arrow[hook, "1 \otimes s", from=4-2, to=4-4]
			\arrow[two heads, from=4-4, to=4-6]
			\arrow[hook, crossing over, from=2-2, to=2-4]
			\arrow[equal, crossing over, from=2-2, to=4-2]
			\arrow[two heads, crossing over, from=2-4, to=2-6]
			\arrow[hook, "1 \otimes s"{pos=0.2}, crossing over, from=2-4, to=4-4]
		\end{tikzcd}
	\end{equation}
	where the first row of the top (resp.\ bottom) layer is the second row of the top (resp.\ bottom) layer of \eqref{eq:filgr_modmu_diagram}, the map $1 \otimes s$ (resp.\ $s$) is as described after diagram \eqref{eq:filgr_modmu_diagram}, and the $L\textrm{-vector space}$ $Q_L$ (resp.\ $R[1/p]\module$ $Q_R$) denotes the cokernel of the left horizontal arrow in the second row of the top (resp.\ bottom) layer of \eqref{eq:filgr_modmu_diagram_s}.
	The right vertical arrow in the top (resp.\ bottom) layer is naturally induced by the middle vertical arrow $1 \otimes s$ (resp.\ $s$), and its injectivity follows from an easy application of the snake lemma.
	The dotted arrows from the bottom to top layer are $R[1/p]\linear$ and were described in \eqref{eq:filgr_modmu_diagram} (except for the lower right corner).
	The $R[1/p]\linear$ map $Q_R \rightarrow Q_L$ is induced by the left commutative square involving the respective second rows of the top and bottom layers, and its injectivity follows from diagram \eqref{eq:filgr_modmu_diagram}: indeed, from the injectivity of the dotted arrow from the top right corner of the bottom layer to the top right corner of the top layer of \eqref{eq:filgr_modmu_diagram} note that we have
	\begin{equation*}
		Q_R \lisomorphic \gr^k M_R \hookrightarrow \gr^k M_L \isomorphic Q_L,
	\end{equation*}
	where the left (resp.\ right) isomorphism is $R[1/p]\linear$ (resp.\ $L\linear$), and the composition coincides with map $Q_R \rightarrow Q_L$.

	Note that by the definition of arrows in diagram \eqref{eq:filgr_modmu_diagram_s}, it is clear that each square in the top (resp.\ bottom) layer is commutative, and each square involving solid and dotted arrows (except the rightmost square) is commutative.
	But, the commutativity of the rest of the diagram implies that the rightmost square in \eqref{eq:filgr_modmu_diagram_s}, i.e.\ the following diagram, is also commutative:
	\begin{equation*}
		\begin{tikzcd}
			\gr^k(M_R/\mu) & \gr^k(M_L/\mu) \\
			Q_R & Q_L.
			\arrow[from=1-1, to=1-2]
			\arrow[hook, from=1-1, to=2-1]
			\arrow[hook, from=1-2, to=2-2]
			\arrow[hook, from=2-1, to=2-2]
		\end{tikzcd}
	\end{equation*}
	From the preceding diagram, we see that the top arrow is injective, and so the dotted arrow between the top right corner of the two layers in \eqref{eq:filgr_modmu_diagram_s} is injective.
	Therefore, we obtain that the dotted arrow in the lower right corner of the two layers in \eqref{eq:filgr_modmu_diagram} is injective, and thus the map in \eqref{eq:rational_filmodmu_induced} is bijective for $k+1$, proving the second claim.
	Finally, using induction, we conclude that \eqref{eq:fil_compatibility} holds for all $k \in \NN$.
\end{proof}

\begin{rem}\label{rem:fil_compatibility_ml_mlbreve}
	Let $N_L$ be a Wach module over $\AL^+$ and let $T$ denote the associated crystalline $\ZZ_p\textrm{-representation}$ of $G_L$ from \cite[Theorem 1.6]{abhinandan-imperfect-wach}.
	In \cite[Theorem 1.8 \& Corollary 3.16]{abhinandan-imperfect-wach}, we have shown that the natural isomorphism $(N_L/\mu)[1/p] \isomorphic \ODcrysL(V)$ is compatible with the respective filtrations, where the left-hand term is equipped with a filtration as described in Section \ref{subsubsec:nygaardfil_modmu} and the right-hand term is equipped with the natural Hodge filtration.
	We claim that this compatibility between filtrations may also be obtained by using the analogous result in the perfect residue field case from \cite[Th\'eor\`eme III.4.4]{berger-limites}.
	Indeed, consider the extension $\Lbreve/L$ with perfect residue field from Remark \ref{rem:nygaard_fil_nl_nlbreve}.
	Then, $\NLbreve \coloneq \ALbreve^+ \otimes_{\AL^+} N_L$ is a Wach module over $\ALbreve^+$ and $T$ is a $\ZZ_p\textrm{-representation}$ of $G_{\Lbreve}$.
	Set $M_L \coloneq N_L[1/p]$, $\MLbreve \coloneq \NLbreve[1/p]$ and $V \coloneq T[1/p]$ and consider the following diagram:
	\begin{equation}\label{eq:fil_compatibility_lbreve}
		\begin{tikzcd}[row sep=15pt]
			\ODcrysL(V) & \Lbreve \otimes_L \ODcrysL(V) & \DcrysLbreve(V) \\
			(M_L/\mu) & \Lbreve \otimes_L (M_L/\mu) & \MLbreve/\mu,
			\arrow[hook, from=1-1, to=1-2]
			\arrow["\wr"', from=1-1, to=2-1]
			\arrow["\sim", from=1-2, to=1-3]
			\arrow["\wr"', from=1-2, to=2-2]
			\arrow["\wr", from=1-3, to=2-3]
			\arrow[hook, from=2-1, to=2-2]
			\arrow["\sim", from=2-2, to=2-3]
		\end{tikzcd}
	\end{equation}
	where the top (resp.\ bottom) left horizontal arrow is the natural inclusion; the top right horizontal arrow is the natural isomorphism of filtered $\varphi\modules$ over $\Lbreve$ (where the source is equipped with the $\Lbreve\linear$ extension of the natural Hodge filtration on $\ODcrysL(V)$ and the target is equipped with the natural Hodge filtration, see \cite[Equation (2.5)]{abhinandan-imperfect-wach}); the bottom right horizontal arrow is the natural isomorphism; the right vertical arrow is the natural isomorphism of filtered $\varphi\modules$ over $\Lbreve$ from \cite[Th\'eor\`eme III.4.4]{berger-limites} (where the source is equipped with the natural Hodge filtration and the target is equipped with a filtration as in Remark \ref{rem:fil_gr_nlbreve_modmu}); the left vertical arrow is induced from \cite[Equation (4.7)]{abhinandan-imperfect-wach}; the middle vertical arrow is the extension along $L \rightarrow \Lbreve$ of the left vertical arrow and it coincides with the right vertical arrow of \eqref{eq:fil_compatibility_l} by \cite[Lemma 4.8, Lemma 4.11 (3), Equations (4.6), (4.7), (4.15) \& (4.19) and Corollary 4.27]{abhinandan-imperfect-wach}.
	The left square commutes by definition and the right square commutes by the compatibilty between the constructions of \cite{abhinandan-imperfect-wach} and \cite{berger-limites}.

	Now, note that similar to Proposition \ref{prop:fil_compatibility}, it is enough to show the claim for $N_L$ effective, so from now onwards we will work under this assumption.
	Then, to obtain the claim, it is enough to show that the left vertical arrow of \eqref{eq:fil_compatibility_lbreve} induces a map $\Fil^k \ODcrysL(V) \rightarrow \Fil^k (M_L/\mu)$, and that the induced map is bijective.
	The case $k=0$ is obvious and we shall proceed by induction on $k \geqslant 1$, i.e.\ we shall assume that the left vertical arrow of \eqref{eq:fil_compatibility_lbreve} induces an isomorphism $\Fil^k \ODcrysL(V) \isomorphic \Fil^k (M_L/\mu)$, for some $k$, and we shall show that the claim holds for $k+1$.
	To that end, let us observe that we have an injective natural $L\linear$ homomorphism $\gr^k (M_L/\mu) \rightarrow \Lbreve \otimes_L \gr^k (M_L/\mu) \isomorphic \gr^k (\MLbreve/\mu)$, where the first homomorphism is injective by definition and the isomorphism of $\Lbreve\textrm{-vector}$ spaces follows from Remark \ref{rem:nygaard_fil_nl_nlbreve_modmu} (after inverting $p$).
	So, it follows that we have the following isomorphism of $L\textrm{-vector}$ spaces:
	\begin{equation*}
		\Fil^{k+1} (M_L/\mu) \isomorphic \Fil^k (M_L/\mu) \cap \Fil^{k+1} (\MLbreve/\mu) \subset \MLbreve/\mu.
	\end{equation*}
	Thus, from the preceding isomorphism, the filtered isomorphisms from \eqref{eq:fil_compatibility_lbreve} (see the top right horizontal arrow and the right vertical arrow), and the induction assumption, it follows that
	\begin{align*}
		\Fil^{k+1} \ODcrysL(V) &\isomorphic \Fil^k \ODcrysL(V) \cap \Fil^{k+1} \DcrysLbreve(V)\\
			&\qquad\isomorphic \Fil^k (M_L/\mu) \cap \Fil^{k+1} (\MLbreve/\mu) \lisomorphic \Fil^{k+1} (M_L/\mu),
	\end{align*}
	where the intersection in second term is taken inside $\ODcrysL(V)$ via the top right filtered isomorphism in \eqref{eq:fil_compatibility_lbreve} and the first isomorphism follows because $\gr^k \ODcrysL(V) \hookrightarrow \gr^k \DcrysLbreve(V)$ using the top horizontal arrows of \eqref{eq:fil_compatibility_lbreve}).
	Hence, we conclude that the left vertical arrow of \eqref{eq:fil_compatibility_lbreve} is a filtered isomorphism.
\end{rem}

Let us note an interesting consequence of Theorem \ref{thm:qdeformation_dcrys}, where we keep the same notations as in the theorem and Proposition \ref{prop:fil_compatibility}:
\begin{cor}\label{cor:fil_nygaard_rational_proj}
	For each $k \in \ZZ$, the $\BR^+\module$ $\Fil^k M$ is finite projective and the $R[1/p]\module$ $\gr^k M$ is also finite projective.
	In particular, there exists (non-canonical) isomorphisms of $R[1/p]\textrm{-modules}$:
	\begin{equation*}
		\gr^k M \isomorphic \oplus_{i \leqslant k} \gr^i(M/\mu) \isomorphic \oplus_{i \leqslant k} \gr^i(\ODcrysR(V)),
	\end{equation*}
	and thus we have that 
	\begin{equation*}
		\rank_{R[1/p]}(\gr^k M) = \textstyle\sum_{i \leqslant k} \rank_{R[1/p]}(\gr^i(M/\mu)) = \textstyle\sum_{i \leqslant k} \rank_{R[1/p]}(\gr^i(\ODcrysR(V))).
	\end{equation*}
\end{cor}
\begin{proof}
	For $r \in \NN$ large enough, note that the Wach module $\mu^r N (-r)$ is always effective (the twist $(-r)$ denotes a Tate twist on which $\Gamma_R$ acts via $\chi^{-r}$, where $\chi$ is the $\padic$ cyclotomic character) and from Lemma \ref{lem:nygaard_fil_twist} (1), it is enough to show the claim for effective Wach modules.
	So, let us assume that $N$ is effective.
	Now, using the filtered isomorphism in \eqref{eq:qdeformation_dcrys} (also see \eqref{eq:fil_compatibility}) and taking the associated graded pieces, we obtain $R[1/p]\linear$ isomorphisms $\gr^k(M_R/\mu) \isomorphic \gr^k \ODcrysR(V)$.
	In particular, from \cite[Proposition 8.3.2]{brinon-relatif}, it follows that $\gr^k(M_R/\mu)$ is a finite projective $R[1/p]\module$, for each $k \in \ZZ$.
	As we assumed that $N_R$ is effective, therefore, by using \eqref{eq:grnmodmu} for $k = 0$ (after inverting $p$), we see that $\gr^0 M_R \isomorphic \gr^0 (M_R/\mu) \isomorphic \gr^0 \ODcrysR(V)$, which implies that $\gr^0 M_R$ is a finite projective $R[1/p]\module$.
	Moreover, using the exact sequence in \eqref{eq:grnmodmu} (after inverting $p$) and an easy induction on $k \geqslant 0$, yields that $\gr^k M_R$ is a finite projective $R[1/p]\module$.

	Next, note that we have $\textrm{Tor}_i^{\BR^+}(R[1/p], P) = 0$, for $i \geqslant 2$ and any $\BR^+\module$ $P$, because $\BR^+/\mu \isomorphic R[1/p]$.
	Since $\gr^k M_R$ is finite projective over $R[1/p]$, therefore, there exists a direct sum presentation $(\gr^k M_R) \oplus M' = R[1/p]^{\oplus n}$, for some finite projective $R[1/p]\module$ $M'$ and $n \in \NN$.
	Recall that tensor products commute with direct sums, so for any $\BR^+\module$ $P$ we have that
	\begin{equation*}
		\textrm{Tor}_i^{\BR^+}(\gr^k M_R, P) \oplus \textrm{Tor}_i^{\BR^+}(M', P) = \textrm{Tor}_i^{\BR^+}(R[1/p]^{\oplus n}, P),
	\end{equation*}
	where the right hand term vanishes for $i \geqslant 2$.
	So, it follows that $\textrm{Tor}_i^{\BR^+}(\gr^k M_R, P) = 0$, for $i \geqslant 2$.
	Now, observe that the following sequence of $\BR^+\modules$ is exact:
	\begin{equation*}
		0 \longrightarrow \Fil^{k+1} M_R \longrightarrow \Fil^k M_R \longrightarrow \gr^i M_R \longrightarrow 0.
	\end{equation*}	
	We shall show that $\Fil^{k+1} M_R$ is finite projective over $\BR^+$ by induction on $k$.
	Indeed, note that $\Fil^0 M_R = M_R$ is a finite projective $\BR^+\module$ and assume that $\Fil^k M_R$ is finite projective over $\BR^+$, for some $k \geqslant 0$.
	Then, from the discussion above and the long exact sequence in Tor, we get that $\textrm{Tor}_i^{\BR^+}(\Fil^{k+1} M_R, P) = 0$, for $i \geqslant 1$ and any $\BR^+\module$ $P$, i.e.\ $\Fil^{k+1} M_R$ is a flat $\BR^+\module$, hence, finite projective as $\Fil^{k+1} M_R$ is a finitely generated module over the noetherian ring $\BR^+$.

	Finally, from the discussion above, we note that after inverting $p$ in the exact sequence in \eqref{eq:grnmodmu}, it splits non-canonically and $R[1/p]\textrm{-linearly}$.
	Thus, it follows that we have non-canonical isomorphisms of $R[1/p]\textrm{-modules}$:
	\begin{equation*}
		\gr^k M_R \isomorphic \oplus_{i \leqslant k} \gr^k(M_R/\mu).
	\end{equation*}
	Consequently, we also get that $\rank_{R[1/p]}(\gr^k M_R) = \sum_{i \leqslant k} \rank_{R[1/p]}(\gr^i(M_R/\mu))$.
	This allows us to conclude.
\end{proof}

\begin{rem}\label{rem:qdeformation_dcrys_imperfect}
	The obvious variation of Theorem \ref{thm:qdeformation_dcrys} and Corollary \ref{cor:fil_nygaard_rational_proj} also hold true in the imperfect residue field case.
	Indeed, to show Theorem \ref{thm:qdeformation_dcrys} for $O_L$, note that all compatibilities except for the connection were already proven in \cite[Corollary 3.15]{abhinandan-imperfect-wach} (see Remark \ref{rem:fil_compatibility_ml_mlbreve} for another proof of compatibility between filtrations).
	To verify the compatibility of connections, similar to Proposition \ref{defi:wachmod_qconnection}, we can define a $\qconnection$ over a Wach module over $\AL^+$.
	Then, using the results of \cite[Section 3.3]{abhinandan-imperfect-wach}, one obtains an obvious variation of Remark \ref{rem:qconnection_arpipd} over $\ALpi^{\PD}$.
	Proceeding exactly as in the proof of Theorem \ref{thm:qdeformation_dcrys} (after replacing each object by the analogous object for $L$), we obtain the desired isomorphism of filtered $(\varphi, \partial)\modules$ over $L$.
	Finally, the claim analogous to Corollary \ref{cor:fil_nygaard_rational_proj} easily follows from \eqref{eq:grnmodmu} (for Wach modules over $\AL^+$).
\end{rem}

Let us summarise the relationship between various categories considered in \eqref{eq:odcrysr_func}, Corollary \ref{cor:crystalline_wach_rat_equivalence_relative} and Theorem \ref{thm:qdeformation_dcrys}.
Recall that $\Rep_{\QQ_p}^{\crys}(G_R)$ is the category of $\padic$ crystalline representations of $G_R$ and $\MF_R\ad(\varphi, \partial)$ denotes the essential image of the functor $\ODcrysR$ restricted to $\Rep_{\QQ_p}^{\crys}(G_R)$.
\begin{cor}\label{cor:cat_equiv_diagram}
	Functors in the following diagram induce exact equivalence of $\otimes\textrm{-categories}$:
	\begin{center}
		\begin{tikzcd}[row sep=36pt]
			\Rep_{\QQ_p}^{\crys}(G_R) \arrow[rr, "\NR", shift left=1mm] \arrow[rd, "\ODcrysR", shift left=1mm] & & (\varphi, \Gamma_R)\Mod_{\BR^+}^{[p]_q} \arrow[ll, "\VR", shift left=1mm] \arrow[ld, "q \mapsto 1"]\\
			& \MF_R\ad(\varphi, \partial) \arrow[ul, "\OVcrysR", shift left=1mm].
		\end{tikzcd}
	\end{center}
\end{cor}
\begin{proof}
	The exact equivalence induced by functors $\NR$ and $\VR$ follows from Corollary \ref{cor:crystalline_wach_rat_equivalence_relative} and the exact equivalence induced by $\ODcrysR$ and $\OVcrysR$ follows from \cite[Th\'eor\`eme 8.5.1]{brinon-relatif}.
	Moreover, from Theorem \ref{thm:qdeformation_dcrys}, note that for a Wach module $M$ over $\BR^+$ we have that $M/(q-1) = M/\mu \isomorphic \ODcrysR(\VR(M))$.
	Hence, from the preceding exact equivalence of $\otimes\textrm{-categories}$, it follows that the slanted arrow labelled ``$q \mapsto 1$'' is also an exact equivalence of $\otimes\textrm{-categories}$.
\end{proof}

\appendix
\section{Some technical lemmas}

\subsection{Commutative algebra}

We begin with an easy observation.
Let $A$ be a commutative ring and $a, b \in A$ nonzerodivisors.
\begin{lem}\label{lem:injectivity_modulo}
	Let $M \hookrightarrow N$ be an injective homomorphism of $a\textrm{-torsion}$ free $A\modules$.
	Then, we have that $M = M[1/a] \cap N \subset N[1/a]$ if and only if $aM = M \cap aN \subset N$ if and only if $M/a \hookrightarrow N/a$.
\end{lem}
\begin{proof}
	The first equivalence follows easily.
	For the second equivalence, consider the following diagram with exact rows and injective vertical arrows in the left and the middle:
	\begin{center}
		\begin{tikzcd}[row sep=14pt]
			0 \arrow[r] & M \arrow[r, "a"] \arrow[d] & M \arrow[r] \arrow[d] & M/a \arrow[r] \arrow[d] & 0\\
			0 \arrow[r] & N \arrow[r, "a"] & N \arrow[r] & N/a \arrow[r]& 0.
		\end{tikzcd}
	\end{center}
	An application of the snake lemma shows that the right vertical arrow is injective if and only if $aM = M \cap aN \subset N$.
\end{proof}

\begin{lem}\label{lem:ab_torsion}
	Assume that $M$ is $a\torsion$ free and $b\torsion$ free.
	Then, we have that $(M/a)[b] = (M/b)[a]$.
\end{lem}
\begin{proof}
	Multiplication by $b$ on the exact sequence $0 \rightarrow M \xrightarrow{\hspace{0.5mm}a\hspace{0.5mm}} M \rightarrow M/a \rightarrow 0$ and an application of the snake lemma yields the claim.
\end{proof}

\begin{defi}\label{defi:strict_regular_seq}
	The sequence $\{a, b\}$ in $A$ is said to be \textit{$M\regular$} if $M$ is $a\textrm{-torsion}$ free and $M/aM$ is $b\textrm{-torsion}$ free.
	The sequence $\{a, b\}$ in $A$ is said to be \textit{strictly $M\regular$} if both $\{a, b\}$ and $\{b, a\}$ are $M\regular$ sequences.
\end{defi}

\begin{rem}\label{rem:compsupp_coh}
	Assume that $A$ is noetherian.
	Let $i \geqslant 1$ be an integer and consider the following commutative diagram of complexes:
	\begin{equation}\label{eq:koszulcomp}
		\begin{tikzcd}[row sep=14pt, column sep=50pt]
			\calc^{\bullet}_i \colon \hspace{12pt} M \arrow[r, "(a^i{,} b^i)"] \arrow[d, shift left=16pt, equal] & M \oplus M \arrow[r, "(b^i{,} -a^i)"] \arrow[d, "(a{,} b)"] & M \arrow[d, "ab"]\\
			\hspace{-6pt}\calc^{\bullet}_{i+1} \colon \hspace{8pt} M \arrow[r, "(a^{i+1}{,} b^{i+1})"] & M \oplus M \arrow[r, "(b^{i+1}{,} -a^{i+1})"] & M,
		\end{tikzcd}
	\end{equation}
	where in the top row, the first map is given as $x \mapsto (a^ix, b^ix)$ and the second map is given as $(x, y) \mapsto b^ix - a^iy$, and similarly for the bottom row.
	Moreover, the middle vertical arrow in \eqref{eq:koszulcomp} is given as multiplication by $a$ and $b$ on the respective copy of $M$ and the right vertical arrow is given as multiplication by $ab$ on $M$.
	Then, the collection of complexes $\{\calc^{\bullet}_i\}_{i \geqslant 1}$ with morphisms $\calc^{\bullet}_i \rightarrow \calc^{\bullet}_{i+1}$ as above, form a direct system of complex of $A\textrm{-modules}$.
	By setting $I = (a, b)$ and $Z := V(I) \subset \Spec(A)$ as a closed subset, from \cite[\href{https://stacks.math.columbia.edu/tag/0956}{Tag 0956}, \href{https://stacks.math.columbia.edu/tag/0G6H}{Tag 0G6H}, \href{https://stacks.math.columbia.edu/tag/0913}{Tag 0913}]{stacks-project} and \cite[Theorem 4.6.8]{weibel}, we have that
	\begin{equation*}
		H^n_I(M) \isomorphic H^n_Z(M) \isomorphic H^n(\colim_i \calc^{\bullet}_i) \isomorphic \colim_i H^n(\calc^{\bullet}_i),
	\end{equation*}
	where the leftmost term denotes the local cohomology of $M$ with respect to $I$, the second term denotes the cohomology of $M$ with support in $Z$, the third term denotes the cohomology of the colimit complex and the rightmost term denotes the colimit of the cohomology of the complex $\calc^{\bullet}_i$, where the transition maps $H^n(\calc^{\bullet}_i) \rightarrow H^n(\calc^{\bullet}_{i+1})$ are induced from the transition maps $\calc^{\bullet}_i \rightarrow \calc^{\bullet}_{i+1}$.
\end{rem}

Assume that $A$ is noetherian, $(a,b)\textrm{-adically}$ complete and the sequence $\{a, b\}$ is strictly $A\regular$.
\begin{lem}\label{lem:strictreg_koszulcomp}
	Let $M$ be a finitely generated $A\module$, and consider the following complex:
	\begin{equation*}
		\calc^{\bullet} \colon M \xrightarrow{(a, b)} M \oplus M \xrightarrow{(b, -a)} M,
	\end{equation*}
	where the first map is given as $x \mapsto (ax, bx)$ and the second map is given as $(x, y) \mapsto bx - ay$.
	Then, the sequence $\{a, b\}$ is strictly $M\regular$ if and only if $H^1(\calc^{\bullet}) = 0$.
	Moreover, under these equivalent conditions we have that $H^0(\calc^{\bullet}) = 0$.
\end{lem}
\begin{proof}
	Let us first note that if $\{a, b\}$ is strictly $M\regular$, then $(M/a)[b] = (M/b)[a] = 0$.
	Therefore, we must have $H^0(\calc^{\bullet}) = H^1(\calc^{\bullet}) = 0$.
	For the converse, consider the following commutative diagram:
	\begin{equation}\label{eq:mmodab}
		\begin{tikzcd}[row sep=14pt]
			M[a, b] \arrow[r, hookrightarrow] \arrow[d, hookrightarrow] & M[a] \arrow[r, "b"] \arrow[d, hookrightarrow] & M[a] \arrow[r] \arrow[d, hookrightarrow] & (M/b)[a] \arrow[d, hookrightarrow]\\
			M[b] \arrow[r, hookrightarrow] \arrow[d, "a"] & M \arrow[r, "b"] \arrow[d, "a"] & M \arrow[r, twoheadrightarrow] \arrow[d, "a"] & M/b \arrow[d, "a"]\\
			M[b] \arrow[r, hookrightarrow] \arrow[d] & M \arrow[r, "b"] \arrow[d, twoheadrightarrow] & M \arrow[r, twoheadrightarrow] \arrow[d, twoheadrightarrow] & M/b \arrow[d, twoheadrightarrow]\\
			(M/a)[b] \arrow[r, hookrightarrow] & M/a \arrow[r, "b"] & M/a \arrow[r, twoheadrightarrow] & M/(a, b),
		\end{tikzcd}
	\end{equation}
	where the second, third and fourth rows (resp.\ columns) are exact.
	Using that $H^1(\calc^{\bullet}) = 0$, let us first show that the top right and the bottom left corners of \eqref{eq:mmodab} are zero, i.e.\ $(M/b)[a] = (M/a)[b] = 0$.
	Indeed, let $x$ be an element $M/a$ such that $bx = 0$.
	We take $y$ in $M$ to be a lift of $x$ such that $by = az$, for some $z$ in $M$.
	Then, we see that $(y, z)$ represents a class in $H^1(\calc^{\bullet}) = 0$, and thus there exists some $w$ in $M$ such that $y = aw$ and $z = bw$.
	But then we have that $x = y\textrm{ mod } a = 0$, hence, $(M/a)[b] = 0$.
	A similar argument starting with an $a\textrm{-torsion}$ element of $M/b$ shows that $(M/b)[a] = 0$.

	Next, from the first row (resp.\ column) of the diagram \eqref{eq:mmodab}, it follows that we have an injective homomorphism $M[a]/b \hookrightarrow (M/b)[a] = 0$ (resp.\ $M[b]/a \hookrightarrow (M/a)[b] = 0$).
	In particular, the topmost horizontal arrow from the second column to the third column (resp.\ the leftmost vertical arrow from the second row to the third row) is surjective.
	Now, let $x$ be an element of $M[a]$, then from the preceding discussion there exists $y$ in $M[a]$ such that $x = by$.
	Proceeding by induction on $n \geqslant 1$, it is easy to see that $x$ is an element of $b^n M[a] \subset b^n M$, for all $n \in \NN$.
	But, as $M$ is finitely generated over the $(a, b)\textrm{-adically}$ complete noetherian ring $A$, it is $(a, b)\textrm{-adically}$ complete, in particular, $M$ is $a\textrm{-adically}$ separated, therefore, we must have $x = 0$ and thus $M[a] = 0$.
	A similar argument shows that $M[b] = 0$.
	Hence, this proves the converse claim and we get that $H^0(\calc^{\bullet}) = M[a, b] = 0$.
\end{proof}

\begin{rem}\label{rem:strictreg_localcoh}
	From Remark \ref{rem:compsupp_coh}, we have that $H^n_I(M) \isomorphic \colim_i H^n(\calc^{\bullet}_i)$, for $I = (a, b)$.
	Therefore, in the setting of Lemma \ref{lem:strictreg_koszulcomp} and using that the sequence $\{a, b\}$ is strictly $M\regular$ if and only if the sequence $\{a^i, b^i\}$ is strictly $M\regular$ for any $i \geqslant 1$ (see \cite[\href{https://stacks.math.columbia.edu/tag/07DV}{Tag 07DV}]{stacks-project}), we conclude that the sequence $\{a, b\}$ is strictly $M\regular$ if and only if $H^1_I(M) = H^1_Z(M) = 0$.
\end{rem}

Next, we note a simple fact from \cite[Lemma IV.3.2.2]{abbes-gros-tsuji}.
Let $A$ be a $p\textrm{-adically}$ complete, i.e.\ $A \isomorphic \lim_n A/p^nA$, and $p\torsion$ free commutative ring.
Note that $A$ need not be noetherian.
\begin{lem}\label{lem:pcomplete_finprojrational}
	Let $M$ be a finitely generated $A\module$ such that $M[1/p]$ is finite projective over $A[1/p]$.
	Let $M[p^{\infty}] \subset M$ denote its $p^{\infty}\torsion$ submodule.
	Then, the following hold:
	\begin{enumerate}
		\item[(1)] There exists an integer $N \gg 0$ such that $M[p^{\infty}] = M[p^N]$.
		\item[(2)] The $A\module$ $M$ is $p\textrm{-adically}$ complete.
	\end{enumerate}
\end{lem}
\begin{proof}
	For completeness, we recall the proof from \cite[Lemma IV.3.2.2]{abbes-gros-tsuji}.
	To prove (1), let us choose a surjective homomorphism of $A\textrm{-modules}$ $f \colon A^{\oplus r} \twoheadrightarrow M$, for some $r \in \NN$.
	As $M[1/p]$ is a finite projective $A[1/p]\module$, therefore, there exists a decomposition $A[1/p]^{\oplus r} = M[1/p] \oplus (\kert f)[1/p]$ as $A[1/p]\textrm{-modules}$.
	Since $M[1/p]$ and $(\kert f)[1/p]$ are finitely generated $A[1/p]\textrm{-modules}$, there exist some $A\textrm{-submodules}$ $M_1 \subset M[1/p]$ and $M_2 \subset (\kert f)[1/p]$, such that $M_1 \oplus M_2 \subset A^{\oplus r}$ and $M_1[1/p] = M[1/p]$ and $M_2[1/p] = (\kert f)[1/p]$.
	So, it follows that there exist $N_1, N_2 \in \NN$ such that $p^{N_1} A^{\oplus r} \subset M_1 \oplus M_2$ and $p^{N_2} f(M_2) = 0$.
	Then, for any $x$ in $A^{\oplus r}$ such that $(f \otimes \QQ_p)(x) = 0$, we have that $p^{N_1} x$ is in $M_2$ and $p^{N_1+N_2}f(x)$ is in $p^{N_2}f(M_2) = 0$.
	This proves (1).

	For (2), set $\overline{M} := M/(M[p^{\infty}]) \subset M[1/p]$.
	Then, $\overline{M}$ is a finitely generated $A\module$ such that $\overline{M}[1/p] = M[1/p]$ is finite projective over $A[1/p]$.
	So, from (1), it follows that the following sequence is exact:
	\begin{equation*}
		0 \longrightarrow M[p^{\infty}] \longrightarrow \lim_n M/p^n M \longrightarrow \lim_n \overline{M}/p^n \overline{M} \longrightarrow 0.
	\end{equation*}
	In particular, we see that it is enough to show the claim for $p\torsion$ free $M$.
	In this case, there exists some finitely generated $p\torsion$ free $A\module$ $M'$ with $M'[1/p]$ finite projective over $A[1/p]$, and an isomorphism $M[1/p] \oplus M'[1/p] \isomorphic A^{\oplus r}$.
	By choosing an integer $N \gg 0$ such that $p^N A^{\oplus r} \hookrightarrow M \oplus M' \hookrightarrow p^{-N} A^{\oplus r}$, we see that $M$ is $p\textrm{-adically}$ complete.
	This completes the proof of (2).
\end{proof}

\subsection{Structure of \texorpdfstring{$\varphi$}{-}-modules}

We will use setup and notations from Section \ref{subsec:setup_nota} and the rings defined in Section \ref{subsec:ainf_relative}.
Let $q$ be an indeterminate and recall that we have a Frobenius-equivariant isomorphism $R\llbracket q-1 \rrbracket \isomorphic \AR^+$, via the map $X_i \mapsto [X_i^{\flat}]$ and $q \mapsto 1+\mu$.
We have the following structural result:
\begin{prop}\label{prop:finiteproj_torus}
	Let $N$ be a finitely generated $\AR^+\module$ and suppose that $N$ is equipped with a Frobenius-semilinear endomorphism $\varphi \colon N \rightarrow N$ such that $1 \otimes \varphi \colon \varphi^*(N)[1/[p]_q] \isomorphic N[1/[p_q]]$.
	Then, $N[1/p]$ is finite projective over $\BR^+$.
\end{prop}
\begin{proof}
	The proof is essentially the same as \cite[Proposition 4.13]{du-liu-moon-shimizu}.
	Compared to loc.\ cit., the Frobenius endomorphism on $\AR^+$ and finite height assumption on $N$ are different and we do not assume $N$ to be torsion free.
	However, one observes that torsion freeness of $N$ is not used in the proof of loc.\ cit.\ and one may use \cite[Lemma 2.14]{abhinandan-imperfect-wach} and Lemma \ref{lem:finiteproj_regularlocal} instead of \cite[Proposition 4.3]{bhatt-morrow-scholze-1} and \cite[Lemma 4.12]{du-liu-moon-shimizu}.
\end{proof}

\begin{lem}\label{lem:finiteproj_regularlocal}
	Let $k$ be a perfect field of characteristic $p$ and $S := W(k)\llbracket u_1, \ldots, u_m \rrbracket$ equipped with a Frobenius endomorphism $\varphi$ extending the Witt vector Frobenius on $W(k)$ such that $\varphi(u_i) \in S$ has zero constant term for each $1 \leqslant i \leqslant m$.
	Let $A := S \llbracket q-1 \rrbracket$ equipped with a Frobenius endomorphism extending the one on $S$ by $\varphi(q) = q^p$ and let $N$ be a finitely generated $A\module$ equipped with a Frobenius-semilinear endomorphism $\varphi \colon N \rightarrow N$ such that $1 \otimes \varphi \colon \varphi^*(N)[1/[p]_q] \isomorphic N[1/[p]_q]$.
	Then $N[1/p]$ is finite projective over $A[1/p]$.
\end{lem}
\begin{proof}
	The proof is essentially the same as \cite[Lemma 4.12]{du-liu-moon-shimizu}, except for a few changes.
	One proceeds by induction on $m$.
	The case $m=0$ follows from \cite[Lemma 2.14]{abhinandan-imperfect-wach}, so let $m \geqslant 1$.
	Take $J$ to be the smallest non-zero Fitting ideal of $N$ over $A$.
	It suffices to show that $JA[1/p] = A[1/p]$.
	Compatibility of Fitting ideals under base change implies that $JA[1/[p]_q] = \varphi(J)A[1/[p]_q]$ as ideals of $A[1/[p]_q]$, therefore, $(A/J)[1/[p]_q] = (A/\varphi(J))[1/[p]_q]$.
	Let us assume that $JA[1/p] \neq A[1/p]$ and we will show a contradiction.

	In our setting, the Frobenius endomorphism on $A$ and the finite height condition are different from \cite[Lemma 4.12]{du-liu-moon-shimizu}.
	Therefore, we need some modifications in the arguments of loc. cit.; let us point out the differences in terms of their notations. 
	Let $K = W(k)[1/p]$, fix $\Kbar$ as an algebraic closure of $K$.
	Consider the $\Kbar\textrm{-valued}$ points of $\Spec(A[1/p]/J)$ and let $Z = \{(|u_1|, \ldots, |u_m|, |q-1|) \in \RR^{m+1}\}$ be the corresponding set of $(m+1)\textrm{-tuple}$ norms.
	Define the set $Z' = \{(|u_1|, \ldots, |u_m|, |q-1|) \in \RR^{m+1} \textrm{ such that } (|\varphi(u_1), \ldots |\varphi(u_m)|, |q^p-1|) \in Z\}$ and take $\zeta_p-1$ as the chosen uniformiser.
	Then, one proceeds as in loc.\ cit.\ to show that $JA[1/p] \subset (u_1, \ldots, u_m, q-1)A[1/p]$ and $JA[1/p] \not\subset IA[1/p]$, where $I = (u_1, \ldots, u_m) \subset A[1/p]$.

	Finally, consider the Frobenius-equivariant projection $A \rightarrow \Abar = A/I = W(k)\llbracket q-1 \rrbracket$ and let $\Jbar \subset \Abar$ denote the image of $J$.
	Since $JA[1/p] \not\subset IA[1/p]$, we get that $\Jbar \neq 0$.
	Moreover, $\Jbar \Abar [1/p] \neq \Abar[1/p]$ since $JA[1/p] \subset (u_1, \ldots, u_m, q-1)A[1/p]$.
	However, the equality $(A/J)[1/[p]_q] = (A/\varphi(J))[1/[p]_q]$ implies that $(\Abar/\Jbar)[1/[p]_q] = (\Abar/\varphi(\Jbar))[1/[p]_q]$, i.e.\ $\Jbar \Abar[1/p] = \Abar[1/p]$ by inductive hypothesis (see \cite[Lemma 2.14]{abhinandan-imperfect-wach}).
	This gives a contradiction.
	Hence, we must have $JA[1/p] = A[1/p]$, thus proving the lemma.
\end{proof}


\setstretch{0.95}
\phantomsection
\printbibliography[heading=bibintoc, title={References}]

@book {abbes-gros-tsuji,
    AUTHOR = {Abbes, Ahmed and Gros, Michel and Tsuji, Takeshi},
     TITLE = {The {$p$}-adic {S}impson correspondence},
    SERIES = {Annals of Mathematics Studies},
    VOLUME = {193},
 PUBLISHER = {Princeton University Press, Princeton, NJ},
      YEAR = {2016},
     PAGES = {xi+603},
      ISBN = {978-0-691-17029-9; 978-0-691-17028-2},
   MRCLASS = {14G22 (14G40)},
  MRNUMBER = {3444777},
MRREVIEWER = {Marco\ A.\ Garuti},
       DOI = {10.1515/9781400881239},
       URL = {https://doi.org/10.1515/9781400881239},
}

@article{abhinandan-relative-wach-i,
    AUTHOR = {Abhinandan},
     TITLE = {Crystalline representations and {W}ach modules in the relative
              case},
   JOURNAL = {Ann. Inst. Fourier (Grenoble)},
  FJOURNAL = {Universit\'e{} de Grenoble. Annales de l'Institut Fourier},
    VOLUME = {75},
      YEAR = {2025},
    NUMBER = {1},
     PAGES = {379--474},
      ISSN = {0373-0956,1777-5310},
   MRCLASS = {11F80 (11S25 14F30)},
  MRNUMBER = {4862349},
       DOI = {10.5802/aif.3670},
       URL = {https://doi.org/10.5802/aif.3670},
}

@article{abhinandan-syntomic,
    AUTHOR = {Abhinandan},
     TITLE = {Syntomic complex and $p$-adic nearby cycles},
   JOURNAL = {Algebra Number Theory},
  FJOURNAL = {Algebra \& Number Theory},
    VOLUME = {20},
      YEAR = {2026},
    NUMBER = {1},
     PAGES = {17--108},
      ISSN = {1937-0652,1944-7833},
       DOI = {10.2140/ant.2026.20.17},
       URL = {https://doi.org/10.2140/ant.2026.20.17},
}

@article{abhinandan-imperfect-wach,
    AUTHOR = {Abhinandan},
     TITLE = {Crystalline representations and {W}ach modules in the
              imperfect residue field case},
   JOURNAL = {Doc. Math.},
  FJOURNAL = {Documenta Mathematica},
    VOLUME = {30},
      YEAR = {2025},
    NUMBER = {6},
     PAGES = {1461--1524},
      ISSN = {1431-0635,1431-0643},
   MRCLASS = {11S23 (14F20 14F30 14F40)},
  MRNUMBER = {4959787},
       DOI = {10.4171/dm/1015},
       URL = {https://doi.org/10.4171/dm/1015},
}

@article{abhinandan-prismatic-wach,
       author = {{Abhinandan}},
        title = "{Prismatic $F$-crystals and Wach modules}",
         year = 2024,
        month = may,
          eid = {arXiv:2405.18245},
}

@article {andre,
    AUTHOR = {Andr\'{e}, Yves},
     TITLE = {Diff\'{e}rentielles non commutatives et th\'{e}orie de {G}alois
              diff\'{e}rentielle ou aux diff\'{e}rences},
   JOURNAL = {Ann. Sci. \'{E}cole Norm. Sup. (4)},
  FJOURNAL = {Annales Scientifiques de l'\'{E}cole Normale Sup\'{e}rieure. Quatri\`eme
              S\'{e}rie},
    VOLUME = {34},
      YEAR = {2001},
    NUMBER = {5},
     PAGES = {685--739},
      ISSN = {0012-9593},
   MRCLASS = {12H05 (16E45 33D15 39A70)},
  MRNUMBER = {1862024},
MRREVIEWER = {Daniel Bertrand},
       DOI = {10.1016/S0012-9593(01)01074-6},
       URL = {https://doi.org/10.1016/S0012-9593(01)01074-6},
}

@article {andreatta-phigamma,
    AUTHOR = {Andreatta, Fabrizio},
     TITLE = {Generalized ring of norms and generalized
              {$(\phi,\Gamma)$}-modules},
   JOURNAL = {Ann. Sci. \'{E}cole Norm. Sup. (4)},
  FJOURNAL = {Annales Scientifiques de l'\'{E}cole Normale Sup\'{e}rieure. Quatri\`eme
              S\'{e}rie},
    VOLUME = {39},
      YEAR = {2006},
    NUMBER = {4},
     PAGES = {599--647},
      ISSN = {0012-9593},
   MRCLASS = {12J10 (13F30)},
  MRNUMBER = {2290139},
MRREVIEWER = {Alan Koch},
       DOI = {10.1016/j.ansens.2006.07.003},
       URL = {https://doi.org/10.1016/j.ansens.2006.07.003},
}

@incollection {andreatta-brinon,
    AUTHOR = {Andreatta, Fabrizio and Brinon, Olivier},
     TITLE = {Surconvergence des repr\'{e}sentations {$p$}-adiques: le cas
              relatif},
      NOTE = {Repr\'{e}sentations $p$-adiques de groupes $p$-adiques. I.
              Repr\'{e}sentations galoisiennes et $(\phi,\Gamma)$-modules},
   JOURNAL = {Ast\'{e}risque},
  FJOURNAL = {Ast\'{e}risque},
    NUMBER = {319},
      YEAR = {2008},
     PAGES = {39--116},
      ISSN = {0303-1179},
      ISBN = {978-2-85629-256-3},
   MRCLASS = {11S15 (11F80 11S20)},
  MRNUMBER = {2493216},
MRREVIEWER = {Alan Koch},
}

@incollection {andreatta-iovita-phigamma,
    AUTHOR = {Andreatta, Fabrizio and Iovita, Adrian},
     TITLE = {Global applications of relative {$(\phi,\Gamma)$}-modules.
              {I}},
      NOTE = {Repr\'esentations $p$-adiques de groupes $p$-adiques. I.
              Repr\'esentations galoisiennes et $(\phi,\Gamma)$-modules},
   JOURNAL = {Ast\'erisque},
  FJOURNAL = {Ast\'erisque},
    NUMBER = {319},
      YEAR = {2008},
     PAGES = {339--420},
      ISSN = {0303-1179,2492-5926},
      ISBN = {978-2-85629-256-3},
   MRCLASS = {14F20 (11S25 14F30)},
  MRNUMBER = {2493222},
MRREVIEWER = {Alan\ Koch},
}

@article {antieau-mathew-morrow-nikolaus,
    AUTHOR = {Antieau, Benjamin and Mathew, Akhil and Morrow, Matthew and
              Nikolaus, Thomas},
     TITLE = {On the {B}eilinson fiber square},
   JOURNAL = {Duke Math. J.},
  FJOURNAL = {Duke Mathematical Journal},
    VOLUME = {171},
      YEAR = {2022},
    NUMBER = {18},
     PAGES = {3707--3806},
      ISSN = {0012-7094,1547-7398},
   MRCLASS = {14F30 (14F40 19D55 19E15)},
  MRNUMBER = {4516307},
MRREVIEWER = {Guillermo\ Corti\~{n}as},
       DOI = {10.1215/00127094-2022-0037},
       URL = {https://doi.org/10.1215/00127094-2022-0037},
}

@article {berger-limites,
    AUTHOR = {Berger, Laurent},
     TITLE = {Limites de repr\'{e}sentations cristallines},
   JOURNAL = {Compos. Math.},
  FJOURNAL = {Compositio Mathematica},
    VOLUME = {140},
      YEAR = {2004},
    NUMBER = {6},
     PAGES = {1473--1498},
      ISSN = {0010-437X},
   MRCLASS = {11S25 (11F80 11R23 13K05 14F30)},
  MRNUMBER = {2098398},
       DOI = {10.1112/S0010437X04000879},
       URL = {https://doi.org/10.1112/S0010437X04000879},
}

@article{bhatt-lurie,
       author = {{Bhatt}, Bhargav and {Lurie}, Jacob},
        title = "{Absolute prismatic cohomology}",
         year = 2022,
        month = jan,
          eid = {arXiv:2201.06120},
}

@article {bhatt-mathew,
    AUTHOR = {Bhatt, Bhargav and Mathew, Akhil},
     TITLE = {Syntomic complexes and {$p$}-adic \'{e}tale {T}ate twists},
   JOURNAL = {Forum Math. Pi},
  FJOURNAL = {Forum of Mathematics. Pi},
    VOLUME = {11},
      YEAR = {2023},
     PAGES = {Paper No. e1, 26},
      ISSN = {2050-5086},
   MRCLASS = {14F30 (14F42)},
  MRNUMBER = {4530091},
       DOI = {10.1017/fmp.2022.21},
       URL = {https://doi.org/10.1017/fmp.2022.21},
}

@article {bhatt-scholze-prisms,
    AUTHOR = {Bhatt, Bhargav and Scholze, Peter},
     TITLE = {Prisms and prismatic cohomology},
   JOURNAL = {Ann. of Math. (2)},
  FJOURNAL = {Annals of Mathematics. Second Series},
    VOLUME = {196},
      YEAR = {2022},
    NUMBER = {3},
     PAGES = {1135--1275},
      ISSN = {0003-486X,1939-8980},
   MRCLASS = {14F30 (14F20 14F40)},
  MRNUMBER = {4502597},
       DOI = {10.4007/annals.2022.196.3.5},
       URL = {https://doi.org/10.4007/annals.2022.196.3.5},
}

@article{bhatt-scholze-crystals,
    AUTHOR = {Bhatt, Bhargav and Scholze, Peter},
     TITLE = {Prismatic {$F$}-crystals and crystalline {G}alois
              representations},
   JOURNAL = {Camb. J. Math.},
  FJOURNAL = {Cambridge Journal of Mathematics},
    VOLUME = {11},
      YEAR = {2023},
    NUMBER = {2},
     PAGES = {507--562},
      ISSN = {2168-0930,2168-0949},
   MRCLASS = {14 (11F80 12 13 16)},
  MRNUMBER = {4600546},
}

@article {bhatt-morrow-scholze-1,
    AUTHOR = {Bhatt, Bhargav and Morrow, Matthew and Scholze, Peter},
     TITLE = {Integral {$p$}-adic {H}odge theory},
   JOURNAL = {Publ. Math. Inst. Hautes \'{E}tudes Sci.},
  FJOURNAL = {Publications Math\'{e}matiques. Institut de Hautes \'{E}tudes
              Scientifiques},
    VOLUME = {128},
      YEAR = {2018},
     PAGES = {219--397},
      ISSN = {0073-8301},
   MRCLASS = {14F30},
  MRNUMBER = {3905467},
MRREVIEWER = {Daniel Robert Gulotta},
       DOI = {10.1007/s10240-019-00102-z},
       URL = {https://doi.org/10.1007/s10240-019-00102-z},
}

@article{bhatt-morrow-scholze-2,
    AUTHOR = {Bhatt, Bhargav and Morrow, Matthew and Scholze, Peter},
     TITLE = {Topological {H}ochschild homology and integral {$p$}-adic
              {H}odge theory},
   JOURNAL = {Publ. Math. Inst. Hautes \'{E}tudes Sci.},
  FJOURNAL = {Publications Math\'{e}matiques. Institut de Hautes \'{E}tudes
              Scientifiques},
    VOLUME = {129},
      YEAR = {2019},
     PAGES = {199--310},
      ISSN = {0073-8301},
   MRCLASS = {14F30 (13A35)},
  MRNUMBER = {3949030},
MRREVIEWER = {Lance Edward Miller},
       DOI = {10.1007/s10240-019-00106-9},
       URL = {https://doi.org/10.1007/s10240-019-00106-9},
}

@article {brinon-imparfait,
    AUTHOR = {Brinon, Olivier},
     TITLE = {Repr\'{e}sentations cristallines dans le cas d'un corps r\'{e}siduel
              imparfait},
   JOURNAL = {Ann. Inst. Fourier (Grenoble)},
  FJOURNAL = {Universit\'{e} de Grenoble. Annales de l'Institut Fourier},
    VOLUME = {56},
      YEAR = {2006},
    NUMBER = {4},
     PAGES = {919--999},
      ISSN = {0373-0956},
   MRCLASS = {11S20 (11F80 11F85 11S15 11S25 14F30)},
  MRNUMBER = {2266883},
MRREVIEWER = {Laurent N. Berger},
       URL = {http://aif.cedram.org/item?id=AIF_2006__56_4_919_0},
}

@article{brinon-relatif,
    AUTHOR = {Brinon, Olivier},
     TITLE = {Repr\'{e}sentations {$p$}-adiques cristallines et de de {R}ham dans le cas relatif},
   JOURNAL = {M\'{e}m. Soc. Math. Fr. (N.S.)},
  FJOURNAL = {M\'{e}moires de la Soci\'{e}t\'{e} Math\'{e}matique de France. Nouvelle S\'{e}rie},
    NUMBER = {112},
      YEAR = {2008},
     PAGES = {vi+159},
      ISSN = {0249-633X},
      ISBN = {978-2-85629-250-1},
   MRCLASS = {14F30 (11S25 11S80)},
  MRNUMBER = {2484979},
MRREVIEWER = {Laurent N. Berger},
       DOI = {10.24033/msmf.424},
       URL = {https://doi.org/10.24033/msmf.424},
}

@article {brinon-trihan,
    AUTHOR = {Brinon, Olivier and Trihan, Fabien},
     TITLE = {Repr\'{e}sentations cristallines et {$F$}-cristaux: le cas d'un
              corps r\'{e}siduel imparfait},
   JOURNAL = {Rend. Semin. Mat. Univ. Padova},
  FJOURNAL = {Rendiconti del Seminario Matematico della Universit\`a di
              Padova. Mathematical Journal of the University of Padua},
    VOLUME = {119},
      YEAR = {2008},
     PAGES = {141--171},
      ISSN = {0041-8994},
      ISBN = {978-88-7784-291-6},
   MRCLASS = {14F30 (11S20)},
  MRNUMBER = {2431507},
MRREVIEWER = {Adrian Vasiu},
       DOI = {10.4171/RSMUP/119-4},
       URL = {https://doi.org/10.4171/RSMUP/119-4},
}

@article {colmez-hauteur,
    AUTHOR = {Colmez, Pierre},
     TITLE = {Repr\'{e}sentations cristallines et repr\'{e}sentations de hauteur
              finie},
   JOURNAL = {J. Reine Angew. Math.},
  FJOURNAL = {Journal f\"{u}r die Reine und Angewandte Mathematik. [Crelle's
              Journal]},
    VOLUME = {514},
      YEAR = {1999},
     PAGES = {119--143},
      ISSN = {0075-4102},
   MRCLASS = {11S20 (11S25 14F30)},
  MRNUMBER = {1711279},
MRREVIEWER = {Abdellah Mokrane},
       DOI = {10.1515/crll.1999.068},
       URL = {https://doi.org/10.1515/crll.1999.068},
}

@article{colmez-niziol,
    AUTHOR = {Colmez, Pierre and Nizio{\l} , Wies{\l}awa},
     TITLE = {Syntomic complexes and {$p$}-adic nearby cycles},
   JOURNAL = {Invent. Math.},
  FJOURNAL = {Inventiones Mathematicae},
    VOLUME = {208},
      YEAR = {2017},
    NUMBER = {1},
     PAGES = {1--108},
      ISSN = {0020-9910},
   MRCLASS = {14F30 (11S25 14F20 14F40 14G20 14G22)},
  MRNUMBER = {3621832},
MRREVIEWER = {Adolfo Quir\'{o}s},
       DOI = {10.1007/s00222-016-0683-3},
       URL = {https://doi.org/10.1007/s00222-016-0683-3},
}

@article{du-liu-moon-shimizu,
    AUTHOR = {Du, Heng and Liu, Tong and Moon, Yong Suk and Shimizu, Koji},
     TITLE = {Completed prismatic {$F$}-crystals and crystalline
              {$Z_p$}-local systems},
   JOURNAL = {Compos. Math.},
  FJOURNAL = {Compositio Mathematica},
    VOLUME = {160},
      YEAR = {2024},
    NUMBER = {5},
     PAGES = {1101--1166},
      ISSN = {0010-437X,1570-5846},
   MRCLASS = {14F30 (11F80 14G45)},
  MRNUMBER = {4733770},
       DOI = {10.1112/S0010437X24007097},
       URL = {https://doi.org/10.1112/S0010437X24007097},
}

@book {fontaine-pdivisibles,
    AUTHOR = {Fontaine, Jean-Marc},
     TITLE = {Groupes {$p$}-divisibles sur les corps locaux.},
    SERIES = {},
 PUBLISHER = {Soci\'{e}t\'{e} Math\'{e}matique de France, Paris,, },
      YEAR = {1977},
     PAGES = {i+262},
   MRCLASS = {14L05},
  MRNUMBER = {498610},
MRREVIEWER = {Loren\ D.\ Olson},
}

@incollection{fontaine-messing,
	AUTHOR = {Fontaine, Jean-Marc and Messing, William},
	TITLE = {{$p$}-adic periods and {$p$}-adic \'{e}tale cohomology},
	BOOKTITLE = {Current trends in arithmetical algebraic geometry ({A}rcata, {C}alif., 1985)},
	SERIES = {Contemp. Math.},
	VOLUME = {67},
	PAGES = {179--207},
	PUBLISHER = {Amer. Math. Soc., Providence, RI},
	YEAR = {1987},
	MRCLASS = {14F30 (14F40 14G20)},
	MRNUMBER = {902593},
	MRREVIEWER = {T. Ekedahl},
	DOI = {10.1090/conm/067/902593},
	URL = {https://doi.org/10.1090/conm/067/902593},
}

@incollection {fontaine-phigamma,
    AUTHOR = {Fontaine, Jean-Marc},
     TITLE = {Repr\'{e}sentations {$p$}-adiques des corps locaux. {I}},
 BOOKTITLE = {The {G}rothendieck {F}estschrift, {V}ol. {II}},
    SERIES = {Progr. Math.},
    VOLUME = {87},
     PAGES = {249--309},
 PUBLISHER = {Birkh\"{a}user Boston, Boston, MA},
      YEAR = {1990},
   MRCLASS = {11S23 (14F30 14L05)},
  MRNUMBER = {1106901},
MRREVIEWER = {Rutger Noot},
}

@article{guo-reinecke,
    AUTHOR = {Guo, Haoyang and Reinecke, Emanuel},
     TITLE = {A prismatic approach to crystalline local systems},
   JOURNAL = {Invent. Math.},
  FJOURNAL = {Inventiones Mathematicae},
    VOLUME = {236},
      YEAR = {2024},
    NUMBER = {1},
     PAGES = {17--164},
      ISSN = {0020-9910,1432-1297},
   MRCLASS = {14F30},
  MRNUMBER = {4712864},
       DOI = {10.1007/s00222-024-01238-4},
       URL = {https://doi.org/10.1007/s00222-024-01238-4},
}

@article{hyodo,
    AUTHOR = {Hyodo, Osamu},
     TITLE = {On the {H}odge-{T}ate decomposition in the imperfect residue
              field case},
   JOURNAL = {J. Reine Angew. Math.},
  FJOURNAL = {Journal f\"{u}r die Reine und Angewandte Mathematik. [Crelle's
              Journal]},
    VOLUME = {365},
      YEAR = {1986},
     PAGES = {97--113},
      ISSN = {0075-4102},
   MRCLASS = {14L05 (11G10 11G25 11S25)},
  MRNUMBER = {826154},
MRREVIEWER = {Kazuya Kato},
       DOI = {10.1515/crll.1986.365.97},
       URL = {https://doi.org/10.1515/crll.1986.365.97},
}

@article {hyodo-hodge-tate,
    AUTHOR = {Hyodo, Osamu},
     TITLE = {On variation of {H}odge-{T}ate structures},
   JOURNAL = {Math. Ann.},
  FJOURNAL = {Mathematische Annalen},
    VOLUME = {284},
      YEAR = {1989},
    NUMBER = {1},
     PAGES = {7--22},
      ISSN = {0025-5831,1432-1807},
   MRCLASS = {14F30 (11G25 11S31 14G20)},
  MRNUMBER = {995378},
MRREVIEWER = {Gerd\ Faltings},
       DOI = {10.1007/BF01443501},
       URL = {https://doi.org/10.1007/BF01443501},
}

@article {joyal,
    AUTHOR = {Joyal, Andr\'{e}},
     TITLE = {{$\delta$}-anneaux et vecteurs de {W}itt},
   JOURNAL = {C. R. Math. Rep. Acad. Sci. Canada},
  FJOURNAL = {La Soci\'{e}t\'{e} Royale du Canada. L'Academie des Sciences.
              Comptes Rendus Math\'{e}matiques. (Mathematical Reports)},
    VOLUME = {7},
      YEAR = {1985},
    NUMBER = {3},
     PAGES = {177--182},
      ISSN = {0706-1994},
   MRCLASS = {13K05 (14F30 14L05 18C99 18F25 19D99)},
  MRNUMBER = {789309},
MRREVIEWER = {R.\ T.\ Hoobler},
}

@article {kisin-deformation-rings,
    AUTHOR = {Kisin, Mark},
     TITLE = {Potentially semi-stable deformation rings},
   JOURNAL = {J. Amer. Math. Soc.},
  FJOURNAL = {Journal of the American Mathematical Society},
    VOLUME = {21},
      YEAR = {2008},
    NUMBER = {2},
     PAGES = {513--546},
      ISSN = {0894-0347,1088-6834},
   MRCLASS = {11S20 (11F80)},
  MRNUMBER = {2373358},
MRREVIEWER = {Laurent\ N.\ Berger},
       DOI = {10.1090/S0894-0347-07-00576-0},
       URL = {https://doi.org/10.1090/S0894-0347-07-00576-0},
}

@article {liu-zhu,
    AUTHOR = {Liu, Ruochuan and Zhu, Xinwen},
     TITLE = {Rigidity and a {R}iemann-{H}ilbert correspondence for
              {$p$}-adic local systems},
   JOURNAL = {Invent. Math.},
  FJOURNAL = {Inventiones Mathematicae},
    VOLUME = {207},
      YEAR = {2017},
    NUMBER = {1},
     PAGES = {291--343},
      ISSN = {0020-9910,1432-1297},
   MRCLASS = {14G22 (14G35 14J60)},
  MRNUMBER = {3592758},
MRREVIEWER = {Marco\ A.\ Garuti},
       DOI = {10.1007/s00222-016-0671-7},
       URL = {https://doi.org/10.1007/s00222-016-0671-7},
}

@article {liu-semistable-lattices,
    AUTHOR = {Liu, Tong},
     TITLE = {A note on lattices in semi-stable representations},
   JOURNAL = {Math. Ann.},
  FJOURNAL = {Mathematische Annalen},
    VOLUME = {346},
      YEAR = {2010},
    NUMBER = {1},
     PAGES = {117--138},
      ISSN = {0025-5831,1432-1807},
   MRCLASS = {11S23 (11S20 14F20 14L05)},
  MRNUMBER = {2558890},
MRREVIEWER = {Fabrizio\ Andreatta},
       DOI = {10.1007/s00208-009-0392-y},
       URL = {https://doi.org/10.1007/s00208-009-0392-y},
}

@book {matsumura,
    AUTHOR = {Matsumura, Hideyuki},
     TITLE = {Commutative ring theory},
    SERIES = {Cambridge Studies in Advanced Mathematics},
    VOLUME = {8},
   EDITION = {Second},
      NOTE = {Translated from the Japanese by M. Reid},
 PUBLISHER = {Cambridge University Press, Cambridge},
      YEAR = {1989},
     PAGES = {xiv+320},
      ISBN = {0-521-36764-6},
   MRCLASS = {13-01},
  MRNUMBER = {1011461},
}

@article {moon,
    AUTHOR = {Moon, Yong Suk},
     TITLE = {A note on purity of crystalline local systems},
   JOURNAL = {Proc. Amer. Math. Soc.},
  FJOURNAL = {Proceedings of the American Mathematical Society},
    VOLUME = {152},
      YEAR = {2024},
    NUMBER = {12},
     PAGES = {5095--5103},
      ISSN = {0002-9939,1088-6826},
   MRCLASS = {11F80 (11F85 14F30)},
  MRNUMBER = {4855873},
MRREVIEWER = {Lei\ Yang},
       DOI = {10.1090/proc/16993},
       URL = {https://doi.org/10.1090/proc/16993},
}

@ARTICLE{morrow-tsuji,
       author = {{Morrow}, Matthew and {Tsuji}, Takeshi},
        title = "{Generalised representations as q-connections in integral $p$-adic Hodge theory}",
         year = 2020,
        month = oct,
          eid = {arXiv:2010.04059},
}

@article{scholze-q-deformations,
    AUTHOR = {Scholze, Peter},
     TITLE = {Canonical {$q$}-deformations in arithmetic geometry},
   JOURNAL = {Ann. Fac. Sci. Toulouse Math. (6)},
  FJOURNAL = {Annales de la Facult\'{e} des Sciences de Toulouse. Math\'{e}matiques. S\'{e}rie 6},
    VOLUME = {26},
      YEAR = {2017},
    NUMBER = {5},
     PAGES = {1163--1192},
      ISSN = {0240-2963},
   MRCLASS = {14F30 (11G25 12H10 14F05 14F40)},
  MRNUMBER = {3746625},
MRREVIEWER = {Nobuo Tsuzuki},
       DOI = {10.5802/afst.1563},
       URL = {https://doi.org/10.5802/afst.1563},
}

@article {scholze-perfectoid,
    AUTHOR = {Scholze, Peter},
     TITLE = {Perfectoid spaces},
   JOURNAL = {Publ. Math. Inst. Hautes \'Etudes Sci.},
  FJOURNAL = {Publications Math\'ematiques. Institut de Hautes \'Etudes
              Scientifiques},
    VOLUME = {116},
      YEAR = {2012},
     PAGES = {245--313},
      ISSN = {0073-8301,1618-1913},
   MRCLASS = {14G99},
  MRNUMBER = {3090258},
MRREVIEWER = {Jean-Marc\ Fontaine},
       DOI = {10.1007/s10240-012-0042-x},
       URL = {https://doi.org/10.1007/s10240-012-0042-x},
}

@article {scholze-rigid,
    AUTHOR = {Scholze, Peter},
     TITLE = {{$p$}-adic {H}odge theory for rigid-analytic varieties},
   JOURNAL = {Forum Math. Pi},
  FJOURNAL = {Forum of Mathematics. Pi},
    VOLUME = {1},
      YEAR = {2013},
     PAGES = {e1, 77},
      ISSN = {2050-5086},
   MRCLASS = {14G22 (14C30 14F30 14G20 32J27 32P05)},
  MRNUMBER = {3090230},
MRREVIEWER = {Hui\ June\ Zhu},
       DOI = {10.1017/fmp.2013.1},
       URL = {https://doi.org/10.1017/fmp.2013.1},
}

@misc{scholze-rigid-erratum,
    author = {Scholze, Peter},
     title = {Erratum to {$p$}-adic {H}odge theory for rigid-analytic varieties},
   howpublished = {\url{https://people.mpim-bonn.mpg.de/scholze/pAdicHodgeErratum.pdf}},
}

@misc{stacks-project,
  author       = {The {Stacks project authors}},
  title        = {The Stacks project},
  howpublished = {\url{https://stacks.math.columbia.edu}},
  year         = {2023},
}

@article {tsuji-cst,
    AUTHOR = {Tsuji, Takeshi},
     TITLE = {{$p$}-adic \'{e}tale cohomology and crystalline cohomology in the
              semi-stable reduction case},
   JOURNAL = {Invent. Math.},
  FJOURNAL = {Inventiones Mathematicae},
    VOLUME = {137},
      YEAR = {1999},
    NUMBER = {2},
     PAGES = {233--411},
      ISSN = {0020-9910},
   MRCLASS = {14F30 (14F20)},
  MRNUMBER = {1705837},
MRREVIEWER = {Abdellah Mokrane},
       DOI = {10.1007/s002220050330},
       URL = {https://doi.org/10.1007/s002220050330},
}

@article {tsuji-hodge-tate-purity,
    AUTHOR = {Tsuji, Takeshi},
     TITLE = {Purity for {H}odge-{T}ate representations},
   JOURNAL = {Math. Ann.},
  FJOURNAL = {Mathematische Annalen},
    VOLUME = {350},
      YEAR = {2011},
    NUMBER = {4},
     PAGES = {829--866},
      ISSN = {0025-5831,1432-1807},
   MRCLASS = {14F30 (11F80 11S20)},
  MRNUMBER = {2818716},
MRREVIEWER = {Kentaro\ Nakamura},
       DOI = {10.1007/s00208-010-0582-7},
       URL = {https://doi.org/10.1007/s00208-010-0582-7},
}

@article {tsuji-syntomic,
    AUTHOR = {Tsuji, Takeshi},
     TITLE = {Syntomic complexes and {$p$}-adic vanishing cycles},
   JOURNAL = {J. Reine Angew. Math.},
  FJOURNAL = {Journal f\"{u}r die Reine und Angewandte Mathematik. [Crelle's
              Journal]},
    VOLUME = {472},
      YEAR = {1996},
     PAGES = {69--138},
      ISSN = {0075-4102,1435-5345},
   MRCLASS = {14F30 (14G20)},
  MRNUMBER = {1384907},
MRREVIEWER = {Abdellah\ Mokrane},
       DOI = {10.1515/crll.1996.472.69},
       URL = {https://doi.org/10.1515/crll.1996.472.69},
}

@article {wach-free,
    AUTHOR = {Wach, Nathalie},
     TITLE = {Repr\'{e}sentations {$p$}-adiques potentiellement cristallines},
   JOURNAL = {Bull. Soc. Math. France},
  FJOURNAL = {Bulletin de la Soci\'{e}t\'{e} Math\'{e}matique de France},
    VOLUME = {124},
      YEAR = {1996},
    NUMBER = {3},
     PAGES = {375--400},
      ISSN = {0037-9484},
   MRCLASS = {11S23 (14F30 14L05)},
  MRNUMBER = {1415732},
MRREVIEWER = {Abdellah Mokrane},
       URL = {http://www.numdam.org/item?id=BSMF_1996__124_3_375_0},
}

@article {wach-torsion,
    AUTHOR = {Wach, Nathalie},
     TITLE = {Repr\'{e}sentations cristallines de torsion},
   JOURNAL = {Compositio Math.},
  FJOURNAL = {Compositio Mathematica},
    VOLUME = {108},
      YEAR = {1997},
    NUMBER = {2},
     PAGES = {185--240},
      ISSN = {0010-437X},
   MRCLASS = {11S23 (14F30 14L05)},
  MRNUMBER = {1468834},
MRREVIEWER = {Abdellah Mokrane},
       DOI = {10.1023/A:1000108818774},
       URL = {https://doi.org/10.1023/A:1000108818774},
}

@book{weibel,
    AUTHOR = {Weibel, Charles A.},
     TITLE = {An introduction to homological algebra},
    SERIES = {Cambridge Studies in Advanced Mathematics},
    VOLUME = {38},
 PUBLISHER = {Cambridge University Press, Cambridge},
      YEAR = {1994},
     PAGES = {xiv+450},
      ISBN = {0-521-43500-5; 0-521-55987-1},
   MRCLASS = {18-01 (16-01 17-01 20-01 55Uxx)},
  MRNUMBER = {1269324},
MRREVIEWER = {Kenneth A. Brown},
       DOI = {10.1017/CBO9781139644136},
       URL = {https://doi.org/10.1017/CBO9781139644136},
}

\Addresses

\end{document}